\documentclass[11pt]{amsart}

\usepackage{amsmath,amsthm}

\usepackage{amssymb}

\usepackage{pdfsync}

\usepackage{pstricks-add}
\usepackage{tikz}

\newcommand*\circled[1]{\tikz[baseline=(char.base)]{
            \node[shape=circle,draw,inner sep=2pt] (char) {#1};}}

\newcommand{\C}{{\mathbb C}}       
\newcommand{\R}{{\mathbb R}}       
\newcommand{\Z}{{\mathbb Z}}       

\newcommand{\DD}{{\mathcal D}}
\newcommand{\HH}{{\mathcal H}}
\newcommand{\LL}{{\mathcal L}}

\newcommand{\AZ}{{\mathcal A}}
\newcommand{\BZ}{{\mathcal B}}

\newcommand{\FF}{{\mathcal F}}

\newcommand{\SSS}{{\mathcal S}}

\newcommand{\MD}{{\mathcal {MD}}}

\newcommand{\CC}{{\mathcal C}}
\newcommand{\CE}{{{\mathcal C}_{\varepsilon}}}

\newcommand{\diam}{{\rm diam}}
\newcommand{\dist}{{\rm dist}}

\newcommand{\fiproof}{{\hspace*{\fill} $\square$ \vspace{2pt}}}

\newcommand{\ra}{\rightarrow}

\newcommand{\rf}[1]{{(\ref{#1})}}
\newcommand{\supp}{\operatorname{supp}}

\newcommand{\vphi}{{\varphi}}
\newcommand{\ve}{{\varepsilon}}
\newcommand{\vv}{{\vspace{2mm}}}
\newcommand{\vvv}{{\vspace{3mm}}}
\newcommand{\wt}[1]{{\widetilde{#1}}}
\newcommand{\wh}[1]{{\widehat{#1}}}

\newcommand{\meas}{{\measuredangle}}

\newcommand{\noi}{\noindent}
\newcommand{\rest}{{\lfloor}}
\newcommand{\lip}{{\rm Lip}}

\newcommand{\GG}{{\mathsf G}}
\newcommand{\LD}{{\mathsf{LD}}}
\newcommand{\HD}{{\mathsf{HD}}}

\newcommand{\BCF}{{\mathsf{BCF}}}
\newcommand{\BCG}{{\mathsf{BCG}}}
\newcommand{\BSD}{{\mathsf{BS}\Delta}}
\newcommand{\BSB}{{\mathsf{BS}\beta}}

\newcommand{\gex}{{\Gamma^k_{ex}}}

\newcommand{\DG}{{\mathcal D(\Gamma^k_{ex})}}
\newcommand{\DGG}{{(\Gamma^k_{ex})}}
\newcommand{\bsb}{{\BSB}}
\newcommand{\sss}{{\mathsf{Stop}}}

\newcommand{\ttt}{{\mathsf{Top}}}

\newcommand{\term}{{\mathsf{Term}}}
\newcommand{\good}{{\mathsf{Good}}}
\newcommand{\reg}{{\mathsf{Reg}}}
\newcommand{\nreg}{{\mathsf{NReg}}}
\newcommand{\tree}{{\mathsf{Tree}}}
\newcommand{\tr}{{\mathsf{Tr}}}

\newcommand{\nterm}{{\mathsf{NTerm}}}
\newcommand{\nqgood}{{\mathsf{NQgood}}}
\newcommand{\qgood}{{\mathsf{Qgood}}}
\newcommand{\ngood}{{\mathsf{NGood}}}
\newcommand{\eend}{{\mathsf{End}}}

\newtheorem{theorem}{Theorem}[section]
\newtheorem*{theorem*}{Theorem}
\newtheorem*{theorema*}{Theorem A}
\newtheorem*{theoremb*}{Theorem B}
\newtheorem*{theoremc*}{Theorem C}
\newtheorem*{mainlemma*}{Main Lemma}
\newtheorem{mainlemma}[theorem]{Main Lemma}
\newtheorem{lemma}[theorem]{Lemma}

\newtheorem{coro}[theorem]{Corollary}
\newtheorem{propo}[theorem]{Proposition}
\newtheorem{claim}[theorem]{Claim}
\theoremstyle{definition}

\theoremstyle{remark}
\newtheorem{remark}[theorem]{\bf Remark}

\numberwithin{equation}{section}

\begin{document}

\begin{abstract}
This paper is devoted to the proof of two related results. The first one asserts that if $\mu$ is a Radon measure in $\R^d$ satisfying 
$$\limsup_{r\to 0} \frac{\mu(B(x,r))}{r}>0\quad \text{ and }\quad
\int_0^1\left|\frac{\mu(B(x,r))}{r} - \frac{\mu(B(x,2r))}{2r}\right|^2\,\frac{dr}r< \infty$$
for $\mu$-a.e.\ $x\in\R^d$, then $\mu$ is rectifiable.
Since the converse implication is already known to hold, this yields the following characterization of
rectifiable sets: a set $E\subset\R^d$ with finite $1$-dimensional Hausdorff measure $\HH^1$ is rectifiable if and only if
$$\int_0^1\left|\frac{\HH^1(E\cap B(x,r))}{r} - \frac{\HH^1(E\cap B(x,2r))}{2r}\right|^2\,\frac{dr}r< \infty
\qquad\mbox{\!
for $\HH^1$-a.e.\ $x\in E$.}$$

The second result of the paper deals with the relationship between the above square function in the complex plane and the Cauchy transform $\CC_\mu f(z) = \int \frac1{z-\xi}\,f(\xi)\,d\mu(\xi)$. Suppose that $\mu$ has  linear growth, that is, 
$\mu(B(z,r))\leq c\,r$ for all $z\in\C$ and all $r>0$.
It is proved that $\CC_\mu$ is bounded in $L^2(\mu)$ if and only if
$$
\int_{z\in Q}\int_0^\infty\left|\frac{\mu(Q\cap B(z,r))}{r} - \frac{\mu(Q\cap B(z,2r))}{2r}\right|^2\,\frac{dr}r\,d\mu(z)\leq
c\,\mu(Q)
\qquad\mbox{\!\!\!for every square $Q\subset\C$.}
$$
\end{abstract}

\title[Rectifiable measures, square functions, and the Cauchy transform]{Rectifiable measures, square functions involving densities, and the Cauchy transform}

\author{Xavier Tolsa}
\address{Xavier Tolsa. Instituci\'{o} Catalana de Recerca i Estudis Avan\c{c}ats (ICREA) and Departament de Ma\-te\-m\`a\-ti\-ques, Universitat Aut\`onoma de Bar\-ce\-lo\-na, Catalonia}
\email{xtolsa@mat.uab.cat}

\thanks{Partially supported by 
the ERC grant 320501 of the European Research
Council (FP7/2007-2013) 
and by the grants
2014-SGR-75 (Catalonia) and MTM-2010-16232, MTM2013-44304-P (Spain). }

\maketitle
\tableofcontents

\section{Introduction}

A set $E\subset \R^d$ is called $n$-rectifiable if there are Lipschitz maps
$f_i:\R^n\to\R^d$, $i=1,2,\ldots$, such that 
\begin{equation}\label{eq001}
\HH^n\biggl(E\setminus\bigcup_i f_i(\R^n)\biggr) = 0,
\end{equation}
where $\HH^n$ stands for the $n$-dimensional Hausdorff measure. 
Also, one says that 
a Radon measure $\mu$ on $\R^d$ is $n$-rectifiable if $\mu$ vanishes out of an $n$-rectifiable
set $E\subset\R^d$ and moreover $\mu$ is absolutely continuous with respect to $\HH^n|_E$.
On the other hand, $E$ is called purely $n$-unrectifiable if $\HH^n(F\cap E)=0$
for any $n$-rectifiable set $F\subset\R^d$.
In the case $n=1$, instead of saying that a set or a measure is $1$-rectifiable, one just says 
that it is rectifiable.

One of the main objectives of geometric measure theory consists in characterizing $n$-rectifiable
sets and measures in different ways. For instance, there are characterizations in terms of the
almost everywhere existence of approximate tangent planes, in terms of the size of projections on
$n$-planes, and in terms of the existence of densities.
To describe the latter characterization in detail, we need to introduce some terminology.

Given a Radon measure $\mu$ and $x\in\R^d$ we denote
$$\Theta^{n,*}(x,\mu)=\limsup_{r\to0}\frac{\mu(B(x,r))}{(2r)^n},
\qquad 
\Theta^{n}_*(x,\mu)=\liminf_{r\to0}\frac{\mu(B(x,r))}{(2r)^n}.
$$
These are the upper and lower $n$-dimensional densities of $\mu$ at $x$. If they coincide, 
they are denoted by $\Theta^{n}(x,\mu)$ and this is called the $n$-dimensional densities of $\mu$ at $x$. In the case when $\mu=\HH^n|_E$ for some set
$E\subset\R^d$, we also write $\Theta^{n,*}(x,E)$, $\Theta^{n}_*(x,E)$, $\Theta^{n}(x,E)$
instead of $\Theta^{n,*}(x,\HH^n|_E)$, $\Theta^{n}_*(x,\HH^n|_E)$, $\Theta^{n}(x,\HH^n|_E)$,
respectively.

The following result is due to Besicovitch for $n=1$, $d=2$, to Marstrand 
 \cite{Marstrand1} for $n=2$, $d=3$, and to Mattila \cite{Mattila-tams} for arbitrary $n,d$.

\begin{theorema*}
Let $n$ be a positive integer and let $E\subset\R^d$ be $\HH^n$-measurable with $\HH^n(E)
<\infty$. We have:
\begin{itemize}
\item[(a)] $E$ is $n$-rectifiable if and only if $\Theta^{n}(x,E)$ exists and equals $1$ 
for $\HH^n$-a.e.\ $x\in E$.
\item[(b)] $E$ is purely $n$-unrectifiable if and only if $\Theta^{n}_*(x,E)<1$
for $\HH^n$-a.e.\  $x\in E$.
\end{itemize}
\end{theorema*}

Another fundamental result concerning the relationship between rectifiability and densities is 
given by the following celebrated theorem of Preiss \cite{Preiss}.

\begin{theoremb*}
A Radon measure $\mu$ in $\R^d$ is 
 $n$-rectificable if and only if the density
$\Theta^{n}(x,\mu)$
exists and is non-zero for $\mu$-a.e.\ $x\in\R^d$.
\end{theoremb*}

In particular, for $\mu=\HH^n|_E$ with $\HH^n(E)<\infty$, the preceding theorem ensures the $n$-rectifiability of $E$ just assuming that the density $\Theta^{n}(x,E)$ exists and is non-zero for
$\HH^n$-a.e.\ $x\in E$. 

Quite recently, in the works \cite{CGLT} and \cite{TT}, the authors have obtained some results which can be considered
as square function versions of Preiss theorem. In particular, in \cite{TT} the following is proved:

\begin{theoremc*}
Let $\mu$ be a Radon measure in $\R^d$ such that $0< \Theta_*^n(x,\mu)\leq \Theta^{n,*}(x,\mu)<\infty$
for $\mu$-a.e.\ $x\in\R^d$. Then $\mu$ is $n$-rectifiable if and only if 
\begin{equation}\label{eqahj37}
\int_0^1 \left|\frac{\mu(B(x,r))}{r^n} - \frac{\mu(B(x,2r))}{(2r)^n}\right|^2\,\frac{dr}r< \infty
\quad\mbox{ for $\mu$-a.e.\ $x\in\R^d$.}
\end{equation}
\end{theoremc*}

This theorem was preceded by the proof of a related result 
in \cite{CGLT} which characterizes the so called uniform $n$-rectifiability in terms of a square function similar to the one in \rf{eqahj37}. See the next section for the precise definition
of uniform rectifiability and the statement of this result.

A natural question is if the condition \rf{eqahj37} above implies the $n$-rectifiability of $E$
just under the assumption that $0< \Theta^{n,*}(x,\mu)<\infty$ $\mu$-a.e. If this were true, 
then we would deduce that a set $E\subset\R^d$ with $\HH^n(E)<\infty$ is $n$-rectifiable if 
and only if
$$
\int_0^1 \left|\frac{\HH^n(E\cap B(x,r))}{r^n} - \frac{\HH^n(E\cap B(x,2r))}{(2r)^n}\right|^2\,\frac{dr}r< \infty
\quad \mbox{ for $\HH^n$-a.e.\ $x\in E$.}
$$
The arguments used in \cite{TT} make an essential use of the assumption that the lower density
$\Theta_*(x,\mu)$ is positive. So different techniques are required if one wants to extend Theorem 
C to the case of vanishing lower density.
In the present paper we solve this  problem in the case $n=1$:
\vv

\begin{theorem}\label{teomain}
Let $\mu$ be a Radon measure in $\R^d$ such that $\Theta^{1,*}(x,\mu)>0$
for $\mu$-a.e.\ $x\in\R^d$. Then $\mu$ is rectifiable if and only if
\begin{equation}\label{eqasjr11}
\int_0^1\left|\frac{\mu(B(x,r))}{r} - \frac{\mu(B(x,2r))}{2r}\right|^2\,\frac{dr}r< \infty\quad
\mbox{\;for $\mu$-a.e.\ $x\in\R^d$.}
\end{equation}
\end{theorem}
\vv

\begin{coro}\label{coromain}
Let $E\subset\R^d$ be a Borel set with $\HH^1(E)<\infty$. The set $E$ is rectifiable if and only if 
$$
\int_0^1\left|\frac{\HH^1(E\cap B(x,r))}{r} - \frac{\HH^1(E\cap B(x,2r))}{2r}\right|^2\,\frac{dr}r< \infty
\quad \mbox{\;for $\HH^1$-a.e.\ $x\in E$.}
$$
\end{coro}
\vv

I do not  know if the analogous result in the case $n>1$ holds.

Note that the ``only if'' part of Theorem \ref{teomain} is an immediate consequence of Theorem C above.
Indeed, if $\mu$ is rectifiable, then it follows easily that $0< \Theta_*^n(x,\mu)\leq \Theta^{n,*}(x,\mu)<\infty$
for $\mu$-a.e.\ $x\in\R^d$. So the assumptions of Theorem C are fulfilled and thus \rf{eqahj37} holds.

In the present paper we prove the ``if'' implication of Theorem \ref{teomain}.
This combines a compactness argument which originates from
\cite{CGLT} and constructive techniques involving stopping time conditions. One of the main
difficulties, which is absent in \cite{TT}, consists in controlling the oscillations of the
densities $\frac{\mu(B(x,r))}r$ as $r\to0$. If the power in the integrand of \rf{eqasjr11} were 
$1$ instead of $2$, then this task would be significantly easier, and we could argue as in \cite{TT} or as in \cite{ADT}.

In our arguments, a basic tool for the control of such oscillations of the density is the construction of suitable measures
$\sigma^k$ supported on some approximating curves $\Gamma^k$ so that, for each $k$, $\sigma^k$ has linear growth with some absolute constant and such that the $L^2(\sigma^k)$ norm of a smooth version of the square function 
 in \rf{eqasjr11}, with $\mu$ replaced by $\sigma^k$, is very small. The main obstacle to extend Theorem \ref{teomain}
to higher dimensions lies in the difficulty to extend this construction to the case $n>1$.
\vv

In the final part of this paper we prove a striking connection between the boundedness in $L^2(\mu)$ of the square function
$$T\mu(x) = \left(\int_0^\infty\left|\frac{\mu(B(x,r))}{r} - \frac{\mu(B(x,2r))}{2r}\right|^2\,\frac{dr}r\right)^{1/2}$$
and the $L^2(\mu)$ boundedness of the Cauchy transform. Recall that given a complex Radon measure $\nu$ on $\C$,
its Cauchy transform is defined by
$$\CC\nu(z) = \int \frac{1}{z-\xi}\,d\nu(\xi),$$
whenever the integral makes sense. 
For $\ve>0$, the $\ve$-truncated Cauchy transform of $\nu$ is given by
$$\CE\nu(z) = \int_{|z-\xi|>\ve} \frac{1}{z-\xi}\,d\nu(\xi).$$
Note that the last integral is absolutely convergent for all $z\in\C$, unlike 
 the integral defining $\CC\nu(z)$, in general. Given $f\in L^p(\mu)$, one denotes
$\CC_\mu f = \CC(f\,\mu)$ and $\CC_{\mu,\ve} f = \CE(f\,\mu)$. One says that $\CC_\mu$ is bounded in
$L^p(\mu)$ if and only if the operators $\CC_{\mu,\ve}$ are bounded in $L^p(\mu)$ uniformly on $\ve>0$.

 In the particular case when $\mu=\HH^1|_E$ with $\HH^1(E)<\infty$, by the
theorem of David-L\'eger \cite{Leger}, the $L^2(\mu)$ boundedness of $\CC_\mu$ implies the rectifiability of $E$. So it is natural to expect some relationship between the behaviors of the Cauchy
transform of $\mu$ and of the square function $T\mu$.
The next theorem, which is the second main result of this paper,
 shows that indeed there is a very strong and precise connection between the $L^2(\mu)$ boundedness
of $\CC_\mu$ and the $L^2(\mu)$ behavior of $T\mu$ for arbitrary measures $\mu$ with linear growth.

\begin{theorem}\label{teocauchy}
Let $\mu$ be a finite Radon measure in $\C$ satisfying the linear growth condition
$$\mu(B(x,r))\leq c\,r\qquad\mbox{for all $x\in\C$ and all $r>0$.}$$
 The Cauchy transform $\CC_\mu$ is bounded in $L^2(\mu)$ if and only if
\begin{equation}\label{eqsq4982}
\int_{x\in Q}\int_0^\infty\left|\frac{\mu(Q\cap B(x,r))}{r} - \frac{\mu(Q\cap B(x,2r))}{2r}\right|^2\,\frac{dr}r\,d\mu(x)\leq
c\,\mu(Q)
\qquad\mbox{\!\!\!for every square $Q\subset\C$.}
\end{equation}
\end{theorem}

The behavior of the square function $T\mu$ is related to the cancellation properties of
the densities $\frac{\mu(B(x,r))}r$, $x\in\C$, $r>0$. On the other hand, heuristically the $L^2(\mu)$ boundedness of $\CC_\mu$
seems to be more connected to the behavior of the approximate tangents to $\mu$. 
So it is quite remarkable (to the author's point of view) that the behavior of $T\mu$ is so strongly connected
to the $L^2(\mu)$ boundedness of $\CC_\mu$, as shown in the preceding theorem.

The proof of Theorem \ref{teocauchy} uses a corona decomposition analogous to the one of \cite{Tolsa-bilip}.
We will see in this paper that, loosely speaking, the condition \rf{eqsq4982} is equivalent to the
existence of a corona decomposition such as the one mentioned above, which in turn is equivalent to the
$L^2(\mu)$ boundedness of the Cauchy transform because of the results of \cite{Tolsa-bilip}.
\vv

The plan of the paper is the following. In Section \ref{secprelim} we introduce some notation and terminology and we review some results which will be needed later.
Section \ref{sec:bddur} contains a blow up argument which, roughly speaking, shows that, given a ball
$B(x_0,r_0)$, if 
$$\int_{\delta r_0}^{\delta^{-1}r_0} \left|\frac{\mu(B(x,r))}{r} - \frac{\mu(B(x,2r))}{2r}\right|^2\,\frac{dr}r
$$
is very small for a big proportion in measure $\mu$ of points $x\in B(x_0,\delta^{-1}r_0)$, with $\delta>0$ sufficiently small, then
the measure $\mu$ is close to a flat measure in $B(x_0,r_0)$. The argument is quite similar to the one
used for the AD-regular case in \cite{CGLT} (see Section \ref{secprelim} for the definition of AD-regular measures). Next, in Section \ref{sec4} we review the construction of the 
dyadic cells from David-Mattila \cite{David-Mattila}, which will be very useful for the proof of
Theorem \ref{teomain}.

In Section \ref{secmlemma} we state the Main Lemma \ref{mainlemma}. 
In a sense, this lemma asserts, in some quantitative way, that given a doubling dyadic cell $R$
with side length $\ell(R)$,
if $$\int_0^{\delta^{-1}\ell(R)} \left|\frac{\mu(B(x,r))}{r} - \frac{\mu(B(x,2r))}{2r}\right|^2\,\frac{dr}r$$
is very small for a big proportion in $\mu$-measure of the points $x$ near $R$ and $\delta$ is small enough, then either a big proportion of the measure $\mu|_R$ 
is concentrated on an AD-regular curve, or $\frac{\mu(B(x,r))}r\ll \frac{\mu(B(x,\ell(R)))}{\ell(R)}$ for many points $x\in\supp\mu$ and some $r=r(x)\leq \ell(R)$.
In the same section, we show how Theorem \ref{teomain} follows from the Main Lemma \ref{mainlemma} by means of a suitable
corona type decomposition. 

Sections \ref{sec7}-\ref{sec14} are devoted to the proof of the Main Lemma. In Sections \ref{sec7}-\ref{sec9*}
we introduce some stopping cells and an auxiliary measure $\wt\mu$ and we prove some related results.
In Section \ref{sec88} we construct some AD-regular curves $\Gamma^k$ and
in Section \ref{sec10} we construct measures $\nu^k$ supported on $\Gamma^k$ which, in a sense, approximate $\mu$.
Section \ref{sec12} deals with the construction of the aforementioned auxiliary measures $\sigma^k$, which are supported 
on $\Gamma^k$. In this section we also obtain some suitable square function estimates involving $\sigma^k$, which will be used in the subsequent section to estimate the $L^2(\sigma^k)$ norm of the density
of $\nu^k$ with respect to $\sigma^k$. This is the main ingredient used in Section \ref{sec14} to show that there are very few stopping cells of high density, and to finish the proof of the Main Lemma.

Sections \ref{secauchy1}-\ref{secauchy3} deal with the proof of
Theorem \ref{teocauchy}. By means of the Main Lemma \ref{mainlemma}, in Section \ref{secauchy1} it is shown that if the 
condition \rf{eqsq4982} holds, then one can construct a corona type decomposition for $\mu$ analogous to
the one of \cite{Tolsa-bilip}, which suffices to show that the Cauchy transform is bounded in $L^2(\mu)$.
In the subsequent section, some Calder\'on-Zygmund type results are obtained for the square function operator $T_\mu f :=
T(f\mu)$, which  will be necessary later to show the remaining implication of Theorem \ref{teocauchy}, namely
that the $L^2(\mu)$ boundedness of $\CC_\mu$ implies \rf{eqsq4982}.
This is proved in the final Section \ref{secauchy3} of this paper, relying on the
corona type decomposition for $\mu$ constructed in \cite{Tolsa-bilip}.

\vvv


\section{Preliminaries}\label{secprelim}

In this paper the letters $c,C$ stand
for some constants which may change their values at different
occurrences. On the other hand, constants with subscripts, such as $c_1$, do not change their values
at different occurrences.
The notation $A\lesssim B$ means that
there is some fixed constant $c$ (usually an absolute constant) such that $A\leq c\,B$. Further, $A\approx B$ is equivalent to $A\lesssim B\lesssim A$. We will
also write $A\approx_{c_1,c_2} B$ and  $A\lesssim_{c_1,c_2} B$ if we want to make explicit 
the dependence on the constants $c_1$ and $c_2$ of the relationships 
``$\approx$'' and ``$\lesssim$''. 

Given two lines $L_1,L_2\subset\R^d$, $\meas(L_1,L_2)$ stands for the (smallest) angle that form $L_1$ and $L_2$.
Also, given $x_1,x_2,x_3\in\R^d$, $\meas(x_1,x_2,x_3)$ is the angle with vertex $x_2$ and sides equal to the
segments $\overline{x_2,x_1}$ and $\overline{x_2,x_3}$.\vv

\subsection{AD-regular and uniformly rectifiable measures}

A measure $\mu$ is called $n$-AD-regular (or just AD-regular or Ahlfors-David regular) if there exists some
constant $c_0>0$ such that
$$c_0^{-1}r^n\leq \mu(B(x,r))\leq c_0\,r^n\quad \mbox{ for all $x\in
\supp(\mu)$ and $0<r\leq \diam(\supp(\mu))$.}$$

A measure $\mu$ is  uniformly  $n$-rectifiable if it is 
$n$-AD-regular and
there exist $\theta, M >0$ such that for all $x \in \supp(\mu)$ and all $r>0$ 
there is a Lipschitz mapping $g$ from the ball $B_n(0,r)$ in $\R^{n}$ to $\R^d$ with $\text{Lip}(g) \leq M$ such that$$
\mu (B(x,r)\cap g(B_{n}(0,r)))\geq \theta r^{n}.$$
In the case $n=1$, $\mu$ is uniformly $1$-rectifiable if and only if $\supp(\mu)$ is contained in a rectifiable curve $\Gamma$ in $\R^d$ such that the arc length measure on $\Gamma$ is $1$-AD-regular.

A set $E\subset\R^d$ is called $n$-AD-regular if $\HH^n|_E$ is $n$-AD-regular, and it is called
uniformly $n$-rectifiable if $\HH^n|_E$ is uniformly  $n$-rectifiable.

The notion of uniform rectifiability was introduced by David and Semmes \cite{DS1}, \cite{DS2}.
In these works they showed that a big class of singular singular integrals with odd kernel
is bounded in $L^2(\mu)$ if $\mu$ is uniformly $n$-rectifiable. See \cite{ntov} for a recent related result
in the converse direction involving the $n$-dimensional Riesz transform.

In \cite{CGLT} it is shown that uniform $n$-rectifiability can be characterized as follows.

\begin{theorem}
Let $\mu$ be an $n$-AD-regular measure. Then $\mu$ is uniformly $n$-rectifiable if and only if there exists a constant $c$ such that, for any ball $B(x_0,R)$ centered at $\supp(\mu)$,
$$
\int_{x\in B(x_0,R)}\int_0^R  \left|\frac{\mu(B(x,r))}{r} - \frac{\mu(B(x,2r))}{2r}\right|^2\,\frac{dr}r\,d\mu(x) \leq c\, R^n.
$$
\end{theorem}

\vv

\subsection{The $\beta$ and $\alpha$ coefficients}

Given a closed or open ball $B\subset\R^d$, we set
$$\beta_{1,\mu}(B) = \inf_L \frac1{\mu(B)}\int_{B} \frac{\dist(y,L)}{r(B)}\,d\mu(y),$$
where $r(B)$ stands for the radius of $B$ and the infimum is taken over all the lines $L$. 
The $L^\infty$ version is the following:
$$\beta_{\infty,\mu}(B) = \inf_L \sup_{y\in B\cap\supp\mu} \frac{\dist(y,L)}{r(B)}.$$

The analogous bilateral coefficients are defined by
$$b\beta_{1,\mu}(B) = \inf_L \biggl[ \frac1{\mu(B)}\int_{B} \frac{\dist(y,L)}{r(B)}\,d\mu(y) + 
\frac1{r(B)}\int_{L\cap B}\!\! \frac{\dist(x,\supp\mu)}{r(B)}\,d\HH^1(x) \biggr],$$
and
$$b\beta_{\infty,\mu}(B) = \inf_L \biggl[ \sup_{y\in B\cap\supp\mu} \frac{\dist(y,L)}{r(B)} + 
\sup_{y\in L\cap B}\!\! \frac{\dist(x,\supp\mu)}{r(B)} \biggr].$$

Below we will use the so called $\alpha$ coefficients from \cite{Tolsa-plms}. They are defined as follows.
Given a closed or open ball $B\subset\R^d$ which intersects $\supp(\mu)$,
 and two Radon measures $\sigma$ and $\nu$
in $\R^d$, we set
$$\dist_B(\sigma,\nu):= \sup\Bigl\{ \Bigl|{\textstyle \int f\,d\sigma  -
\int f\,d\nu}\Bigr|:\,{\rm Lip}(f) \leq1,\,\supp f\subset
B\Bigr\},$$
where ${\rm Lip}(f)$ stands for the Lipschitz constant of $f$.
It is easy to check that this is indeed a distance in the space of finite Borel measures supported in the interior of 
$B$. See [Chapter 14, Ma] for other properties of this distance. In fact, this is a variant of the well known Wasserstein distance $W_1$ from mass transport.
Given a subset $\AZ$ of Radon measures in $\R^d$, we set
$$\dist_B(\mu,\AZ) := \inf_{\sigma\in\AZ}\dist_B(\mu,\,\sigma).$$
We define
$$
\alpha_\mu(B) := \frac1{r(B)\,\mu(B)}\,\inf_{c\geq0,L} \,\dist_{B}(\mu,\,c\HH^1|_{L}),$$
where the infimum is taken over all the constants $c\geq0$ and all the lines $L$.  Also, we denote by $c_B$ and $L_B$ a constant
and a line that minimize $\dist_{B_Q}(\mu,\,c\HH^1|_L)$, respectively (it is easy to check that this minimum is attained). 
We also write 
$\LL_B:=c_B\HH^1|_{L_B}$, so that
$$\alpha_\mu(B) = \frac1{r(B)\mu(B)}\,\dist_{B}(\mu,\,\LL_B).$$
Let us remark that $c_B$ and $L_B$ (and so $\LL_B$) may be not unique. Moreover, we may (and will) assume that 
$L_B\cap \bar B\neq\varnothing$.

\begin{lemma} \label{lempr0}
Let $B,B'\subset \R^d$ be two balls. The coefficients $\alpha_\mu(\cdot)$ satisfy the following properties:
\begin{itemize}
\item[(a)] $\alpha_\mu(B) \lesssim 1$.
\item[(b)] If $B\subset B'$, $r(B)\approx r(B')$, and $\mu(B)\approx \mu(B')$, then $\alpha_\mu(B) \lesssim\alpha_\mu(B')$.
\item[(c)] If $\mu(\frac14B)\approx \mu(B)$ and $\alpha_\mu(B)\leq c_1$, where $c_1$ is some constant small enough, then $L_B\cap \frac12B \neq\varnothing$ and 
$c_B\approx\frac{\mu(B)}{r(B)}$.
\end{itemize}
\end{lemma}

\begin{proof}
The statements (a) and (b) are direct consequences of the definitions. 

Let us turn our attention to (c). To show that $L_B\cap \frac12B\neq\varnothing$ if $c_1$ is small enough, take a smooth function function 
$\vphi$ such that $\chi_{\frac14 B}\leq\varphi\leq \chi_{\frac12B}$ with
$\|\nabla\vphi\|_\infty\lesssim 1/r(B)$. Then we have $\|\nabla(\vphi\,\dist(\cdot,L_B))\|_\infty\lesssim 1$, and since
$\vphi\,\dist(\cdot,L_B)$ vanishes on $L_B$, we have
$$\biggl|\int\vphi(x)\dist(x,L_B)\,d\mu(x)\biggr| = \biggl|\int\vphi(x)\dist(x,L_B)\,d(\mu-\LL_B)(x)\biggr| \lesssim \alpha_\mu(B)\,r(B)\,\mu(B).$$
On the other hand, 
\begin{align*}
\int\vphi(x)\dist(x,L_B)\,d\mu(x) &\geq \dist(\supp(\vphi),\,L_B) \int\vphi\,d\mu \\
& \gtrsim \dist(\supp(\vphi),\,L_B) \,\mu(\tfrac14B)\\ &\approx \dist(\supp(\vphi),\,L_B) \,\mu(B).
\end{align*}
If $\alpha_\mu(B)$ is small enough we infer that
$\dist(\supp(\vphi),L_B) \leq r(B)/10$, and so $L_B\cap \tfrac12 B \neq\varnothing.$

Let us check now that $c_B\approx\frac{\mu(B)}{r(B)}$.
Let $\psi$ be a smooth function such that $\chi_{\frac12 B} \leq \psi\leq \chi_{B}$ 
and $\|\nabla\psi\|_\infty\lesssim 1/r(B)$. Then
$$\biggl|\int\psi\,d\mu - \int\psi\,d\LL_{B}\biggr| \lesssim\alpha_\mu(B)\mu(B).$$
Thus, 
$$\int \psi\,d\mu - C\alpha_\mu(B)\mu(B) \leq c_B\int\psi\,d\HH^1|_{L_{B}} \leq \int \psi\,d\mu + C\alpha_\mu(B)\mu(B).$$
From the second inequality, we deduce easily that $c_B\lesssim\frac{\mu(B)}{r(B)}$. From the first one, we see that if $\alpha_\mu(B) \leq c_1$, where $c_1$ is 
small enough, then
$$ c_B\int\psi\,d\HH^1|_{L_{B}}  \geq \frac12\,\mu(B) - C\alpha_\mu(B)\mu(B) \geq \frac14 \mu(B),$$
which implies that $c_B\gtrsim\frac{\mu(B)}{r(B)}$.
\end{proof}

We have the following relationship between $\beta_{1,\mu}(B)$, $b\beta_{1,\mu}(B)$ and $\alpha_\mu(B)$:

\begin{lemma} \label{lempr1}
Let $B\subset \R^d$ be a ball such that $\mu(\frac12B)\approx \mu(2B)$. Then we have
$$\beta_{1,\mu}(B)\leq b\beta_{1,\mu}(B)\lesssim\alpha_\mu(2B).$$
In fact,
$$\int_{B} \frac{\dist(y,L_B)}{r(B)\mu(B)}\,d\mu(y) + 
\int_{L_B\cap B}\!\! \frac{\dist(x,\supp\mu)}{r(B)^2}\,d\HH^1|_{L_B}(x)\lesssim \alpha_\mu(2B).$$
\end{lemma}

This result has been proved in the case that $\mu$ is AD-regular in \cite{Tolsa-plms}. Almost the same arguments work in the present situation. 



\begin{lemma}\label{lempr2}
Let $B,B'\subset \R^d$ be balls with $B\subset B'$ which satisfy $\mu(\frac14 B)\approx\mu(\frac14 B')
\approx\mu(B')$, with
$r(B)\approx r(B')$. Then we have
\begin{equation}\label{dh1}
\dist_H\bigl(L_B\cap B' , \, L_{B'}\cap B') \leq C\alpha_\mu(B')\,r(B'),
\end{equation}
where $\dist_H$ stands for the Hausdorff distance. Also,
\begin{equation} \label{dh2}
|c_B - c_{B'}| \leq C\alpha_\mu(B')\,\frac{\mu(B')}{r(B')}.
\end{equation}
\end{lemma}

This result has also been proved for $\mu$ being AD-regular in \cite{Tolsa-plms}, and again the
same arguments are valid in the present situation.

\vvv


\section{A compactness argument} \label{sec:bddur}

Let $\FF$ denote the family of $1$-flat measures, that is, the family of measures $\sigma$ of the form
$$\sigma= c\,\HH^1|_L,$$
where $L$ is a line and $c>0$.
 Given $x\in\R^d$ and $r>0$, we denote
$$\Delta_\mu(x,r)= \left|\frac{\mu(B(x,r))}{r} - \frac{\mu(B(x,2r))}{2r}\right|,$$
and for a ball $B\subset\R^d$,
$$\Theta_\mu(B)=\frac{\mu(B)}{r(B)}.$$
The main objective of this section is to prove the following result:

\begin{lemma} \label{lemcpt1}
Let $\ve>0$ and let $\mu$ be a Radon measure on $\R^d$ and $B_0\subset \R^d$ some closed ball.
Given $\delta>0$, denote by $G(B_0,\delta)$ the collection of points $x\in\R^d$ such that
$$\int_{\delta r(B_0)}^{\delta^{-1}r(B_0)} |\Delta_\mu(x,r)|^2\,\frac{dr}r
\leq \delta^4\,\Theta_\mu(B_0)^2.$$
Suppose that $\mu(B_0\cap G(B_0,\delta))>0$ and that
$$\mu(\delta^{-1}B_0\setminus G(B_0,\delta))\leq \delta^4 \,\mu(\delta^{-1}B_0).$$
If $\delta$ is small enough, depending only on $d$ and $\ve$,
then
$$\alpha_\mu(2B_0) <\ve.$$
\end{lemma}

Before proving this we will need to prove some auxiliary results and to introduce some additional notation.
 For any Borel function $\vphi: \R \ra \R$, let 
$$\vphi_t(x)=\frac{1}{t} \vphi \left(\frac {|x|}t \right), \, t>0$$
and define
\begin{equation}\label{eqdeltafi}
\Delta_{\mu,\vphi} (x,t):= \int\bigr(\vphi_t (y-x)-\vphi_{2t}(y-x)\bigr)\,d\mu(y),
\end{equation}
whenever the integral makes sense.

\begin{lemma}\label{lemconvex}
 Let $\vphi:[0,\infty)\to\R$ be a $\CC^\infty$ function supported in 
$[0,2]$ which is constant in $[0,1/2]$. Let $x\in\R^d$ and $0\leq r_1<r_2$. For any $1\leq p<\infty$ we have
$$\int_{r_1}^{r_2} |\Delta_{\mu,\vphi}(x,r)|^p\,\frac{dr}r \leq c
\int_{r_1/2}^{2r_2} |\Delta_{\mu}(x,r)|^p\,\frac{dr}r,$$
where $c$ depends only on $\vphi$ and $p$.
\end{lemma}

\begin{proof}
This follows by writing $\vphi$ as a suitable convex combination of functions of the form $\chi_{[0,r]}$.
For completeness we show the details. For $s\geq0$, we write
$$\frac1R\vphi\left(\frac sR\right) = -\int_0^\infty \frac1{R^2}\,\vphi'\left(\frac rR\right) \,\chi_{[0,r]}(s)\,dr,$$
so that, by Fubini and changing variables,
\begin{align}\label{eqff22}
\Delta_{\mu,\vphi}(x,R) & = - \int_0^\infty \!\frac1{R^2} \,\vphi'\!\left(\frac rR\right) \chi_{[0,r]}(|\cdot|)* \mu(x) \,dr
+ \int_0^\infty \!\!\frac1{(2R)^2}\, \vphi'\!\left(\frac r{2R}\right) \chi_{[0,r]}(|\cdot|)* \mu(x) \,dr\\
& = 
- \int_0^\infty  \vphi'(t) \left(\frac1{R}\,\chi_{[0,tR]}(|\cdot|)* \mu(x) - \frac1{2R}\,\chi_{[0,2tR]}(|\cdot|)* \mu(x)
\right) \,dt \nonumber\\
&=
- \int_{1/2}^2  t\,\vphi'(t) \,\Delta_\mu(x,tR) \,dt,\nonumber
\end{align}
taking into account that $\vphi'$ is supported on $[1/2,2]$ in the last identity.
As a consequence, since $\int |t\,\vphi'(t)|^{p'}dt\lesssim1$,
by Cauchy-Schwarz we get
$$|\Delta_{\mu,\vphi}(x,r)|^p\leq \left|
\int_{1/2}^2  t\,\vphi'(t) \,\Delta_\mu(x,tr) \,dt\right|^p
\lesssim \int_{1/2}^2  |\Delta_\mu(x,tr)|^p \,dt = \int_{r/2}^{2r}  |\Delta_\mu(x,s)|^p \,\frac{ds}r.$$
Thus
\begin{align*}
\int_{r_1}^{r_2} |\Delta_{\mu,\vphi}(x,r)|^p\,\frac{dr}r & \lesssim  
\int_{r_1}^{r_2} \!
\int_{r/2}^{2r} |\Delta_\mu(x,s)|^p \,ds\,\frac{dr}{r^2} \lesssim
\int_{r_1/2}^{2r_2} 
 |\Delta_\mu(x,s)|^p \,\frac{ds}s.
\end{align*}
\end{proof}

\begin{lemma}\label{lemfac31} 
Let $\mu$ be a non-zero Radon measure in $\R^d$. Then $\mu$ is $1$-flat if and only if
$\Delta_\mu(x,r)=0$ for all $x\in\supp\mu$ and all $r>0$.
\end{lemma}

\begin{proof}
It is clear if $\mu$ $1$-flat, then
$\Delta_\mu(x,r)=0$ for all $x\in\supp\mu$ and all $r>0$. To prove the converse implication
it is enough to show that
 $\mu$ is $1$-uniform, that is, there exists some constant $c>0$
such that 
$$\mu(B(x,r)) = c\,r\qquad \mbox{for all $x\in\supp\mu$ and all $r>0$}.$$
It is well known that $1$-uniform measures are $1$-flat (see \cite[Chapter 17]{Mattila-llibre},
for example).

We intend to apply Theorem 3.10 from \cite{CGLT}, which asserts that, 
if $\mu$ is AD-regular and $\Delta_{\mu,\vphi}(x,r)=0$  for all $x\in\supp\mu$ and all $r>0$,
with $\vphi(y) = e^{-|y|^2}$, then $\mu$ is $1$-flat.
To prove the AD-regularity of $\mu$, assume for simplicity that $0\in\supp\mu$. Since $\Delta_\mu(0,r)=0$ for all $r>0$, we deduce that
$\mu(B(0,2^n))= 2^n\,\mu(B(0,1))$ for all $n\geq1$.
For $x\in\supp\mu\cap B(0,n)$ and any integer $m\leq n$, using now that $\Delta_\mu(x,r)=0$ for all $r>0$, we infer that $\mu(B(x,2^m))= 2^{m-n}\mu(B(x,2^n))$.
Since $B(0,2^{n-1})\subset B(x,2^n)\subset B(0,2^{n+1})$, we have
$$2^{n-1} \mu(B(0,1))\leq \mu(B(x,2^n)) \leq 2^{n+1} \mu(B(0,1)).$$
Thus
$$c_2\,2^{m-1}\leq \mu(B(x,2^m)) \leq c_2\,2^{m+1} ,$$
with $c_2=\mu(B(0,1))$.
Since $n$ can be take arbitrarily large and the preceding estimate holds for all $m\leq n$,
the AD-regularity of $\mu$ follows.

On the other hand, as in \rf{eqff22}, we have
$$\Delta_{\mu,\vphi}(x,r)  = - \int_{1/2}^2  t\,\vphi'(t) \,\Delta_\mu(x,tr) \,dt,$$
and so $\Delta_{\mu,\vphi}(x,r)$ vanishes identically on $\supp\mu$ for all $r>0$, as wished.
\end{proof}

\begin{lemma}\label{lemcompact}
Let $\mu$ be a Radon measure in $\R^d$ such that $1\leq \mu(\bar B(0,1))\leq\mu(B(0,2))\leq 9$. For all
$\ve>0$ there exists  $\delta>0$ depending only on $d$ and $\ve$ such that
if 
$$\int_{\delta}^{\delta^{-1}}\!\!\!\!
 \int_{x\in B(0,\delta^{-1})} |\Delta_{\mu} (x,r)|\,d\mu(x)\,\frac{dr}r \leq \delta^{1/2},$$
then
$$\dist_{B(0,2)}(\mu,\FF) <\ve.$$
\end{lemma}

\begin{proof}
Suppose that there exists an $\ve>0$, and for each $m\geq 1$ there exists a Radon
 measure $\mu_m$ such that $1\leq \mu_m(\bar B(0,1))\leq\mu_m(B(0,2))\leq 9$, which satisfies
\begin{equation}\label{eqass32}
\int_{1/m}^{m}
 \int_{x\in  B(0,m)} |\Delta_{\mu_m} (x,r)|\,d\mu_m(x)\,\frac{dr}r \leq \frac1{m^{1/2}},
\end{equation}
and
\begin{equation}\label{equu12}
\dist_{B(0,2)}(\mu_m,\FF) \geq\ve.
\end{equation}

We will first show that the sequence $\{\mu_m\}$ has a subsequence which is weakly * convergent to some Radon measure $\mu$
(i.e. when tested against 
compactly supported continuous functions).
This follows from standard compactness arguments once we show that $\mu_m$ is uniformly bounded on compact sets. That is,
for any compact $K\subset\R^d$, $\sup_m\mu_m(K)<\infty$. To prove this, 
for $n\geq 4$, $1/4<r<1/2$, and $x\in B(0,1)$, 
we write
\begin{align*}
\frac{\mu_m(B(0,2^{n-3}))}{2^{n+2}}  \leq
\frac{\mu_m(B(x,2^nr))}{2^nr}& \leq \sum_{k=1}^n |\Delta_{\mu_m}(x,2^{k-1}r)| + \frac{\mu_m(B(x,r))}r\\
& \leq \sum_{k=1}^n |\Delta_{\mu_m}(x,2^{k-1}r)| + 4\,\mu_m(B(0,2)).
\end{align*}
Integrating this estimate with respect to $\mu_m$ on $B(0,1)$ and with respect to $r\in[1/4,1/2]$, 
using \rf{eqass32} for $m$ big enough (depending on $n$)
we obtain
$$\mu_m(B(0,2^{n-3})) \leq 2^{n+2} \left[\sum_{k=1}^n \int_{1/4}^{1/2}\!\!\int_{B(0,1)}
\!\!\!|\Delta_{\mu_m}(x,2^{k-1}r)|
d\mu_m(x) \frac{dr}r + 4\,\mu_m(B(0,2))\right]\! \leq c(n),$$
which proves the uniform boundedness of $\mu_m$ on compact sets.

Our next objective consists in proving that $\mu$ is a $1$-flat measure. As shown in Lemma 
\ref{lemfac31}, it is enough to show that $\mu$ is a non-zero measure such that $\Delta_\mu(x,r)=0$ for all $x\in\supp\mu$ and all $r>0$.
Indeed, it is easy to check that
 $1\leq \mu(\bar B(0,1))\leq\mu(B(0,2))\leq 9$, and thus $\mu$ is not identically zero.

To prove that $\Delta_\mu(x,r)$ vanishes identically on $\supp\mu$ for all $r>0$, we will show first that, given any $\CC^\infty$ function $\vphi:[0,\infty)\to\R$ which is supported in 
$[0,2]$ and constant in $[0,1/2]$, we have
\begin{equation}\label{eqlimit}
\int_{0}^\infty\!
 \int_{x\in\R^d} |\Delta_{\mu,\vphi} (x,r)|\,d\mu(x)\,\frac{dr}r=0.
 \end{equation}
The proof of this fact is elementary.  Suppose that $\mu_{m_j}$ converges weakly to $\mu$.
Fix  $m_0$ and 
 let $\eta > 0$. 
 Set  $K = [1/m_0,\,m_0] \times \bar B(0,2m_0)$.  
Now $\{y \to  \vphi_t(x-y)- \vphi_{2t}(x-y),~ (t,x) \in K\}$ is an equicontinuous family of continuous functions supported inside a fixed compact set,  which implies that $(\vphi_t- \vphi_{2t}) * \mu_{m_j}(x)$ converges to $ (\vphi_t- \vphi_{2t}) * \mu(x)$ uniformly on $K$.  It therefore follows that 
\begin{equation*}
\iint_K  |(\vphi_t- \vphi_{2t})  *\mu(x)|d\mu(x)\frac{dt}t =
 \lim_j \int_{1/m_0}^{m_0}
 \int_{x\in \bar B(0,m_0)} |(\vphi_t- \vphi_{2t}) *\mu_{m_j}(x)|d\mu_{m_j}(x)\frac{dt}t =0,
\end{equation*}
by \rf{eqass32}.
Since  this holds for any $m_0\geq 1$, our claim \eqref{eqlimit}  is proved.

Denote by $G$ the subset of those points $x\in\supp(\mu)$ such that
$$\int_{0}^\infty\!
 | \Delta_{\mu,\vphi} (x,r)|\,\frac{dr}r=0.$$
It is clear now that $G$ has full $\mu$-measure. By continuity, it follows that $\Delta_{\mu,\vphi} (x,r)=0$ for all $x\in
\supp\mu$ and all $r>0$. Finally, by taking a suitable sequence of $\CC^\infty$ functions $\vphi_k$ which
converge to $\chi_{[0,1]}$ we infer that $\Delta_{\mu} (x,r)=0$ for all $x\in\supp\mu$ and $r>0$, and thus
$\mu$ is $1$-flat.

 However, by condition \rf{equu12}, letting $m\to\infty$, we have 
$$\dist_{B(0,2)}(\mu,\FF) \geq\ve,$$
because $\dist_{B(0,2)}(\cdot,\FF) $ is continuous under the weak * topology, see \cite[Lemma 14.13]{Mattila-llibre}.
So $\mu\not\in\FF$, which is a contradiction.
\end{proof}

By renormalizing the preceding lemma we get:

\begin{lemma}\label{lemcompactbol}
Let $\mu$ be a Radon measure in $\R^d$ and let $B_0\subset\R^d$ be some
ball such that
 $0<\mu(\bar B_0)\leq\mu(2B_0)\leq 9\,\mu(\bar B_0)$. For all
$\ve>0$ there exists  $\delta>0$ depending only on $d$ and $\ve$ such that
if 
$$\int_{\delta\,r(B_0)}^{\delta^{-1}\,r(B_0)}\!\!
 \int_{x\in \delta^{-1} B_0} |\Delta_{\mu} (x,r)|\,d\mu(x)\,\frac{dr}r \leq \delta^{1/2}\,\frac{\mu(\bar B_0)^2}{r(B_0)},$$
then
$$\dist_{2B_0}(\mu,\FF) <\ve\,r(B_0)\,\mu(\bar B_0).$$
\end{lemma}

\begin{proof}
Let $T:\R^d\to\R^d$ be an affine map which maps $\bar B_0$ to $\bar B(0,1)$. Consider the measure
$\sigma=\frac1{\mu(\bar B_0)}\,T\#\mu$, where as usual $T\#\mu(E):= \mu (T^{-1}(E))$, and apply the preceding lemma to $\sigma$.
\end{proof}

\vv
\begin{lemma}\label{lemdob}
Let $\mu$ be a Radon measure on $\R^d$ and let $x\in\R^d$, $r>0$, be such that $\mu(B(x,r/2))>0$.
If
$$\int_{r/2}^{2r} \Delta_\mu(x,t)^2\,\frac{dt}t
\leq \frac1{200}\,\Theta_\mu(B(x,r))^2,$$
then
$$\mu(B(x,2r))\leq 9\,\mu(B(x,r)).$$
\end{lemma}

\begin{proof}
Observe that
$$\int_{r/2}^{2r} \Delta_\mu(x,t)^2\,\frac{dt}t
= \int_{r/2}^{r} \left[ \Delta_\mu(x,t)^2 +  \Delta_\mu(x,2t)^2\right]
\,\frac{dt}t\leq \frac1{200}\,\Theta_\mu(B(x,r))^2.$$
Denote by $\sigma$ the measure $dt/t$ on $(0,\infty)$.
Then, by Chebyshev,
$$\sigma\left(\left\{t\in[r/2,r]:\left[\Delta_\mu(x,t)^2 +  \Delta_\mu(x,2t)^2\right]>\lambda\right\}\right) \leq \frac{1}{200\lambda}\,\Theta_\mu(B(x,r))^2.$$
Thus, if we choose $\lambda=\Theta_\mu(B(x,r))^2/100$, then there exists some $t\in[r/2,r]$ such that
$$ \Delta_\mu(x,t)^2 +  \Delta_\mu(x,2t)^2\leq \frac1{100}\,\Theta_\mu(B(x,r))^2,$$
taking into account that $\sigma([r/2,r])=\log2>1/2$.
This implies that 
$$\max\bigl(\Delta_\mu(x,t),\Delta_\mu(x,2t)\bigr)\leq \frac1{10}\,\Theta_\mu(B(x,r)),$$
and so
$$\bigl|\Theta_\mu(B(x,4t))- \Theta_\mu(B(x,t))\bigr|\leq 
\Delta_\mu(x,t)+\Delta_\mu(x,2t) \leq\frac1{5}\,\Theta_\mu(B(x,r)).$$
Then we deduce that
\begin{align*}
\Theta_\mu(B(x,2r))&\leq 2\,\Theta_\mu(B(x,4t))\leq 2\,\Theta_\mu(B(x,t))+ \frac2{5}\,\Theta_\mu(B(x,r))\\
& \leq \biggl(4 +\frac25\biggr)\,\Theta_\mu(B(x,r)) = \frac{22}5\,\Theta_\mu(B(x,r)),
\end{align*}
which is equivalent to saying that $\mu(B(x,2r))\leq \frac{44}5\,\mu(B(x,r))$.
\end{proof}

\vv
\begin{proof}[\bf Proof of Lemma \ref{lemcpt1}]
We set $B(x_0,r_0):=B_0$. We will assume first that $x_0\in G(B_0,\delta)\cap\supp\mu$. We will show that if $\delta>0$ is small enough,
the assumptions in the lemma imply that
$
0<\mu(B_0)\leq\mu(2B_0)\leq 9\,\mu(B_0)$ and
\begin{equation}\label{eqcl19}
\int_{4\delta r_0}^{(4\delta)^{-1}\,r_0}\!\!
 \int_{x\in (4\delta)^{-1} B_0} |\Delta_{\mu} (x,r)|\,d\mu(x)\,\frac{dr}r \leq (4\delta)^{1/2}\,\frac{\mu( B_0)^2}{r(B_0)}.
\end{equation}
 Then the application of Lemma \ref{lemcompactbol} finishes the proof (in the case $x_0\in G(B_0,\delta)$).
 
The constant $\delta$ will be chosen smaller than $1/10$, and so Lemma \ref{lemdob} ensures that
\begin{equation}\label{eqakpo09}
0<\mu(2B_0)\leq 9\,\mu(B_0)\leq 81\,\mu(\tfrac12B_0).
\end{equation}

For any $x\in G(B_0,\delta)$, we write
\begin{align} \label{eqsa111'}
\int_{\delta\,r_0}^{\delta^{-1}\,r_0}
  |\Delta_{\mu} (x,r)|\,\frac{dr}r   & \leq 
  (2\,\log\delta^{-1})^{1/2} \left(\int_{\delta\,r_0}^{\delta^{-1}r_0}\!\!
  \Delta_\mu(x,r)^2\,\frac{dr}r\right)^{1/2}\\
  & \leq  \left(2\,\delta^4\,\log\delta^{-1}\right)^{1/2}\,\Theta_\mu(B(x,r)).
  \nonumber\end{align}
For $x\in (4\delta)^{-1}B_0\setminus G(B_0,\delta)$ and $4\delta r_0\leq r\leq (4\delta)^{-1}r_0$ we use the brutal estimate
\begin{equation}\label{eqgdelta'}
|\Delta_{\mu} (x,r)| \leq \frac{\mu(B(x,(2\delta)^{-1}r_0))}{4\,\delta\,r_0}\leq 
\frac{\mu(B(x_0,\delta^{-1}r_0))}{4\,\delta\,r_0}
.
\end{equation}

By integrating the estimate \rf{eqsa111'} on $(4\delta)^{-1}B_0\cap G(B_0,\delta)$ and \rf{eqgdelta'}  on 
$(4\delta)^{-1}B_0\setminus G(B_0,\delta)$ and using that $\mu(\delta^{-1}B_0\setminus G(B_0,\delta))\leq \delta^4\,
\mu(\delta^{-1}B_0)$, we get
\begin{align}\label{eqplug55}
\int_{4\delta\,r_0}^{(4\delta)^{-1}\,r_0} \!\int_{x\in (4\delta)^{-1} B_0}
  |\Delta_{\mu} (x,r)|\,d\mu(x)\,\frac{dr}r & \leq c\,\delta^2\,(\log\delta^{-1})^{1/2}
\,\frac{\mu(B(x_0,r_0))}{r_0}\,\mu(B(x_0,\delta^{-1}r_0))\\
&\quad
+ \delta^4\,\frac{\mu(B(x_0,\delta^{-1}r_0))^2}{4\,\delta\,r_0}.\nonumber
\end{align}

We will estimate $\mu(B(x_0,\delta^{-1}r_0))$ now.
 Without loss of generality we 
assume that $\delta=2^{-n}$, for some big integer $n$. By changing variables, we obtain
$$\int_{2^{-n}r_0}^{2^nr_0}\!\!
  \Delta_\mu(x_0,r)^2\,\frac{dr}r = \sum_{k=-n+1}^n \int_{2^{k-1}r_0}^{2^kr_0}
  \Delta_\mu(x_0,r)^2\,\frac{dr}r
  = 
  \int_{r_0/2}^{r_0} \sum_{k=-n+1}^n \Delta_\mu(x_0,2^kr)^2\,\frac{dr}r.
  $$
Denote by $\sigma$ the measure $dr/r$ on $(0,\infty)$.
Then, by Chebyshev,
$$\sigma\Bigl(\Bigl\{r\in[r_0/2,r_0]:\sum_{k=-n+1}^n \Delta_\mu(x_0,2^kr)^2>\lambda\Bigr\}\Bigr) \leq \frac{\delta^4}{\lambda}\,\Theta_\mu(B_0)^2.$$
Thus, if we choose, for instance $\lambda=\delta^2\,\Theta_\mu(B_0)^2$, then there exists some $t\in[r_0/2,r_0]$ such that
$$\sum_{k=-n+1}^n \Delta_\mu(x_0,2^kt)^2\leq \delta^2\,\Theta_\mu(B_0)^2,$$
taking into account that $\sigma([r_0/2,r_0])=\log2>\delta^2$, for $\delta$ small enough. From the fact that $\Delta_\mu(x_0,2^kt)\leq\delta\,\Theta_\mu(B_0)$
for $-n+1\leq k\leq n$, we infer that
$$\Theta_\mu(B(x_0,2^{n+1}t))\leq \Theta_\mu(B(x_0,t)) + \sum_{k=0}^n\Delta_\mu(x_0,2^kt)
\leq 2\,\Theta_\mu(B(x_0,r_0)) + (n+1)\,\delta\,\Theta_\mu(B_0).$$
Using that $n=\log(\delta^{-1})/\log2$ and that $\Theta_\mu(B(x_0,\delta^{-1}r_0))\leq 2\,
\Theta_\mu(B(x_0,2^{n+1}t))$, we get
$$\Theta_\mu(B(x_0,\delta^{-1}r_0))\leq (4+c\,\delta\,\log\delta^{-1})\,\Theta_\mu(B_0)\leq
5\,\Theta_\mu(B_0),$$
for $\delta$ small enough. This is equivalent to saying that
$\mu(B(x_0,\delta^{-1}r_0))\leq 5\,\delta^{-1}\,\mu(B_0)$.
Plugging this estimate into \rf{eqplug55}, we obtain
$$\int_{4\delta\,r_0}^{(4\delta)^{-1}\,r_0} \int_{x\in (4\delta)^{-1} B_0}
  |\Delta_{\mu} (x,r)|\,d\mu(x)\,\frac{dr}r  \leq \bigl(c\,\delta\,(\log\delta^{-1})^{1/2}
  + c\,\delta\bigr)
\,\frac{\mu(B_0)^2}{r_0}.$$
For $\delta$ small enough the right hand side above is smaller than $\frac{(4\delta)^{1/2}\mu(B_0)^2}{r_0}$, as wished, and thus \rf{eqcl19} holds and we are done.
\vv

Suppose now that $x_0\not\in G(B_0,\delta)\cap\supp\mu$. Let $x_1\in B_0\cap G(B_0,\delta)\cap\supp\mu$ and consider the ball
$B_1=B(x_1,2r_0)$. Since $\Theta_\mu(B_0)\leq 2\Theta_\mu(B_1)$, every $x\in G(B_0,\delta)$ satisfies 
\begin{align*}
\int_{4\delta r(B_1)}^{(4\delta^{-1})r(B_1)} |\Delta_\mu(x,r)|^2\,\frac{dr}r & \leq 
\int_{\delta r(B_0)}^{\delta^{-1}r(B_0)} |\Delta_\mu(x,r)|^2\,\frac{dr}r
\leq \delta^4\,\Theta_\mu(B_0)^2 \leq (4\delta)^4\,\Theta_\mu(B_1)^2,
\end{align*}
and thus $x\in G(B_1,4\delta)$. Therefore,
$$\mu((4\delta)^{-1})B_1\setminus G(B_1,4\delta))\leq \mu(\delta^{-1}B_0\setminus G(B_0,\delta))
\leq \delta^4 \,\mu(\delta^{-1}B_0) \leq \delta^4 \,\mu(\delta^{-1}B_1)
\leq (4\delta)^4 \,\mu(\delta^{-1}B_1).$$
Thus, applying the conclusion of the lemma to the ball $B_1$, with $\delta$ small enough, we deduce
that $\alpha_\mu(2B_1)\leq \ve$. 
Taking also into account that $\frac12 B_1\subset 2B_0$, by \rf{eqakpo09} applied to $B_1$ we have
$$\mu(2B_1)\leq 81\,\mu(\tfrac12\,B_1) \leq 81\,\mu(2B_0),$$ and thus we get
$$\alpha(2B_0) =
\frac1{2r_0\,\mu(2B_0)}\,\inf_{c\geq0,L} \,\dist_{2B_0}(\mu,\,c\HH^1|_{L}) \lesssim
\frac1{2r_1\,\mu(2B_1)}\,\inf_{c\geq0,L} \,\dist_{2B_1}(\mu,\,c\HH^1|_{L}) \lesssim \ve.$$
\end{proof}

\vvv
\begin{remark}\label{remdob1}
By arguments very similar to the ones used in the preceding proof, one shows that under the assumptions of Lemma \ref{lemcpt1},
for all $x\in G(B_0,\delta)\cap B_0$,
we have
$$\mu(B(x,r)) \approx \frac{\mu(B_0)}{r_0}\,r\qquad \mbox{for $\delta\,r_0\leq r\leq \delta^{-1}r_0$,}$$
assuming $\delta$ small enough.
\end{remark}
\vv


\section{The dyadic lattice of cells with small boundaries}\label{sec4}

In our proof of Theorem \ref{teomain} we will use the dyadic lattice of cells
with small boundaries constructed by David and Mattila in \cite[Theorem 3.2]{David-Mattila}.
The properties of this dyadic lattice are summarized in the next lemma.
\vv

\begin{lemma}[David, Mattila]
\label{lemcubs}
Let $\mu$ be a Radon measure on $\R^d$, $E=\supp\mu$, and consider two constants $C_0>1$ and $A_0>5000\,C_0$. Then there exists a sequence of partitions of $E$ into
Borel subsets $Q$, $Q\in \DD_k$, with the following properties:
\begin{itemize}
\item For each integer $k\geq0$, $E$ is the disjoint union of the cells $Q$, $Q\in\DD_k$, and
if $k<l$, $Q\in\DD_l$, and $R\in\DD_k$, then either $Q\cap R=\varnothing$ or else $Q\subset R$.
\vv

\item The general position of the cells $Q$ can be described as follows. For each $k\geq0$ and each cell $Q\in\DD_k$, there is a ball $B(Q)=B(z_Q,r(Q))$ such that
$$z_Q\in E, \qquad A_0^{-k}\leq r(Q)\leq C_0\,A_0^{-k},$$
$$E\cap B(Q)\subset Q\subset E\cap 28\,B(Q)=E \cap B(z_Q,28r(Q)),$$
and
$$\mbox{the balls $5B(Q)$, $Q\in\DD_k$, are disjoint.}$$

\vv
\item The cells $Q\in\DD_k$ have small boundaries. That is, for each $Q\in\DD_k$ and each
integer $l\geq0$, set
$$N_l^{ext}(Q)= \{x\in E\setminus Q:\,\dist(x,Q)< A_0^{-k-l}\},$$
$$N_l^{int}(Q)= \{x\in Q:\,\dist(x,E\setminus Q)< A_0^{-k-l}\},$$
and
$$N_l(Q)= N_l^{ext}(Q) \cup N_l^{int}(Q).$$
Then
\begin{equation}\label{eqsmb2}
\mu(N_l(Q))\leq (C^{-1}C_0^{-3d-1}A_0)^{-l}\,\mu(90B(Q)).
\end{equation}
\vv

\item Denote by $\DD_k^{db}$ the family of cells $Q\in\DD_k$ for which
\begin{equation}\label{eqdob22}
\mu(100B(Q))\leq C_0\,\mu(B(Q)),
\end{equation}
and set $\BZ_k = \DD_k\setminus \DD_k^{db}$. We have that $r(Q)=A_0^{-k}$ when $Q\in\BZ_k$
and
\begin{equation}\label{eqdob23}
\mu(100B(Q))\leq C_0^{-l}\,\mu(100^{l+1}B(Q))\quad
\mbox{for all $l\geq1$ such that $100^l\leq C_0$ and $Q\in\BZ_k$.}
\end{equation}
\end{itemize}
\end{lemma}

\vv
We use the notation $\DD=\bigcup_{k\geq0}\DD_k$. For $Q\in\DD$, we set $\DD(Q) =
\{P\in\DD:P\subset Q\}$.
Given $Q\in\DD_k$, we denote $J(Q)=k$. We set
$\ell(Q)= 56\,C_0\,A_0^{-k}$ and we call it the side length of $Q$. Note that 
$$\frac1{28}\,C_0^{-1}\ell(Q)\leq \diam(Q)\leq\ell(Q).$$
Observe that $r(Q)\approx\diam(Q)\approx\ell(Q)$.
Also we call $z_Q$ the center of $Q$, and the cell $Q'\in \DD_{k-1}$ such that $Q'\supset Q$ the parent of $Q$.
 We set
$B_Q=28 \,B(Q)=B(z_Q,28\,r(Q))$, so that 
$$E\cap \tfrac1{28}B_Q\subset Q\subset B_Q.$$

We assume $A_0$ big enough so that the constant $C^{-1}C_0^{-3d-1}A_0$ in 
\rf{eqsmb2} satisfies 
$$C^{-1}C_0^{-3d-1}A_0>A_0^{1/2}>10.$$
Then we deduce that, for all $0<\lambda\leq1$,
\begin{equation}\label{eqfk490}
\mu\bigl(\{x\in Q:\dist(x,E\setminus Q)\leq \lambda\,\ell(Q)\}\bigr) + 
\mu\bigl(\bigl\{x\in 4B_Q:\dist(x,Q)\leq \lambda\,\ell(Q)\}\bigr)\leq
c\,\lambda^{1/2}\,\mu(3.5B_Q).
\end{equation}

We denote
$\DD^{db}=\bigcup_{k\geq0}\DD_k^{db}$ and $\DD^{db}(Q) = \DD^{db}\cap \DD(Q)$.
Note that, in particular, from \rf{eqdob22} it follows that
$$\mu(100B(Q))\leq C_0\,\mu(Q)\qquad\mbox{if $Q\in\DD^{db}.$}$$
For this reason we will call the cells from $\DD^{db}$ doubling. 

As shown in \cite[Lemma 5.28]{David-Mattila}, any cell $R\in\DD$ can be covered $\mu$-a.e.\
by a family of doubling cells:
\vv

\begin{lemma}\label{lemcobdob}
Let $R\in\DD$. Suppose that the constants $A_0$ and $C_0$ in Lemma \ref{lemcubs} are
chosen suitably. Then there exists a family of
doubling cells $\{Q_i\}_{i\in I}\subset \DD^{db}$, with
$Q_i\subset R$ for all $i$, such that their union covers $\mu$-almost all $R$.
\end{lemma}

The following result is proved in \cite[Lemma 5.31]{David-Mattila}.
\vv

\begin{lemma}\label{lemcad22}
Let $R\in\DD$ and let $Q\subset R$ be a cell such that all the intermediate cells $S$,
$Q\subsetneq S\subsetneq R$ are non-doubling (i.e.\ belong to $\bigcup_{k\geq0}\BZ_k$).
Then
\begin{equation}\label{eqdk88}
\mu(100B(Q))\leq A_0^{-10(J(Q)-J(R)-1)}\mu(100B(R)).
\end{equation}
\end{lemma}

Let us remark that the constant $10$ in \rf{eqdk88} can be replaced by any other positive 
constant if $A_0$ and $C_0$ are chosen suitably in Lemma \ref{lemcubs}, as shown in (5.30) of
\cite{David-Mattila}.

From the preceding lemma we deduce:

\vv
\begin{lemma}\label{lemcad23}
Let $Q,R\in\DD$ be as in Lemma \ref{lemcad23}.
Then
$$\Theta_\mu(100B(Q))\leq C_0\,A_0^{-9(J(Q)-J(R)-1)}\,\Theta_\mu(100B(R))$$
and
$$\sum_{S\in\DD:Q\subset S\subset R}\Theta_\mu(100B(S))\leq c\,\Theta_\mu(100B(R)),$$
with $c$ depending on $C_0$ and $A_0$.
\end{lemma}

\begin{proof}
By \rf{eqdk88},
$$\Theta_\mu(100\,B(Q))\leq A_0^{-10(J(Q)-J(R)-1)}\,\frac{\mu(100\,B(R))}{r(100\,B(Q))} = 
A_0^{-10(J(Q)-J(R)-1)}\Theta_\mu(100\,B(R))\,\frac{r(B(R))}{r(B(Q))}.$$ 
The first inequality in the lemma follows from this estimate and the fact that
$r(B(R))\leq C_0\,A_0^{(J(Q)-J(R))}\,r(B(Q))$.

The second inequality in the lemma is an immediate consequence of the first one.
\end{proof}

\vv
From now on we will assume that $C_0$ and $A_0$ are some big fixed constants so that the
results stated in the lemmas of this section hold.

\vv


\section{The Main Lemma}\label{secmlemma}

\subsection{Statement of the Main Lemma}

Let $\mu$ be the measure in Theorem \ref{teomain} and $E=\supp\mu$, and consider the
dyadic lattice associated with $\mu$ described in Section \ref{sec4}.
Let $F\subset E$ be an arbitrary compact set such that
\begin{equation}\label{eqsk771}
\int_{F}\int_0^1 \Delta_\mu(x,r)^2\,\frac{dr}r\, d\mu(x)<\infty.
\end{equation}

Given $Q\in\DD$, we denote by $\GG(Q,\delta,\eta)$ the set of the points 
$x\in\R^d$ such that
\begin{equation}\label{eqgqd}
\int_{\delta\,\ell(Q)}^{\delta^{-1}\ell(Q)} \Delta_\mu(x,r)^2\,\frac{dr}r\leq 
\eta\,\Theta_\mu(2B_Q)^2.
\end{equation}

The next lemma concentrates the main difficulties for the proof of the ``if'' implication of Theorem \ref{teomain}.
\vv

\begin{mainlemma}\label{mainlemma}
Let $0<\ve<1/100$. Suppose that 
 $\delta$ and $\eta$ are small enough positive constants (depending only on $\ve$).
Let $R\in\DD^{db}$ be a doubling cell with $\ell(R)\leq \delta$ such that 
\begin{equation}\label{eqassu1}
\mu(R\setminus F)\leq\eta\,\mu(R),\qquad\quad
\mu(\lambda B_R\setminus F)\leq \eta\,\mu(\lambda B_R) \quad\mbox{for all $2<\lambda\leq\delta^{-1}$},\end{equation}
and
\begin{equation}\label{eqassu2}
\mu\bigl(\delta^{-1}B_R\cap F\setminus \GG(R,\delta,\eta)\bigr)
\leq \eta\,\mu(R\cap F).
\end{equation}
Then there exists an AD-regular curve $\Gamma_R$ (with the AD-regularity constant bounded
by some absolute constant) and a family of pairwise disjoint cells $\sss(R)\subset \DD(R)
\setminus \{R\}$ such that, denoting by
$\tree(R)$ the subfamily of the cells from $\DD(R)$ which are not strictly contained in any cell
from $\sss(R)$,
the following holds:
\vv
\begin{itemize}

\item[(a)] $\mu$-almost all $F\cap R\setminus \bigcup_{Q\in\sss(R)}Q$ is contained in 
$\Gamma_R$ and moreover $\mu|_{F\cap R\setminus \bigcup_{Q\in\sss(R)}Q}$ is absolutely continuous with
respect to $\HH^1|_{\Gamma_R}$.
\vv

\item[(b)] For all $Q\in\tree(R)$, $\Theta(1.1B_Q)\leq A\,\Theta_\mu(1.1B_R)$, where $A\geq 100$ is
some absolute constant.\vv

\item[(c)] The cells from $\sss(R)$ satisfy
\begin{align*}
\sum_{Q\in\sss(R)}\Theta_\mu(1.1B_Q)^2\,\mu(Q)& \leq\ve\,\Theta_\mu(B_R)^2\,\mu(R) \\
&\quad + 
c(\ve)\sum_{Q\in\tree(R)} \int_{F\cap\delta^{-1}B_Q}\int_{\delta\ell(Q)}^{\delta^{-1}\ell(Q)}
\Delta_\mu(x,r)^2\,\frac{dr}r\,d\mu(x).
\end{align*}
\end{itemize}
\end{mainlemma}

\vvv
Let us remark that the assumption that $\ell(R)\leq\delta$ can be removed if we assume that
$$\int_{F}\int_0^\infty \Delta_\mu(x,r)^2\,\frac{dr}r\, d\mu(x)<\infty,$$
instead of \rf{eqsk771}.

We will prove the Main Lemma in Sections \ref{sec7}--\ref{sec14}. 
Before proving it, we show how Theorem \ref{teomain} follows from this.


\vv
\subsection{Proof of Theorem \ref{teomain} using the Main Lemma \ref{mainlemma}}\label{secprovam}

As remarked in the Introduction, we only have to prove the ``if'' implication of the theorem.
First we prove the following auxiliary result, which will be used to deal
with some cells $R\in\DD$ such that  \rf{eqassu2}
does not hold.

\vv

\begin{lemma}\label{lemaux29}
Let $R\in\DD$ be a cell such that
\begin{equation}\label{eqaux934} 
\mu\bigl(\delta^{-1}B_R\cap F\setminus \GG(R,\delta,\eta)\bigr)
> \eta\,\mu(R\cap F).
\end{equation}
Then
$$\Theta_\mu(2B_R)^2 \,\mu(R\cap F)\leq \frac1{\eta^2}\int_{\delta^{-1}B_R\cap F}
\int_{\delta\ell(R)}^{\delta^{-1}\ell(R)}\Delta_\mu(x,r)^2\,\frac{dr}r\,d\mu(x).$$
\end{lemma}

\begin{proof}
For all $x\in \delta^{-1}B_R\cap F\setminus \GG(R,\delta,\eta)$ we have
$$\int_{\delta\,\ell(R)}^{\delta^{-1}\ell(R)} \Delta_\mu(x,r)^2\,\frac{dr}r>
\eta\,\Theta_\mu(2B_R)^2.$$
Thus, integrating on $\delta^{-1}B_R\cap F\setminus \GG(R,\delta,\eta)$ and applying
\rf{eqaux934}, we derive
\begin{align*}
\int_{\delta^{-1}B_R\cap F\setminus \GG(R,\delta,\eta)} \int_{\delta\,\ell(R)}^{\delta^{-1}\ell(R)} \Delta_\mu(x,r)^2\,\frac{dr}r \,d\mu(x)& \geq 
\eta\,\Theta_\mu(2B_R)^2\,\mu\bigl(\delta^{-1}B_R\cap F\setminus \GG(R,\delta,\eta)\bigr)\\
&
\geq \eta^2\,\Theta_\mu(2B_R)^2\,\mu(R\cap F),
\end{align*}
and the lemma follows.
\end{proof}

\vv

To prove the ``if'' implication of Theorem \ref{teomain} clearly it is enough to show that $\mu|_F$ is rectifiable. To this end, 
let $x_0$ be a
point of density of $F$ and for $\eta>0$ let $B_0=B(x_0,r_0)$ be some ball such that
\begin{equation}\label{eqdd3}
\mu(B_0\setminus F)\leq\eta^2\mu(B_0)\qquad \mbox{and}\qquad \mu(\tfrac12B_0)
\geq \frac1{2^{d+1}}\,\mu(B_0).
\end{equation}
Taking into account that for $\mu$-almost every $x_0\in F$ there exists a sequence of balls
like $B_0$ 
centered at $x_0$ with radius tending to $0$
fulfilling \rf{eqdd3} (see Lemma 2.8 of \cite{Tolsa-llibre} for example), it suffices to prove that any ball like $B_0$ contains a 
rectifiable subset with positive $\mu$-measure. 

Denote by $\BZ^1$ the family of cells $R\in\DD$ with $\delta^{-1}B_R\subset B_0$
such that 
$$\mu(\lambda B_R\setminus F)\leq \eta\,\mu(\lambda B_R) \quad\mbox{for some $\lambda$ with $2<\lambda\leq\delta^{-1}$},$$
and by $\BZ^2$ the family of cells $R\in\DD$ contained in $B_0$ such that 
$$\mu(R\setminus F)\geq \eta\,\mu(R).$$
Next we show that union of the cells from $\BZ^1\cup\BZ^2$ has very small $\mu$-measure.
\vv

\begin{lemma}\label{lemaux21}
We have
\begin{equation}\label{eqmaxx4}
\mu\Biggl(\bigcup_{R\in\BZ^1\cup \BZ^2} R\Biggr)\leq c\,\eta\,\mu(B_0).
\end{equation}
\end{lemma}

\begin{proof}
To deal with the cells from $\BZ^1$ we consider the maximal operator
\begin{equation}\label{eqm*}
M_*f(x)=\sup_{B\text{ ball}:\,x\in \frac12 B} \frac1{\mu(B)}\int_B|f|\,d\mu.
\end{equation}
This operator is known to be bounded from $L^1(\mu)$ to $L^{1,\infty}(\mu)$. Note that for all $x\in R\in\BZ^1$,
$$M_*\chi_{B_0\setminus F}(x) \geq \eta.$$
Then, using the first estimate in \rf{eqdd3} we get
$$\mu\Biggl(\bigcup_{R\in\BZ^1} R\Biggr)\leq \mu\bigl(\{x\in\R^d:\,M_*\chi_{B_0\setminus F}
(x)\geq \eta\}\bigr)\leq c\,\frac{\mu(B_0\setminus F)}\eta \leq c\,\eta\,\mu(B_0),$$
as wished.

To deal with the cells from $\BZ^2$, we argue analogously, by taking the maximal dyadic 
operator
\begin{equation}\label{eqmd*}
M^df(x) = \sup_{Q\in\DD:x\in Q}\frac1{\mu(Q)}\int|f|\,d\mu.
\end{equation}
\end{proof}

\vv

Let us continue with the proof of Theorem \ref{teomain}.
From \rf{eqmaxx4} and the fact that $\mu(B_0)\approx\mu(\frac12\,B_0)$ we infer that,
for $\eta$ small enough, there exists some cell $R_0\in\DD^{db}$ satisfying $R_0\subset \frac34 B_0$,
$\ell(R_0)\leq\delta$, $\delta^{-1}B_{R_0}\subset \frac9{10}B_0$, and 
$$\mu\Biggl(R_0\setminus \bigcup_{Q\in\BZ^1\cup\BZ^2}Q\Biggr)>0.$$
We are going now to construct a family of cells $\ttt$ contained in $R_0$ inductively, by applying the Main
Lemma \ref{mainlemma}. To this end, we need to introduce some additional notation.

Recall that the Main Lemma asserts that if $R\in\DD^{db}$, with $\ell(R)\leq\delta$, satisfies the assumptions \rf{eqassu1} and \rf{eqassu2}, then it generates some family of cells
$\sss(R)$ fulfilling the properties (a), (b) and (c). Now it is convenient to define $\sss(R)$
also if the assumptions \rf{eqassu1} or \rf{eqassu2} do not hold.
In case that $R$ is a descendant of $R_0$ such that
 $R\in\DD^{db}\setminus(\BZ^1\cup\BZ^2)$ does not satisfy  \rf{eqassu2}, that is, $$\mu\bigl(\delta^{-1}B_R\cap F\setminus \GG(R,\delta,\eta)\bigr)> \eta\,\mu(R\cap F),$$ 
we let 
$\sss(R)$ be the family of the sons of $R$. In other words, for $R\in\DD_k$, $\sss(R)=\DD_{k+1}\cap \DD(R)$.

On the other hand, if $R$ is a descendant of $R_0$ such that $R\in\DD^{db}\cap(\BZ^1\cup\BZ^2)
$ (note that
this means that some of the inequalities in \rf{eqassu1} does not hold), we set $\sss(R)=\varnothing$. 

Given a cell $Q\in\DD$, we denote by $\MD(Q)$ the family  of maximal cells 
(with respect to inclusion) from $P\in \DD^{db}(Q)$ such that $2B_P\subset 1.1B_Q$. Recall that, by Lemma \ref{lemcobdob}, this family covers $\mu$-almost
all $Q$. Moreover, by Lemma \ref{lemcad23} it follows that if $P\in\MD(Q)$, then $\Theta_\mu(2B_P)
\leq c\,\Theta_\mu(1.1B_Q)$.

We are now ready to construct the aforementioned family $\ttt$. We will have
$\ttt=\bigcup_{k\geq0}\ttt_k$. First we set
$$\ttt_0=\{R_0\}.$$
Assuming $\ttt_k$ to be defined, we set
$$\ttt_{k+1} = \bigcup_{R\in\ttt_k} \bigcup_{Q\in\sss(R)} \MD(Q).$$
Note that the families $\MD(Q)$ with $Q\in\SSS(R)$, $R\in\ttt_k$ are pairwise disjoint. 

Next we prove a key estimate.
\vv

\begin{lemma}\label{lemkey62}
If $\ve$ is chosen small enough in the Main Lemma, then
\begin{equation}\label{eqsum441}
\sum_{R\in\ttt}\Theta_\mu(2B_R)^2\,\mu(R)\leq 2\,\Theta_\mu(2B_{R_0})^2\,\mu(R_0)+
c(\ve,\eta,\delta)\,\int_F\int_0^1\Delta_\mu(x,r)^2\,\frac{dr}r\,d\mu(x).
\end{equation}
\end{lemma}

\begin{proof}
For $k\geq0$ we have
\begin{align*}
\sum_{P\in\ttt_{k+1}}\Theta_\mu(2B_P)^2\,\mu(P) & = 
\sum_{R\in\ttt_k}\sum_{Q\in\sss(R)}\sum_{P\in\MD(R)}\Theta_\mu(2B_P)^2\,\mu(P).
\end{align*}
From Lemma \ref{lemcad22} we infer that any $P\in\MD(Q)$ satisfies $\Theta_\mu(2B_P)\leq c\,
\Theta_\mu(1.1B_Q)$. So we get
\begin{equation}\label{eqssk73}
\sum_{P\in\ttt_{k+1}}\Theta_\mu(2B_P)^2\,\mu(P)\leq 
c\sum_{R\in\ttt_k}\sum_{Q\in\sss(R)}\Theta_\mu(1.1B_Q)^2\,\mu(Q).
\end{equation}
If the conditions \rf{eqassu1} and \rf{eqassu2} hold, then (c) in the Main Lemma tells us that
\begin{equation}\label{eqdkg89}
\sum_{Q\in\sss(R)}\!\!\!\!\Theta_\mu(1.1B_Q)^2\mu(Q)\leq\ve\,\Theta_\mu(2B_R)^2\mu(R) + 
c(\ve)\!\!\!\!\sum_{Q\in\tree(R)} \int_{F\cap\delta^{-1}B_Q}\int_{\delta\ell(Q)}^{\delta^{-1}\ell(Q)}
\!\!\Delta_\mu(x,r)^2\,\frac{dr}r\,d\mu(x).
\end{equation}

In the case $R\not\in\BZ^1\cup\BZ^2$ and $\mu\bigl(\delta^{-1}B_R\cap F\setminus \GG(R,\delta,\eta)\bigr)
> \eta\,\mu(R\cap F)$, recalling that $\sss(R)$ is the family of the sons of $R$, we derive
$$\sum_{Q\in\sss(R)}\Theta_\mu(1.1B_Q)^2\,\mu(Q)\leq c\,\Theta_\mu(2B_R)^2\,\mu(R).$$
On the other hand, by Lemma \ref{lemaux29}
$$\Theta_\mu(2B_R)^2 \,\mu(R)\leq 2\,
\Theta_\mu(2B_R)^2 \,\mu(R\cap F)\leq \frac2{\eta^2}\int_{\delta^{-1}B_R\cap F}
\int_{\delta\ell(R)}^{\delta^{-1}\ell(R)}\Delta_\mu(x,r)^2\,\frac{dr}r\,d\mu(x),$$
taking into account that $\mu(R)\leq 2\,\mu(R\cap F)$, as $R\not\in\BZ^2$, $\eta<1/2$.
So \rf{eqdkg89} also holds in this case, replacing $c(\ve)$ by $2/\eta^2$.

Finally, if $R\in\BZ^1\cup\BZ^2$, by construction we have
$$\sum_{Q\in\sss(R)}\Theta_\mu(1.1B_Q)^2\,\mu(Q)=0,$$
since $\sss(R)=\varnothing$.

Plugging the above estimates into \rf{eqssk73} we obtain
\begin{align*}
\sum_{P\in\ttt_{k+1}}\Theta_\mu(2B_P)^2\,\mu(P) &\leq 
c_3\,\ve\sum_{R\in\ttt_k} \Theta_\mu(2B_R)^2\,\mu(R) \\
&\quad +c(\ve,\eta)\!\sum_{R\in\ttt_k}\sum_{Q\in\tree(R)} \int_{F\cap\delta^{-1}B_Q}\int_{\delta\ell(Q)}^{\delta^{-1}\ell(Q)}
\Delta_\mu(x,r)^2\,\frac{dr}r\,d\mu(x).
\end{align*}
Choosing $\ve$ such that $c_3\,\ve\leq 1/2$, we deduce that
\begin{align}\label{eqfif23}
\sum_{R\in\ttt}\Theta_\mu(2B_R)^2\,\mu(R)&\leq 2\,\Theta_\mu(2B_{R_0})^2 \\&\quad +
c(\ve,\eta)\!\sum_{R\in\ttt}\sum_{Q\in\tree(R)} \int_{F\cap\delta^{-1}B_Q}
\int_{\delta\ell(Q)}^{\delta^{-1}\ell(Q)}
\Delta_\mu(x,r)^2\,\frac{dr}r\,d\mu(x).\nonumber
\end{align}
By the finite overlap of the domains of the last integrals for $Q\in\DD(R_0)$, we derive
\begin{align*}
\sum_{R\in\ttt}\sum_{Q\in\tree(R)} \int_{F\cap\delta^{-1}B_Q}
\int_{\delta\ell(Q)}^{\delta^{-1}\ell(Q)}&
\Delta_\mu(x,r)^2\,\frac{dr}r\,d\mu(x)\\
& \leq \sum_{Q\in\DD} \int_{F\cap\delta^{-1}B_Q}
\int_{\delta\ell(Q)}^{\delta^{-1}\ell(Q)}
\Delta_\mu(x,r)^2\,\frac{dr}r\,d\mu(x)\\
&\leq c(\delta)\int_F\int_0^1\Delta_\mu(x,r)^2\,\frac{dr}r\,d\mu(x),
\end{align*}
which together with \rf{eqfif23} yields \rf{eqsum441}.
\end{proof}
\vv

From the preceding lemma we deduce that for $\mu$-a.e.\ $x\in R_0$,
\begin{equation}\label{eqsug9}
\sum_{R\in\ttt:x\in R} \Theta_\mu(2B_R)^2<\infty.
\end{equation}
For a given $x\in R_0\setminus \bigcup_{Q\in\BZ^1\cup\BZ^2}Q$ such that
\rf{eqsug9} holds, 
let $R_0,R_1,R_2,\ldots$ be the cells from $\ttt$ such that $x\in R_i$. Suppose that
this is an infinite sequence and assume that $R_0\supset R_1\supset R_2\supset\ldots$,
so that for each $i\geq0$, $R_{i+1}\in\MD(Q)$ for some $Q\in\sss(R_{i})$.
From the property (b) in the Main Lemma and Lemma \ref{lemcad22} it follows that
$$\Theta_\mu(B(x,r))\leq c\,\Theta_\mu(2B_{R_i})\qquad \mbox{ for
$\ell(R_{i+1})\leq r\leq \ell(R_i)$,}$$
with $c$ depending on the constant $A$. 
As a consequence, 
$$\Theta^{1,*}(x,\mu)\leq c\,\limsup_{i\to\infty}\Theta_\mu(2B_{R_i}).$$
From \rf{eqsug9}, we infer that the limit on the right hand side above
vanishes and so $\Theta^{1,*}(x,\mu)=0$.
So we have shown that for any $x\in R_0$ satisfying \rf{eqsug9}, the condition $\Theta^{1,*}(x,\mu)>0$ implies that the collection
of cells $R\in\ttt$ which contain $x$ is finite.

Given $R\in\ttt$, denote by $\ttt(R)$ the collection of cells from $\ttt$ which are strictly contained
in $R$ and are maximal with respect to the inclusion. That is, 
$$\ttt(R) = \bigcup_{Q\in\sss(R)}\MD(Q).$$
Note that by the property (a) in the Main Lemma and the above construction, if $R\in\ttt\setminus
\BZ^1\cup\BZ^2$ and \rf{eqassu2} holds, then there exists a set $Z_R$ of $\mu$-measure $0$ and a set 
$W_R\subset \Gamma_R$ such that
\begin{equation}\label{eqfhe2}
R\subset Z_R\cup W_R\cup\bigcup_{Q\in\ttt(R)} Q,
\end{equation}
with $\mu|_{W_R}$ being absolutely continuous with respect to $\HH^1|_{\Gamma_R}$.
On the other hand, if $R\in\ttt\setminus
\BZ^1\cup\BZ^2$ and \rf{eqassu2} does not hold, then
\begin{equation}\label{eqfhe3}
R=Z_R\cup \bigcup_{Q\in\ttt(R)} Q
\end{equation}
for some set $Z_R$ of $\mu$-measure $0$.

Suppose now that $\Theta^{1,*}(x,\mu)>0$, that
\begin{equation}\label{eqshh1}
x\in R_0\setminus \Biggl(\bigcup_{R\in \ttt} Z_R \bigcup_{Q\in\BZ^1\cup\BZ^2}Q\Biggr),
\end{equation} 
and that \rf{eqsug9} holds. Note that the set of such points is a subset of full $\mu$-measure of 
$R_0\setminus \bigcup_{Q\in\BZ^1\cup\BZ^2}Q$. Let $R_n$ be the smallest cell from $\ttt$ which
contains $x$. Since $x\not\in\bigcup_{Q\in\BZ^1\cup\BZ^2}Q$, we have $R_n\not\in \BZ^1\cup\BZ^2$. 
So either \rf{eqfhe2} or \rf{eqfhe3} hold for $R_n$. Since $x\not\in Z_{R_n}$ and 
$x$ does not belong to any cell from $\ttt(R_n)$ (by the choice of $R_n$), we infer that 
we are in the case \rf{eqfhe2} (i.e.\ $R_n$ is a cell for which \rf{eqfhe2} holds) and $x\in W_{R_n}\subset
\Gamma_{R_n}$.
Thus the subset of points $x$ with $\Theta^{1,*}(x,\mu)>0$
satisfying \rf{eqshh1} and \rf{eqsug9} is contained in $\bigcup_{n}W_{R_n}$, which is a rectifiable set such
that $\mu|_{\bigcup_{n}W_{R_n}}$ is absolutely continuous with respect to $\HH^1$.
\fiproof
\vv

\vv


\section{The stopping cells for the proof of Main Lemma \ref{mainlemma}} \label{sec7}

\subsection{The good and the terminal cells}

The remaining part of this paper, with the exception of Sections \ref{secauchy1}-\ref{secauchy3}, is devoted to the proof of Main Lemma \ref{mainlemma}.
We assume that the constants $A_0$ and $C_0$ from the construction of the dyadic lattice of David and Mattila are fixed (unless otherwise stated) and we consider them to be absolute constants. On other hand, we will keep track
of the parameters $\ve$, $\delta$, $\eta$, $A$ in the statement of the Main Lemma.

The main task in this section consists in the construction of the stopping cells from 
$\DD$, which later will be used in the construction of the curve $\Gamma_R$ of the Main Lemma.

First we introduce the notation $\GG(Q_1,Q_2,\delta,\eta)$ for $Q_1,Q_2\in\DD$ and $\delta,\eta>0$.
This is the set of the points 
$x\in\R^d$ such that
\begin{equation}\label{eqgqd2}
\int_{\delta\,\ell(Q_1)}^{\delta^{-1}\ell(Q_2)} \Delta_\mu(x,r)^2\,\frac{dr}r\leq 
\eta\,\Theta_\mu(2B_{Q_2})^2.
\end{equation}
Note that $\GG(Q,Q,\delta,\eta)=\GG(Q,\delta,\eta)$.

Let $R\in\DD$ be as in the Main Lemma \ref{mainlemma}. 
We denote $x_0=z_R$ (this is the center of $R$) and $r_0=r(B_R)$, so that $B(x_0,r_0)=B_R$,
and thus 
$$R\subset B(x_0,r_0),\qquad r_0\approx \ell(R).$$

Now we need to define some families of stopping cells which are not good for the
construction of the curve mentioned above. Let $A, \tau>0$ be some constants to be fixed below.
We assume $\tau$ to be very small, with $\tau\leq 10^{-30}$ say, and $A\geq100$. 
Moreover, we let $K\geq100$ be some big absolute constant (probably $K=10^4$ suffices) which depends on the ambient dimension $d$ but not on the 
 other constants $\delta,\eta,\tau,A$. 
 The reader should think that  
  \begin{equation}\label{eqconstants*} 
\eta\ll\delta\ll\tau\ll A^{-1}\ll K^{-1}\ll 1.
\end{equation}

\begin{itemize}
\vv

\item A cell $Q\in\DD$ belongs to $BCF_0$ if $\ell(Q)\leq \ell(R)$ and either
$$\mu(Q\setminus F) \geq \eta^{1/2}\,\mu(Q)\quad\text{ or }
\quad\mu(\lambda B_Q \setminus F)\geq \eta^{1/2}\,\mu(\lambda B_Q) \quad\mbox{for some $1.1\leq\lambda \leq\delta^{-1/2}$.}
$$
\item A cell $Q\in\DD$ belongs to $LD_0$ if $\ell(Q)\leq \ell(R)$, $Q\not\in 
BCF_0$, and
$$\Theta_\mu(1.1B_Q)\leq \tau \,\Theta_\mu(B_R).$$

\item A cell $Q\in\DD$ belongs to $HD_0$ if $\ell(Q)\leq \ell(R)$, $Q\not\in BCF_0$, and
$$\Theta_\mu(1.1B_Q)\geq A \,\Theta_\mu(1.1B_R).$$

\item A cell $Q\in\DD$ belongs to $BCG_0$ if $Q\not\in BCF_0\cup LD_0\cup HD_0$, $\ell(Q)\leq \ell(R)$, and
$$\mu(\delta^{-1/2}B_Q\cap F\setminus \GG(Q,R,\delta^{1/2},\eta))\geq \eta \,\mu(\delta^{-1/2}B_Q\cap F).$$

\item A cell $Q\in\DD$ belongs to $BS\Delta_0$ if $Q\not\in BCF_0\cup LD_0\cup HD_0\cup BCG_0$, $\ell(Q)\leq \ell(R)$, and
$$\sum_{P\in \DD: Q\subset P\subset R} \frac1{\mu(1.1B_P)} \int_{1.1B_P\cap F} \int_{\delta\,\ell(P)}^{\delta^{-1}
\ell(P)}\Delta_\mu(x,r)^2\,\frac{dr}r\,d\mu
\geq \eta\,\Theta_\mu(B_R)^2.$$


\vv

 \end{itemize}

Next we consider the subfamily of $BCF_0\cup LD_0\cup HD_0\cup BCG_0\cup BS\Delta_0$ of the cells which are maximal with respect to inclusion (thus they are disjoint), and we call it $\term$. 
We denote by $\BCF$ the subfamily 
of the cells from $\term$ which belong by $BCF_0$, and by $\LD$, $\HD$, $\BCG$, $\BSD$, $\BSB$ the subfamilies of the cells from $\term$ which belong to  $LD_0$, $HD_0$, $BCG_0$, and $BS\Delta_0$, respectively.
Notice that we have the disjoint union
$$\term = \BCF \cup \LD \cup\HD\cup\BCG\cup \BSD.$$
The notations $\BCF$, $\LD$, $\HD$, $\BCG$, and $\BSD$ stand for ``big complement of $F$'', ``low density'', and ``high density'',  ``big
complement of $G$'', and ``big sum of $\Delta$ coefficients'', respectively; and $\term$ for
``terminal''.

We denote by $\good$ the subfamily of the cells $Q\subset B(x_0,\frac1{10} K r_0)$ 
with $\ell(Q)\leq\ell(R)$ such that there
does not exist any cell $Q'\in\term$ with $Q'\supset Q$. Notice that $\term\not\subset\good$ while, on the other hand, $R\in\good$.




\subsection{Some basic estimates}
The following statement is an immediate consequence of the construction.

\begin{lemma}\label{lemdobbb1}
If $Q\in\DD$, $\ell(Q)\leq\ell(R)$, and $Q$ is not contained in any cell from $\term$ 
(and so in particular, if $Q\in\good$), then 
$$\tau\,\Theta_\mu(B_R)\leq \Theta_\mu(1.1B_Q) \leq A\,\Theta_\mu(B_R).$$
\end{lemma}
\vv

\begin{lemma}\label{lemdobbb}
If $Q\in\DD$, $Q\subset B(x_0,K^2r_0)$, $\ell(Q)\leq\ell(R)$, and $Q$ is not contained in any cell from $\term$ (and so in particular, if $Q\in\good$), then
\begin{equation}\label{eqdob75}
\mu(aB_Q)\leq c(a)\,\mu(Q)
\end{equation}
for any $a\geq1$ such that $r(a\,B_Q)\leq \delta^{-3/4}r(B_R)$, assuming that the constant $C_0$ in the construction of 
the lattice $\DD$ in Lemma \ref{lemcubs} is big enough.
\end{lemma}

\begin{proof}
First we will show that
\begin{equation}\label{eqdjg32}
\mu(aB_Q)\leq c(a)\,\mu(3.3B_Q)
\end{equation}
for $a$ as in the lemma.
Since $Q$ is not contained in any cell from $\BCF\cup\BSD$, we have
$$\mu(1.1B_Q\setminus F) < \eta^{1/2}\,\mu(1.1B_Q)\quad \mbox{and}\quad
\frac1{\mu(1.1B_Q)} \int_{1.1B_Q\cap F} \int_{\delta\,\ell(Q)}^{\delta^{-1}
\ell(Q)}\!\!\Delta_\mu(x,r)^2\,\frac{dr}r\,d\mu\leq \eta\,\Theta_\mu(B_R)^2.$$
Thus
$$\int_{1.1B_Q\cap F} \int_{\delta\,\ell(Q)}^{\delta^{-1}
\ell(Q)}\Delta_\mu(x,t)^2\,\frac{dt}t\,d\mu\leq \eta\,\Theta_\mu(B_R)^2\,\mu(1.1B_Q)\leq 
2\,\eta\,\Theta_\mu(B_R)^2\,\mu(1.1B_Q\cap F).$$
Hence there exists $y_0\in 1.1B_Q\cap F$ such that 
$$\int_{\delta\,\ell(Q)}^{\delta^{-1}
\ell(Q)}\Delta_\mu(y_0,t)^2\,\frac{dt}t\,d\mu\leq 2\,\eta\,\Theta_\mu(B_R)^2.$$
Take $r$ such that $2.2r(B_Q)\leq r\leq \delta^{-1}\ell(Q)/2$. For these $r$'s we have
$B(y_0,r)\supset 1.1B_Q$ and thus
$$\Theta_\mu(B(y_0,r))\geq c(\tau,\delta)\,\Theta_\mu(B_R),$$
and thus, by Lemma \ref{lemdob}, 
\begin{equation}\label{eqclau721}
\mu(B(y_0,2r))\leq 9\,\mu(B(y_0,r)) \quad \mbox{\;for \,$2.2r(B_Q)\leq r\leq \delta^{-1}\ell(Q)/2$.}
\end{equation}
Iterating this estimate we deduce that
$$\mu(B(y_0,a\,r)) \leq c(a)\,\mu(B(y_0,r))\leq c(a)\,\mu(3.3B_Q)\qquad \mbox{for $a\,r\leq \delta^{-1}\ell(Q)/4$,}$$
since $B(y_0,2.2r(B_Q))\subset 3.3B_Q$.
Applying this estimate also to the ancestors of $Q$, \rf{eqdjg32} follows. Observe that $c(a)$ does not depend on $C_0$.

To prove \rf{eqdob75} it is enough to show that 
\begin{equation}\label{eqdobn48}
\mu(3.3B_Q)\leq c\,\mu(Q).
\end{equation}
 Note that by the property \rf{eqdob23}
of the cells of David and Mattila, if $Q\in\DD\setminus \DD^{db}$, then 
\begin{equation}\label{eqdobhfk80}
\mu(3.3B_Q)= \mu(28\cdot3.3B(Q))\leq
\mu(100B(Q))\leq C_0^{-1}\,\mu(100^{2}B(Q))\quad
\mbox{if $C_0\geq 100$.}
\end{equation}
By \rf{eqdjg32}, for a cell $Q$ satisfying the assumptions in the lemma we have 
$$\mu(100^2B(Q))\leq c\,\mu(3.3B_Q)$$
with $c$ independent of $C_0$.
Thus \rf{eqdobhfk80} does not hold if $C_0$ is chosen big enough. Hence $Q\in\DD^{db}$ and then 
$$\mu(3.3B_Q)\leq \mu(100B(Q)) \leq C_0\,\mu(B(Q))\leq C_0\,\mu(Q),$$
and so
\rf{eqdobn48} holds.
\end{proof}

\vv

\begin{remark}\label{remdens}
Let $Q$ be as in the preceding lemma. That is, $Q\in\DD$, $Q\subset B(x_0,K^2r_0)$, $\ell(Q)\leq\ell(R)$, and $Q$ is not contained in any cell from $\term$.
We showed in \rf{eqclau721} that there exists some $y_0\in1.1B_Q\cap F$ such that
$$\mu(B(y_0,2r))\leq 9\,\mu(B(y_0,r)) \qquad \mbox{$2.2r(B_Q)\leq r\leq \delta^{-1}\ell(Q)/2$.}$$
From this estimate it follows that
\begin{equation}\label{eqclau723}
\mu(b\,B_Q)\leq C(a,b)\,\mu(a\,B_Q)
\qquad \mbox{for $3.3\leq a\leq b\leq \delta^{-1/2}$,}
\end{equation}
with the constant $C(a,b)$ independent of the constant $C_0$ from the construction of the David-Mattila cells. This fact will 
be very useful later.
On the contrary, the constant in the inequality \rf{eqdobn48} depends on $C_0$.
\end{remark}

\vv
\begin{lemma}\label{lempocbc}
If $\eta$ is small enough (with $\eta\ll\delta$), then
$$\mu\biggl(\,\bigcup_{Q\in\BCF:Q\subset R}Q\biggr)\leq c\,\eta^{1/4}\,\mu(R).$$
\end{lemma}

\begin{proof}
The arguments are similar to the ones of Lemma \ref{lemaux21}. Denote by
$\BZ^1_R$ the family of cells $Q\in\DD$ which are contained in $R$
and satisfy 
$$\mu(\lambda B_Q \setminus F)\geq \eta^{1/2}\,\mu(\lambda B_Q)$$
for some $1.1\leq\lambda \leq\delta^{-1/2}$,
and by $\BZ^2_R$ the family of the ones that are contained in $R$
and satisfy 
$$
\mu(Q\setminus F) \geq \eta^{1/2}\,\mu(Q).$$

To deal with the cells from $\BZ^1_R$ we consider the maximal operator $M_{**}$ (which is a variant of $M_*$, introduced in \rf{eqm*}):
$$
M_{**}f(x)=\sup_{B\text{ ball}:\,x\in  B} \frac1{\mu(1.1B)}\int_{1.1B}|f|\,d\mu.
$$
Similarly to $M_*$, this operator is  bounded from $L^1(\mu)$ to $L^{1,\infty}(\mu)$.
 It turns out that for all $x\in Q\in\BZ^1_R$,
$M_{**}\chi_{c\,\delta^{-1/2}B_R\setminus F}(x) \geq \eta^{1/2}$,
because $\delta^{-1/2}B_Q\subset c\,\delta^{-1/2}B_R$, for some absolute constant $c$.
So we have
\begin{align*}
\mu\Biggl(\bigcup_{R\in\BZ^1_R} R\Biggr) &\leq \mu\bigl(\{x\in\R^d:\,M_{**}\chi_{c\,\delta^{-1/2}B_R\setminus F}
(x)\geq \eta^{1/2}\}\bigr)\\
& \leq c\,\frac{\mu(c\,\delta^{-1/2}B_R\setminus F)}{\eta^{1/2}}
\leq c\,\eta^{1/2}\,\mu(c\,\delta^{-1/2}B_R).
\end{align*}
By Lemma \ref{lemdobbb}, we know that 
$$\mu(c\,\delta^{-1/2}B_R)\leq  c(\delta)\,\mu(R).$$
Hence,
\begin{equation}\label{eqdhh31}
\mu\Biggl(\bigcup_{R\in\BZ^1_R} R\Biggr)\leq c'(\delta)\,\eta^{1/2}\,\mu(R)\leq \frac12\,\eta^{1/4}\,\mu(R),
\end{equation}
assuming $\eta$ enough (depending on $\delta$). 

To deal with the cells from $\BZ^2_R$, we argue with the maximal dyadic 
operator $M^d$ defined in \rf{eqmd*}.
Indeed, since every $Q\in\BZ_R^2$ is contained in 
$\{x\in\R^d:\,M^d\chi_{R\setminus F}
(x)\geq \eta^{1/2}\}$, we have
\begin{equation}\label{eqdhh32}
\mu\Biggl(\bigcup_{R\in\BZ^2_R} R\Biggr) \leq \mu\bigl(\{x\in\R^d:\,M^d\chi_{R\setminus F}
(x)\geq \eta^{1/2}\}\bigr) \leq c\,\frac{\mu(R\setminus F)}{\eta^{1/2}}
\leq c\,\eta^{1/2}\,\mu(R).
\end{equation}
Adding the estimates \rf{eqdhh31} and \rf{eqdhh32} the lemma follows.
\end{proof}

\vv
\begin{lemma}\label{lembcs}
For all $Q\in\BCG$,
$$ \Theta_\mu(B_R)^2\mu(\delta^{-1/2}B_Q)\leq 
 \frac{2}{\eta^2}\int_{\delta^{-1}B_Q\cap F}
\int_{\delta\ell(Q)}^{\delta^{-1}\ell(R)}\Delta_\mu(x,r)^2\,\frac{dr}r\,d\mu(x).$$
\end{lemma}

\begin{proof}
For all $x\in \delta^{-1}B_Q\cap F\setminus \GG(Q,R,\delta,\eta)$ we have
$$\int_{\delta\,\ell(Q)}^{\delta^{-1}\ell(R)} \Delta_\mu(x,r)^2\,\frac{dr}r>
\eta\,\Theta_\mu(B_R)^2.$$
Thus, integrating on $\delta^{-1}B_Q\cap F\setminus \GG(Q,R,\delta,\eta)$ and 
taking into account that $Q\in\BCG$ we get
\begin{align*}
\int_{\delta^{-1}B_Q\cap F\setminus \GG(Q,R,\delta,\eta)} \int_{\delta\,\ell(Q)}^{\delta^{-1}\ell(R)} \Delta_\mu(x,r)^2\,\frac{dr}r\,d\mu(x) & \geq 
\eta\,\Theta_\mu(B_R)^2\,\mu\bigl(\delta^{-1}B_Q\cap F\setminus \GG(R,\delta,\eta)\bigr)\\
&
\geq \eta^2\,\Theta_\mu(B_R)^2\,\mu(\delta^{-1}B_Q\cap F).
\end{align*}
Since $Q\not\in\BCF$, we have $\mu(\delta^{-1/2}B_Q\cap F)\geq (1-\eta^{1/2})\,\mu(\delta^{-1/2}B_Q)\geq\frac12\mu(\delta^{-1/2}B_Q)$,
and the lemma follows.
\end{proof}
\vv


\subsection{The regularized family $\reg$ and the family $\qgood$.}

The cells from $\term$ have the inconvenient that their side lengths may change drastically even if 
they are close to each other. For this reason it is appropriate to introduce a regularized version of
this family, which we will call $\reg$. 
 The first step for the construction consists in introducing the following
auxiliary function $d:\R^d\to[0,\infty)$:
\begin{equation}\label{eqdefdx}
d(x) = \inf_{Q\in\good} \bigl(|x-z_Q| + \ell(Q)\bigr).
\end{equation}
Note that $d(\cdot)$ is a $1$-Lipschitz function because it is the infimum of a family of $1$-Lipschitz
functions. 

We denote 
$$W_0=\{x\in\R^d:d(x)=0\}.$$
For each $x\in E\setminus W_0$ we take the largest cell $Q_x\in\DD$ 
such that $x\in Q_x$ with
\begin{equation}\label{eqdefqx}
\ell(Q_x) \leq \frac1{60}\,\inf_{y\in Q_x} d(y).
\end{equation}
We consider the collection of the different cells $Q_x$, $x\in E\setminus W_0$, and we denote it by $\reg$. Also, we let $\qgood$ (this stands
for ``quite good'') be
the family of cells $Q\in \DD$ such that $Q$ is contained in $B(x_0,2 K r_0)$ and $Q$ is not strictly contained in any
cell of the family $\reg$. Observe that $\reg\subset \qgood$.

The family $\sss(R)$ described in the Main Lemma is made up of the cells from the family $\reg$ which are contained in $R$. That is,
$$\sss(R):=\DD(R)\cap \reg.$$
 To simplify notation, from now on we will write $\sss$ instead of $\sss(R)$ and, analogously,
 $\tree$ instead of $\tree(R)$.

\vv

\begin{lemma}\label{lem74}
The cells from $\reg$ are pairwise disjoint and satisfy the following properties:
\begin{itemize}
\item[(a)] If $P\in\reg$ and $x\in B(z_{P},50\ell(P))$, then $10\,\ell(P)\leq d(x) \leq c\,\ell(P)$,
where $c$ is some constant depending only on $A_0$. In particular, $B(z_{P},50\ell(P))\cap W_0=\varnothing$.

\item[(b)] There exists some absolute constant $c$ such that if $P,P'\in\reg$ and $B(z_{P},50\ell(P))\cap B(z_{P'},50\ell(P'))
\neq\varnothing$, then
$$c^{-1}\ell(P)\leq \ell(P')\leq c\,\ell(P).$$
\item[(c)] For each $P\in \reg$, there at most $N$ cells $P'\in\reg$ such that
$$B(z_{P},50\ell(P))\cap B(z_{P'},50\ell(P'))
\neq\varnothing,$$
 where $N$ is some absolute constant.
 
 \item[(d)] If $x\not\in B(x_0,\frac1{8} K r_0)$, then $d(x)\approx |x-x_0|$. Thus,
 if $P\in\reg$ and $B(z_{P},50\ell(P))\not\subset  B(x_0,\frac1{8} K r_0)$, then $\ell(P)\gtrsim  K r_0$.
\end{itemize}
\end{lemma}

\begin{proof}
To prove (a), consider  $x\in B(z_{P},50\ell(P))$. Since $d(\cdot)$ is $1$-Lipschitz
and, by definition, $d(z_{P})\geq 60\,\ell(P)$, we have
$$d(x)\geq d(z_{P}) - |x-z_{P}| \geq d(z_{P}) - 50\,\ell(P)\geq 10 \,\ell(P).$$

To prove the converse inequality, by the definition of $\reg$, there exists some $z'\in
\wh P$, the parent of $P$, such that 
$$d(z')\leq 60\,\ell(\wh P)\leq 60\,A_0\,\ell(P).$$
Also, we have
$$|x-z'|\leq |x-z_{P}| + |z_{P}-z'|\leq 50\,\ell(P) + A_0\,\ell(P).$$
Thus,
$$d(x)\leq d(z') + |x-z'| \leq (50+ 61\,A_0)\,\ell(P).$$

The statement (b) is an immediate consequence of (a), and (c) follows easily from (b).

Finally, the first assertion in (d) follows from the fact that all the cells from $\good$ are contained in $B(x_0,\frac1{10} K r_0)$, by definition. Together with (a), this yields that
 if $B(z_{P},50\ell(P))\not\subset  B(x_0,\frac1{8} K r_0)$, then $\ell(P)\gtrsim  K r_0$.
\end{proof}

\vv
\begin{lemma}\label{lemregterm}
Every cell $Q\in\reg$ with $Q\subset B(x_0,\frac1{10} K r_0)$ is contained in some cell $Q'\in\term$.
\end{lemma}

\begin{proof}
Suppose that $Q$ is not contained in such a cell $Q'$. This means that
$Q\in\good$. Then, by the definition of $d(\cdot)$ in \rf{eqdefdx}, for every $x\in Q$ we have
$d(x)\leq \diam(Q)+\ell(Q)\leq 2\ell(Q)$. Thus, by \rf{eqdefqx}, $\ell(Q_x)< \ell(Q)$, and so $Q\not\in
\reg$.
\end{proof}

\vv

\begin{lemma}\label{claf22}
There exists some absolute constant $c_4>2$ such that for every cell $Q\in \qgood$ 
contained in $B(x_0,2K\,r_0)$ there exist $Q'\in\good$ such with $\ell(Q')\approx
\ell(Q)$ such that 
$2B_{Q'}\subset c_4 B_Q$. Further
the following holds:
\begin{equation*}\label{eqaa319''}
\tau\,\Theta_\mu(B_{R})\lesssim \Theta_{ \mu}(c_4 B_Q)\lesssim A\,\Theta_\mu(B_R).
\end{equation*}
\end{lemma}

\begin{proof}
The first statement is consequence of the construction of the family $\reg$.
The second one  follows from the first one, together with the doubling properties of $1.1B_{Q'}$
(by Lemma \ref{lemdobbb}) and the fact that
$$\tau\,\Theta_\mu(B_{R})\leq \Theta_{ \mu}(1.1 B_{Q'})\leq A\,\Theta_\mu(B_R).$$
\end{proof}

\vv
\begin{lemma}\label{lemqs11}
If $Q\in\qgood$ and $Q\subset B(x_0,K\,r_0)$, then there exists some ball $\wt B_Q$ containing
$2B_Q$, with radius $r(\wt B_Q)\leq c_5\,\ell(Q)$ (where $c_5\geq1$ is some absolute constant)
which satisfies the following properties:
\begin{itemize}
\item[(a)] Denote by $G(\wt B_Q)$ the subset of points $x\in\R^d$ such that
$$\int_{\delta^{1/2} r(\wt B_Q))}^{\delta^{-1/2}r(\wt B_Q)} |\Delta_\mu(x,r)|^2\,\frac{dr}r
\leq \eta^{1/4}\,\Theta_\mu(\wt B_Q)^2.$$
Then we have
$$\mu(\delta^{-1/4}\wt B_Q\setminus G(\wt B_Q)) \leq \eta^{1/4}\,\mu(\delta^{-1/4}B_{Q'}),$$
and moreover $\mu(\wt B_Q\cap G(\wt B_Q))>0$.

\item[(b)]
If $\ve_0>0$ is some arbitrary (small) constant, assuming $\eta$ and $\delta$ small enough (depending on $\ve_0$), we have
$$\alpha_\mu(2\wt B_Q)\leq \ve_0.$$ 

\item[(c)] For any $a\geq1$ such that $r(a\,\wt B_Q)\leq \delta^{-3/4}r(B_R)$.
$$\mu(a\wt B_Q)\leq c(a)\,\mu(\wt B_Q).$$
\end{itemize}
\end{lemma}

\begin{proof}
By the definition of the cells from $\qgood$, there exists some $Q'\in\good$ such that
$2B_Q\subset c\,B_{Q'}$, for some absolute constant $c\geq 2$.
Since $Q'$ is good, by construction it satisfies
$$\mu(\delta^{-1/2} B_{Q'} \setminus F)< \eta^{1/2}\,\mu(\delta^{-1/2}B_{Q'})$$
and
$$\mu(\delta^{-1/2}B_{Q'}\cap F\setminus \GG(Q',R,\delta^{1/2},\eta))< \eta \,\mu(\delta^{-1/2}B_{Q'}\cap F).$$
From these two estimates we infer that 
$$\mu(\delta^{-1/2}B_{Q'}\setminus \GG(Q',R,\delta^{1/2},\eta))< 2
\eta^{1/2}\,\mu(\delta^{-1/2}B_{Q'}).$$
Further, by Lemma \ref{lemdobbb}, if $\eta\ll\delta$ we get $2
\eta^{1/2}\,\mu(\delta^{-1/2}B_{Q'})<\mu(B_{Q'})$, and thus 
$$\mu(B_{Q'}\cap \GG(Q',R,\delta^{1/2},\eta))>0.$$
The first assertion of the lemma follows if we take $\wt B_Q=B_{Q'}$, noting that $\GG(Q',R,\delta^{1/2},\eta)
\subset G(\wt B_Q)$ if $\eta$ is small enough depending on $\tau$ (using that $\Theta_\mu(Q')\geq \tau\,\Theta_\mu(R)$). The second assertion is an immediate corollary of the first one and Lemma
\ref{lemcpt1}. The last one follows from Lemma \ref{lemdobbb}.
\end{proof}
\vv

From the preceding results, we obtain the following easily. We leave the proof for the reader.

\begin{lemma}\label{lemqs12}
Given $\ve_0>0$,
assume that $\eta$ and $\delta$ are small enough.
If $Q\in\qgood$ and $Q\subset B(x_0,2K\,r_0)$, then there exists some absolute constant $c_6\geq4$ such that, for any $a\geq c_6$ such that $a\,\ell(Q)\leq \ell(R)$,
\begin{equation}\label{eqaa319}
\tau\,\Theta_\mu(B_{R})\lesssim \Theta_{ \mu}(a B_Q)\lesssim A\,\Theta_\mu(B_R)
\end{equation}
and
\begin{equation}\label{eqaa3188}
\alpha_\mu(aB_Q)\lesssim\ve_0.
\end{equation}
\end{lemma}

\vv


\section{The measure $\wt\mu$ and some estimates about its flatness}\label{sec6.5}

We consider the set
\begin{equation}\label{eqetild}
\wt E = B(x_0,2Kr_0)\cap \biggl(W_0\cup \bigcup_{P\in\reg} \bigl[4B_{P} \cap F\cap
\GG(P,R,\delta^{1/4},\eta)\bigr]\biggr).
\end{equation}
 Then we set
$$\wt \mu = \mu|_{\wt E}.$$
Our first objective consists in showing that, in a sense, $\mu(E\cap B(x_0,2Kr_0)\setminus \wt E)$ is very small.

\vv
\begin{lemma}\label{lemk10}
If $P\in \reg$, then we have
$$\mu(4B_{P} \setminus \wt E)\leq \eta^{1/4}\,\mu(c_7 \,B_{P}),$$
with $c_7=5c_6$ (where $c_6$ appears in Lemma \ref{lemqs12}), where $c_7$ is some absolute constant, and 
we assume $\eta$ mall enough.
\end{lemma}

\begin{proof}
Note that 
$$\mu(4B_{P} \setminus \wt E) \leq 
\mu\bigl(4B_P \setminus (F\cap \GG(P,R,\delta^{1/4},\eta))\bigr),$$
By the definition of the cells from $\reg$, there exists some cell $Q\in\good$ with $\ell(Q)
\approx \ell(P)$ such that $4B_{P}\subset c_8\,B_Q$, where $c_8$ is some absolute constant. Since $\ell(Q)\approx\ell(P)$, we deduce that
$\GG(Q,R,\delta^{1/2},\eta)\subset \GG(P,R,\delta^{1/4},\eta)$, and thus
\begin{equation}\label{eqsj434}
\mu(4B_{P}\setminus (F\cap \GG(P,R,\delta^{1/4},\eta)))\leq
\mu(\delta^{-1/2}B_Q \setminus (F\cap \GG(Q,R,\delta^{1/2},\eta))).
\end{equation}

To estimate the right hand side above we take into account that
since $Q\not\in \BCF$,
$$\mu(\delta^{-1/2}B_Q\setminus F)\leq\eta^{1/2}\,\mu(\delta^{-1/2}B_Q),$$
and
 as $Q\not\in \BCG$, 
$$\mu(\delta^{-1/2}B_Q\cap F\setminus \GG(Q,R,\delta^{1/2},\eta))\leq \eta \,\mu(\delta^{-1/2}B_Q\cap F).$$
So we get
\begin{align*}
\mu(\delta^{-1/2}B_Q \setminus (F\cap \GG(Q,R,\delta^{1/2},\eta))) &\leq 
\mu(\delta^{-1/2}B_Q\setminus F) + \mu(\delta^{-1/2}B_Q\cap F\setminus \GG(Q,R,\delta^{1/2},\eta))\\
& \leq \eta^{1/2}\,\mu(\delta^{-1/2}B_Q) +\eta\,\mu(\delta^{-1/2}B_Q\cap F)
\\& \leq 2\,\eta^{1/2}\,\mu(\delta^{-1/2}B_Q).
\end{align*}

Gathering the estimates above, we obtain
$$\mu(4B_{P} \setminus \wt E) \leq 2\,\eta^{1/2}\,\mu(\delta^{-1/2}B_Q).$$
By Lemma \ref{lemdobbb}, we know that 
$$\mu(\delta^{-1/2}B_Q)\leq c(\delta)\,\mu(B_Q)\leq c(\delta)\,A\,\Theta_\mu(R)\,\ell(Q).$$
On the other hand, since $c_7>c_6$, by Lemma \ref{lemqs12},
$$\mu(c_7\,B_{P})\gtrsim \tau\,\Theta_\mu(R)\,\ell(P)\gtrsim\,\tau\,\Theta_\mu(R)\,\ell(Q).$$
Thus we derive
$$\mu(4 B_{P} \setminus \wt E) \leq c(\delta)\,\eta^{1/2}\,A\,\tau^{-1}\,\mu(c_7\,B_{P}).$$
If $\eta$ is small enough, we get the desired conclusion.
\end{proof}

\vv
\begin{lemma}\label{lemk11}
Let $Q\in \good$ and 
let $a\geq 2$ with $r(aB_Q)\leq K\,\ell(R)$. Denote by $\BZ(aB_Q)$ the subcollection of cells $P$ from 
$\reg$ which intersect $aB_Q$ and satisfy
\begin{equation}\label{eqhyp4}
\mu(4B_P)\leq \eta^{1/10} \mu(c_7\,B_{P}).
\end{equation}
Then we have
$$\sum_{P\in\BZ(aB_Q)}\mu(P)\leq \eta^{1/4}\,\mu(aB_Q),$$
\end{lemma}

\begin{proof}
Note first that every $P\in\BZ(aB_Q)$ satisfies $r(P)\leq c\,r(aB_Q)$. In fact, $Q$ contains either some point from
$W_0$ or some cell $P'\in\reg$, and if $P$ were too big, we would have too close cells with very different sizes
(or a cell and a point from $W_0$), which would contradict the properties (a) or (b) of 
Lemma \ref{lem74}. As a consequence of the fact that $r(P)\leq c\,r(aB_Q)$, we infer that $P\subset c'aB_Q$, for some absolute constant $c'$.

We consider two types of cells $P\in\BZ(aB_Q)$. We set $P\in \BZ_1(aB_Q)$ if $\mu(P\setminus F)>\frac12\,\mu(P)$, and
$P\in \BZ_2(aB_Q)$ otherwise. Taking into account that $Q\not\in BF_0$ (because $Q\in\good$), we derive
$$\sum_{P\in \BZ_1(aB_Q)}\mu(P)\leq 2\sum_{P\in \BZ_1(aB_Q)}\mu(P\setminus F)\leq 2\,\mu(c'aB_Q\setminus F)\leq c\,\eta^{1/2}\mu(c'aB_Q).$$
By Lemma \ref{lemdobbb}, we have
$\mu(c'aB_Q)\leq c''\,\mu(aB_Q),$
and so we get
\begin{equation}\label{eqba111}
\sum_{P\in \BZ_1(aB_Q)}\mu(P)\leq c\,\eta^{1/2}\,\mu(aB_Q)\leq 
\frac12\,\eta^{1/4}\,\mu(aB_Q).
\end{equation}

Now turn our attention to the cells from $\BZ_2(aB_Q)$. Take $P\in\BZ_2(aB_Q)$ such that $\mu(P)>0$. We claim
that for every $x\in B_P$
\begin{equation}\label{eqcla9342}
\int_{2r(B_P)}^{4c_7r(B_P)} \Delta_\mu(x,t)^2\,\frac{dt}t\gtrsim \Theta_\mu(c_7B_P)^2.
\end{equation}
To see this, note first that for such $x$ and for $1\leq t\leq 2$ we have 
$B(x,t\,r(B_P))\subset 4B_P$. Let $N$ be the minimal integer such that 
$c_7B_P\subset B(x,2^Nr(B_P))$ for every $x\in B_P$. Obviously, $N$ is an absolute constant depending on $c_7$.
We write
\begin{align*}
c^{-1}\Theta_\mu(c_7B_P)-c\,\Theta_\mu(4B_P) & \leq \Theta_\mu(B(x,t2^Nr(B_P))) - \Theta_\mu(B(x,t\,r(B_P)))\\
&\leq \sum_{k=1}^{N-1}\bigl|\Theta_\mu(B(x,t2^k r(B_P)) - \Theta_\mu(B(x,t2^{k+1}r(B_P))\bigr|.
\end{align*}
From the assumption \rf{eqhyp4} it turns out that the left hand above is comparable to $\Theta_\mu(c_7B_P)$.
Therefore, we deduce that 
$$\Theta_\mu(c_7B_P)^2\leq c\sum_{k=1}^{N-1}\bigl|\Theta_\mu(B(x,t2^k r(B_P)) - \Theta_\mu(B(x,t2^{k+1}r(B_P))\bigr|^2= c\sum_{k=1}^{N-1}\Delta_\mu(x,t2^k r(B_P))^2,$$
with the constant $c$ possibly depending on $N$ and thus on $c_7$.
Integrating with respect to $t\in[1,2]$, \rf{eqcla9342} follows easily.

Integrating now \rf{eqcla9342} with respect to $\mu$ on $P\cap F$ and recalling that $\mu(P)\approx\mu(P\cap F)$, we get
\begin{align*}
\int_{P\cap F}\int_{2r(B_P)}^{4c_7r(B_P)} \Delta_\mu(x,r)^2\,\frac{dr}r\,d\mu(x) & \gtrsim \Theta_\mu(c_7B_P)^2\,\mu(P\cap F)\\
&\approx \Theta_\mu(c_7B_P)^2\,\mu(P)\gtrsim c(A,\tau)\,\Theta_\mu(B_R)^2\,\mu(P).
\end{align*}
Consider $S(P)\in\good$ such that $P\subset 4B_{S(P)}$ and $\ell(S(P))\approx\ell(P)$. Then, for $\delta$ small enough, we have
\begin{align*}
\int_{4B_{S(P)}\cap F}\int_{\delta\ell(S(P))}^{\delta^{-1}\ell(S(P))} \Delta_\mu(x,r)^2\,\frac{dr}r\,d\mu(x)
& \geq
\int_{P\cap F}\int_{2r(B_P)}^{4c_7r(B_P)} \Delta_\mu(x,r)^2\,\frac{dr}r\,d\mu(x)\\
&\gtrsim c(A,\tau)\,\Theta_\mu(B_R)^2\,\mu(P).
\end{align*}
Since $\ell(P)\approx\ell(S(P))$ and $P\subset 1.1B_{S(P)}$, for a given $S\in\good$, the number of cells $P\in\DD$ 
such that $S=S(P)$ does not exceed some fixed absolute constant. Moreover, it is easy to check that $S\subset caB_Q$
for some fixed $c>1$. Then we infer that
\begin{equation}\label{eqba222}
\Theta_\mu(B_R)^2\sum_{P\in \BZ_2(aB_Q)} \mu(P) \leq c(A,\tau) \sum_{S\in\good:S\subset caB_Q}
\int_{1.1B_{S}\cap F}\int_{\delta\ell(S)}^{\delta^{-1}\ell(S)} \Delta_\mu(x,r)^2\,\frac{dr}r\,d\mu(x).
\end{equation}
We estimate the right hand side above using the fact that the good cells are not in $BS\Delta_0$.
For a big positive integer
$k>J(Q)$, we write
\begin{align*}
\sum_{\substack{S\in\good: J(S)\leq k\\  S\subset caB_Q}}&
\int_{1.1B_{S}\cap F}\int_{\delta\ell(S)}^{\delta^{-1}\ell(S)}  \Delta_\mu(x,r)^2\,\frac{dr}r\,d\mu(x)\\
& = 
\sum_{\substack{S\in\good:J(S)\leq k\\ S\subset caB_Q}} \sum_{\substack{T\in\DD_k:\\T\subset S}}\frac{\mu(T)}{\mu(S)}
\int_{1.1B_{S}\cap F}\int_{\delta\ell(S)}^{\delta^{-1}\ell(S)} \Delta_\mu(x,r)^2\,\frac{dr}r\,d\mu(x)\\
& \leq c
\sum_{\substack{T\in\DD_k:\\T\subset caB_Q}} \mu(T)
\sum_{\substack{S\in\good:T\subset S \subset caB_Q}} \frac1{\mu(1.1B_S)}
\int_{1.1B_{S}\cap F}\int_{\delta\ell(S)}^{\delta^{-1}\ell(S)} \Delta_\mu(x,r)^2\,\frac{dr}r\,d\mu(x)\\
& \leq c
\sum_{\substack{T\in\DD_k:\\T\subset caB_Q}} \eta\,\Theta_\mu(B_R)^2\,\mu(T) \leq c\,\eta\,\Theta_\mu(B_R)^2\,\mu(aB_Q)\leq \frac12\,\eta^{1/4}\,\Theta_\mu(B_R)^2\,\mu(aB_Q).
\end{align*}
Since this holds uniformly for all $k>J(Q)$, we derive
$$
\sum_{\substack{S\in\good:  S\subset caB_Q}}
\int_{1.1B_{S}\cap F}\int_{\delta\ell(S)}^{\delta^{-1}\ell(S)}  \Delta_\mu(x,r)^2\,\frac{dr}r\,d\mu(x)
\leq \frac12\,\eta^{1/4}\,\Theta_\mu(B_R)^2\,\mu(aB_Q).
$$
This estimate, together with \rf{eqba111} and \rf{eqba222}, proves the lemma.
\end{proof}

\vv
\begin{lemma}\label{lemk12}
Let $Q\in \good$ 
and let $a\geq 1$ be such that $r(aB_Q)\leq K\,\ell(R)$. Then
$$\mu(aB_{Q} \setminus \wt E)\leq \eta^{1/10}\,\mu(a B_{Q}),$$
assuming $\eta$ small enough.
\end{lemma}

\begin{proof}
Denote by $\AZ(aB_Q)$ the subfamily of the cells from $\reg$ which intersect $aB_Q\setminus \wt E$. We have
$$
\mu(aB_Q\setminus \wt E)  \leq \sum_{P\in\AZ(aB_Q)} \mu(P\setminus \wt E).$$
We distinguish two cases according to wether or not the cells $P$ belong to $\BZ(aB_Q)$ (this is the family of cells introduced in Lemma \ref{lemk11}).
For the cells $P\in\AZ(aB_Q)$ which belong to $\BZ(aB_Q)$, we have shown that
\begin{equation}\label{eqsh255}
\sum_{P\in \BZ(aB_Q)} \mu(P) \leq \eta^{1/4}\,\mu(aB_Q).
\end{equation}
For the ones that do not belong to $\BZ(aB_Q)$, by Lemma \ref{lemk10} and
the finite superposition of the balls $4B_P$, $P\in\reg$, we have
\begin{align*}
\sum_{P\in\AZ(aB_Q)\setminus \BZ(aB_Q)} \mu(P\setminus \wt E)
&\leq \eta^{1/4} \sum_{P\in\AZ(aB_Q)\setminus \BZ(aB_Q)} \mu(c_7\,B_P) \\
& \leq \eta^{1/4}\,\eta^{-1/10}\,\sum_{P\in\AZ(aB_Q)\setminus \BZ(aB_Q)}\mu(4B_P)\\
& \leq
\eta^{3/20}\,\mu\biggl(\,\bigcup_{P\in \AZ(aB_Q)\setminus \BZ(aB_Q)} 4B_P\biggr).
\end{align*}

The same argument used in the previous lemma for the cells of $\BZ(aB_Q)$ shows that the cells from
$\AZ(aB_Q)$ are contained in $c'aB_Q$, for some absolute constant $c'$.
Thus we have
$$
\mu\biggl(\,\bigcup_{P\in \AZ(aB_Q)\setminus \BZ(aB_Q)} 4B_P\biggr)
 \leq c\,\eta^{3/20} \,\mu(c'aB_Q)\leq 
c''\,\eta^{3/20} \,\mu(aB_Q).$$
Adding this estimate and \rf{eqsh255}, the lemma follows, assuming $\eta$ small enough.
\end{proof}
\vv

Notice that, by the preceding lemma,
we have
\begin{equation}\label{eqaa320a}
\wt\mu(a B_{Q}) \geq (1-\eta^{1/10})\,\mu(a B_{Q})\qquad \mbox{for $Q\in\good$, $a\geq 1$, with $r(a\,B_Q)\leq K\,\ell(R)$.}
\end{equation}

\vv

\begin{lemma}\label{lemk88}
Let $Q\in\qgood$ and and let $Q'\in\good$ be such that $2B_Q\subset 2B_{Q'}$ and $\ell(Q)\approx
\ell(Q')$. For any $y\in Q\cap \wt E$, we have
\begin{equation}\label{eqdii21**}
\wt \mu(B(y,r)) \approx \mu(B(y,r))\approx \frac{\mu(Q')}{\ell(Q')}\,r \qquad \mbox{for $\delta^{1/5}\,\ell(Q)\leq r\leq \min(\delta^{-1/5}\ell(Q),\frac34K\ell(R))$.}
\end{equation}
Further, if $\wt\mu(Q)>0$, then 
\begin{equation}\label{eqdj743}
\wt \mu(2B_Q)\approx \mu(2B_Q)\approx \mu(Q').
\end{equation}
As a consequence,  for any $Q\in\qgood$ such that $\wt\mu(Q)>0$, we have
$$\Theta_{\wt\mu}(2B_Q)\gtrsim \tau\,\Theta_\mu(B_R).$$
\end{lemma}

\begin{proof}
By the definition of $\wt E$, either $y\in W_0$ or
 there exists some $P\in\reg$ such that
$$y\in 4B_{P}\cap Q\cap \GG(P,R,\delta^{1/4},\eta).$$
In the latter case, from Lemma \ref{lem74} it follows that $\ell(P)\lesssim \ell(Q)\approx\ell(Q')$.
Together with the definition of $\GG(P,R,\delta^{1/4},\eta)$ in \rf{eqgqd2}, this yields
$$
\int_{c\delta^{1/4}\,\ell(Q')}^{\delta^{-1/4}\ell(R)} \Delta_\mu(y,r)^2\,\frac{dr}r\leq
\int_{\delta^{1/4}\,\ell(P)}^{\delta^{-1/4}\ell(R)} \Delta_\mu(y,r)^2\,\frac{dr}r\leq 
\eta\,\Theta_\mu(2B_R)^2\leq c\,\eta\, \tau^{-2}\,\Theta_\mu(2B_{Q'})^2.
$$

In the case that $y\in W_0$, it is immediate to check that the last estimate also holds.
So in any case, by Remark \ref{remdob1}, we get
\begin{equation}\label{eqdii21}
\mu(B(y,r)) \approx \frac{\mu(2B_{Q'})}{\ell(Q')}\,r \qquad \mbox{for $c\,\delta^{1/4}\,\ell(Q')\leq r\leq c\,\delta^{-1/4}\ell(R)$,}
\end{equation}
assuming $\eta$ and $\delta$ small enough. This proves one of the comparabilities in \rf{eqdii21**}. For the remaining one, we apply Lemma \ref{lemk12}. Indeed, for $\delta^{1/5}\,\ell(Q)\leq r\leq \min(\delta^{-1/5}\ell(Q),\frac34K\ell(R))$
we have $B(y,r)\subset a\,B_{Q'}$, for some $a>0$ such that $a\,r(B_{Q'})\leq \min(\delta^{-1/5}r,K\ell(R))$ (assuming $K$ big enough if necessary), and thus
$$\mu(B(y,r)\setminus \wt E)\leq \mu(aB_{Q'}\setminus \wt E) \leq \eta^{1/10}\,\mu(aB_{Q'}).$$
Arguing as in \rf{eqdii21}, we easily get $\mu(aB_{Q'})\leq \mu(B(y,r))\,{a\,r(B_{Q'})}/r$, and then we obtain
$$\mu(B(y,r)\setminus \wt E)\lesssim \eta^{1/10}\,\mu(B(y,r))\,\frac{a\,r(B_{Q'})}r
\lesssim \eta^{1/10}\,\delta^{-1/5}\,\mu(B(y,r)).
$$
Thus, assuming $\eta\ll\delta$, we deduce that $\wt \mu(B(y,r))=\mu(B(y,r)\cap \wt E) \approx \mu(B(y,r))$ and so we are done with \rf{eqdii21**}.

To prove \rf{eqdj743}, we take $y$ as above and note that, in particular, by \rf{eqdii21**}
$\wt \mu(B(y,r(B_Q)))\approx \mu(B(y,r(B_Q)))\approx \mu(Q')$. Since $B(y,r(B_Q))\subset 2B_Q$, this implies that
$$\wt\mu(2B_Q)\approx\mu(2B_Q)\approx \mu(2B_Q')\approx \mu(Q').$$

The last statement of the lemma follows from \rf{eqdj743} and the fact that 
$\Theta_\mu(2B_{Q'})\gtrsim \tau\,\Theta_\mu(B_R)$.
\end{proof}

\vv
\begin{lemma}\label{lemk12**}
For given $\ve_0,\,\ve_0'>0$, if $\eta$ and $\delta$ are taken small enough, the following holds
for all $Q\in\good$ and $a\geq 1$ such that $r(a\,B_Q)\leq \frac34K\,\ell(R)$:
\begin{equation}\label{eqaa320}
\alpha_\mu(a B_Q)\lesssim \ve_0 \qquad \mbox{and}\qquad\alpha_{\wt\mu}(aB_Q)\lesssim \ve_0'.
\end{equation}
\end{lemma}

\begin{proof}
The first estimate in 
\rf{eqaa320} has already been seen in \rf{eqaa3188}.

To show that $\alpha_{\wt\mu}(a\, B_Q)\lesssim \ve_0'$, take 
a $1$-Lipschitz function $f$ supported on $aB_Q$.
Then we have
$$\left|\int f\,d\mu  -
\int f\,d\wt\mu\right| \leq 2r(a\,B_Q)\,\mu(a\,B_Q\setminus \wt E)\leq 2\,\eta^{1/10}\,r(a\,B_Q)\,\mu(a\,B_Q),$$
by \rf{eqaa320a}.
Thus,
$\dist_{a\,B_Q}(\mu,\wt\mu)\leq 2\,\eta^{1/10}\,r(a\,B_Q)\,\mu(a\,B_Q)$, and so
$$\alpha_{\mu}(a\,B_Q)\leq \alpha_{\wt\mu}(a\,B_Q) + 2\,\eta^{1/10}
\leq \ve_0',
$$
if $\eta$ is taken small enough.
\end{proof}
\vv

\begin{remark}
From the preceding proof, it follows that if $c_Q,L_Q$ minimize $\alpha_\mu(aB_Q)$, then
$$\dist_{a\,B_Q}(\wt\mu,c_Q\,\HH^1|_{L_Q})\leq c\,\ve_0'.$$
\end{remark}

\vv
In next lemma we extend the result stated in Lemma \ref{lemk12**} to the cells from $\qgood$.

\vv
\begin{lemma}\label{lemk16}
For given $\ve_0,\,\ve_0'>0$, if $\eta$ and $\delta$ are taken small enough, the following holds
for all $Q\in\qgood$ with $\wt\mu(Q)>0$ and $a\geq 2$ such that $r(a\,B_Q)\leq \frac12K\,\ell(R)$:
\begin{equation}\label{eqaa320qq}
\alpha_\mu(a B_Q)\lesssim \ve_0 \qquad \mbox{and}\qquad\alpha_{\wt\mu}(aB_Q)\lesssim \ve_0'.
\end{equation}
\end{lemma}

\begin{proof}
Let $S\in\qgood$ be such that $2B_S \subset aB_Q$ and $r(aB_{Q})\approx \ell(S)$, and let
$Q'\in\good$ be such that $aB_Q\subset 2B_{Q'}$ and $\ell(Q')\approx r(aB_{Q})$. 
Since $2B_S\subset 2B_{Q'}$ and $\ell(Q')\approx \ell(S)$, by Lemma \ref{lemk88} we have
$\wt \mu(2B_S)\approx  \mu(2B_{Q'})$. As $2B_S\subset aB_Q\subset 2B_{Q'}$, we infer that
$$\wt \mu(2B_S)\approx \wt\mu(aB_Q)\approx  \mu(2B_{Q'}).$$
Then we deduce
$$\alpha_\mu(aB_Q)\leq c\,\alpha_\mu(2B_{Q'}) \qquad\mbox{and}\qquad
\alpha_{\wt\mu}(aB_Q)\leq c\,\alpha_{\wt\mu}(2B_{Q'}),$$
and by Lemma \ref{lemk12**} we are done.
\end{proof}

We also have:

\begin{lemma}\label{claf23}
Let $\ve_0''>0$ be an arbitrary (small) constant.
Let $Q\in \qgood$ be such that $\wt\mu(Q) >0$. If $4\leq a\leq \delta^{-1/5}$ and $r(a\,B_Q)\leq \frac14K\,\ell(R)$, then
$$
b\beta_{\infty,\wt\mu}(aB_Q)\leq \ve_0'',$$
assuming $\delta$ and $\eta$ small enough. In fact,
\begin{equation}\label{eqdj56}
\sup_{x\in aB_Q} \frac{\dist(x,L_{a,Q})}{r(a\,B_Q)} + 
\sup_{x\in L_{a,Q}\cap aB_Q} \frac{\dist(x,\supp\wt\mu)}{r(a\,B_Q)}\leq \ve_0'',
\end{equation}
where $L_{a,Q}$ is the same line minimizing $\alpha_\mu(2aB_Q)$. 
\end{lemma}

\begin{proof}
We can assume $\alpha_{\mu}(2a\,B_Q)\leq \ve_0'''$, with $\ve_0'''$ as small as wished if $\eta$ and $\delta$ are
small enough. 
As shown in Lemma \ref{lempr1} this implies that 
\begin{equation}\label{eqdj57}
\frac1{\mu(a\,B_Q)}
\int_{aB_Q} \frac{\dist(y,L_{a,Q})}{r(a\,B_Q)}\,d\mu(y) + 
\int_{L_{a,Q}\cap aB_Q}\!\! \frac{\dist(x,\supp\mu)}{r(a\,B_Q)^2}\,d\HH^1|_{L_{a,Q}}(x)\lesssim \ve_0'''.
\end{equation}
From Lemma \ref{lemk12**} and the subsequent remark we also have
\begin{equation}\label{eqdj58}
\frac1{\mu(a\,B_Q)}\int_{aB_Q} \frac{\dist(y,L_{a,Q})}{r(a\,B_Q)}\,d\wt\mu(y) + 
\int_{L_{a,Q}\cap aB_Q}\!\! \frac{\dist(x,\supp\wt\mu)}{r(a\,B_Q)^2}\,d\HH^1|_{L_{a,Q}}(x)\lesssim \ve_0'''.
\end{equation}
Moreover, minor modifications in the proofs of these results show that the ball $a\,B_Q$ can be replaced by $\frac32a\,B_Q$ in \rf{eqdj57} and \rf{eqdj58}, at the cost of worsening the constants implicit in the ``$\lesssim$'' relation.
 
We will now estimate the first sup on the the left hand side of \rf{eqdj56}. To this end, recall that $\wt\mu
=\mu|_{\wt E}$, take $x\in \wt E\cap aB_Q\setminus L_{a,Q}$ and set
$$d_x:=\frac12\,\min(\dist(x,L_{a,Q}),r(aB_Q))\approx \dist(x,L_{a,Q}).$$
Then we have $B(x,d_x)\subset \tfrac32a\,B_Q$, and  thus
\begin{equation}\label{eqdjg33}
\frac1{\mu(\tfrac32a\,B_Q)}
\int_{\tfrac32a\,B_Q} \frac{\dist(y,L_{a,Q})}{r(\tfrac32a\,B_Q)}\,d\mu(y)\geq \frac{d_x\,\mu(B(x,d_x))}{r(\tfrac32a\,B_Q)\,\mu(\tfrac32a\,B_Q)}
\gtrsim \frac{d_x\,\mu(B(x,d_x))}{r(a\,B_Q)\,\mu(a\,B_Q)}.
\end{equation}
 
By Lemma \ref{lemk88} we have
$$
\mu(B(x,r)) \approx \frac{\mu(2B_{Q})}{\ell(Q)}\,r\gtrsim \tau\,\Theta(B_R)\,r \qquad \mbox{for $\delta^{1/5}\,\ell(Q)\leq r\leq \min(\delta^{-5}\ell(Q),r(aB_Q))$.}
$$
So we infer that if $d_x\geq\delta^{1/5}\ell(Q)$, then $\mu(B(x,d_x))\gtrsim \tau\,\Theta_\mu(B_R)\,d_x$, and by \rf{eqdjg33},
$$\ve_0'''\gtrsim
\frac1{\mu(\tfrac32a\,B_Q)}
\int_{\tfrac32a\,B_Q} \frac{\dist(y,L_{a,Q})}{r(\tfrac32a\,B_Q)}\,d\mu(y)\gtrsim \frac{\tau\,\Theta_\mu(B_R)\,d_x^2\,}{r(a\,B_Q)\,\mu(a\,B_Q)}\gtrsim \frac{A^{-1}\,\tau\,d_x^2}{r(a\,B_Q)^2}.
$$
Therefore,
$$d_x\lesssim (\ve_0''')^{1/2} A\,\tau^{-1}\,r(a\,B_Q),$$
and then in either case
$$d_x\leq\max\bigl(\delta^{1/5}\,\ell(Q),\, c\,(\ve_0''')^{1/2} A\,\tau^{-1}\,r(a\,B_Q)\bigr)
\lesssim \max\bigl(\delta^{1/5},\, (\ve_0''')^{1/2} A\,\tau^{-1}\bigr)\,r(a\,B_Q).$$
Taking the supremum on all $x\in\wt E\cap a\,B_Q$, we deduce that
\begin{equation}\label{eqdk297}
\sup_{x\in \supp\wt\mu \cap a\,B_Q} \frac{\dist(x,L_{a,Q})}{r(a\,B_Q)}\lesssim \max\bigl(\delta^{1/5},\, (\ve_0''')^{1/2} A\,\tau^{-1}\bigr).
\end{equation}

To estimate the second sup on the left side of \rf{eqdj56}, take $x\in L_{a,Q}\cap aB_Q$, and let
$\wt d_x=\dist(x,\supp\wt\mu)$. Then it follows that for all $y\in L_{a,Q}\cap \tfrac32a\,B_Q\cap B(x,\wt d_x/2)$,
$\dist(y,\supp\wt\mu)\geq \wt d_x/2$. Thus,
$$\ve_0''\gtrsim\int_{L_{a,Q}\cap \tfrac32a\,B_Q} \frac{\dist(x,\supp\wt\mu)}{r(\tfrac32a\,B_Q)^2}\,d\HH^1|_{L_{a,Q}}(x)
\gtrsim \frac{\wt{d}_x\,\HH^1(B(x,\wt d_x/2)\cap \tfrac32a\,B_Q)}{r(\tfrac32a\,B_Q)^2}
\gtrsim \frac{(\wt{d}_x)^2}{r(\tfrac32a\,B_Q)^2}.$$
Taking the sup on all the points $x\in L_{a,Q}\cap aB_Q$, we obtain
\begin{equation}\label{eqdk298}
\sup_{x\in L_{a,Q}\cap a\,B_Q} \frac{\dist(x,\supp\wt\mu)}{r(a\,B_Q)}\leq \ve_0'''^{1/2}.
\end{equation}

The lemma follows from \rf{eqdk297} and  \rf{eqdk298}, assuming $\eta$ and $\delta$ small enough.
\end{proof}
\vv

From now on, we assume that for some small constant $\ve_0>0$, we have
\begin{equation}\label{eqahjd398}
\alpha_\mu(a\,B_Q)\leq \ve_0,\quad
\alpha_{\wt\mu}(a\,B_Q)\leq \ve_0,\quad
b\beta_{\infty,\wt\mu}(a\,B_Q)\leq \ve_0,
\end{equation}
for any $Q\in\qgood$ with $\wt\mu(Q)>0$ and for $a\geq 2$ with $r(a\,B_Q)\leq \frac12K\,\ell(R)$. To this end, we will need the constants $\delta$ and $\eta$ to be 
chosen small enough in the Main Lemma.
\vvv


\section{The measure of the cells from $\BCF$, $\LD$, $\BSD$ and $\BCG$}\label{sec8}

Recall that
$$\sss=\sss(R)=\DD(R)\cap \reg.$$
To prove the property (c) stated Main Lemma \ref{mainlemma}, we have to estimate the sum
$$\sum_{P\in\sss(R)}\Theta_\mu(1.1B_P)^2\,\mu(P).$$
According to Lemma \ref{lemregterm}, this sum can be split as follows:
\begin{align*}
\sum_{P\in\sss}\ldots  & =
\sum_{Q\in\DD(R)\cap \BCF} \sum_{P\in\sss:P\subset Q} \ldots + \sum_{Q\in\DD(R)\cap \LD} \sum_{P\in\sss:P\subset Q}\ldots + \sum_{Q\in\DD(R)\cap \HD} \sum_{P\in\sss:P\subset Q}\ldots\\
& \quad + \sum_{Q\in\DD(R)\cap \BCG} \sum_{P\in\sss:P\subset Q}\ldots +\sum_{Q\in\DD(R)\cap \BSD} \sum_{P\in\sss:P\subset Q}\ldots\,,
\end{align*}
where we denoted $\;\ldots = \Theta_\mu(1.1B_P)^2\,\mu(P)$.
In this section we will estimate all the sums on the right hand side above, with the exception of the one involving the cells $Q\in\HD$.

Regarding the sum involving the family $\BCF$, we have:

\begin{lemma}\label{lempocbc*}
If $\eta$ is small enough, we have
$$\sum_{Q\in\DD(R)\cap \BCF} \sum_{P\in\sss:P\subset Q} \Theta_\mu(1.1B_P)^2\,\mu(P)\lesssim \eta^{1/5}\,\Theta_\mu(B_R)^2\,\mu(R).$$
\end{lemma}

\begin{proof}
Recall that, by Lemma \ref{lempocbc},
$$\mu\biggl(\,\bigcup_{Q\in\BCF:Q\subset R}Q\biggr)\leq c\,\eta^{1/4}\,\mu(R).$$
On the other hand, by Lemma \ref{claf22}, any cell $P\in\sss$ satisfies $\Theta_\mu(1.1B_P)\lesssim A\,\Theta_\mu(B_R)$, and thus
\begin{align*}
\sum_{Q\in\DD(R)\cap \BCF} \sum_{P\in\sss:P\subset Q} \Theta_\mu(1.1B_P)^2\,\mu(P) &
\lesssim A\,\Theta_\mu(B_R)^2\sum_{Q\in\DD(R)\cap \BCF}\mu(Q)\\
& \lesssim   A\,\eta^{1/4}\,\Theta_\mu(B_R)^2\,\mu(R)\lesssim \eta^{1/5}\,\Theta_\mu(B_R)^2\,\mu(R).
\end{align*}
\end{proof}

Concerning the family $\LD$ we have:

\begin{lemma}\label{lemld*}
We have
$$\sum_{Q\in\DD(R)\cap \LD} \sum_{P\in\sss:P\subset Q} \Theta_\mu(1.1B_P)^2\,\mu(P)\lesssim A^2\,\tau^{1/4}\,
\Theta_\mu(B_R)^2\,\mu(R).$$
\end{lemma}

\begin{proof}
Let $Q\in\DD(R)\cap \LD$. To estimate the sum $\sum_{P\in\sss:P\subset Q} \Theta_\mu(1.1B_P)^2\,\mu(P)$
we distinguish two cases according to wether $\ell(P)>\tau^{1/2}\,\ell(Q)$ or not.

Suppose first that $\ell(P) >\tau^{1/2}\,\ell(Q)$.
 If the parameters $A_0,C_0$  in the construction of the David-Mattila cells are chosen appropriately (with $1\ll C_0\ll A_0$), then $1.1B_P\subset 1.1B_Q$
and so it follows that
$$\Theta_\mu(1.1B_P)\lesssim \frac{\ell(Q)}{\ell(P)}\,\Theta_\mu(1.1B_Q) 
\lesssim \tau^{-1/2}\,\tau\,\Theta_\mu(B_R) \approx \tau^{1/2}\,\Theta_\mu(B_R).$$
Therefore,
\begin{equation}\label{eqcasv44}
\sum_{P\in\sss:P\subset Q} \Theta_\mu(1.1B_P)^2\,\mu(P) \lesssim \tau\,\Theta_\mu(B_R)^2\,\mu(Q).
\end{equation}

For the case when $\ell(P)\leq\tau^{1/2}\,\ell(Q)$ we will use the small boundaries condition of $Q$.
By Lemma \ref{lem74} we have $d(x)\lesssim \tau^{1/2}\,\ell(Q)$ for all
$x\in P$. Since $P\subset Q
\not\in\good$, then we deduce that
$$\dist(P,E\setminus Q)\lesssim \tau^{1/2}\,\ell(Q).$$
Thus, by \rf{eqfk490},
\begin{equation}\label{eqakk42}
\sum_{\substack{P\in \sss:P\subset Q\\ \ell(P)\leq \tau^{1/2}\ell(Q)}} \!\!\Theta_\mu(1.1B_P)^2\,\mu(P) \lesssim
A^2\,\Theta_\mu(B_R)^2\!\!\!\!\!\!
\sum_{\substack{P\in \sss:P\subset Q\\ \dist(P,E\setminus Q)\lesssim \tau^{1/2}\,\ell(Q)}}\!\!\!\!\! \mu(P) \lesssim A^2\,\tau^{1/4}\,\Theta_\mu(B_R)^2\,\mu(3.5B_Q).
\end{equation}

Next we claim that if there exists some cell $P\in \sss$ contained in $Q$ such that $\ell(P)\leq\tau^{1/2}\,\ell(Q)$, then $Q$ is doubling, i.e.\ $Q\in\DD^{db}$. Indeed,
by the definition of $\reg$, assuming $\tau$ small enough, the existence of such cell $P$ implies the existence
of some cell $Q'\in\good$ such that $\ell(Q')\approx\ell(P)$ and $3.3B_{Q'}\subset 1.02B_Q$. Taking a
suitable ancestor of $Q'$, we deduce that there exists some $Q''\in\good$ such that $3.3B_{Q''}\subset
1.05B_Q$ and $\ell(Q'')\approx\ell(Q)$.

Let $a\geq3.3$ be the maximal number such that $a\,B_{Q''}\subset 1.1B_Q$. Notice that
$r(a\,B_{Q''})\geq r(1.1B_Q)-r(1.05B_Q) = 0.05\,r(B_Q)$.
By Remark \ref{remdens}, we know that
$\mu(bB_{Q''})\leq C(a,b)\,\mu(aB_{Q''})$ for $3.3\leq a\leq b\leq\delta^{-1/2}\ell(Q)$, with $C(a,b)$ not depending on $C_0$. So we have
\begin{equation}\label{eqfjd39}
\mu(400B_Q)\leq c\,\mu(a\,B_{Q''})\leq c\,\mu(1.1B_Q),
\end{equation}
with $c$ independent of $C_0$.
By arguments analogous to the ones of Lemma \ref{lemdobbb}, this implies that $Q\in\DD^{db}$. Indeed,
if $Q\not\in\DD^{db}$, we have
$$\mu(1.1B_Q) =\mu(28\cdot 1.1 B(Q))\leq \mu(100B(Q)) \leq C_0^{-1}\,\mu(100^2B(Q)),$$
while by \rf{eqfjd39},
$$\mu(100^2B(Q)) = \mu(\tfrac1{28}\,100^2\,B_Q) \leq \mu(400B_Q)\leq c\,\mu(1.1B_Q),$$
and so we get a contradiction if $C_0$ is assumed big enough in the construction of the
David-Mattila cells.

Since $Q$ is doubling, we have $\mu(3.5B_Q)\leq \mu(100B(Q))\leq c\,\mu(B(Q))\leq c\,\mu(Q)$.
Then, by \rf{eqakk42}, we deduce that
$$
\sum_{\substack{P\in \sss:P\subset Q\\ \ell(P)\leq \tau^{1/2}\ell(Q)}} \Theta_\mu(1.1B_P)^2\,\mu(P) 
\lesssim A^2\,\tau^{1/4}\,\Theta_\mu(B_R)^2\,\mu(Q).
$$
Together with \rf{eqcasv44}, this yields the desired conclusion.
\end{proof}

Next we will deal with the cells from $\BSD$:
\vv

\begin{lemma}\label{lembsd}
We have
$$\sum_{Q\in\BSD:Q\subset R} \Theta_\mu(B_R)^2\,\mu(Q) \lesssim \frac1\eta
\sum_{Q\in\tree} \int_{1.1B_Q\cap F}\int_{\delta\,\ell(Q)}^{\delta^{-1}
\ell(Q)}\Delta_\mu(x,r)^2\,\frac{dr}r\,d\mu(x).
$$
\end{lemma}

\begin{proof}
Recall that the cells $Q\in \BSD$ satisfy 
$$\sum_{P\in \DD: Q\subset P\subset R} \frac1{\mu(1.1B_P)} \int_{1.1B_P\cap F} \int_{\delta\,\ell(P)}^{\delta^{-1}
\ell(P)}\Delta_\mu(x,r)^2\,\frac{dr}r\,d\mu
\geq \eta\,\Theta_\mu(B_R)^2.$$
So we have
\begin{align*}
\sum_{\substack{Q\in\BSD:\\Q\subset R}} \Theta_\mu(B_R)^2\,\mu(Q) &\leq \frac1\eta
\sum_{\substack{Q\in\term:\\Q\subset R}}
\sum_{\substack{P\in\DD:\\Q\subset P\subset R}}
\frac{\mu(Q)}{\mu(P)}
\int_{1.1B_P\cap F} \int_{\delta\,\ell(P)}^{\delta^{-1}
\ell(P)}\Delta_\mu(x,r)^2\,\frac{dr}r\,d\mu(x).
\end{align*}
Denote now by $\wt\tree$ the family of cells $P$ which are contained in $R$ and are not strictly contained
in any cell from $\term$. By interchanging the order of summation, the right hand side above
equals
\begin{multline*}
\frac1\eta
\sum_{P\in\wt\tree} 
\int_{1.1B_P\cap F} \int_{\delta\,\ell(P)}^{\delta^{-1}
\ell(P)}\Delta_\mu(x,r)^2\,\frac{dr}r\,d\mu(x)
\sum_{\substack{Q\in\term:\\Q\subset P}}
\frac{\mu(Q)}{\mu(P)}\\
\leq
\frac1\eta
\sum_{P\in\wt\tree} 
\int_{1.1B_P\cap F} \int_{\delta\,\ell(P)}^{\delta^{-1}
\ell(P)}\Delta_\mu(x,r)^2\,\frac{dr}r\,d\mu(x).
\end{multline*}
Since $\wt\tree\subset\tree$, we are done.
\end{proof}
\vv

Now we turn our attention to the cells from $\BCG$:
\vv

\begin{lemma}\label{lembcg}
Suppose that $\delta$ is small enough. Then we have
\begin{align*}
\sum_{\substack{Q\in\BCG:\\Q\subset R}} \Theta_\mu(B_R)^2\,\mu(Q) & \lesssim \delta^{1/2}\,\Theta_\mu(B_R)^2\,\mu(R)+ \frac1{\eta^3}
\sum_{Q\in\tree} \int_{\delta^{-1}B_Q\cap F}\int_{\delta^{5}\,\ell(Q)}^{\delta^{-1}
\ell(Q)}\Delta_\mu(x,r)^2\,\frac{dr}r\,d\mu(x).
\end{align*}
\end{lemma}

\begin{proof}
We need to distinguish three types of cells from $\BCG$:
\begin{itemize}
\item $Q\in\BCG_1$ if $Q\in\DD(R)$ and $\ell(Q)\geq \delta^4\,\ell(R)$.
\item $Q\in\BCG_2$ if $Q\in\DD(R)$, $\ell(P)\leq\delta^4\,\ell(R)$, and $\dist(Q,E\setminus R)\leq\delta\,\ell(R)$.
\item $Q\in\BCG_3$ if $Q\in\DD(R)$, $\ell(P)\leq\delta^4\,\ell(R)$, and $\dist(Q,E\setminus R)>\delta\,\ell(R)$.
\end{itemize}

First we will estimate the measure of the cells from $\BCG_2$. To this end we will use the fact
that $R$ has ``small boundaries''. More precisely, recall that by \rf{eqfk490} we have
$$\mu\bigl(\{x\in R:\dist(x,E\setminus R)\leq \lambda\,\ell(R)\}\bigr) \leq c\,\lambda^{1/2}\,\mu(R).
$$
By definition, every cell $Q\in\BCG_2$ satisfies
$$Q\subset \{x\in R:\dist(x,E\setminus R)\leq (\delta+c\,\delta^4)\,\ell(R)\},$$
and thus
$$\sum_{Q\in\BCG_2} \mu(Q) \leq \mu\bigl(\{x\in R:\dist(x,E\setminus R)\leq (\delta+c\,\delta^4)\,\ell(R)\}\bigr) \lesssim (\delta+c\,\delta^4)^{1/2}\,\mu(R) \lesssim \delta^{1/2}\,\mu(R).$$

To deal with $\BCG_1$ recall that 
 the cells $Q\in\BCG$ satisfy
\begin{equation}\label{eqdkhe40}
\mu(\delta^{-1/2}B_Q\cap F\setminus \GG(Q,R,\delta^{1/2},\eta))\geq \eta \,\mu(\delta^{-1/2}B_Q\cap F)\geq \frac\eta2\,\mu(\delta^{-1/2}B_Q)\geq \frac\eta2\,\mu(Q),
\end{equation}
where we used that $Q\not\in\BCF$ for the last inequality. Taking also
into account that 
\begin{equation}\label{eqdkhe41}
\int_{\delta^{1/2}\ell(Q)}^{\delta^{-1/2}\ell(R)}\Delta_\mu(x,r)^2\,\frac{dr}r\geq \eta\,\Theta_\mu(2B_R)^2
\qquad\mbox{for all $x\not\in \GG(Q,R,\delta^{1/2},\eta)$,}
\end{equation}
by Chebyshev we infer that
\begin{align*}
\mu(Q) & \leq \frac2\eta\,
\mu(\delta^{-1/2}B_Q\cap F\setminus \GG(Q,R,\delta^{1/2},\eta)) \\
&\leq
\frac2{\eta^2\,\Theta_\mu(2B_R)^2}\int_{\delta^{-1/2}B_Q\cap F}
\int_{\delta^{1/2}\ell(Q)}^{\delta^{-1/2}\ell(R)}\Delta_\mu(x,r)^2\,\frac{dr}r\,d\mu(x).
\end{align*}
Using that $\ell(Q)\geq\delta^4\,\ell(R)$ for $Q\in\BCG_1$, we infer that
\begin{align*}
\sum_{Q\in\BCG_1}\Theta_\mu(B_R)^2\,\mu(Q) & \lesssim \frac1{\eta^2} \sum_{Q\in\BCG_1}
\int_{\delta^{-1/2}B_Q\cap F}
\int_{\delta^{1/2}\ell(Q)}^{\delta^{-1/2}\ell(R)}\Delta_\mu(x,r)^2\,\frac{dr}r\,d\mu(x)\\
& \lesssim
\frac1{\eta^2} \int_{\delta^{-1}B_R\cap F}
\int_{\delta^{5}\ell(R)}^{\delta^{-1/2}\ell(R)}\Delta_\mu(x,r)^2\,\frac{dr}r\,d\mu(x).
\end{align*}

Finally we will estimate the measure of the cells from $\BCG_3$. To this end we consider the function
$$f(x) = \!\sum_{P\in\sss}\! \left(\int_{\delta^2\ell(P)}^{\delta^{-1}\ell(R)}\! \Delta_\mu(x,r)^2\frac{dr}r\right)^{1/2}\!\!\chi_{P\cap F}(x)
+ 
\left(\int_0^{\delta^{-1}\ell(R)}\! \Delta_\mu(x,r)^2\frac{dr}r\right)^{1/2}\!\!\chi_{R\cap F\setminus 
\bigcup_{P\in\sss} P}(x)
.$$
We claim that 
\begin{equation}\label{eqj778}
Q\subset \bigl\{x\in\R^d: M_*f(x)>\tfrac12\,\eta^{3/2}\,\Theta_\mu(2B_R)\bigr\}\qquad \mbox{for all $Q\in\BCG_3$,}
\end{equation}
where $M_*$ is the maximal operator introduced in \rf{eqm*}. To prove the claim, consider $Q\in
\BCG_3$ and notice that
by \rf{eqdkhe41} 
\begin{equation}\label{eqgjdew36}
\left(\int_{\delta^{1/2}\,\ell(Q)}^{\delta^{-1/2}\ell(R)} \Delta_\mu(x,r)^2\,\frac{dr}r\right)^{1/2}
\geq \eta^{1/2}\,\Theta_\mu(2B_R)\qquad \mbox{for all $x\in \delta^{-1/2}B_Q\cap F\setminus\GG(Q,R,\delta^{1/2},\eta)$.}
\end{equation}
Let us see that
\begin{equation}\label{eqgjdew37}
f(x) \geq \left(\int_{\delta^{1/2}\,\ell(Q)}^{\delta^{-1/2}\ell(R)} \Delta_\mu(x,r)^2\,\frac{dr}r\right)^{1/2}
\qquad\mbox{for all $x\in \delta^{-1/2}B_Q\cap F\setminus\GG(Q,R,\delta^{1/2},\eta)$.}
\end{equation}
Suppose first that $x$ belongs to some cell $P\in\reg$. Then, by the construction of the family $\reg$,
\begin{equation}\label{eqdjue43}
\ell(P)\leq c\,\delta^{-1/2}\ell(Q)\leq \delta^{-1}\ell(Q)\leq \delta^3\,\ell(R).
\end{equation}
This implies that
\begin{align*}
\dist(P,E\setminus R) & \geq \dist(Q,E\setminus R) - \diam(\delta^{-1/2}B_Q) - \diam(P)\\
& \geq \delta\,\ell(R) - c\,\delta^{-1/2}\,\delta^4 \,\ell(R)- c\,\delta^3\,\ell(R)>0,
\end{align*}
assuming $\delta$ small enough, and thus $P\subset R$, which in particular tells us that $P\in\sss$.
Further, from \rf{eqdjue43} it also follows that $\delta^2\ell(P)\leq \delta^{1/2}\ell(Q)$. Hence \rf{eqgjdew37} holds in this case. If $x\in \delta^{-1/2}B_Q\setminus\bigcup_{P\in\reg} P$, \rf{eqgjdew37} follows by an analogous argument (just letting 
$\ell(P)=0$ in the preceding estimates).

From \rf{eqgjdew36} and \rf{eqgjdew37}, we infer that
$$f(x)
\geq \eta^{1/2}\,\Theta_\mu(2B_R)\qquad \mbox{for all $x\in \delta^{-1/2}B_Q\cap F\setminus\GG(Q,R,\delta^{1/2},\eta)$,}
$$
and as \rf{eqdkhe40} also holds in this case, we obtain
$$
\frac1{\mu(\delta^{-1/2}B_Q)} \int_{\delta^{-1/2}B_Q\cap F} f(x)
\,d\mu(x)\geq \frac12\,\eta^{3/2}\,\Theta_\mu(2B_R),
$$
which proves \rf{eqj778}.

From the claim just proved and Chebyshev we deduce that
\begin{align*}
\sum_{Q\in\BCG_3}\Theta_\mu(B_R)^2\,\mu(Q) &\lesssim \sum_{Q\in\BCG_3} \frac1{\eta^3}  \int_Q M_*f(x)^2\,d\mu(x) \lesssim \frac1{\eta^3}  \int M_*f(x)^2\,d\mu(x)  \lesssim \frac1{\eta^3} \int |f|^2\,d\mu.
\end{align*}
To conclude with the family $\BCG_3$ it just remains to note that
\begin{align*}
\int |f|^2\,d\mu & =\!
\sum_{P\in\sss} \int_{P\cap F}\int_{\delta^2\ell(P)}^{\delta^{-1}\ell(R)} \!\!\Delta_\mu(x,r)^2\frac{dr}r\,d\mu(x)
+ \int_{R\cap F\setminus 
\bigcup_{P\in\sss} P}\int_0^{\delta^{-1}\ell(R)}\!\! \Delta_\mu(x,r)^2\frac{dr}r\,d\mu(x)
\\
& \leq
\sum_{P\in\sss} \,\sum_{S:P\subset S\subset R} \int_{P\cap F}\int_{\delta^2\ell(S)}^{\delta^{-1}\ell(S)} \Delta_\mu(x,r)^2\,\frac{dr}r\,d\mu(x)\\
& \quad+ 
\int_{R\cap F\setminus 
\bigcup_{P\in\sss} P}\sum_{S:x\in S\subset R}
\int_{\delta^2\ell(S)}^{\delta^{-1}\ell(S)} \Delta_\mu(x,r)^2\,\frac{dr}r\,d\mu(x)
\\
& = \sum_{S\in\tree} \int_{S\cap F}\int_{\delta^2\ell(S)}^{\delta^{-1}\ell(S)} \Delta_\mu(x,r)^2\,\frac{dr}r\,d\mu(x).
\end{align*}

Gathering the estimates we obtained for the families $\BCG_1$, $\BCG_2$ and $\BCG_3$, the lemma follows.
\end{proof}
\vv

By combining the results obtained in Lemmas \ref{lempocbc*}, \ref{lemld*}, \ref{lembsd}, and
\ref{lembcg}, and taking into account that $\Theta_\mu(1.1B_Q)\lesssim A\,\Theta_\mu(B_R)$ for all $Q\in\sss(R)$,
we get the following.

\vv

\begin{lemma}\label{lempoctot}
If $\eta$ and $\delta$ are small enough, then
\begin{align*}
\sum_{\substack{Q\in\DD(R):\\
Q\subset \BCF\cup\LD\cup\BCG\cup\BSD}} 
 \Theta_\mu(1.1B_Q)^2\,\mu(Q) &\lesssim A^2\,(\eta^{1/5}+
 \tau^{1/4} +\delta^{1/2})
  \,\Theta_\mu(B_R)^2\,\mu(R)\\
 &\quad + \frac{A^2}{\eta^3}
\sum_{Q\in\tree} \int_{\delta^{-1}B_Q\cap F}\int_{\delta^{5}\,\ell(Q)}^{\delta^{-1}
\ell(Q)}\Delta_\mu(x,r)^2\,\frac{dr}r\,d\mu(x).
\end{align*}
\end{lemma}

\vvv


\section{The new families of cells $\bsb$, $\nterm$, $\ngood$, $\nqgood$ and $\nreg$}\label{sec9*}

To complete the proof of the Main Lemma \ref{mainlemma}, it remains to construct the curve $\Gamma_R$
and to estimate the sum 
$\sum_{\substack{Q\in\DD(R)\cap\HD}}
\sum_{P\in\sss:P\subset Q}
 \Theta_\mu(1.1B_P)^2\,\mu(P)$.
To this end, we need first to introduce a new type of terminal cells. Let $M$ be some very big constant 
to be fixed below (in particular, $M\gg
A\,\tau^{-1}$).
We say that a  
 cell $Q\in\DD$ belongs to $BS\beta_0$ if $Q\not\in BCF_0\cup LD_0\cup HD_0\cup BCG_0\cup BC\Delta_0$, $\ell(Q)\leq \ell(R)$, and
 \begin{equation}\label{eqajr77}
\sum_{P\in \DD: Q\subset P\subset R}  \beta_{\infty,\wt\mu}(2B_P)^2\geq M.
\end{equation}
\vv

Next we consider the subfamily of $\term \cup BS\beta_0$ of the cells which are maximal with respect to inclusion (thus they are disjoint), and we call it $\nterm$. 
We denote by $\BSB$ the subfamily 
of the cells from $\nterm$ which belong by $BS\beta_0$.
The notation $\nterm$ stands for ``new term'', and $\bsb$ for ``big sum of $\beta's$''.
Note that 
$$\nterm \subset \term\cup \bsb.$$
The definition of the family $BS\beta_0$, and so of $\bsb$, depends on the measure $\wt\mu$.
This is the reason why $BS\beta_0$ and $\bsb$ were not introduced in Section \ref{sec7}, like the others cells 
of $\term$. The introduction of the new family $\bsb$ is necessary to guaranty the lower Ahlfors-David regularity of the measures $\sigma_k$ in the forthcoming Section \ref{sec12}.

Similarly to Section \ref{sec7}, 
we denote by $\ngood$ the subfamily of the cells $Q\subset B(x_0,\frac1{10} K r_0)$ 
with $\ell(Q)\leq\ell(R)$ such that there
does not exist any cell $Q'\in\nterm$ with $Q'\supset Q$. 
Notice that $R\in\ngood$, $\ngood \subset\good$, and $\nterm\not\subset\ngood$. 

We need now to define a regularized version of $\nterm$ which we will call $\nreg$. To this end, we proceed
exactly as in Section \ref{sec7}.
First we consider the
auxiliary $1$-Lipschitz function $\wt d:\R^d\to[0,\infty)$:
\begin{equation}\label{eqdefdxwt}
\wt d(x) = \inf_{Q\in\ngood} \bigl(|x-z_Q| + \ell(Q)\bigr).
\end{equation}

We denote 
$$N\!W_0=\{x\in\R^d:\wt d(x)=0\}.$$
For each $x\in E\setminus N\!W_0$ we take the largest cell $Q_x\in\DD$ 
such that $x\in Q_x$ with
$$\ell(Q_x) \leq \frac1{60}\,\inf_{y\in Q_x} \wt d(y).$$
We denote by $\nreg$ the collection of the different cells $Q_x$, $x\in E\setminus W_0$.
Further, we consider the subcollection of the cells from $\nreg$ with non-vanishing $\wt\mu$-measure and
we relabel it as $\{Q_i\}_{i\in I}$.
 Also, we denote by $\nqgood$ the family of cells $Q\in \DD$ such that $Q$ is contained in $B(x_0,2 K r_0)$ and $Q$ is not strictly contained in any
cell of the family $\nreg$. Note that $\nreg\subset \nqgood$. Moreover, since
$\wt d(x)\geq d(x)$ for all $x\in\R^d$, it follows that $\nqgood\subset\qgood$. Thus all the properties
proved in Sections \ref{sec7} and \ref{sec6.5} for the cells from $\qgood$ also hold for the ones from $\nqgood$.


The following result and its proof, which we omit, are analogous to the ones of Lemma \ref{lem74}. 
\vv

\begin{lemma}\label{lem74**}
The cells $\{Q_i\}_{i\in I}$ are pairwise disjoint and satisfy the following properties:
\begin{itemize}
\item[(a)] If $x\in B(z_{Q_i},50\ell(Q_i))$, then $10\,\ell(Q_i)\leq \wt d(x) \leq c\,\ell(Q_i)$,
where $c$ is some constant depending only on $A_0$. In particular, $B(z_{Q_i},50\ell(Q_i))\cap N\!W_0=\varnothing$.

\item[(b)] There exists some constant $c$ such that if $B(z_{Q_i},50\ell(Q_i))\cap B(z_{Q_j},50\ell(Q_j))
\neq\varnothing$, then
$$c^{-1}\ell(Q_i)\leq \ell(Q_j)\leq c\,\ell(Q_i).$$
\item[(c)] For each $i\in I$, there at most $N$ cells $Q_j$, $j\in I$, such that
$$B(z_{Q_i},50\ell(Q_i))\cap B(z_{Q_j},50\ell(Q_j))
\neq\varnothing,$$
 where $N$ is some absolute constant.
 
 \item[(d)] If $x\not\in B(x_0,\frac1{8} K r_0)$, then $\wt d(x)\approx |x-x_0|$. As a consequence,
 if $B(z_{Q_i},50\ell(Q_i))\not\subset  B(x_0,\frac1{8} K r_0)$, then $\ell(Q_i)\gtrsim  K r_0$.
\end{itemize}
\end{lemma}

\vvv


\section{The approximating curves $\Gamma^k$}
\label{sec88}

In this section we will construct somes curves $\Gamma^k$ which, in a sense, approximate $\supp\wt\mu$
on $B(x_0,\frac14 K r_0)$ up to the scale of the cubes $\{Q_i\}_{i\in I}$. This curves will be
used to show that the measure of the cells from $\HD$ is small.

The curves $\Gamma^k$ are constructed inductively in the following way. Let $$d_0=\diam\bigl(\supp\wt\mu\cap \bar B(x_0,\tfrac14 K r_0)\bigr),$$ and take $z_A,z_B\in \supp\wt\mu\cap \bar B(x_0,\frac14 K r_0)$ such that $|z_A- z_B| =d_0$. The curve $\Gamma^1$ is just the segment $L_1^1$ with endpoints $x_0^1\equiv z_A$ and
$x_1^1\equiv z_B$. 

For $k\geq1$ we assume that $\Gamma^k$ contains points $z_A\equiv x_0^k,\,x_1^k,\ldots,x_{N_k-1}^k,\,x_{N_k}^k\equiv z_B$ from $\supp\wt\mu\cap \bar B(x_0,\frac12 K r_0))$ and
that $\Gamma^k$ is the union of the segments $L_j^k:=\bigl[x_{j-1}^k,x_{j}^k\bigr]$, for $j=1,\ldots,N_k$. Then $\Gamma^{k+1}$ is constructed as follows. Each one of the segments $L_j^k$, 
$j=1,\ldots,N_k$, that constitutes $\Gamma^k$ is replaced by a curve $\Gamma_j^k$ with the same end points
as $L_j^k$ by the following rules:
\vv

\begin{itemize}
\item[(A)] If $\HH^1(L_j^k)\leq 2^{-(k+1)/2}\,d_0$, we  set $\Gamma_j^k= L_j^k$.\vv

\item[(B)] If there exists some cell $Q_i$, $i\in I$, such that $2B_{Q_i}\cap L_j^k\neq\varnothing$ and 
$\HH^1(L_j^k)\leq \ell(Q_i)$, then we also set 
$\Gamma_j^k= L_j^k$.\vv

\item[(C)] If the conditions in (A) and (B) do not hold, that is to say, if 
$\HH^1(L_j^k)> 2^{-(k+1)/2}\,d_0$ and also $\HH^1(L_j^k)> \ell(Q_i)$ for all $i\in I$ such that $2B_{Q_i}\cap L_j^k\neq\varnothing$, then we consider the mid point of the segment $L_j^k$, which we denote by
 $z_j^k$, and we take a point $p_j^k\in\supp\wt\mu$ such that
 \begin{equation}\label{eqsok91}
|p_j^k - z_j^k|\leq c\,\ve_0\,\HH^1(L_j^k).
\end{equation}
The existence of $p_j^k$ is ensured by the fact that the ball $B$ centered at $x_{j-1}^k$ with radius $2\HH^1(L_j^k)$   satisfies 
$b\beta_{\infty,\wt\mu}(B)\lesssim \ve_0$ (recall that
the end points of $L_j^k$ are $x_{j-1}^k$ and $x_{j}^k$ and they belong to $\supp\wt\mu$). This follows from the fact that if $Q\in\DD$ is the smallest cell containing 
$x_{j-1}^k$ such that $\ell(Q)>\HH^1(L_j^k)$ and $x_{j-1}^k$ belongs to some cell $Q_i$, $i\in I$, then
we have $Q_i\subset Q$, and so we can apply Lemma \ref{claf23} to $Q$.
Then we set
$$\Gamma_j^k = [x_{j-1}^k,p_j^k] \cup [p_j^k,x_j^k].
$$
\end{itemize}
\vv

The points $z_A\equiv x_0^{k+1},\,x_1^{k+1},\ldots,x_{N_k}^{k+1},\,x_{N_{k+1}}^{k+1}\equiv z_B$ are obtained from the sequence
$$x_0^k,\,x_1^k,\ldots,x_{N_k-1}^k,\,x_{N_k}^k$$ just by inserting the point $p_j^k$ between 
$x_{j-1}^k$ and $x_{j}^k$  when $\Gamma_j^{k}$ is constructed as in (C), for every $j\in[1,N_k]$, and relabeling the points from the resulting sequence suitably.
Note that in the cases (A) and (B), the segment $L_j^k$ will coincide with some segment $L_h^{k+1}$ from $\Gamma^{k+1}$, while in the case (C) $L_j^k$ is replaced by two new segments $L_h^{k+1}$, $L_{h+1}^{k+1}$, satisfying
$$\frac13\,\HH^1(L_j^k)< \HH^1(L_{h'}^{k+1})< \frac1{2^{1/2}} \HH^1(L_j^k),$$
both for $h'=h$ and $h'=h+1$. 
In the cases (A) and (B) we say that $L_h^{k+1}$ is generated by $L_j^k$ and in the case (C), that both $L_h^{k+1}$ and $L_{h+1}^{k+1}$ are generated by $L_j^k$.


We will call the points $z_A\equiv x_0^k,\,x_1^k,\ldots,x_{N_k-1}^k,\,x_{N_k}^k\equiv z_B$ vertices of 
$\Gamma^k$.

Next we define the auxiliary map $\Pi_k:\Gamma^k\to\Gamma^{k+1}$ as follows. Given $x\in L_j^k$,
we let $\Pi_k(x)$ be the unique point in $\Gamma_j^{k}\subset \Gamma^{k+1}$ whose orthogonal projection to $L_j^k$ is $x$.
In particular, note that if $\Gamma_j^k=L_j^k$, then $\Pi_k(x)=x$. To simplify notation, we denote $\ell_j^k = \HH^1(L_j^k)$. Observe that the condition (A) guaranties that
\begin{equation}\label{eqguai1}
\ell_j^k\geq 2^{-(k+2)/2}\,d_0\qquad \mbox{for all $k\geq1$, $1\leq j\leq N_k$.}
\end{equation}

We denote by $\rho_j^k$ the line which contains $L_j^k$.

Also, we consider
the (open) ball
$$B_j^k=B(z_j^k,\ell_j^k)$$
(recall that $z_j^k$ stands for the mid point of $L_j^k$).  
Observe that $L_j^k\subset \frac12\overline{B_j^k}$.
By the argument just below \rf{eqsok91}, it is clear that $\wt E\cap \frac1{10}\,B_j^k\neq\varnothing$ is
$\ve_0$ is small enough. By Lemma \ref{lemk88} this guaranties that if $Q\in\good$ fulfils
$2B_Q\cap B_j^k\neq\varnothing$ and $\ell(Q)\approx r(B_j^k)$, then
\begin{equation}\label{eqdob842}
\mu(\tfrac16B_j^k)\approx\mu(B_j^k) \approx \mu(2B_j^k)\approx \mu(Q),
\end{equation}
assuming $\eta$, $\delta$ and $\ve_0$ small enough. That such a cell $Q$ exists follows easily from
the construction of $\Gamma^k$ and (b) in the next lemma.

\vv

\begin{lemma}\label{lemnofac}
The following properties hold for all $L_j^k\subset \Gamma^k$, with $k\geq1$:
\begin{itemize}
\item[(a)] If $x\in L_j^k$, then
$$|\Pi_k(x)-x|\leq c\,\ve_0\,\ell_j^k.$$

\item[(b)] If there exists some $Q_{i_0}$, $i_0\in I$, such that
 $\dist(Q_{i_0},L_j^k)\leq 2\,\ell_j^k+ 2\,\ell(Q_{i_0})$,
then 
$$\ell_j^k\approx \max\bigl(\ell(Q_{i_0}),\,2^{-k/2} d_0 \bigr).$$

\item[(c)] If there exists some point $x\in N\!W_0$ (i.e.\ $\wt d(x)=0$) such that
 $\dist(x,L_j^k)\leq 2\,\ell_j^k$,
then 
$$\ell_j^k\approx 2^{-k/2} d_0 .$$

\item[(d)] If $L_h^k$ satisfies $\dist(L_j^k,L_h^k)\leq 2\,\ell_j^k$,
then 
$\ell_j^k\approx\ell_h^k.$
\end{itemize}
\end{lemma}

\begin{proof}
The statement in (a) is an immediate consequence of \rf{eqsok91}. To prove (b), consider
a sequence of segments $[z_A,z_B]= L_{j_1}^1,\;L_{j_2}^2,\ldots,L_{j_k}^k=L_j^k$, so that for each $m$   $L_{j_{m+1}}^{m+1}$ is one of the segments that form
$\Gamma_{j_m}^m$ (in particular, we may have $L_{j_{m+1}}^{m+1}=L_{j_m}^m$). 

Suppose first that in the construction described above, the option (B) holds for some $m=1,
\ldots,k$. That is,
there exists some cell $Q_i$, $i\in I$, such that $2B_{Q_i}\cap L_{j_m}^m\neq\varnothing$ and 
$\HH^1(L_{j_m}^m)\leq \ell(Q_i)$. Take the minimal index $m\in [1,k]$ such that this holds.
By construction, we have $L_{j_m}^m=L_{j_{m+1}}^{m+1}=\ldots = L_j^k$. So
\begin{equation}\label{eqs111}
\ell_j^k\leq \ell(Q_i).
\end{equation}
Note now that $L_{j_{m-1}}^{m-1}\neq L_{j_m}^m$ (otherwise this would contradict the definition of $m$).
 Suppose  that $x_{j_{m-1}}^{m-1}$ is a common
endpoint both of $L_{j_m}^m$ and $L_{j_{m-1}}^{m-1}$. Then
$$\dist(x_{j_{m-1}}^{m-1},2B_{Q_i})\leq \ell_{j_m}^m\leq \ell(Q_i),$$
 which implies that
$x_{j_{m-1}}^{m-1}\in B(z_{Q_i},50\ell(Q_i))$.
From Lemma \ref{lem74**} we infer that $\wt d(x_{j_{m-1}}^{m-1})>0$ and that there exists some $i'\in I$ such that
$x_{j_{m-1}}^{m-1}\in Q_{i'}$ with $\ell(Q_i')\approx\ell(Q_i)$.
Since the option (B) in the construction of $\Gamma_{m-1}$ does not hold for $L_{j_{m-1}}^{m-1}$,
we have $\HH^1(L_{j_{m-1}}^{m-1})> \ell(Q_{i'})$. Thus
$$\ell_{j}^k = \ell_{j_m}^m \approx \ell_{j_{m-1}}^{m-1}> \ell(Q_{i'})\approx\ell(Q_i).
$$
Together with \rf{eqs111}, this estimate shows that
\begin{equation}\label{eqa33}
\ell_{j}^k \approx\ell(Q_i).
\end{equation}
Moreover, the fact that $L_{j_{m-1}}^{m-1}\neq L_{j_m}^m$ also implies that the option (A)
does not hold for $m-1$, and thus $\HH^1(L_{j_{m-1}}^{m-1})>2^{-m/2} d_0 $. Hence,
$$\ell_{j}^k =\ell_{j_{m}}^{m} \approx \ell_{j_{m-1}}^{m-1}>2^{-m/2} d_0 \geq2^{-k/2} d_0 .$$
That is, $\ell_j^k\approx \max(\ell(Q_i),\,2^{-k/2} d_0 )$ 
if the option (B) of the algorithm holds for some $m$. Moreover, if $Q_{i_0}$ is as in (b), by \rf{eqs111} we get
\begin{align*}
\dist(Q_{i_0},2B_{Q_i}) & \leq \dist(Q_{i_0},L_j^k) + \ell_j^k + \dist(L_j^k,2B_{Q_i}) \\
&\leq 2\ell_j^k + 2\,\ell(Q_{i_0}) + \ell_j^k + 0 \leq  3\,\ell(Q_i)+ 2\,\ell(Q_{i_0}),
\end{align*}
which implies that $B(z_{Q_i},50\ell(Q_i))\cap B(z_{Q_{i_0}},50\ell(Q_{i_0}))\neq\varnothing$. So $\ell(Q_i)\approx\ell(Q_{i_0})$ and 
$\ell_j^k\approx \max(\ell(Q_{i_0}),\,2^{-k/2} d_0 )$, as wished.

If the option (B) does not hold for any $m\in[1,k]$, then we claim that 
$$\ell_j^k\approx 2^{-k/2} d_0 .$$
This follows easily form the fact that $\ell_{j_1}^1=d_0$, and for any $m$ we have:
\begin{itemize}
\item If $\ell_{j_m}^m\leq 2^{-(m+1)/2} d_0 $, then $\ell_{j_{m+1}}^{m+1}=\ell_{j_m}^m$.
\item If $\ell_{j_m}^m> 2^{-(m+1)/2} d_0 $, then $\frac1{3}\ell_{j_m}^m<\ell_{j_{m+1}}^{m+1}<\frac1{2^{1/2}}\ell_{j_m}^m$.
\end{itemize}
We leave the details for the reader. 

To complete the proof of (b) it remains to check that 
$\ell(Q_{i_0})\leq A_1\,2^{-k/2} d_0 $ for some absolute big enough constant $A_1$. Suppose not and let 
$Q_{i'}$, $i'\in I$, such that $x_{j_{m-1}}^{m-1}\in Q_{i'}$. 
 Then we have
$$\dist(Q_{i_0},Q_{i'})\leq \ell_j^k +
\dist(Q_{i_0},L_j^k)\leq 3\,\ell_j^k +2\,\ell(Q_{i_0})\leq c\,2^{-k/2} d_0  +2\,\ell(Q_{i_0})\leq \left(\frac cA_1 + 2\right)\,\ell(Q_{i_0}).$$
For $A_1$ big enough this tells us $B(z_{Q_i},\ell(Q_i))\cap B(z_{Q_{i'}},\ell(Q_{i'}))\neq \varnothing$ and 
thus 
$$\ell(Q_{i'})\approx \ell(Q_i)\geq A_1\,2^{-k/2} d_0.$$
So $\ell(Q_{i'})> A_1\,2^{-k/2} d_0$ for $A_1$ big enough, 
 which is not possible in this case (as we assumed that  the option (B) does not hold for any $m\in[1,k]$).
\vv


The statement in (c) can be considered as a particular case of the one in (b). Indeed, when $\wt d(x)=0$,
one can think that of the point $x$ as a cell from the family $\{Q_i\}_{i\in I}$ with side length $0$.
We leave the details for the reader.

\vv
Finally we turn our attention to (d). So we consider $L_j^k$ and $L_h^k$ such that $\dist(L_j^k,L_h^k)\leq 2\,\ell_j^k$ and we have to show that $\ell_j^k\approx\ell_h^k$. We intend to apply the statement just proved in (b). If $\ell_j^k\approx 2^{-k/2} d_0 $ and $\ell_h^k\approx 2^{-k/2} d_0 $ we clearly have $\ell_j^k\approx\ell_h^k$. Suppose now that $\ell_j^k\geq A_2\,\ell_h^k$ for some big constant $A_2$. By (b) this implies that
$\ell_j^k\geq c\,A_2\,2^{-k/2} d_0 $ and so there exists some cell $Q_i$, $i\in I$, such that 
 $\dist(L_j^k,Q_{i})\leq 2\,\ell_j^k + 2\,\ell(Q_i)$ with $\ell(Q_i)\approx \ell_j^k $.
Then we have
\begin{align*}
\dist(L_h^k,Q_i) & \leq \dist(L_h^k,L_j^k) + \ell_j^k +   \dist(L_j^k,Q_i) \leq 2\,\ell_j^k + \ell_j^k +2\,\ell_j^k 
+
2\,\ell(Q_i)\leq c\,\ell_j^k \leq \frac cA_2\,\ell_h^k.
\end{align*}
Assuming $A_2$ big enough again, this yields $\dist(L_h^k,Q_i)  \leq 2\,\ell_h^k$ and then, by (b), 
$$\ell_h^k\gtrsim \ell(Q_i)\approx \ell_j^k.$$
So we get $\ell_h^k\approx \ell_j^k$. 

If we suppose that $\ell_j^k\geq A_2\,\ell_h^k$, by interchanging the roles of $j$ and $h$ we derive analogously that $\ell_j^k\gtrsim \ell_h^k$, and thus $\ell_j^k\approx \ell_h^k$. 
\end{proof}

\vv

\begin{remark}\label{rem31}
Note that, from the statements (b) and (c) in Lemma \ref{lemnofac}, in particular one deduces that
if $\dist(x,L_j^k)\leq 2\,\ell_j^k)$,
then
$$\ell_j^k\approx \max\bigl(\wt d(x),\,2^{-k/2} d_0 \bigr).$$

Since $\wt d(z_A)\approx \wt d(z_B)\approx K\,r_0$, for every $k\geq1$ we have
\begin{equation}\label{eqrem31}
\ell_1^k\approx \ell_{N_k}^k\approx d_0\approx K\,r_0.
\end{equation}
\end{remark}

\vv
\begin{lemma}\label{lemfac32}
For all $k\ge1$ and $1\leq j\leq N_k$,
$\supp\wt\mu\cap 2B_j^k$ is contained in the $(c\,\ve_0\,\ell_j^k)$-neighborhood of the line $\rho_j^k$ (recall that $B_j^k=B(z_j^k,\ell_j^k)$). Moreover, if $L_h^k$ satisfies $\dist(L_j^k,L_h^k)\leq 2\,\ell_j^k$,
then 
$$\dist_H(\rho_j^k\cap 2B_j^k,\,\rho_h^k\cap 2B_j^k)\lesssim \ve_0\,\ell_j^k.$$
In particular, 
\begin{equation}\label{eqang33}
\meas (\rho_{j-1}^k,\,\rho_j^k)\lesssim \ve_0.
\end{equation}
\end{lemma}

\begin{proof} For $k\geq 1$ and $1\leq j
\leq N_k$, consider the ball $B_j^k=B(z_j^k,\ell_j^k)$ and the segment $L_j^k$ with endpoints $x_{j-1}^k,x_j^k\in\supp\wt\mu$.
 Suppose that
$\wt d(x_{j-1}^k)>0$. Then there exists some cell $Q_i$, $i\in I$, such that $x_{j-1}^k\in Q_i$. By (b) in the preceding lemma, $\ell_j^k\gtrsim\ell(Q_i)$. So there exists some cell $P\supset Q_i$ such that
$4B_P\supset 2B_j^k$, with $\ell(P)\approx\ell_j^k$. By Lemma \ref{claf23} and \rf{eqahjd398},
\begin{equation}\label{eqbsb}
b\beta_{\infty,\wt\mu}(P)\leq \ve_0.
\end{equation}
Moreover, since the endpoints of $L_j^k$ are both in $\supp\wt\mu$ and $\ell_j^k\approx \ell(P)$, it easily follows that 
\begin{equation}\label{eqws42}
\dist_H(\rho_j^k\cap 4B_P,\rho_P\cap 4B_P)\leq c\,\beta_{\infty,\wt\mu}(P)\,\ell(P)\leq c\,\ve_0\,\ell(P),
\end{equation}
where $\rho_P$ stands for a best approximating line for $b\beta_{\infty,\wt\mu}(P)$. From this fact and
\rf{eqbsb} one infers that $\supp\wt\mu\cap 2B_j^k$ is contained in the $(c'\ve_0\,\ell_j^k)$-neighborhood of $\rho_j^k$. The proof is analogous if $\wt d(x_{j-1}^k)=0$.

\vv
The second statement of the lemma follows as above, just taking the cell $P$ big enough so that
$2B_j^k\cup 2B_h^k \subset4B_P$, still with $\ell(P)\approx\ell_j^k$. We leave the details for the reader. 
\end{proof}

\vv
Next we intend to show that each curve $\Gamma_k$ is AD-regular, with a constant uniform on $k$.
The lemma below is the first step.

\begin{lemma}\label{lemcla11}
For fixed $k\geq1$ and $j\in[1,N_k]$, the only segment of the family $\{L_h^k\}_{1\leq h \leq N_k}$ that intersects the open ball $\tfrac12B_j^k= B(z_j^k,\tfrac12\ell_j^k)$ is $L_j^k$. In other words,
$$\tfrac12\,B_j^k\cap \Gamma_j^k = \tfrac12\,B_j^k\cap L_j^k.$$
\end{lemma}

\begin{proof}
Suppose that the statement above does not hold and let us argue
by contradiction. Consider the least integer $k\geq1$ such that there exists $h,j\in[1,N_k]$, with
$h\neq j$, such that
\begin{equation}\label{eqint53}
L_h^k\cap \tfrac12 B_j^k \neq\varnothing.
\end{equation}
By construction, we must have $k\geq2$. By the preceding lemma, for each $m\in[1,N_{k-1}-1]$, the angle
between the lines 
$\rho_{m}^{k-1}$ and $\rho_{m+1}^{k-1}$ is bounded by $c\,\ve_0$. This implies that either
$\meas(x_{m-1}^{k-1},x_{m}^{k-1},x_{m+1}^{k-1})$ is very close to $0$ or very close to $\pi$.
Since $L_{m}^{k-1}$ does not intersect
$\tfrac12 B_{m+1}^{k-1}$, this angle must be very close to $\pi$. That is,
\begin{equation}\label{eqoo1}
|\meas(x_{m-1}^{k-1},x_{m}^{k-1},x_{m+1}^{k-1})-\pi|\lesssim\ve_0.
\end{equation} 
Because of the way $\Gamma^k$ is generated from $\Gamma^{k-1}$, we infer that the angles
$\meas(x_{m-1}^{k},x_{m}^{k},x_{m+1}^{k})$ are also  very close to $\pi$ for all $m\in[1,N_k-1]$.
As a consequence, if $L_h^k$ and $B_j^k$ satisfy \rf{eqint53}, then $|h-j|\geq N(\ve_0)$, where 
$N(\ve_0)$ is some big integer depending only on $\ve_0$ which tends to $\infty$ as $\ve_0\to0$.

Consider the segments $L_{h'}^{k-1}$ and $L_{j'}^{k-1}$ which
generate $L_h^k$ and $L_j^k$ respectively. Notice that
$$|h'-j'|\geq \frac{|h-j|}2 - 1 \geq \frac14\,N(\ve_0),$$
for $N(\ve_0)$ big enough (i.e.\ $\ve_0$ small enough).
Take $y\in L_h^k\cap \frac12B_j^k$ and $y'\in L_{h'}^{k-1}$ with $\Pi_{k-1}(y')=y$, so that
$$|y-y'|\lesssim \ve_0\,\ell_{h'}^{k-1}\approx \ve_0\,\ell_{h}^{k}\approx \ve_0\,\ell_{j}^{k},$$
by Lemma \ref{lemnofac} (a), (c). By Lemma \ref{lemfac32},
we deduce that 
$$\dist_H(\rho_h^k\cap B_j^k,\rho_j^k\cap B_j^k)\lesssim \ve_0\,\ell_j^k,$$
and so
$\dist(y,L_j^k\cap\tfrac12 B_j^k)\lesssim\ve_0\,\ell_j^k$. Thus,
there exists some $x\in L_j^k\cap\tfrac12 B_j^k$ such that $|x-y|\lesssim\ve_0\,\ell_j^k$.
We take now $x'\in L_{j'}^{k-1}$ such that $\Pi_{k-1}(x')=x$, which, in particular, implies that
$$|x-x'|\lesssim \ve_0\,\ell_{j}^{k}.$$
Then we have
$$|x'-y'|\leq |x'-x| + |x-y| + |y-y'|\leq c\,\ve_0\,\ell_j^{k}\leq \frac{1}{10}\,\ell_{j'}^{k-1},$$
assuming $\ve_0$ small enough.
Therefore, 
$$L_{h'}^{k-1}\cap \tfrac34\,B_{j'}^{k-1}\neq \varnothing.$$
Then from Lemma \ref{lemfac32} we deduce that 
$$\dist_H(\rho_{j'}^{k-1}\cap 2B_{j'}^{k-1},\,\rho_{h'}^{k-1}\cap 2B_{j'}^{k-1})\lesssim \ve_0\,\ell_{j'}^{k-1}.$$
From this estimate and \rf{eqoo1} 
we infer that there exists some $h''\in[1,N_{k-1}]$, with $|h''-h'|\leq c_{10}$ (where
$c_{10}$ is some absolute constant), such that
\begin{equation}\label{claim932}
L_{h''}^{k-1}\cap \tfrac12 B_{j'}^{k-1} \neq\varnothing.
\end{equation}
The fact that $|h''-h'|\leq c_{10}$ and $|h'-j'|\geq N(\ve_0)$ ensures that $h''\neq j'$. This contradicts the minimality of $k$ and proves the lemma.
\end{proof}

\vv
\begin{lemma}\label{lemfac33}
For all $k\ge1$ and $1\leq j\leq N_k-1$,
\begin{equation}\label{eqoo2}
|\meas\bigl(x_{j-1}^{k},x_{j}^{k},x_{j+1}^{k}\bigr)-\pi|\lesssim\ve_0.
\end{equation} 
\end{lemma}

\begin{proof}
This has been shown in \rf{eqoo1}.
\end{proof}

\vv
\begin{lemma}\label{lemfac33.5}
For every fixed $k\geq1$, the balls $\tfrac16B_j^k$, $1\leq j\leq N_k$, are pairwise disjoint.
\end{lemma}

\begin{proof}
Suppose not. Let $1\leq j,h\leq N_k$ be such that $\tfrac16\,B_j^k\cap \tfrac16\,B_h^k\neq \varnothing$, with $h\neq j$, and $\ell_j^k\geq\ell_h^k$, say.
Then $\tfrac16B_h^k\subset \frac12B_j^k$ and thus $L_h^k$ intersects $\frac12B_j^k$, which
contradicts Lemma~\ref{lemcla11}.
\end{proof}

\vv
\begin{lemma}\label{lemfac34}
For all $k\geq 1$, we have
$$\supp\wt\mu \cap B(x_0,\tfrac14 K r_0)\subset \bigcup_{j=1}^{N_k} B_j^k.$$
\end{lemma}

\begin{proof} We will argue by induction on $k$. This clearly holds for $k=1$, 
taking into account Lemma \ref{claf23}. Suppose now this holds for $k$ and let us see how this follows for
$k+1$. Consider the ball $B_j^k$, for some $k\geq 1$ and $1\leq j
\leq N_k$. Take a segment $L_h^{k+1}$ generated by $L_j^k$. By construction, we have $\ell_h^{k+1}
\approx \ell_j^k$, and arguing as in Lemma \ref{lemfac32},
\begin{equation}\label{eqav12}
\dist_H(\rho_h^{k+1}\cap B_j^k, \rho_j^{k}\cap B_j^k)\leq c\,\ve_0\,\ell_j^k.
\end{equation}
Consider now the maximal integers $m,n\geq 0$ such that all the balls
\begin{equation}\label{eqboles99}
B_{h-m}^{k+1},\,B_{h-m+1}^{k+1},\ldots, B_h^{k+1},\ldots,B_{h+n-1}^{k+1},\,B_{h+n}^{k+1}
\end{equation}
intersect $B_j^k$. By (c) from Lemma \ref{lemnofac} it follows easily that $\ell_p^{k+1}\approx
\ell_j^k$ for $h-m\leq p \leq h+n$ and moreover $m$ and $n$ are uniformly bounded. 
Further \rf{eqav12} also holds replacing $\rho_h^{k+1}$ by $\rho_p^{k+1}$ and by Lemma \ref{lemfac33},
$$|\meas(x_{p-1}^{k+1},x_{p}^{k+1},x_{p+1}^{k+1})-\pi|\lesssim\ve_0,$$
for all $p$.
By elementary geometry,
the segments
$L_{h-m}^{k+1},\,L_{h-m+1}^{k+1},\ldots, L_h^{k+1},\ldots,L_{h+n-1}^{k+1},\,L_{h+n}^{k+1}$
form a polygonal line $\gamma$ such that
$$\dist_H(\gamma\cap B_j^k, \rho_j^k\cap B_j^k)\lesssim \ve_0\,\ell_j^k.$$
Moreover, one can also verify that, for $\ve_0$ small enough,
the intersection of the $(c'\ve_0\,\ell_j^k)$-neighborhood of $\rho_j^k$ with $B_j^k$ is contained
in the union of the balls \rf{eqboles99}, and so 
$$\supp\wt\mu \cap  B_j^k \subset \bigcup_{p=h-m}^{h+n} B_p^{k+1},$$
which yields
$$\supp\wt\mu \cap  \bigcup_{j=1}^{N_k} B_j^k \subset \supp\wt\mu\cap \bigcup_{h=1}^{N_{k+1}} B_h^{k+1}.$$
\end{proof}

\vv

\begin{lemma}\label{lemad}
The curves $\Gamma^k$ are AD-regular uniformly on $k$, with the AD-regularity constant bounded
by $c\,A\,\tau^{-1}$. 
\end{lemma}

\begin{proof}
Since $\Gamma^k$ is a curve, we only have to check the upper AD-regularity.  Let $B(x,r)$ be a 
ball centered at some point $x\in L_j^k$. Suppose first that $r\leq 2\,\ell_j^k$. If $B(x,r)$ intersects
another segment $L_h^k$, then $\ell_h^k\approx\ell_j^k$ because $\dist(L_j^k,L_h^k)\leq r\leq 2
\,\ell_j^k$. Therefore, there exists some absolute constant $c\geq1$ such that $B_h^k\subset c\,B_j^k$. Since the balls $\frac16B_h^k$, $1\leq h \leq N_k$, are pairwise disjoint, it follows
 that the number of balls $B_h^k$ contained in $c\,B_j^k$ which satisfy $r(B_h^k)\approx
 r(B_j^k)$ is uniformly bounded above. Then we derive
 $$\HH^1(\Gamma^k\cap B(x,r))\leq \sum_{h:B_h^k\subset c\,B_j^k}\HH^1(L_h^k\cap B(x,r))\leq c\,r.$$

Suppose now that $r> 2\,\ell_j^k$. First we claim that if $B(x,r)$ intersects
another segment $L_h^k$, then $\ell_h^k\leq M_0\,r$, for some absolute constant $M_0$. Indeed, if 
$r\leq \ell_h^k$, then we obtain
$$\dist(L_h^k,L_j^k)\leq r\leq \ell_h^k,$$
which implies that $\ell_h^k\approx\ell_j^k\leq \frac12\,r$, and proves the claim.
So we deduce that the ball $B_h^k$ is contained in $B(x,C\,r)$, for some $C\geq1$.

Now we write
$$\HH^1(\Gamma^k\cap B(x,r))\leq \sum_{h:B_h^k\subset B(x,C\,r)}\HH^1(L_h^k).$$
Observe now that
$\mu(\tfrac16B_h^k)\gtrsim \tau\,\Theta_\mu(B_R)\,\ell_h^k$
by \rf{eqdob842}, and thus
\begin{equation}\label{eqplug6}
\HH^1(\Gamma^k\cap B(x,r)) \lesssim \frac1{\tau\,\Theta_\mu(B_R)}
\sum_{h:B_h^k\subset B(x,C\,r)}\mu(\tfrac16B_h^k).
\end{equation}
Since, for a fixed $k$, the balls $\tfrac16B_h^k$ are disjoint, we have
$$\sum_{h:B_h^k\subset B(x,C\,r)}\mu(\tfrac16B_h^k)\leq \mu(B(x,C\,r))\leq c\,\,A\,\Theta_\mu(B_R)\,r.$$
Plugging this estimate into \rf{eqplug6} we obtain
$$ \HH^1(\Gamma^k\cap B(x,r)) \lesssim A\,\tau^{-1}\,r.$$
\end{proof}

\begin{remark}
It is easy to check that the limit in the Hausdorff metric of the sequence of curves $\{\Gamma^k\}_k$
exists. By the preceding lemma, it is an AD-regular curve $\Gamma$ with the AD-regularity constant bounded
by $c\,A\,\tau^{-1}$. 
\end{remark}

\vv
The next lemma asserts that, in a sense, $\supp\wt\mu$ is very close to $\Gamma^k$.

\begin{lemma}\label{lemaprop1}
If $x\in\supp\wt\mu \cap B(x_0,\tfrac14 K r_0)$, then
$$\dist(x,\Gamma^k)\lesssim \ve_0\,\max\bigl(\wt d(x),\,2^{-k/2} d_0 \bigr),$$
for all $k\geq 1$.
\end{lemma}

\begin{proof}
By Lemma \ref{lemfac34}, there exists some ball $B_j^k$ which contains $x$. Therefore,
$$\dist(x,\rho_j^k)\lesssim\ve_0\,\ell_j^k.$$
By Lemmas \ref{lemnofac} (d)
and \ref{lemfac33}, we deduce that
$$\dist(x,\Gamma^k)\lesssim\ve_0\,\ell_j^k.$$
On the other hand, by Lemma \ref{lemnofac} (b), (c), since $\dist(x,L_j^k)\leq 2\,\ell_j^k$,
we have
$$\ell_j^k\approx \max\bigl(\wt d(x),\,2^{-k/2} d_0 \bigr),$$
and thus we are done.
\end{proof}

\vv

Note that, in particular, from the preceding lemma one deduces that $N\!W_0$ is supported in the limiting curve $\Gamma$. So we have:

\begin{lemma}\label{lemz0}
The set $N\!W_0$ is rectifiable.
\end{lemma}

\vv
The next result can be understood as a kind of converse of Lemma \ref{lemaprop1}.
Roughly speaking, it asserts that for each $x\in\Gamma^k$ there exists some point from
$\supp\wt \mu$ which is very close.

\begin{lemma}\label{lemaprop2}
Let $k\geq 1$ and $1\leq j\leq N_k$. For every $x\in\Gamma^k\cap B_j^k$ there exists some $x'\in
\supp\wt\mu$ such that
$$|x-x'|\lesssim \ve_0\,\ell_j^k.$$
\end{lemma}

\begin{proof}
This follows from the fact that $b\beta_{\infty,\wt\mu}(2B_j^k)\lesssim \ve_0$ and since $x_{j-1}^k,x_j^k\in \supp\wt\mu
\cap 2B_j^k$ and $|x_{j-1}^k-x_j^k|\approx\diam(B_j^k)$ we infer that $\dist_H(\rho_j^k\cap 2B_j^k,L_{2B_j^k}\cap 2B_j^k)\lesssim
\ve_0$, where $L_{2B_j^k}$ is the best approximating line for $b\beta_{\infty,\wt\mu}(2B_j^k)$.
\end{proof}





\vv

Finally we have:

\vv
\begin{lemma}\label{lembeta44}
Let $L_{j_1}^1,L_{j_2}^2,\ldots,L_{j_k}^k$ be a sequence of segments such that $L_{j_{m+1}}^{m+1}$ is generated by
$L_{j_{m}}^{m}$ for $m=1,\ldots,k-1$. Then
$$\sum_{m=1}^{k-1} \meas(\rho_{j_m}^m,\rho_{j_{m+1}}^{m+1})^2 \leq c\,M.$$
\end{lemma}

\begin{proof}
Is is easy to check that
$$\meas(\rho_{j_m}^m,\rho_{j_{m+1}}^{m+1})\lesssim \beta_{\infty,\wt\mu}(B_{j_m}^m).$$
Let $Q\in\qgood$ be a cell such that $x_{j_k}^k\in Q$ with $\ell(Q)\approx \ell_{j_k}^k$. By the construction of the cells from $\bsb$, we
have
$$\sum_{P\in\DD:Q\subset P\subset B(x_0,Kr_0)} \beta_{\infty,\wt\mu}(2B_P)^2\lesssim M.$$
Then we deduce that
$$\sum_{m=1}^{k-1}\beta_{\infty,\wt\mu}(B_{j_m}^m)^2 \lesssim \sum_{P\in\DD:Q\subset P\subset B(x_0,Kr_0)} \beta_{\infty,\wt\mu}(2B_P)^2\lesssim M,$$
and we are done.
\end{proof}

\vv

\begin{remark}\label{remgr}
If in the construction of the curves $\Gamma^k$ above we replace the function $\wt d(\cdot)$
by $d(\cdot)$ and the
cells $\{Q_i\}_{i\in I}$ by the cells
from the family $\reg$ which have positive $\wt\mu$ measure, we will get curves $\Gamma_R^k$ which 
satisfy properties analogous to the ones of $\Gamma^k$, with the exception of the one stated in Lemma
\ref{lembeta44}. So very similar versions of Lemmas \ref{lemnofac}-\ref{lemaprop2} will hold for 
$\Gamma_R^k$, $k\geq1$. Moreover, letting
$\Gamma_R$ be the limit in the Hausdorff metric of the curves $\Gamma_R^k$, one gets
$W_0\subset \Gamma_R$ and so $W_0$ is rectifiable. Using the fact that $\Theta_\mu(1.1B_Q)\lesssim A\,
\Theta_\mu(B_R)$ for any $Q\in\DD$ with $\ell(Q)\leq \ell(R)$ such that $\mu(Q\cap W_0)>0$, it follows easily that
 $\mu|_{W_0}$ is absolute
continuous with respect to $\HH^1|_{\Gamma_R}$.
\end{remark}
\vvv


\section{The small measure $\wt\mu$  of the cells from $\bsb$}\label{sec9}

Recall that $Q\in\bsb_0$ 
 if $Q\not\in BCF_0\cup LD_0\cup HD_0\cup BCG_0\cup BC\Delta_0$, $\ell(Q)\leq \ell(R)$, and
\begin{equation}\label{eqsk88}
\sum_{P\in \DD: Q\subset P\subset R}  \beta_{\infty,\wt\mu}(2B_P)^2\geq M.
\end{equation}
The cells from $\bsb$ are the ones from $\nterm$ which belong to $\bsb_0$.
We denote by $\bsb_1$ the cells from $\bsb$ which are contained in $B(x_0, \frac1{10}K r_0)$.

In this section we will prove the following:

\begin{lemma}\label{lembsb}
Assume that $M$ is big enough (depending only $A$ and $\tau$).
Then
\begin{equation}\label{eqatau**}
\wt\mu\biggl(\bigcup_{Q\in\bsb_1}Q\biggr)\leq \frac{c(A,\tau ,K )}M\,\mu(R).
\end{equation}
\end{lemma}

Recall that in \rf{eqconstants*} we said that the parameters $\eta,\delta,\tau,A,K$ will be chosen so that 
 $$\eta\ll\delta\ll\tau\ll A^{-1}\ll K^{-1}\ll 1.$$
On the other hand, for the constant $M$ we will need that
\begin{equation}\label{eqconstants*2}
M\gg c(A,\tau,K),
\end{equation}
where $c(A,\tau,K)$ is the constant in \rf{eqatau**}. Further we will also require that $\eta,\delta\ll M^{-1}$.

To prove the lemma we will use the usual lattice $\DD(\R^d)$ of dyadic cubes of $\R^d$. Given a cube $Q\in\DD(\R^d)$, we denote by $\ell(Q)$ its side length
and by $z_Q$ its center. We define 
$$\beta_{\infty,\Gamma^k}(Q) = \inf_L \sup_{y\in 3Q\cap\Gamma^k} \frac{\dist(y,L)}{\ell(Q)},$$
where $3Q$ stands for the cube concentric with $Q$ with side length $3\ell(Q)$.
\vv

\begin{proof}[Proof of Lemma \ref{lembsb}]
Consider the following auxiliary curve:
$$\wt\Gamma^k =\Gamma^k \cup \bigcup_{j=1}^{N_k} \partial B_j^k.$$
Since $\HH^1(\wt\Gamma^k)\lesssim \HH^1(\Gamma^k)\lesssim_{A,\tau} K\,\ell(R)$,
by Jones' traveling salesman theorem \cite{Jones}, \cite{Okikiolu}, it follows that
\begin{equation}\label{eqh4329}
\sum_{Q\in\DD(\R^d)}\beta_{\infty,\wt\Gamma^k}(Q)^2\,\ell(Q) \lesssim K\,\ell(R).
\end{equation}

Now we claim that
\begin{equation*}
\sum_{\substack{P\in\qgood:\\P\subset  B(x_0,\tfrac14 K r_0)\\
\ell(P)\geq 2^{-k/2}d_0}} \beta_{\infty,\wt\mu}(2B_P)^2\,\ell(P)
\lesssim_{A,\tau} \sum_{\substack{Q\in\DG:\\ Q\subset B(x_0,Kr_0)}}\beta_{\infty,\wt\Gamma^k}(Q)^2\,\ell(Q).
\end{equation*}
To this end, recall that 
$$\supp\wt\mu \cap B(x_0,\tfrac14 K r_0)\subset \bigcup_{j=1}^{N_k} B_j^k.$$
Take a cell $P\in\qgood$ such that $\ell(P)\geq 2^{-k/2}d_0$. To such a
cell we can associate a cube $Q(P)\in\DD(\R^d)$ such that $2B_P\subset 3Q(P)$ and $\ell(Q(P))\approx\ell(P)$. Then it follows that
$$\beta_{\infty,\wt\mu}(2B_P)\lesssim \beta_{\infty,\wt\Gamma^k}(Q(P)),$$
and since for a given $Q\in\DD(\R^d)$, the number of cells $P\in\qgood$ such that $Q=Q(P)$ does not
exceed some absolute constant, the claim follows. Together with  
 \rf{eqh4329}, this gives
$$\sum_{\substack{P\in\qgood:\\P\subset  B(x_0,\tfrac14 K r_0)\\
\ell(P)\geq 2^{-k/2}d_0}} \beta_{\infty,\wt\mu}(2B_P)^2\,\ell(P)
\lesssim_{A,\tau,K} \ell(R).$$

From the last estimate, taking into account that $\bsb_1\subset\qgood$, by
\rf{eqsk88} and
Chebyshev, we derive
\begin{align*}
\sum_{\substack{Q\in\bsb_1:\\
\ell(Q)\geq 2^{-k/2}d_0}}
\wt\mu(Q) & \leq \frac1M \sum_{\substack{Q\in\bsb_1:\\
\ell(Q)\geq 2^{-k/2}d_0}}\,\sum_{P\in \DD:Q\subset P\subset B(x_0, K r_0)} \beta_{\infty,\wt\mu}(2B_P)^2 \wt\mu(Q)\\
&\lesssim \frac1M \Biggl(\mu(B(x_0, K r_0)) + \sum_{\substack{Q\in\bsb_1
\\ \ell(Q)\geq 2^{-k/2}d_0}
}\,\sum_{P\in \DD:Q\subset P\subset B(x_0, \frac14\,K r_0)} \beta_{\infty,\wt\mu}(2B_P)^2 \wt\mu(Q)
\Biggl)
\\
& \lesssim_{A,\tau,K} \frac{\Theta_\mu(B_R)}M 
\Biggl(\ell(R) + 
\sum_{\substack{P\in\qgood:\\P\subset  B(x_0,\tfrac14 K r_0)\\
\ell(P)\geq 2^{-k/2}d_0}}
\beta_{\infty,\wt\mu}(2B_P)^2\,\ell(P)\Biggr)\\
&\lesssim_{A,\tau,K} \frac{\Theta_\mu(B_R)\,\ell(R)}M\approx_{A,\tau,K} \frac{\mu(R)}M.
\end{align*}
Letting $k\to\infty$, the lemma follows.
\end{proof}

\vv


\section{The approximating measure $\nu^k$ on $\Gamma^k_{ex}$}\label{sec10}

For technical reasons, it is convenient to define an {\bf extended curve $\Gamma^k_{ex}$}.
Recall that the endpoints of $\Gamma^k$ coincide with the endpoints $z_A$, $z_B$ of the
 segment $L_1^1$, which is contained in the line $\rho_1^1$. We set
 $$\Gamma_{ex}^k= \Gamma^k \cup (\rho^1_1\setminus L_1^1).$$ 
We define analogously $\Gamma_{ex}=\Gamma \cup (\rho^1_1\setminus L_1^1)$. Notice that 
$\Gamma_{ex}^1=\rho_1^1$.

In this section, we will construct a measure $\nu^k$ supported on $\Gamma^k_{ex}$ which will approximate 
$\wt\mu$ at the level of the balls $B_j^k$, $1\leq j\leq N_k$. Taking a weak * limit of the measures
$\nu^k$ we will get a measure $\nu$ supported on $\Gamma_{ex}$ which approximates $\wt\mu$ on
$B(x_0,\tfrac14Kr_0)$.

Consider a radial $\CC^\infty$ function $\wt\theta$ which is supported on the ball $B(0,\frac32)$ and
equals $1$ on $\bar B(0,1)$. For $k\geq1$ and $1\leq j\leq N_k$, we set
$$\wt \theta_j^k =\wt\theta\biggl(\frac{x-z_j^k}{\ell_j^k}\biggr).$$
Recall that $B_j^k=B(z_j^k,\ell_j^k)$ and thus $\wt \theta_j^k$ equals $1$ on $B_j^k$ and is supported on
$\frac32 B_j^k$. Notice that 
$$\sum_{j=1}^{N_k}\wt\theta_j^k\approx1 \qquad\mbox{on\quad $\bigcup_{j=1}^{N_k} B_j^k$.}$$
Next we modify the functions $\wt\theta_j^k$ in order to get functions $\theta_j^k$ satisfying
$
\sum_{j=1}^{N_k}\theta_j^k = 1$ on $\bigcup_{j=1}^{N_k} B_j^k$.
For a fixed $k$, we define $\theta_j^k$ inductively on $j$ as follows. First we set
$\theta_1^k=\wt\theta_1^k$. Then we write
$$\theta_2^k = (1-\theta_1^k)\wt\theta_2^k.$$
In general, if $\theta_1^k,\ldots,\theta_j^k$ have already been defined, we set
$$\theta_{j+1}^k = \left(1-\sum_{h=1}^j\theta_j^k\right)\,\wt\theta_{j+1}^k.$$
Also, we define
$$\theta_0^k=1-\sum_{j=1}^{N_k} \theta_j^k.$$

\begin{lemma}\label{lem10.1}
For each $k\geq1$, the functions $\theta_j^k$, $1\leq j\leq N_k$ satisfy the following properties:
\begin{itemize}
\item[(a)] $\theta_j^k$ is a non-negative and it is supported on $\frac32 \bar B_j^k$, and for all $n\geq0$,
$$\|\nabla^n \theta_j^k\|_\infty\leq c(n)\,\frac{1}{(\ell_j^k)^{n}}.$$
\item[(b)] For all $x\in\R^d$,
$$\sum_{1\leq j \leq N_k} \theta_j^k(x) \leq 1.$$
\item[(c)] For all $x\in\bigcup_{1\leq j \leq N_k} B_j^k$,
$$\sum_{1\leq j \leq N_k} \theta_j^k(x) = 1.$$
\end{itemize}
\end{lemma}
\vv

We leave the easy proof for the reader.

\vv
\begin{remark}
Concerning the function $\theta_0^k$, let us remark that 
\begin{equation}\label{eqnogg32}
|\nabla^n \theta_0^k(x)|\leq c(n)\,\frac{1}{r_0^{n}}\qquad\mbox{for all $x\in \Gamma^k$.}
\end{equation}
This is due to the fact that 
$$\Gamma^k\cap \supp(\nabla \theta_0^k) \subset B(z_1^k,C\,\ell_1^k)\cup B(z_{N_k}^k,C\,\ell_{N_k}^k) 
,$$
for some absolute constant $C$, recalling also that $\ell_1^k\approx \ell_{N_k}^k\approx K\,r_0$ for 
all $k\geq1$, by \rf{eqrem31}.

On the other hand, one should not expect \rf{eqnogg32} to hold for all $x\in\R^d$. In this case, one can only
ensure that
$$|\nabla^n \theta_0^k(x)|\leq c(n)\,\frac{1}{\min_j(\ell_j^k)^{n}}.$$
\end{remark}

\vv
We are ready now to define the measures $\nu^k$. For $1\leq j\leq N_k$, we set
\begin{equation}\label{eqnujk}
\nu_j^k = 
c_j^k\,\theta_j^k \,\HH^1|_{\Gamma^k_{ex}}, \quad\mbox{with} \quad c_j^k=\frac{\int\theta_j^k \,d\wt\mu}{\int \theta_j^k \,d\HH^1|_{\Gamma^k_{ex}}}.
\end{equation}
Also, we set $c_0^k=c_1^k$ and $\nu_0^k =c_0^k\,\theta_0^k \,\HH^1|_{\Gamma^k_{ex}}$.
Then we write
$$\nu^k = \sum_{j=0}^{N_k}\nu_j^k.$$ 

\vv

\begin{lemma}\label{lemadnu}
The measure $\nu^k$ is AD-regular. Indeed, there exists some constant $c=c(A,\tau)$ such that
$$c^{-1}\Theta_\mu(B_R)\,r\leq \nu^k(B(x,r))\leq c\,\Theta_\mu(B_R)\,r\qquad
\mbox{for all $x\in \Gamma^k$.}
$$
\end{lemma}

\begin{proof}
This follows easily from the fact that for all $j,k$,
$$c_j^k\approx_{A,\tau}\Theta_\mu(B_R)$$
\end{proof}

\vv


\section{Square function estimates for $\nu^k$}\label{sec11}

Let $\vphi:[0,\infty)\to\R$ be a non-negative $\CC^\infty$ function supported in 
$[0,1]$ which equals $1$ in $[0,1/2]$. 
We denote $\psi_r(z) = \vphi_r(z) - \vphi_{2r}(z)$, with
$$\vphi_r(x)=\frac{1}{r} \vphi \left(\frac {|x|}r \right), \, r>0,$$
so that we have $\Delta_{\mu,\vphi} (x,r)=
\psi_r*\mu(x)$.
In Lemma \ref{lemconvex} we showed that, for
 $0\leq r_1<r_2$, we have
\begin{equation}\label{eqdd923}
\int_{r_1}^{r_2} |\Delta_{\mu,\vphi}(x,r)|^2\,\frac{dr}r \leq c
\int_{r_1/2}^{2r_2} |\Delta_{\mu}(x,r)|^2\,\frac{dr}r.
\end{equation}
Recall that $\wt \mu=\mu|_{\wt E}$, with
$$\wt E = B(x_0,2Kr_0)\cap \biggl(W_0\cup \bigcup_{Q\in\reg} \bigl[4B_{Q} \cap F\cap
\GG(Q,R,\delta^{1/4},\eta)\bigr]\biggr).$$
If $x\in \wt E$, then either $x\in W_0$ or there exists some some $Q'\in \reg$   such that 
$x\in 4B_{Q'}\cap F\cap
\GG(Q',R,\delta^{1/4},\eta)$. If $Q$ is the cell from $\reg$ which contains $x$, then $\ell(Q)\approx
\ell(Q')$, and by \rf{eqdd923} and the definition of $\GG(Q',R,\delta^{1/4},\eta)$ it follows that
\begin{equation}\label{eqaj43}
\int_{c\delta^{1/4}\ell(Q)}^{c^{-1}\delta^{-1/4}\ell(R)} \Delta_{\mu,\vphi}(x,r)^2\,\frac{dr}r
\lesssim \int_{\delta^{1/4}\ell(Q')}^{\delta^{-1/4}\ell(R)} \Delta_\mu(x,r)^2\,\frac{dr}r \lesssim \eta\,
\Theta_\mu(B_R)^2.
\end{equation}

The next objective consists in proving the following.

\begin{lemma}\label{lem101}
Let $k\geq 1$ and $1\leq j\leq N_k$. For every $x\in\Gamma^k\cap B_j^k$,
$$\int_{\ell_j^k}^{ K r_0/4} \left|\Delta_{\mu,\vphi}(x,r)\right|^2\,\frac{dr}r \lesssim
 (\eta+\ve_0^2)\,A^2\,\Theta_\mu(B_R)^2.$$
\end{lemma}

\begin{proof}
By Lemma \ref{lemaprop2}, there exists some $x'\in
\wt E$ such that
$|x-x'|\lesssim \ve_0\,\ell_j^k$.
From Remark \ref{rem31} it follows easily that there exists some cell $Q\in\qgood$ with
$\ell(Q)\approx\ell_j^k$ which contains $x'$. Together with \rf{eqaj43}
this gives
$$\int_{\ell_j^k}^{ K r_0} \left|\Delta_{\mu,\vphi}(x',r)\right|^2\,\frac{dr}r\lesssim
\eta\,\Theta_\mu(B_R)^2.$$
By Lemma \ref{claf22}, we know that
$\mu(B(x',r))\lesssim A\,\Theta_\mu(B_R)\,r$ for $\ell_j^k\leq r\leq K r_0$.
 Next
 note that, for $r$ such that $|x-x'|\leq r\leq K r_0$,
$$|\psi_r*\mu(x) - \psi_r*\mu(x')|\leq |x-x'|\sup_{z\in[x,x']}|\nabla\,(\psi_r*\mu)(z)|\lesssim 
\frac{|x-x'|}{r^2}\,\mu(B(x',4r))\lesssim \frac{\ve_0\,\ell_j^k}{r}A\,\Theta_\mu(B_R).
$$
Therefore,
$$\left|\Delta_{\mu,\vphi}(x,r)-\Delta_{\mu,\vphi}(x',r)\right|\lesssim \frac{\ve_0\,\ell_j^k}{r}A\,\Theta_\mu(B_R)$$
for $r$ such that $|x-x'|\leq r\leq K r_0$ and $\ell_j^k\leq r$. Thus,
\begin{align*}
\int_{\ell_j^k}^{ K r_0/4} \left|\Delta_{\mu,\vphi}(x,r)\right|^2\,\frac{dr}r
& \leq 2\int_{\ell_j^k}^{ K r_0/4} \left|\Delta_{\mu,\vphi}(x',r)\right|^2\,\frac{dr}r +
c\,\bigl(A\,\ve_0\Theta_\mu(B_R)\bigr)^2\int_{\ell_j^k}^\infty \frac{(\ell_j^k)^2}{r^2}\,
\,\frac{dr}r\\
&\leq 2\,\eta \,A^2\,\Theta_\mu(B_R)^2+ c\,\ve_0^2\,A^2\,\Theta_\mu(B_R)^2.
\end{align*}
\end{proof}
\vv

Recall the definition of $\nu_j^k$ and $c_j^k$ in \rf{eqnujk}.

\begin{lemma}\label{lemalfa230}
Suppose that, for some $k\geq 1$ and some $1\leq j\leq N_k$, $2B_j^k\subset B(x_0,\tfrac16 K r_0)$. For all $x\in\R^d$ and all $r\geq c^{-1}\ell_j^k$, we have
\begin{equation}\label{eqalf1000}
\left|\int\psi_r(x-y)\,d\nu_j^k(y) -  \int\psi_r(x-y)\,\theta_j^k(y)\,d\wt\mu(y)\right|\lesssim 
\ve_0\,\frac{\wt\mu(B_j^k)\,\ell_j^k}{r^2}.
\end{equation}
\end{lemma}

Let us remark that the condition $2B_j^k\subset B(x_0,\tfrac16 K r_0)$ guaranties that
$2B_j^k$ is far from the endpoints of $\Gamma^k$.

\begin{proof}
Recalling that 
$\nu_j^k = c_j^k\,\theta_j^k \HH^1|_{\Gamma^k}$ and that 
$\int d\nu_j^k=\int \theta_j^k \,d\wt\mu$, 
we have
\begin{align}\label{eqdkh312}
\int\psi_r(x-y)\,d\nu_j^k(y) -  \int&\psi_r(x-y)\,\theta_j^k(y)\,d\wt\mu(y) \\
& =
\int\bigl(\psi_r(x-y)-\psi_r(x-z_j^k)\bigr)\,\,d(\nu_j^k - \theta_j^k\wt\mu)(y)\nonumber\\
& =
\int\bigl(\psi_r(x-y)-\psi_r(x-z_j^k)\bigr)\,\theta_j^k(y)\,d(c_j^k\HH^1|_{\Gamma^k} - \wt\mu)(y).
\nonumber
\end{align}
To estimate the last integral we wish to apply that
$\alpha_{\wt\mu}(2B_j^k)\lesssim\ve_0$. Denote by $c_{2B_j^k}$ and $L_{2B_j^k}$ the constant and the
line minimizing $\alpha_{\wt\mu}(2B_j^k)$.
For fixed $x$, denote $f(y)=\psi_r(x-y)-\psi_r(x-z_j^k)$.
Then the left side of \rf{eqdkh312} can be written as follows
\begin{align*}
\int f(y)\,\theta_j^k(y)\,d(c_j^k\HH^1|_{\Gamma^k} - \wt\mu)(y)
& = (c_j^k- c_{2B_j^k})\int f(y)\,\theta_j^k(y)\,d\HH^1|_{\Gamma^k}(y)\\
&\quad + c_{2B_j^k}\int f(y)\,\theta_j^k(y)\,d(\HH^1|_{\Gamma^k} - \HH^1|_{L_{2B_j^k}})(y)\\
&\quad + \int f(y)\,\theta_j^k(y)\,d(c_{2B_j^k}\HH^1|_{L_{2B_j^k}} - \wt\mu)(y)\\
&= T_1 + T_2 + T_3.
\end{align*}

To estimate $T_2$ we use the fact that $|c_{2B_j^k}|\lesssim \Theta(2B_j^k)$,
by Lemma \ref{lempr0} (c). Then we have
\begin{equation}\label{eqt235}
|T_2|\lesssim \Theta_{\wt\mu}(2B_j^k)\,{\rm Lip}\bigl(f\,\theta_j^k\bigr) \,\dist_{2B_j^k}(\HH^1|_{\Gamma^k} ,\HH^1|_{L_{2B_j^k}}).\end{equation}
Observe that
$$
\|f\|_{\infty,2 B_j^k} =
\|\psi_r(x-\cdot)-\psi_r(x-z_j^k)\|_{\infty,2 B_j^k}\,\lesssim \frac{\ell_j^k}{r^2}
$$
and
\begin{equation}\label{eqdjf37}
{\rm Lip}\bigl(f\,\theta_j^k\bigr) 
 \leq \|\nabla f\|_\infty \|\theta_j^k\|_\infty + 
\|f\|_{\infty,2 B_j^k}\,\|\nabla\theta_j^k\|_\infty \lesssim \frac1{r^2} + \frac{\ell_j^k}{r^2}\,\frac1{\ell_j^k}\approx \frac1{r^2}.
\end{equation}
From Lemmas \ref{claf23} and \ref{lemfac33} and the construction of $\Gamma^k$, one can easily check that
$$\dist_H(2B_j^k\cap \Gamma^k,2B_j^k\cap L_{2B_j^k})\lesssim\ve_0\,\ell_j^k$$
and also that
\begin{equation}\label{eqdk441}
\dist_{2B_j^k}(\HH^1|_{\Gamma^k} ,\HH^1|_{L_{2B_j^k}})\lesssim \ve_0\,(\ell_j^k)^2.
\end{equation}
Therefore, by \rf{eqt235}, \rf{eqdjf37}, \rf{eqdk441}, and \rf{eqdob842}, we obtain
$$|T_2|\lesssim \Theta_{\wt\mu}(2B_j^k)\,\frac1{r^2}\,\ve_0\,(\ell_j^k)^2 \approx
\ve_0\,\wt\mu(B_j^k)\,\frac{\ell_j^k}{r^2}.$$

Concerning $T_3$, using \rf{eqdjf37} and \rf{eqdob842} again, we get
$$
|T_3| \lesssim {\rm Lip}\bigl(f\,\theta_j^k\bigr)\,\alpha_{\wt\mu}(2B_j^k)\,\wt\mu(2B_j^k)\,\ell_j^k
\lesssim \alpha_{\wt\mu}(2B_j^k)\,\wt\mu(B_j^k)\,\frac{\ell_j^k}{r^2}\lesssim
\ve_0\,\wt\mu(B_j^k)\,\frac{\ell_j^k}{r^2}.$$

To deal with $T_1$ we need first to estimate $|c_j^k- c_{2B_j^k}|$. To this end, we write
\begin{multline}\label{eqsjk2}
\biggl|\int \theta_j^k\,d\wt\mu - c_{2B_j^k}\int_{\Gamma^k}\theta_j^k\,d\HH^1\biggr|\\
\leq
\left|\int \theta_j^k\,d\wt\mu - c_{2B_j^k}\int_{L_{2B_j^k}}\theta_j^k\,d\HH^1\right|
+
\left| c_{2B_j^k}\int_{L_{2B_j^k}}\theta_j^k\,d\HH^1 -
c_{2B_j^k}\int_{\Gamma^k}\theta_j^k\,d\HH^1 \right|.
\end{multline}
Since $\|\nabla\theta_j^k\|_\infty\lesssim1/\ell_j^k$, the first term on the right hand
side does not exceed
$$c\,\frac1{\ell_j^k}\,\alpha_{\wt\mu}(2B_j^k)\,\wt\mu(B_j^k)\,\ell_j^k\lesssim
\ve_0\,\wt\mu(B_j^k).$$
Arguing as in \rf{eqt235}, we deduce that the last term on the right hand side of \rf{eqsjk2}
is bounded by
$$\Theta_{\wt\mu}(2B_j^k)\,{\rm Lip}\bigl(\theta_j^k\bigr)\,\dist_{2B_j^k}(\HH^1|_{\Gamma^k} ,\HH^1|_{L_{2B_j^k}})\lesssim \Theta_{\wt\mu}(2B_j^k)\,\frac1{\ell_j^k}\,\ve_0\,(\ell_j^k)^2
\lesssim\ve_0\,\wt\mu(B_j^k).$$
Then we deduce that
$$\biggl|\int \theta_j^k\,d\wt\mu - c_{2B_j^k}\int_{\Gamma^k}\theta_j^k\,d\HH^1\biggr|
\lesssim\ve_0\,\wt\mu(B_j^k).$$
Recalling that
$c_j^k=\frac{\int\theta_j^k \,d\wt\mu}{\int \theta_j^k \,d\HH^1|_{\Gamma^k}},$
we obtain
$$\bigl|c_j^k- c_{2B_j^k}\bigr|  = \frac1{\int_{\Gamma^k} \theta_j^k \,d\HH^1}\,
\biggl|\int \theta_j^k\,d\wt\mu - c_{2B_j^k}\int_{\Gamma^k}\theta_j^k\,d\HH^1\biggr|
\lesssim \frac{\ve_0\,\wt\mu(B_j^k)}{\int_{\Gamma^k} \theta_j^k \,d\HH^1}.$$
Therefore, we have
$$|T_1|\lesssim 
|c_j^k- c_{2B_j^k}|\,\|f\|_{\infty,2 B_j^k}\int_{\Gamma^k} \theta_j^k\,d\HH^1
\lesssim
\frac{\ve_0\,\wt\mu(B_j^k)}{\int_{\Gamma^k} \theta_j^k \,d\HH^1}\,
\frac{\ell_j^k}{r^2}\int_{\Gamma^k} \theta_j^k\,d\HH^1
= \frac{\ve_0\,\wt\mu(B_j^k)\ell_j^k}{r^2} 
.$$
Gathering the estimates obtained for $T_1$, $T_2$ and $T_3$, the lemma follows.
\end{proof}

\vv
\begin{lemma}\label{lem103}
For $x\in \Gamma^k$, let $\ell^k(x)$ denote the segment $L_j^k$ which contains $x$ (if this is not unique, the choice does not matter).
We have
\begin{equation}\label{eqlemmac1}
\int_{\gex\cap B(x_0,\frac18 K r_0)}\int_{\ell^k(x)}^{ K r_0/100} |
\Delta_{\wt\mu,\vphi}(x,r) - \Delta_{\nu^k,\vphi}(x,r)|^2\,\frac{dr}r\,d\HH^1(x) \lesssim_{A,\tau,K}\,\ve_0^2\,\Theta_\mu(B_R)^2\,\ell(R).
\end{equation}
\end{lemma}

Note that in the integral above $\supp\nu^k\cap B(x_0,\frac18 K r_0)\subset \Gamma^k$.

\begin{proof}
Let $x\in B(x_0,\tfrac18 K r_0)\cap \Gamma^k$, $r\geq\ell^k(x)$, and write
$$
\Delta_{\wt\mu,\vphi}(x,r) - \Delta_{\nu^k,\vphi}(x,r) =
\sum_{j=1}^{N_k} \left(\int\psi_r(x-y)\,d\nu_j^k(y) - \int\psi_r(x-y)\,\theta_j^k(y)\,d\wt\mu(y)\right).$$
Since $\supp\theta_j^k\subset \frac32 B_j^k$, the integral on the right hand side vanishes unless
$\frac32 B_j^k$ intersects $B(x,2r)$. Since $r\geq\ell^k(x)$, it follows easily that the latter condition
implies $2 B_j^k\subset B(x,c_{11}r)$, for some absolute constant $c_{11}$, by Lemma \ref{lemnofac}.

For a ball $B_j^k$ such that $2 B_j^k\subset B(x,c_{11}r)$, by Lemma \ref{lemalfa230} we have
$$
\left|\int\psi_r(x-y)\,d\nu_j^k(y) -  \int\psi_r(x-y)\,\theta_j^k(y)\,d\wt\mu(y)\right|\lesssim 
\ve_0\,\frac{\wt\mu(B_j^k)\,\ell_j^k}{r^2}.
$$
Hence,
$$\bigl|\Delta_{\wt\mu,\vphi}(x,r) - \Delta_{\nu^k,\vphi}(x,r) \bigr|
\lesssim \ve_0\sum_{j:2B_j^k\subset B(x,c_{11}r)}\frac{\wt\mu(B_j^k)\,\ell_j^k}{r^2}.$$

By Cauchy-Schwarz, from the last estimate we infer that
\begin{align*}
\bigl|\Delta_{\wt\mu,\vphi}(x,r) - \Delta_{\nu^k,\vphi}(x,r) \bigr|^2
&\lesssim \ve_0^2\Biggl(\,\sum_{j:2B_j^k\subset B(x,c_{11}r)}\wt\mu(B_j^k)\Biggr)
\Biggl(\,\sum_{j:2B_j^k\subset B(x,c_{11}r)}\frac{\wt\mu(B_j^k)\,(\ell_j^k)^2}{r^4}\Biggr)\\
& \lesssim \ve_0^2\, \frac{\wt\mu(B(x,c_{11}r))}{r}\sum_{j:2B_j^k\subset B(x,c_{11}r)}\frac{\wt\mu(B_j^k)\,(\ell_j^k)^2}{r^3},
\end{align*}
where in the last inequality we took into account that $\wt\mu(B_j^k)\approx\wt\mu(\frac16B_j^k)$ and that the balls $\frac16B_j^k$ are pairwise disjoint for every fixed $k$.
Since $\mu(B(x,c_{11}r))\lesssim A\,\Theta_\mu(B_R)\,r$, we obtain
$$\bigl|\Delta_{\wt\mu,\vphi}(x,r) - \Delta_{\nu^k,\vphi}(x,r) \bigr|^2
\lesssim \ve_0^2\, A\,\Theta_\mu(B_R)\sum_{j:2B_j^k\subset B(x,c_{11}r)}\frac{\wt\mu(B_j^k)\,(\ell_j^k)^2}{r^3}.
$$
Now we use this inequality to estimate the left hand side of 
\rf{eqlemmac1}:
\begin{align}\label{eqtog3}
\int_{\gex\cap B(x_0,\frac18 K r_0)}&\int_{\ell^k(x)}^{ K r_0/100} |
\Delta_{\wt\mu,\vphi}(x,r) - \Delta_{\nu^k,\vphi}(x,r)|^2\,\frac{dr}r\,d\HH^1(x)\\
&
\lesssim  \ve_0^2\, A\,\Theta_\mu(B_R) \sum_{h=1}^{N_k} \int_{L_h^k}\int_{\ell_h^k}^{ K r_0/100}
\!\!\!\!\sum_{j:2B_j^k\subset B(x,c_{11}r)}\frac{\wt\mu(B_j^k)\,(\ell_j^k)^2}{r^3}\,\frac{dr}r\,d\HH^1(x).
\nonumber
\end{align}
Note now that if $x\in B_h^k$,
$r\geq\ell_h^k$ and $2B_j^k\subset B(x,c_{11}r)$, then
$$r\gtrsim\dist(B_h^k,B_j^k) + r(B_j^k)+ r(B_h^k)=:D(B_j^k,B_h^k).$$
Then, by Fubini,
\begin{align*}
\int_{\ell_h^k}^{ K r_0/100}\!\!\!\!
\sum_{j:2B_j^k\subset B(x,c_{11}r)}\frac{\wt\mu(B_j^k)\,(\ell_j^k)^2}{r^3} \,\frac{dr}r& \leq \sum_{j=1}^{N_k}\wt\mu(B_j^k) (\ell_j^k)^2
\int_{c^{-1}D(B_j^k,B_h^k)}^\infty \frac1{r^3}\,\frac{dr}r\\
&\lesssim  \sum_{j=1}^{N_k}\frac{\wt\mu(B_j^k) (\ell_j^k)^2}{D(B_j^k,B_h^k)^3}.
\end{align*}
Plugging this estimate into \rf{eqtog3}, we get
\begin{align}\label{eqfi396}
\int_{\gex\cap B(x_0,\frac18 K r_0)}\int_{\ell^k(x)}^{ K r_0/100} &|
\Delta_{\wt\mu,\vphi}(x,r) - \Delta_{\nu^k,\vphi}(x,r)|^2\,\frac{dr}r\,d\HH^1(x)\\& \lesssim
\ve_0^2\, A\,\Theta_\mu(B_R) \sum_{h=1}^{N_k}\ell_h^k\sum_{j=1}^{N_k}\frac{\wt\mu(B_j^k) (\ell_j^k)^2}{D(B_j^k,B_h^k)^3}\nonumber\\
& = \ve_0^2\, A\,\Theta_\mu(B_R)\sum_{j=1}^{N_k}\wt\mu(B_j^k) (\ell_j^k)^2 \sum_{h=1}^{N_k}\frac{\HH^1(L_h^k)}{D(B_j^k,B_h^k)^3}.
\nonumber
\end{align}
Since the curve $\gex$ is AD-regular with some constant depending on $A$ and $\tau$,
it easily follows that
$$\sum_{h=1}^{N_k}\frac{\HH^1(L_h^k)}{D(B_j^k,B_h^k)^3}\lesssim_{A,\tau}
\frac{1}{(\ell_j^k)^2}.$$
Then, going back to \rf{eqfi396}, we obtain
\begin{multline*}
\int_{\gex\cap B(x_0,\frac18 K r_0)}\int_{\ell^k(x)}^{ K r_0/100}|
\Delta_{\wt\mu,\vphi}(x,r) - \Delta_{\nu^k,\vphi}(x,r)|^2\,\frac{dr}r\,d\HH^1(x)\\ 
\lesssim_{A,\tau}
\ve_0^2\, \Theta_\mu(B_R)\sum_{j=1}^{N_k}\wt\mu(B_j^k) 
\leq_{A,\tau ,K}\ve_0^2\, \Theta_\mu(B_R)\,\mu(R)
\approx_{A,\tau ,K}\ve_0^2\, \Theta_\mu(B_R)^2\,\ell(R),
\end{multline*}
as wished.
\end{proof}
\vv

\begin{lemma}\label{lem104}
Let $H^k$ be the subset of those points $x\in\Gamma^k\cap B(x_0,\frac18 K r_0)$ such that
$$\int_{\ell^k(x)}^{ K r_0/100} \left|\Delta_{\nu^k,\vphi}(x,r)\right|^2\,\frac{dr}r >\ve_0^{1/2}\,\Theta_\mu(B_R)^2
.$$
Then 
$$\HH^1(H^k)\leq \ve_0^{1/2}\,\ell(R),$$
assuming $\eta$ small enough.
\end{lemma}

\begin{proof}
For $x\in\Gamma^k\cap B(x_0,\frac18 K r_0)$ we write
\begin{align*}
\left(\int_{\ell^k(x)}^{ K r_0/100} \left|\Delta_{\nu^k,\vphi}(x,r)\right|^2\,\frac{dr}r \right)^{1/2}
&\leq
\left(\int_{\ell^k(x)}^{ K r_0/100} \left|\Delta_{\nu^k,\vphi}(x,r) - \Delta_{\wt\mu,\vphi}(x,r)\right|^2\,\frac{dr}r \right)^{1/2}\\
&\quad + \left(\int_{\ell^k(x)}^{ K r_0/100} \left|\Delta_{\wt\mu,\vphi}(x,r) - \Delta_{\mu,\vphi}(x,r)\right|^2\,\frac{dr}r \right)^{1/2}\\
&\quad +\left(\int_{\ell^k(x)}^{ K r_0/100} \left|\Delta_{\mu,\vphi}(x,r)\right|^2\,\frac{dr}r \right)^{1/2}\\
&=: I_1(x) + I_2(x) + I_3(x).
\end{align*}
By Lemma \ref{lem101}, if $\eta$ and $\ve_0$ are assumed small enough,
$$I_3(x)\leq \frac{\ve_0^{1/2}}3\,\Theta_\mu(B_R)\qquad\mbox{for all $x\in\Gamma^k\cap B(x_0,\frac18 K r_0)$.}$$
Thus,
\begin{align}\label{eqsu53}
\HH^1(H^k) & \leq \HH^1\left(\left\{x \in\Gamma^k\cap B(x_0,\tfrac18 K r_0):
I_1(x)>\tfrac13\ve_0^{1/2}\,\Theta_\mu(B_R)\right\}\right) \\
& \quad + 
\HH^1\left(\left\{x\in\Gamma^k\cap B(x_0,\tfrac18 K r_0):
I_2(x)>\tfrac13\ve_0^{1/2}\,\Theta_\mu(B_R)\right\}\right).\nonumber
\end{align}
By Chebyshev and Lemma \ref{lem103}, the first term on the right hand side does not exceed
\begin{align*}
\frac9{\ve_0\,\Theta_\mu(B_R)^2}\int_{\gex\cap B(x_0,\frac18 K r_0)} & \int_{\ell^k(x)}^{ K r_0/100} |
\Delta_{\wt\mu,\vphi}(x,r) - \Delta_{\nu^k,\vphi}(x,r)|^2\,\frac{dr}r\,d\HH^1(x) \\ 
& \leq \frac{9\,c(A,\tau,K)\,\ve_0^2\,\Theta_\mu(B_R)}{\ve_0\,\Theta_\mu(B_R)^2}\,\mu(R) \leq 
\frac{\ve_0^{1/2}}2\,\ell(R),
\end{align*}
where in the last inequality we took into account that $\ve_0\ll c(A,\tau,K)$ and that $\mu(R)\approx_K\nu^k(\Gamma^k)$.

To estimate the last term in \rf{eqsu53}, we consider the operator $T_\vphi$ defined as follows for a measure $\lambda\in M(\R^d)$:
$$T_\vphi\lambda(x) =
\left(\int_0^\infty \left|\Delta_{\lambda,\vphi}(x,r)\right|^2\,\frac{dr}r\right)^{1/2}.$$
As shown in \cite[Theorem 5.1]{TT}, $T_\vphi$ is bounded from $M(\R^d)$ to $L^{1,\infty}(\HH^1|_{\Gamma^k})$ when
$\Gamma^k$ is an AD-regular curve, with the norm bounded by some constant depending only on the AD-regularity constant of $\Gamma^k$, and so on $A$ and $\tau$.
Take the measure 
$$\lambda =\chi_{B(x_0, K r_0)}\,(\mu-\wt\mu).$$
Using the aforementioned boundedness of $T_\vphi$, we deduce that
the last term in \rf{eqsu53} is bounded by
\begin{align*}
\HH^1&\left(\left\{x\in\Gamma^k: T_\vphi\lambda(x)>\tfrac13\ve_0^{1/2}\,\Theta_\mu(B_R)\right\}\right)\leq c(A,\tau)\,\frac{\|\lambda\|}{\ve_0^{1/2}\,\Theta_\mu(B_R)}.
\end{align*}
Note now that by Lemma \ref{lemk12}
$$\|\lambda\| = \mu(B(x_0, K r_0)\setminus \wt E)
\leq \eta^{1/10}\,\mu(B(x_0, K r_0)) \lesssim_K \eta^{1/10}\,\mu(R).$$
Thus, for $\eta$ small enough,
\begin{align*}
\HH^1\left(\left\{x\in\Gamma^k\cap B(x_0,\tfrac18 K r_0):
I_2(x)>\tfrac13\ve_0^{1/2}\,\Theta_\mu(B_R)\right\}\right)&
\leq \frac{c(A,\tau,K)\,\eta^{1/10}}{\ve_0^{1/2}\,\Theta_\mu(B_R)}\,\mu(R) \leq \frac{\ve_0^{1/2}}2\,\ell(R),
\end{align*}
which completes the proof of the lemma.
\end{proof}

\vv

\begin{lemma}\label{lem105}
We have
\begin{multline*}
\left(\int_\gex\int_{ \frac{K r_0}{100}}^\infty + \int_{\gex\cap B(x_0,\frac{Kr_0}8)}\int_0^{\ell_k(x)} \!+ \int_{\gex\setminus
B(x_0,\frac{Kr_0}8)}\int_0^{\frac{ K r_0}{100}}\right)
\left|\Delta_{\nu^k,\vphi}(x,r)\right|^2\frac{dr}r \,d\HH^1(x)\\
 \qquad\qquad\lesssim\ve_0^{1/2}\Theta_\mu(B_R)^2\,\ell(R).
\end{multline*}
\end{lemma}

\begin{proof} 
First we will estimate the integral
$$\int_\gex\int_{ \frac{K r_0}{100}}^\infty 
\left|\Delta_{\nu^k,\vphi}(x,r)\right|^2\frac{dr}r \,d\HH^1(x).$$
To this end, take $r\geq  K r_0/100$ and let
$\chi$ by a $\CC^\infty$ bump function $\chi$ which equals $1$ on $B(x_0, \frac12K r_0)$ and vanishes on
$\R^d\setminus B(x_0, K r_0)$, with $\|\nabla \chi\|_\infty\leq c(\delta)/r_0$. For $x\in\rho_1^1$, we have $\psi_r*\HH^1|_{\rho_1^1}(x)=0$, and thus
$$\Delta_{\nu^k,\vphi}(x,r) = \int \psi_r(x-y)\,d\nu^k(y) - c_0^k\int \psi_r(x-y)\,d\HH^1|_{\rho_1^1}(y).$$
As $\nu^k$ coincides with $c_0^k\,d\HH^1|_{\rho_1^1}$ out of a small neighborhood of $B(x_0, \frac14K r_0)$, we have
$$\Delta_{\nu^k,\vphi}(x,r) = \int \chi(y)\,\psi_r(x-y)\,d(\nu^k - c_0^k\,d\HH^1|_{\rho_1^1})(y).$$
Therefore,
$$|\Delta_{\nu^k,\vphi}(x,r)|\leq \lip\bigl(\chi\,\psi_r(x-\cdot)\bigr) \,\dist_{B(x_0, K r_0)}
(\nu^k,c_0^k\,d\HH^1|_{\rho_1^1}).$$
By the construction of $\nu^k$ and the definition of $c_0^k$, it is not difficult to see that
\begin{equation}\label{eqahf37}
\dist_{B(x_0, K r_0)}(\nu^k,c_0^k\,\HH^1|_{\rho_1^1})\leq c(K)\,\ve_0\,\nu^k(B(x_0, K r_0))\,r_0.
\end{equation}
We leave the details for the reader. Using also that $\lip\bigl(\chi\,\psi_r(x-\cdot))\bigr)\leq
c(K)/(r\,r_0)$, we obtain
\begin{equation}\label{eqdes451}
|\Delta_{\nu^k,\vphi}(x,r)|\leq \frac{c(K)\,\ve_0\,\mu(R)\,r_0}{r\,r_0}
\leq c(K)\,\ve_0\,\Theta_\mu(B_R)\,\frac{r_0}r.
\end{equation}

If $x\in\Gamma_{ex}^k\setminus \rho_1^1$, then we consider the point $x'$ which is the orthogonal projection of $x$ on $\rho_1^1$, and we  write
\begin{align*}
|\Delta_{\nu^k,\vphi}(x,r)| & \leq |\Delta_{\nu^k,\vphi}(x',r)| +
|\Delta_{\nu^k,\vphi}(x',r) - \Delta_{\nu^k,\vphi}(x,r)|\\
& \leq c(K)\,\ve_0\,\Theta_\mu(B_R)\,\frac{r_0}r+|\Delta_{\nu^k,\vphi}(x',r) - \Delta_{\nu^k,\vphi}(x,r)|,
\end{align*}
by applying \rf{eqdes451} to $x'$.
To estimate the last term note that 
$$|x-x'|\leq  \sup_{1\leq j \leq N_k} \dist(x_j^k,\rho_1^1)\lesssim
\beta_{\infty,\wt\mu}(B(x_0, K r_0))\, K \,r_0\leq c(K)\,\ve_0\,r_0.$$
So we have
$$
|\Delta_{\nu^k,\vphi}(x',r) - \Delta_{\nu^k,\vphi}(x,r)|\leq |x-x'|\,\|\nabla (\psi_r*\nu^k)\|_\infty
\leq c(K)\,\ve_0\,r_0\,\frac{\Theta(B_R)}{r}.
$$
Thus \rf{eqdes451} also holds in this case. 

Note also that $\Delta_{\nu^k,\vphi}(x,r)$ vanishes for $x\in\gex$ such that 
$B(x,4r)\cap B(x_0,\frac12Kr_0)\neq \varnothing$. So we may assume that $r\gtrsim K\,r_0+|x-x_0|$,
and thus
\begin{align*}
\int_{ K r_0/100}^\infty \left|\Delta_{\nu^k,\vphi}(x,r)\right|^2\,\frac{dr}r & \leq
c(K)\,\ve_0^2\,\Theta_\mu(B_R)^2\int_{c\, K r_0/100+ c\,|x-x_0|}^\infty  \frac{r_0^2}{r^3}\,dr \\&\leq c(K)\,\ve_0^2\,\frac{r_0^2\,\Theta_\mu(B_R)^2}{r_0^2
+|x-x_0|^2}.
\end{align*}
From the preceding estimate, it follows immediately that
$$\int_\gex\int_{ \frac{K r_0}{100}}^\infty 
\left|\Delta_{\nu^k,\vphi}(x,r)\right|^2\frac{dr}r \,d\HH^1(x)\leq c(K)\,\ve_0^2\,\Theta_\mu(B_R)^2\,\ell(R)\leq \ve_0\,\Theta_\mu(B_R)^2\,\ell(R).$$

By arguments in the same spirit, one can show that
\begin{equation*}
\left(\int_{\gex\cap B(x_0,\frac{Kr_0}8)}\int_0^{\ell_k(x)} + \int_{\gex\setminus
B(x_0,\frac{Kr_0}8)}\int_0^{\frac{ K r_0}{100}}\right)
\left|\Delta_{\nu^k,\vphi}(x,r)\right|^2\frac{dr}r \,d\HH^1(x) \lesssim\ve_0^{1/2}\,\Theta_\mu(B_R)^2\,\ell(R).
\end{equation*}
We leave the details for the reader. 
\end{proof}

\vv
\begin{remark}\label{remb22}
For the record, note that from \rf{eqahf37} it follows easily that
$$c_0^k\approx \Theta_\mu(B_R)$$
with the comparability constant not depending on $A$, $\tau$, $K$ or $M$.
\end{remark}
\vv

Now we need the following auxiliary result.

\begin{propo}\label{propolp}
For $g\in L^p(\HH^1|_\gex)$, consider the operator
\begin{equation}\label{eqdjka22}
T_{\vphi,\HH^1|_\gex} g(x) =
\left(\int_0^\infty 
\left|\psi_r*(g\,\HH^1|_\gex)(x)\right|^2\,\frac{dr}r\right)^{1/2}.
\end{equation}
Then $T_{\vphi,\HH^1|_\gex}$ is bounded in $L^p(\HH^1|_\gex)$ for $1<p<\infty$.
\end{propo}

\begin{proof}
We consider the operator
$$
T_{\HH^1|_\gex} g(x) = \left(\int_0^\infty \Delta_{g\HH^1|_\gex}(x,r)^2
\frac{dr}r \right)^{1/2}.$$
As shown in \cite{TT}, $T_{\HH^1|_\gex}$ is bounded in $L^2(\HH^1|_{\Gamma^k})$ and from
$L^1(\HH^1|_{\Gamma^k})$ to $L^{1,\infty}(\HH^1|_{\Gamma^k})$. Thus by interpolation it is bounded
in $L^p(\HH^1|_{\Gamma^k})$ for $1<p<2$.
By applying Lemma \ref{lemconvex} to the measure $g{\HH^1|_\gex}$, it follows that 
\begin{equation}\label{eqajhf44}
T_{\vphi,\HH^1|_\gex} g(x)\leq c\,T_{\HH^1|_\gex} g(x),
\end{equation}
 and thus $T_{\vphi,\HH^1|_\gex}$ is also bounded in $L^p(\HH^1|_\gex)$ for $1<p\leq2$.

We will show in Proposition \ref{propo175} that the $L^2(\HH^1|_{\Gamma^k})$ boundedness of 
$T_{\HH^1|_\gex}$ implies its boundedness in $L^p(\HH^1|_{\Gamma^k})$ for $2<p<\infty$. Thus 
again by \rf{eqajhf44}, $T_{\vphi,\HH^1|_\gex}$ is bounded in $L^p(\HH^1|_\gex)$ for $2<p<\infty$.

An alternative argument to show that $T_{\vphi,\HH^1|_\gex}$ is bounded in $L^p(\HH^1|_\gex)$ for $2<p<\infty$  consists in proving its boundedness from $L^\infty(\HH^1|_{\Gamma^k})$ to $BMO(\HH^1|_{\Gamma^k})$ (which follows by rather stander arguments). Then by interpolation between the pairs
$(L^2(\HH^1|_{\Gamma^k}),L^2(\HH^1|_{\Gamma^k}))$ and $(L^\infty(\HH^1|_{\Gamma^k}),BMO(\HH^1|_{\Gamma^k}))$
we are done.
\end{proof}

\vv
\begin{lemma}\label{lem116}
We have
$$\int_\gex\int_0^\infty \left|\Delta_{\nu^k,\vphi}(x,r)\right|^2\,\frac{dr}r \,d\HH^1(x)\leq 
\ve_0^{1/10}\,\Theta_\mu(B_R)^2\,\ell(R).$$
\end{lemma}

\begin{proof} 
Because of Lemma \ref{lem105} we only have to estimate
$$\int_{\Gamma^k\cap B(x_0,\frac18 K r_0)} \int_{\ell_k(x)}^{\frac{ K r_0}{100}} \left|\Delta_{\nu^k,\vphi}(x,r)\right|^2\,\frac{dr}r
\,d\HH^1(x).$$
To this end, recall that the subset $H^k$ of the points $x\in\Gamma^k\cap B(x_0,\frac18 K r_0)$ such that
$$\int_{\ell^k(x)}^{ K r_0/100} \left|\Delta_{\nu^k,\vphi}(x,r)\right|^2\,\frac{dr}r >\ve_0^{1/2}\,\Theta_\mu(B_R)^2
$$
satisfies 
$$\HH^1(H^k)\leq \ve_0^{1/2}\,\ell(R),$$
by Lemma \ref{lem104}.

By the definition of $H^k$ we have
\begin{equation}\label{eqdjk33}
\int_{\Gamma^k\cap B(x_0,\frac18 K r_0)\setminus H^k} \int_{\ell_k(x)}^{\frac{ K r_0}{100}} \left|\Delta_{\nu^k,\vphi}(x,r)\right|^2\,\frac{dr}r
\,d\HH^1(x)\lesssim_{A,\tau,K}
\ve_0^{1/2}\,\Theta_\mu(B_R)^2\ell(R).
\end{equation}
To deal with the analogous integral on $H^k$, denote by $g_k$ the density of $\nu^k$ with respect to $\HH^1|_{\Gamma^k}$. Note that $\Delta_{\nu^k,\vphi}(x,r)=
\psi_r*(\chi_{B(x_0,2Kr_0)}\nu_k)$ for $x\in \Gamma^k\cap B(x_0,\frac18 K r_0)$ and 
$r\leq K r_0/100$. Then, by H\"older's inequality and the $L^4(\HH^1|_\gex)$ boundedness of the operator
$T_{\vphi,\HH^1|_\gex}$ from \rf{eqdjka22}, we get
\begin{align*}
\int_{H^k} \int_{\ell_k(x)}^{\frac{ K r_0}{100}} \bigl|\Delta_{\nu^k,\vphi}(x,r)\bigr|^2\,\frac{dr}r\,d\HH^1(x) &\leq \int_{H^k}
|T_{\vphi,\HH^1|_\gex}(g_k\,\chi_{B(x_0,2Kr_0)})|^2\,d\HH^1(x)\\
&\leq \HH^1(H^k)^{1/2}\,
\|T_{\vphi,\HH^1|_\gex}(g_k\,\chi_{B(x_0,2Kr_0)})\|_{L^4(\HH^1|_\gex)}^2\\
& \lesssim_{A,\tau,K} (\ve_0^{1/2}\,\ell(R))^{1/2}\,
\|g_k\,\chi_{B(x_0,2Kr_0)}\|_{L^4(\HH^1|_\gex)}^2\\
& \lesssim _{A,\tau,K} \ve_0^{1/4}\,\,\Theta_\mu(B_R)^2\,\ell(R).
\end{align*}
From this estimate and \rf{eqdjk33} we deduce
\begin{align*}
\int_{\Gamma^k\cap B(x_0,\frac18 K r_0)} \int_{\ell_k(x)}^{\frac{ K r_0}{100}} \left|\Delta_{\nu^k,\vphi}(x,r)\right|^2\,\frac{dr}r
\,d\HH^1(x) &\lesssim_{A,\tau,K}
(\ve_0^{1/2} + \ve_0^{1/4})\,\Theta_\mu(B_R)^2\ell(R) \\
&\lesssim_{A,\tau,K} \ve_0^{1/4}\,\Theta_\mu(B_R)^2\ell(R).
\end{align*}
In combination with Lemma \ref{lem105}, this concludes the proof of the lemma.
\end{proof}

\vv


\section{The good measure $\sigma^k$ on $\Gamma^k$}\label{sec12}

In this section, for each $k$ we will construct a measure $\sigma^k$ supported on $\Gamma^k$ having linear growth (with an absolute constant), so that moreover
\begin{equation}\label{eqsq89}
\int_{\gex}\int_0^\infty \left|\Delta_{\sigma^k,\vphi}(x,r)\right|^2\,\frac{dr}r\,d\HH^1(x)
\end{equation}
is very small. The measure $\sigma^k$ will be used as a kind of reference measure in Section \ref{sec13}, where we will estimate the wavelet coefficients of the density of $\nu^k$ with respect to
$\sigma^k$ in terms of the square function \rf{eqsq89} and of the analogous square function involving the measure $\nu^k$.
By means of these estimates we will prove later that the cells from $\HD$ have small $\mu$-mass.

To define the measures $\sigma^k$ we will use the maps $\Pi_k:\Gamma^k\to\Gamma^{k+1}$ introduced at the 
beginning of Section \ref{sec88}.
Recall that, given $x\in L_j^k$, $\Pi_k(x)$ is the unique point in $\Gamma_j^k\subset\Gamma^{k+1}$ such that the orthogonal projection of $\Pi_k(x)$ on $L_j^k$ is $x$.
We extend $\Pi_k$ to the whole curve $\Gamma_{ex}^k$ just by setting $\Pi_k(x)=x$ for
$x\in\Gamma_{ex}^k\setminus\Gamma^k$. Note that $\Gamma_{ex}^k\setminus\Gamma^k=
\Gamma_{ex}^{k+1}\setminus\Gamma^{k+1}$ and so the definition is correct.
By abusing notation, we continue to denote by $\Pi_k$ this extension.

We set $\sigma^1=\HH^1|_{\Gamma^1_{ex}}
= \HH^1|_{\rho_1^1}$, and then by induction, $\sigma^{k+1}=\Pi_{k,\#}(\sigma^k)$
for $k\geq1$,
where $\Pi_{k,\#}(\sigma^k)$ is the image measure of $\sigma^k$ by $\Pi_{k}$. Note that
$\sigma^1$ is just the arc length measure on the line $\rho_1^1$ (which coincides with $\Gamma^1$), and then for $k\geq 1$, $\sigma^k= g_k\,\HH^1|_{\Gamma^k}$, with $\|g_k\|_\infty\leq 1$. This follows easily from the fact that $\|g_{k+1}\|_\infty\leq \|g_{k}\|_\infty$. Taking into account that $\Gamma^k$ is AD regular, it
follows that $\sigma^k$ has linear growth with some constant depending on $A$ and $\tau$ (analogously to 
$\Gamma^k$). Next we show that the linear growth of $\sigma^k$ does not depend on these constants. 
This fact will play an important role later.

\vv

\begin{lemma}\label{lemsigmagrow}
There exists an absolute constant $c_0$ such that
$$\sigma^k(B(x,r))\leq c_0\,r\qquad\mbox{for all $x\in\R^d$ and $r>0$.}$$
\end{lemma}

\begin{proof}
It is enough to show that $\sigma^k|_{\Gamma^k}$ has linear growth with some absolute constant because
$\sigma^k$ coincides with the arc length measure on $\Gamma^k_{ex}\setminus \Gamma^k$. So it suffices to prove that
\begin{equation}\label{eqlin692}
\sigma^k(B(x,r)\cap\Gamma^k)\leq c\,r\qquad\mbox{for $x\in\Gamma^k$ and $0<r\leq \diam(\Gamma^k)$.}
\end{equation}
Suppose first that $r\leq 2^{-k/2}d_0$ (recall that $d_0=\ell_1^1\approx \diam(\Gamma^k)$). In this case, by Lemma \ref{lemfac33.5} it turns out that
$B(x,r)$ intersects a number of segments $L_j^k$ bounded above by an absolute constant. Since
$\sigma^k|_{\Gamma^k}= g_k\,\HH^1|_{\Gamma^k}$ with $\|g_k\|_\infty\leq 1$, \rf{eqlin692} holds in this case.

Suppose now that $r> 2^{-k/2}\,d_0$.  
Let $0\leq m\leq k$ be the integer such that 
$$2^{-(m-1)/2}\,d_0\leq r\leq 2^{-m/2}\,d_0.$$
Note that
$$\sigma^k = \Pi_{k,\#}(\Pi_{k-1,\#}(\ldots(\Pi_{m,\#}(\sigma^m)))).$$
Let $y\in B(x,r)$.
For $m\leq n\leq k$, let $x_n,y_n$ be such that $\Pi_{k-1}(\Pi_{k-2}(\ldots(\Pi_n(x_n))))=x$ and
$\Pi_{k-1}(\Pi_{k-2}(\ldots(\Pi_n(y_n))))=y$.
Since $|x_n-x_{n+1}|\lesssim \ve_0\,2^{-n/2} d_0 $ for all $n$, writing $x=x_k$ we get
$$|x-x_m| \leq \sum_{n=m}^{k-1} |x_n-x_{n+1}|\lesssim \ve_0\,2^{-m/2}\,d_0.$$
Analogously, $|y-y_m|\lesssim \ve_0\,2^{-m/2}\,d_0$.
Therefore,
$$\Pi_m^{-1}(\ldots(\Pi_{k-2}^{-1}(\Pi_{k-1}^{-1}(B(x,r)))))\subset B(x_m, (1+c\,\ve_0)2^{-m/2}d_0),$$
and so
\begin{align*}
\sigma^k(B(x,r)\cap\Gamma^k) & = \sigma^m(\Pi_m^{-1}(\ldots(\Pi_{k-2}^{-1}(\Pi_{k-1}^{-1}(B(x,r)\cap\Gamma^m)))))\\
&
\leq \sigma^m(B(x_m, (1+c\,\ve_0)2^{-m/2}d_0)\cap\Gamma^m).
\end{align*}
Arguing as above, since $\ell^m_j\lesssim 2^{-m/2}d_0$ for $1\leq j\leq N_m$ and the number of segments 
$L_j^m$ that intersect $B(x_m, (1+c\,\ve_0)2^{-m/2}d_0)$ is bounded by some absolute constant, we deduce that
$$\sigma^m(B(x_m, (1+c\,\ve_0)2^{-m/2}d_0)\cap\Gamma^m)\leq c\,2^{-m/2} d_0 \leq c\,r.$$
\end{proof}

\vv
Next we show that $\sigma^k$ is also lower AD-regular, with a constant depending on $M$ now.

\begin{lemma}\label{lemsigmad}
The density $g_k$ of $\sigma^k$ with respect to $\HH^1|_{\Gamma^k_{ex}}$ satisfies
$$g_k(x)\geq c(M)>0\qquad\mbox{for all $x\in\Gamma_{ex}^k$.}$$
\end{lemma}

\begin{proof}
For $x\in \Gamma_{ex}^k\setminus \Gamma^k$ we have $g^k(x)=1$.

Take now $x\in L_j^k\subset\Gamma^k$, and consider the sequence of segments
 $L_1^1=L_{j_1}^1,L_{j_2}^2,\ldots,L_{j_k}^k=L_j^k$  such that $L_{j_{m+1}}^{m+1}$ is generated by
$L_{j_{m}}^{m}$ for $m=1,\ldots,k-1$. By Lemma \ref{lembeta44} (see \rf{eqws42}) we know that 
$\meas(\rho_{j_m}^m,\rho_{j_{m+1}}^{m+1})$
is bounded
by the $\beta_{\infty,\wt\mu}$ coefficient of a suitable cell $P\in\DD$ such that $2B_P$ contains
$L_{j_m}^m$ and $L_{j_{m+1}}^{m+1}$. Thus we have
\begin{equation}\label{eqa202}
\sum_{m=1}^{k-1} (\meas(\rho_{j_m}^m,\rho_{j_{m+1}}^{m+1}))^2 \leq c\,M.
\end{equation}
Denote by $g_{(m)}$ the constant value of the density $g_m$ on $L_{j_m}^m$.
We claim that
$$\left|\frac{g_{(m+1)}}{g_{(m)}}-1\right|\leq (\meas(\rho_{j_m}^m,\rho_{j_{m+1}}^{m+1}))^2.$$
Indeed, let $I_{m+1}$ be an arbitrary interval contained in $L_{j_{m+1}}^{m+1}$ and denote by $I_m$ the interval from
$L_{j_m}^m$ such that $\Pi_m(I_m)=I_{m+1}$, so that $\sigma^{m+1}(I_{m+1})=\sigma^m(I_m)$ and thus
$$g_{(m+1)}\,\HH^1(I_{m+1}) = g_{(m)}\,\HH^1(I_{m}),$$
Since
$\HH^1(I_m)=\cos\meas(\rho_{j_m}^m,\rho_{j_{m+1}}^{m+1})\,\HH^1(I_{m+1})$, we get
$$\frac{g_{(m+1)}}{g_{(m)}} = \frac{\HH^1(I_{m})}{\HH^1(I_{m+1})} = \cos\meas(\rho_{j_m}^m,\rho_{j_{m+1}}^{m+1}).$$
Thus,
\begin{equation}\label{eqdifdens}
\left|\frac{g_{(m+1)}}{g_{(m)}}-1\right|= |\cos\meas(\rho_{j_m}^m,\rho_{j_{m+1}}^{m+1})-1| \leq
(\sin\meas(\rho_{j_m}^m,\rho_{j_{m+1}}^{m+1}))^2\leq (\meas(\rho_{j_m}^m,\rho_{j_{m+1}}^{m+1}))^2,
\end{equation}
and the claim follows.

From the previous claim and \rf{eqa202} we derive
$$\sum_{m=1}^{k-1}\left|\frac{g_{(m+1)}}{g_{(m)}}-1\right|\leq c\,M.$$
which implies that
$$C(M)^{-1}\leq \prod_{m=1}^{k-1}\frac{g_{(m+1)}}{g_{(m)}}\leq C(M).$$
As $g_{(1)}=1$, we get $g_{(k)}\geq C(M)^{-1},$ as wished.
\end{proof}

\vv
To estimate the integral \rf{eqsq89} it is convenient to introduce a dyadic lattice
over $\Gamma^k_{ex}$, which we will denote by $\DG$. 
This lattice is made up of subsets of $\Gamma^k_{ex}$ and
is analogous to the lattice $\DD$ associated with $\mu$ which has been introduced in Section \ref{sec4}.
However, since the arc-length measure on $\Gamma^k_{ex}$ is AD-regular, the arguments for the
construction of $\DG$ are easier than the ones for $\DD$. There are many references where the 
reader can find such a construction. For example, see the classical works of \cite{Christ} and \cite{David-wavelets}, or the more recent \cite{ntov}
for the precise version that we state below:
\begin{itemize}
\item
The family $\DG$ is the disjoint union of families $\DD_m\DGG$ (families of level $m$ cells, which are subsets of $\Gamma^k_{ex}$), $m\in\mathbb Z$.
\item
If $Q',Q''\in\DD_m\DGG$, then either $Q'=Q''$ or $Q'\cap Q''=\varnothing$.
\item
Each $Q'\in\DD_{m+1}\DGG$ is contained in some $Q\in\DD_m\DGG$ (necessarily unique due to the previous
property). We say that $Q'$ is the son of $Q$, and that $Q$ is the parent of $Q'$.
\item
For each $m\in\Z$, $\Gamma^k_{ex} = \bigcup_{Q\in\DD_m\DGG}Q$.

\item 
For each $Q\in\DD_m\DGG$, there exists $z_Q\in Q$ (the ``center'' of $Q$) such that 
$
Q\subset B(z_Q,2^{-4m+2})$ and $\dist(z_Q,Q')\geq 2^{-4m-3}
$
for any $Q'\in\DD_m\DGG$ different from $Q$.

\end{itemize} 

We write $\ell(Q)=2^{-4m}$, and we call it the side length of $Q$. 
Also, we set
$$B_Q = B(z_Q,4\ell(Q)),$$
so that we have
$$\Gamma^k_{ex}\cap \frac1{32}B_Q \subset Q\subset \Gamma^k_{ex}\cap B_Q.$$
We define
$$\beta_{\infty,\gex}(P) = \beta_{\infty,\gex}(4B_P).$$

\vv

\begin{lemma}\label{lemakey99}
We have
\begin{equation}\label{eqsq99}
\int_{\gex}\int_0^\infty \left|\Delta_{\sigma^k,\vphi}(x,r)\right|^2\,\frac{dr}r\,d\HH^1(x)
\lesssim_{A,\tau}
\sum_{Q\in\DG}\beta_{\infty,\gex}(Q)^4\,\ell(Q).
\end{equation}
\end{lemma}

Note the power $4$ over $\beta_{\infty,\gex}(Q)$ in the last equation. At first sight, it 
may seem surprising because the usual power is $2$ in most square function type estimates.
The fact that we get a power larger than $2$ will allow us to show that the left hand side
of \rf{eqsq99} is small if $\ve_0$ is also small.

\begin{proof}[Proof of Lemma \ref{lemakey99}]
By convenience, for $i\leq0$, we denote $\Gamma_{ex}^i=\rho_1^1$ and $\sigma^i=\HH^1|_{\Gamma_{ex}^i}$, and
$\Pi_i$ is the identity map on $\rho_1^1$.

The first step to prove the lemma consist in estimating $\Delta_{\sigma^k,\vphi}(x,r)=\psi_r*\sigma^k(x)$ in terms of 
the $\beta_\infty$ coefficients of $\gex$.
Suppose first that  $r\geq 2^{-(k+2)/2}d_0$.
Let $m\leq k$ the maximal integer such that $2^{-(m+2)/2}\,d_0\geq 10r$. Note 
that 
$\ell_j^m\geq 10r$ for all $1\leq j\leq N_m$, by \rf{eqguai1}. Consider the sequence of points
$x_m,x_{m+1},\ldots,x_k=x$ such that $x_i\in \Gamma_{ex}^i$ and $\Pi_i(x_i)=x_{i+1}$ for $i=m,\,m+1,\ldots,k-1$. Then we write
\begin{equation*}
|\psi_r*\sigma^k(x)| \leq |\psi_r*\sigma^k(x) -\psi_r*\sigma^{m}(x_m)| + |\psi_r*\sigma^m(x_m)|,
\end{equation*}
so that
\begin{align}\label{eqsusu80}
\int_{\gex}\!\int_{2^{-(k+2)/2}d_0}^\infty \! |\psi_r*\sigma^k(x)|^2\,\frac{dr}r\,d\HH^1(x) &\lesssim
\int_{\gex}\!\int_{2^{-(k+2)/2}d_0}^\infty \!|\psi_r*\sigma^k(x) -\psi_r*\sigma^{m}(x_m)|^2\,\frac{dr}r\,d\HH^1(x)  \\ \quad &+ 
\int_{\gex}\int_{2^{-(k+2)/2}d_0}^\infty |\psi_r*\sigma^m(x_m)|^2\,\frac{dr}r\,d\HH^1(x)
\nonumber\\
& = \circled{1} +\circled{2}.\nonumber
\nonumber
\end{align}
Notice that, although it is not stated explicitly, in the integrals above $m$ depends on $r$, and thus $x_m$ depends
on $x$ and $r$.

\vv
\noi{\bf Estimate of \circled{1}}.\\
We write
\begin{equation}\label{eqsusu8}
|\psi_r*\sigma^k(x) -\psi_r*\sigma^{m}(x_m)| \leq  \sum_{i=m}^{k-1}|\psi_r*\sigma^i(x_i) -\psi_r*\sigma^{i+1}(x_{i+1})|.
\end{equation}
 Since
$\sigma^{i+1} = \Pi_{i,\#}\sigma^i$, we have
\begin{align*}
|\psi_r*\sigma^i(x_i) -\psi_r*\sigma^{i+1}(x_{i+1})|& = 
\left|\int \psi_r(y-x_i)\,d\sigma^i(y) - \int \psi_r(y-\Pi_i(x_i))\,d\sigma^{i+1}(y)\right|\\
& = \left|\int \bigl[\psi_r(y-x_i)- \psi_r(\Pi_i(y)-\Pi_i(x_i))\bigr]\,d\sigma^i(y)\right|.
\end{align*}
To deal with the last integral, recall that $\psi_r(z) = \tfrac1r\,\vphi\Bigl(\frac{|z|}r\Bigr) - \frac1{2r}\,\vphi\Bigl(\frac{|z|}{2r}\Bigr)$
and that $\vphi$ is supported on $[-1,1]$ and constant in $[-1/2,1/2]$. 
So we have
$$\bigl|\psi_r(\Pi_i(y)-\Pi_i(x_i))- \psi_r(y-x_i)\bigr|\lesssim\frac c{r^2}\,
\bigl||\Pi_i(y)-\Pi_i(x_i)|- |y-x_i|\bigr|,$$
and moreover the left hand side vanishes unless $|y-x_i|\approx r$ or $|\Pi_i(y)-\Pi_i(x_i)|\approx r$, 
which is equivalent to saying just that $|y-x_i|\approx r$ (because 
$|y-x_i|\approx |\Pi_i(y)-\Pi_i(x_i)|$ for $y\in\Gamma^i_{ex}$).
Therefore,
\begin{equation}\label{eqash44}
|\psi_r*\sigma^i(x_i) -\psi_r*\sigma^{i+1}(x_{i+1})|\lesssim 
\frac1{r^2}\int_{c^{-1}r\leq|y-x_i|\leq 5\,r}
\bigl||\Pi_i(y)-\Pi_i(x_i)| - |y-x_i|\bigr|\,d\sigma^i(y).
\end{equation}

Now we have:

\begin{claim}\label{clageom}
For $m\leq i \leq k$, let $x_i,y\in\Gamma_{ex}^i$ be as in \rf{eqash44}, with
\begin{equation}\label{eqas402}
c^{-1}r\leq|y-x_i|\leq 5\,r.
\end{equation}
Let $Q^i(x_i),Q^i(y)\in\DG$ be the largest cells with $\ell(Q^i(x_i)),\ell(Q^i(y))\leq 2^{-i/2}\,d_0$
such that $2B_{Q^i(x_i)}$ contains $x_i$ and $2B_{Q^i(y)}$ contains $y$. Let $S\in \DG$ be the smallest cell such that $2B_S$ contains $Q^i(x_i)$ and $Q^i(y)$ for all $y\in \Gamma_{ex}^i$ satisfying \rf{eqas402}
and $m\leq i \leq k$ (so $\diam(B_S)\approx r$).
Then 
\begin{align}\label{eqdf227}
\bigl||\Pi_i(y)-\Pi_i(x_i)| - |y-x_i|\bigr| 
&\lesssim 
\ell(Q^i(y)) \Biggl(\sum_{\substack{Q\in\DG:\\
Q^i(y)\subset Q\subset 2B_S}} \beta_{\infty,\gex}(Q)\Biggr)^2 \\
&\quad +
\ell(Q^i(x_i)) \Biggl(\sum_{\substack{Q\in\DG:\\
Q^i(x_i)\subset Q\subset 2B_S}} \beta_{\infty,\gex}(Q)\Biggr)^2,\nonumber
\end{align}
for $m\leq i \leq k$.
\end{claim}
\vv

Let us assume the claim for the moment and let us continue the proof of the lemma. Let $j(i)$ the level of the largest cells $P\in\DG$ such that
$\ell(P)\leq 2^{-i/2}\,d_0$.
Plugging the above estimate into \rf{eqash44}
we derive
\begin{align}\label{eqsusu99}
|\psi_r*\sigma^i(x_i) \,- \,&\psi_r*\sigma^{i+1}(x_{i+1})|\\
& \lesssim \sum_{\substack{P\in\DD_{j(i)}(\gex):\\
P\subset 2B_S}} \frac1{\ell(S)^2} \int_{y\in 2B_P:c^{-1}r\leq|y-x_i|\leq 5\,r}
\bigl||\Pi_i(y)-\Pi_i(x_i)| - |y-x_i|\bigr|\,d\sigma^i(y)  \nonumber\\
&
\lesssim \sum_{\substack{P\in\DD_{j(i)}(\gex):\\
P\subset 2B_S}} \frac{\ell(P)^2}{\ell(S)^2} 
 \Biggl(\sum_{\substack{Q\in\DG:\\ P\subset Q\subset 2B_S}} \beta_{\infty,\gex}(Q)\Biggr)^2 
 \nonumber\\
 &\quad
+
\sum_{\substack{P\in\DD_{j(i)}(\gex):\\
P\subset 2B_S}} \frac{\ell(P)^2}{\ell(S)^2}   
 \Biggl(\sum_{\substack{Q\in\DG:\\ Q^i(x_i)\subset Q\subset 2B_S}} \beta_{\infty,\gex}(Q)\Biggr)^2.
 \nonumber
\end{align}

Note that 
$$\sum_{\substack{P\in\DD_{j(i)}(\gex):\\
P\subset 2B_S}} \frac{\ell(P)^2}{\ell(S)^2} \leq C(A,\tau) \frac{\ell(Q^i(x_i))}{\ell(S)}.$$
So the last sum in \rf{eqsusu99} does not exceed
$$C(A,\tau)\,\frac{\ell(Q^i(x_i))}{\ell(S)}\Biggl(\sum_{\substack{Q\in\DG:\\ Q^i(x_i)\subset Q\subset 2B_S}} \beta_{\infty,\gex}(Q)\Biggr)^2.$$

Going back to equation \rf{eqsusu8}, we get
\begin{align}\label{eqsusu74}
|\psi_r*\sigma^k(x) -\psi_r*\sigma^{m}(x_m)|
& \lesssim_{A,\tau}
 \sum_{\substack{P\in\DD(\gex):\\
P\subset 2B_S}} \frac{\ell(P)^2}{\ell(S)^2} 
 \Biggl(\sum_{\substack{Q\in\DG:\\ P\subset Q\subset 2B_S}} \beta_{\infty,\gex}(Q)\Biggr)^2 \\
 &\quad
+
\sum_{\substack{P\in\DD(\gex):\\
x\in P\subset 2B_S}} \frac{\ell(P)}{\ell(S)}   
 \Biggl(\sum_{\substack{Q\in\DG:\\ P\subset Q\subset 2B_S}} \beta_{\infty,\gex}(Q)\Biggr)^2,
 \nonumber
\end{align}
where 
we took into account that $\#\{i\in\Z:j(i)=j_0\}$ is bounded independently of $j_0$.

To prove \rf{eqsq99} we will need to square the preceding inequality. Let us deal with the 
first sum on the right hand side. By Cauchy-Schwarz, we obtain
\begin{align*}
\Biggl( \sum_{\substack{P\in\DD(\gex):\\
P\subset 2B_S}} &\frac{\ell(P)^2}{\ell(S)^2} 
 \Biggl(\sum_{\substack{Q\in\DG:\\ P\subset Q\subset 2B_S}} \beta_{\infty,\gex}(Q)\Biggr)^2\,\Biggr)^2 \\
 & \leq
\Biggl( \sum_{\substack{P\in\DD(\gex):\\
P\subset 2B_S}} \frac{\ell(P)^2}{\ell(S)^2} 
 \Biggl(\sum_{\substack{Q\in\DG:\\ P\subset Q\subset 2B_S}} \beta_{\infty,\gex}(Q)\Biggr)^4\,\Biggr)
\Biggl( \sum_{\substack{P\in\DD(\gex):\\
P\subset 2B_S}} \frac{\ell(P)^2}{\ell(S)^2} 
\Biggr).
\end{align*}
The last factor on the right side does not exceed some constant depending on $A$ and $\tau$. Also, by H\"older's inequality, it easily follows that
\begin{equation}\label{eqasg33}
\Biggl(\sum_{\substack{Q\in\DG:\\ P\subset Q\subset 2B_S}} \beta_{\infty,\gex}(Q)\Biggr)^4
\lesssim
\sum_{\substack{Q\in\DG:\\ P\subset Q\subset 2B_S}} \beta_{\infty,\gex}(Q)^4
\,\frac{\ell(Q)^{1/2}}{\ell(P)^{1/2}}.
\end{equation}
So we get
\begin{align*}
\Biggl( \sum_{\substack{P\in\DD(\gex):\\
P\subset 2B_S}} \frac{\ell(P)^2}{\ell(S)^2} &
 \Biggl(\sum_{\substack{Q\in\DG:\\ P\subset Q\subset 2B_S}} \beta_{\infty,\gex}(Q)\Biggr)^2\,\Biggr)^2 \\
 & \lesssim_{A,\tau}
\sum_{\substack{P\in\DD(\gex):\\
P\subset 2B_S}} \frac{\ell(P)^2}{\ell(S)^2} 
 \sum_{\substack{Q\in\DG:\\ P\subset Q\subset 2B_S}} \beta_{\infty,\gex}(Q)^4
\,\frac{\ell(Q)^{1/2}}{\ell(P)^{1/2}}\\
&=_{A,\tau} \sum_{\substack{Q\in\DD(\gex):\\
Q\subset 2B_S}} \beta_{\infty,\gex}(Q)^4
 \sum_{\substack{P\in\DG:\\ P\subset Q}} 
\frac{\ell(Q)^{1/2}\,\ell(P)^{3/2}}{\ell(S)^2}\\
& \lesssim_{A,\tau}\sum_{\substack{Q\in\DD(\gex):\\
Q\subset 2B_S}} \beta_{\infty,\gex}(Q)^4\,\frac{\ell(Q)^2}{\ell(S)^2}
.
\end{align*}

Now we turn our attention to the last sum on the right side of \rf{eqsusu74}.
By Cauchy-Schwarz, we obtain
\begin{align*}
\Biggl(\sum_{\substack{P\in\DD(\gex):\\
x\in P\subset 2B_S}} \frac{\ell(P)}{\ell(S)} &  
 \Biggl(\sum_{\substack{Q\in\DG:\\ P\subset Q\subset 2B_S}} \beta_{\infty,\gex}(Q)\Biggr)^2
\,\Biggr)^2 \\
& \leq 
\Biggl(\sum_{\substack{P\in\DD(\gex):\\
x\in P\subset 2B_S}} \frac{\ell(P)}{\ell(S)}   
 \Biggl(\sum_{\substack{Q\in\DG:\\ P\subset Q\subset 2B_S}} \beta_{\infty,\gex}(Q)\Biggr)^4
\,\Biggr)  
\Biggl(\sum_{\substack{P\in\DD(\gex):\\
x\in P\subset 2B_S}} \frac{\ell(P)}{\ell(S)}
\,\Biggr)\\
&\lesssim 
\sum_{\substack{P\in\DD(\gex):\\
x\in P\subset 2B_S}} \frac{\ell(P)}{\ell(S)}   
 \Biggl(\sum_{\substack{Q\in\DG:\\ P\subset Q\subset 2B_S}} \beta_{\infty,\gex}(Q)\Biggr)^4.
\end{align*}
By \rf{eqasg33}, the right hand side above is bounded by
\begin{align*}
c\,\sum_{\substack{P\in\DD(\gex):\\
x\in P\subset 2B_S}} \frac{\ell(P)}{\ell(S)} 
 \sum_{\substack{Q\in\DG:\\ P\subset Q\subset 2B_S}} \beta_{\infty,\gex}(Q)^4
\,\frac{\ell(Q)^{1/2}}{\ell(P)^{1/2}}
& \approx\!
 \sum_{\substack{Q\in\DD(\gex):\\
x\in Q\subset 2B_S}} \beta_{\infty,\gex}(Q)^4\!
\sum_{\substack{P\in\DD(\gex):\\
x\in P\subset Q}} \!\frac{\ell(Q)^{1/2}\,\ell(P)^{1/2}}{\ell(S)}\\
&\lesssim
\sum_{\substack{Q\in\DD(\gex):\\
x\in Q\subset 2B_S}} \beta_{\infty,\gex}(Q)^4
\, \frac{\ell(Q)}{\ell(S)}.
\end{align*}

Gathering the above estimates, we obtain
\begin{align*}
\Bigl(|\psi_r*\sigma^k(x) -\psi_r*\sigma^{m}(x_m)|\Bigr)^2
 \lesssim_{A,\tau}
\sum_{\substack{Q\in\DD(\gex):\\
Q\subset 2B_S}} \beta_{\infty,\gex}(Q)^4\,\frac{\ell(Q)^2}{\ell(S)^2} + 
\sum_{\substack{Q\in\DD(\gex):\\
x\in Q\subset 2B_S}} \beta_{\infty,\gex}(Q)^4
\, \frac{\ell(Q)}{\ell(S)}.
\end{align*}

The preceding inequality holds for all $x\in\Gamma^k$ and $r\geq 2^{-(k+2)/2}d_0$ with $S\in \DG$ being the smallest cell such that $2B_S$
contains $B(x,4r)$. If these conditions hold, then we write $(x,r)\in I_S$.
Then it follows that
\begin{align*}
\int_{\gex}\int_{2^{-(k+2)/2}d_0}^\infty &|\psi_r*\sigma^k(x) -\psi_r*\sigma^{m}(x_m)|^2\,\frac{dr}r\,d\HH^1(x)\\
&\lesssim_{A,\tau} \sum_{S\in\DG} \iint_{(x,r)\in I_S} 
\sum_{\substack{Q\in\DD(\gex):\\
Q\subset 2B_S}} \beta_{\infty,\gex}(Q)^4\,\frac{\ell(Q)^2}{\ell(S)^2} \,\frac{dr}r\,d\HH^1(x)
\\
&\quad + 
 \sum_{S\in\DG} \iint_{(x,r)\in I_S} 
\sum_{\substack{Q\in\DD(\gex):\\
x\in Q\subset 2B_S}} \beta_{\infty,\gex}(Q)^4
\, \frac{\ell(Q)}{\ell(S)} \,\frac{dr}r\,d\HH^1(x).
\end{align*}
Applying Fubini for the last term on the right hand side, we infer that
\begin{align}\label{eqhos33}
\int_{\gex}\int_{2^{-(k+2)/2}d_0}^{\infty} &|\psi_r*\sigma^k(x) -\psi_r*\sigma^{m}(x_m)|^2\,\frac{dr}r\,d\HH^1(x)\\
&\lesssim_{A,\tau} \! \sum_{S\in\DG}\!
\Biggl(\,\sum_{\substack{Q\in\DD(\gex):\\
Q\subset 2B_S}}\!\beta_{\infty,\gex}(Q)^4\,\frac{\ell(Q)^2}{\ell(S)^2} \,\ell(S)+ \!\!\sum_{\substack{Q\in\DD(\gex):\\
 Q\subset 2B_S}}\! \beta_{\infty,\gex}(Q)^4
\, \frac{\ell(Q)}{\ell(S)}\,\ell(Q)\!\Biggr)\nonumber\\
&\lesssim_{A,\tau} \sum_{S\in\DG}\,
\sum_{\substack{Q\in\DD(\gex):\\
Q\subset 2B_S}} \beta_{\infty,\gex}(Q)^4\,\frac{\ell(Q)}{\ell(S)} \,\ell(Q)\nonumber\\
&=_{A,\tau}
\sum_{Q\in\DG}\beta_{\infty,\gex}(Q)^4\,\ell(Q)
\sum_{\substack{S\in\DD(\gex):\\
2B_S\supset Q}}\frac{\ell(Q)}{\ell(S)} \nonumber\\&
\lesssim_{A,\tau}
\sum_{Q\in\DG}\beta_{\infty,\gex}(Q)^4\,\ell(Q).\nonumber
\end{align}
For the record, note that the preceding estimate is also valid if we replace $\psi_r$ by $\vphi_r$. Indeed,
above we did not use any cancellation property of $\psi_r$. Instead, we just took into account that $\psi_r$ is smooth, radial, supported on $B(0,4r)$, and constant on $B(0,r/2)$. All these properties are also
satisfied by $\vphi_r$.

\vvv
\noi{\bf Estimate of \circled{2}}.\\
Recall that
$$\circled{2}=\int_{\gex}\int_{2^{-(k+2)/2}d_0}^\infty |\psi_r*\sigma^m(x_m)|^2\,\frac{dr}r\,d\HH^1(x).$$
Since $\psi_r*\sigma^m(x_m)=0$ for $m\leq0$, we can assume that $m\geq1$, which implies that
$10r\leq\ell_1^1$ by the dependence of $m$ on $r$.
Recall that $\Gamma^m=\bigcup_{i=1}^{N_m} L_i^m$. 
For convenience, for each $m\in[1,k]$,
we will consider two additional segments, $L_0^m$ and $L_{N_{m}+1}^m$, of length 
$\ell_0^m=\ell_{N_m+1}^m=2^{-m/2}d_0$, so that they are contained in $\rho_1^1\setminus L_1^1$,
and one of the endpoints of $L_0^m$ is $x_0^m=z_A$ and one of the endpoints of $L_{N_{m}+1}^m$
is $x_{N_{m}}^m$. So joining these segments to $\Gamma^m$ we obtain a small extension of
$\Gamma^m$ which we denote by ${\Gamma_{ex}^m}'$ and is contained in $\Gamma_{ex}^m$. Note that
$\psi_r*\sigma^m(x_m)=0$ if $x_m\not\in{\Gamma_{ex}^m}'$ (since $m\leq k$ is the maximal integer such that $10r\leq2^{-(m+2)/2}\,d_0$).
For convenience again, we say that $L_0^{m-1}$ generates
$L_0^m$, and that $L_{N_{m-1}+1}^{m-1}$ generates $L_{N_{m}+1}^{m}$.

On each segment $L_j^m$, $0\leq j \leq N_m+1$, $\sigma^m$ equals
some constant multiple of the arc length measure. So it turns out that, for $x\in\Gamma^m$,  $\psi_r*
\sigma^m(x_m)$ vanishes unless $\supp \psi_r(x_m-\cdot)$ intersects more than one segment $L_j^m$.
Since $\ell^m_j\geq 10r$ for all $j\in [1,N_m]$ and that $\supp \psi_r(x_m-\cdot)
\subset \bar B(x_m,2r)$, by Lemma \ref{lemcla11}
it follows easily that $\supp \psi_r(x_m-\cdot)$ can intersect
at most two segments $L_j^m$, $L_{j+1}^m$. 
We have:

\vv
\begin{claim}\label{claimgd3}
For $0\leq j \leq N_m$, let $x_m\in L^m_j\subset{\Gamma_{ex}^m}'$ be such that $B(x_m,4r)\cap L^m_{j+1}\neq\varnothing$.
Denote by $g_j^m$ and $g_{j+1}^m$ the constant densities of $\sigma^m$ on $L_j^m$ and
$L_{j+1}^m$ respectively.
Then
\begin{equation}\label{eqsk330}
|\psi_r*\sigma^m(x_m)|\lesssim \meas(\rho_j^m,\rho_{j+1}^m)^2 + |g_j^m-g_{j+1}^m|.
\end{equation}
\end{claim}

\vv
Of course, the same estimate holds if $x\in L^m_{j+1}$ and $B(x_m,4r)\cap L^m_{j}\neq\varnothing$.
\vv

Again, we will assume the claim for the moment and we will continue the proof of the lemma.
To this end, denote by $L^1_1,L^2_a,L^3_a,\ldots,L^{m-1}_a$ the ancestors of $L^m_j$. That is to say, for each $1\leq i \leq m-1$, $L^i_a$ is one of the segments $L^i_h$ that constitutes $\Gamma^i$ and $L^i_a$ generates $L^{i+1}_a$  (with $L^{m}_a=L^m_j$ for some $j$). 

Analogously, let
$L^1_1,L^2_b,L^3_b,\ldots,L^{m-1}_b$ be the ancestors of $L^m_{j+1}$. Let $n\geq 0$ be maximal
integer such that $L^n_a=L^n_b$, and if this does not exist (this only happens in the cases $j=0$ and
$j=N_m$), set $n=0$. Thus, either $L^n_a=L^n_b$ is the  closest common ancestor of $L_j^m$
and $L_{j+1}^m$, or $n=0$. We denote by $g^i_a$ and $g^i_b$  the constant density of $\sigma^i$ on $L^i_a$
and $L_b^i$, respectively. Then we write
\begin{equation}\label{eqsk331}
|g_j^m-g_{j+1}^m|\leq \sum_{i=n}^{m-1} |g_a^i - g_a^{i+1}| +  
\sum_{i=n}^{m-1} |g_b^i - g_b^{i+1}|.
\end{equation}
As in \rf{eqdifdens}, for each $i$ we have
\begin{equation}\label{eqsk332}
|g_a^i - g_a^{i+1}|\lesssim \meas(\rho_a^i,\rho_a^{i+1})^2\lesssim \beta_{\infty,\gex}(4B_a^i)^2.
\end{equation}
So from \rf{eqsk330}, \rf{eqsk331} and \rf{eqsk332} we deduce that
$$|\psi_r*\sigma^m(x_m)|\lesssim \meas(\rho_j^m,\rho_{j+1}^m)^2 +
\sum_{i=n}^{m-1} \bigl(\beta_{\infty,\gex}(4B_a^i)^2 + \beta_{\infty,\gex}(4B_b^i)^2\bigr)
\leq \sum_{i=n}^{m-1} \beta_{\infty,\gex}(c_{12}B_a^i)^2,$$
for some absolute constant $c_{12}$.

We  need now to introduce some additional notation. We write $L\sim \Gamma^k$ if $L=L^i_j$ for some
$1\leq i \leq k$, $0\leq j \leq N_i+1$. For such $L$, we write $\wt\ell(L) = 2^{-i/2}\,d_0$ and 
$\beta_{\infty,\gex}(L) = \beta_{\infty,\gex}(c_{12}B_j^i)$. We say that $L^i_{j-1}$ and 
$L^i_{j+1}$ are neighbors of $L^i_j$.
Also, given $L\sim \Gamma^k$ and $L'\sim \Gamma^k$, we write $L'\prec L$ if $L'$ is an ancestor
of $L$ such that $L'$ is not the ancestor of all the neighbors of $L$.

Using the above notation, given $x_m\in L^m_j=L$, by Cauchy-Schwarz, we get
\begin{equation}\label{eqshr33}
|\psi_r*\sigma^m(x_m)|^2\lesssim \Bigl(\sum_{\substack{L'\sim\Gamma^k: L'\prec L}} \beta_{\infty,\gex}(L')^2
\Bigr)^2 \lesssim \sum_{\substack{L'\sim\Gamma^k: L'\prec L}} \beta_{\infty,\gex}(L')^4\,\frac{\wt\ell(L')^{1/2}}
{\wt\ell(L)^{1/2}}.
\end{equation}
Then we deduce
\begin{align}\label{eqakgd22}
\circled{2}& =\int_{\Gamma^k}\int_{2^{-(k+2)/2}d_0}^{d_0/10} |\psi_r*\sigma^m(x_m)|^2\,\frac{dr}r\,d\HH^1(x)\\
& \lesssim \sum_{L\sim\Gamma^k} 
\sum_{\substack{L'\sim\Gamma^k: L'\prec L}} \beta_{\infty,\gex}(L')^4\,\frac{\wt\ell(L')^{1/2}}
{\wt\ell(L)^{1/2}}\,\wt\ell(L) \nonumber\\& =
\sum_{L'\sim\Gamma^k} \beta_{\infty,\gex}(L')^4
\sum_{\substack{L\sim\Gamma^k: L'\prec L}} \,\frac{\wt\ell(L')^{1/2}}
{\wt\ell(L)^{1/2}}\,\wt\ell(L)\nonumber
.
\end{align}
To deal with the last sum, note that for any given $L'\sim\Gamma^k$
the number of segments $L^i_h\sim\Gamma^k$ such that $L'\prec L^i_h$ of a fixed generation $i$ is at most $2$. Then it 
follows that 
$$\sum_{\substack{L\sim\Gamma^k: L'\prec L}} \,\frac{\wt\ell(L')^{1/2}}
{\wt\ell(L)^{1/2}}\,\wt\ell(L)\lesssim \wt\ell(L'),$$
So we deduce that
$$\circled{2}\lesssim 
\sum_{L'\sim\Gamma^k} \beta_{\infty,\gex}(L')^4\,\wt\ell(L')\lesssim_{A,\tau} \sum_{Q\in\DG}
\beta_{\infty,\geq}(Q)^4\,
\ell(Q).$$

\vvv
\noi{\bf The remaining term}.\\
It remains to estimate the integral
\begin{equation}\label{eq333*}
\circled{3}=\int_{\gex}\int_0^{2^{-(k+2)/2}d_0}|\psi_r*\sigma^k(x)|^2\,\frac{dr}r\,d\HH^1(x).
\end{equation}
The arguments will be quite similar to the ones we used for \circled{2}. We will use the same 
notation as the one for that case. 

Note that $\psi_r*\sigma^k(x)=0$ for $x\in\gex\setminus\gex'$. In fact, $\psi_r*\sigma^k(x)$
vanishes unless $x$ belongs to some segment $L_j^m$, $0\leq j\leq N_k+1$ and $x$ is the
at a distance at most $4r$ from one of the endpoints of $L_j^m$.
Moreover, arguing as in \rf{eqshr33} (setting $m=k$ and $x_m=x$), for $x\in L_j^k$
 we get
$$
|\psi_r*\sigma^k(x)|^2\lesssim \sum_{\substack{L'\sim\Gamma^k: L'\prec L_j^k}} \beta_{\infty,\gex}(L')^4\,\frac{\wt\ell(L')^{1/2}}
{\wt\ell(L_j^k)^{1/2}}.
$$
As a consequence,
$$\int_{L_j^k} |\psi_r*\sigma^k(x)|^2\,d\HH^1(x)\lesssim r\sum_{\substack{L'\sim\Gamma^k: L'\prec L_j^k}} \beta_{\infty,\gex}(L')^4\,\frac{\wt\ell(L')^{1/2}}
{\wt\ell(L_j^k)^{1/2}},$$
using also that $\psi_r*\sigma^k(x)=0$ far from the endpoints of $L_j^k$, as explained above.
Then we obtain
\begin{align*}
\circled{3}& = \sum_{j=0}^{N_{k+1}}\int_0^{2^{-(k+2)/2}d_0} \int_{L_j^k}
|\psi_r*\sigma^k(x)|^2\,d\HH^1(x)\,\frac{dr}r\\
&\lesssim \sum_{j=0}^{N_{k+1}}\int_0^{2^{-(k+2)/2}d_0}r\sum_{\substack{L'\sim\Gamma^k: L'\prec L_j^k}} \beta_{\infty,\gex}(L')^4\,\frac{\wt\ell(L')^{1/2}}
{\wt\ell(L_j^k)^{1/2}}\,\frac{dr}r\\
&= \sum_{j=0}^{N_{k+1}}\sum_{\substack{L'\sim\Gamma^k: L'\prec L_j^k}} \beta_{\infty,\gex}(L')^4\,\frac{\wt\ell(L')^{1/2}}
{\wt\ell(L_j^k)^{1/2}}\,\wt\ell(L_j^k).
\end{align*}
Note that the right had side above does not exceed the right hand side of \rf{eqakgd22}. So 
arguing as we did for \circled{2}, we deduce
$$\circled{3}\lesssim_{A,\tau} \sum_{Q\in\DG}\beta_{\infty,\gex}(Q)^4\,
\ell(Q).$$
\end{proof}

\vv

\begin{proof}[\bf Proof of Claim \ref{clageom}]
 To simplify notation, we set
$y_i:=y$ and $y_{i+1}:=\Pi_i(y)$. In this way, the left side of \rf{eqdf227} becomes 
 $\bigl||y_{i+1}-x_{i+1}| - |y_i-x_i|\bigr|$. Denote by $L$ the line 
through $x_i$ and $y_i$.
If $x_i\in\Gamma^i$,
let $\rho^i(x_i)$ be the line which supports the segment $L^i_h$ that contains $x_i$,
and in the case that $x_i\in\Gamma_{ex}^i\setminus \Gamma^i$, let $\rho^i(x_i)=\rho_1^1$.
Let 
$\rho^i(y_i)$ the analogous one that contains $y_i$.
Denote by $\alpha_x$ the angle between $L$ and $\rho^i(x_i)$, and by $\alpha_y$ the one between $L$ and $\rho^i(y_i)$.

\begin{figure}
\begin{center}
\psset{xunit=1.0cm,yunit=1.0cm,algebraic=true,dimen=middle,dotstyle=o,dotsize=3pt 0,linewidth=0.8pt,arrowsize=3pt 2,arrowinset=0.25}
\begin{pspicture*}(-5.19,0)(8.5,5)
\psplot{-5.19}{7.87}{(--8.24-0*x)/8.24}
\psline(6.66,0.52)(3.54,1.56)
\psline[linestyle=dashed,dash=3pt 3pt](5.24,1)(6.02,3.14)
\psplot{-5.19}{7.87}{(--16.81-0*x)/8.24}
\psline[linestyle=dashed,dash=3pt 3pt](5.24,1)(5.24,3.73)
\psline(4.13,4.57)(7.27,2.23)
\psline[linestyle=dashed,dash=3pt 3pt](-1.67,1)(-2.24,2.04)
\psline[linestyle=dashed,dash=3pt 3pt](4.76,4.09)(4.72,2.04)
\psline[linestyle=dashed,dash=3pt 3pt](4.72,2.04)(5.24,1)
\psline(-3.32,0.11)(-0.18,1.83)
\psline(-0.46,2.6)(-4.04,1.52)
\parametricplot{2.825440179007492}{3.141592653589793}{1*1.13*cos(t)+0*1.13*sin(t)+5.24|0*1.13*cos(t)+1*1.13*sin(t)+1}
\parametricplot{0.0}{0.5144739061246572}{1*1.05*cos(t)+0*1.05*sin(t)+-1.67|0*1.05*cos(t)+1*1.05*sin(t)+1}
\parametricplot[linestyle=dashed,dash=3pt 3pt]{1.5707474849570173}{2.0372783075598386}{1*0.63*cos(t)+0*0.63*sin(t)+5.24|0*0.63*cos(t)+1*0.63*sin(t)+1}

\begin{scriptsize}
\psdots[dotstyle=*,linecolor=black](5.24,1)
\rput[bl](5.4,1.1){\black{$y_i$}}
\rput[bl](-5.08,0.66){$L$}
\psdots[dotstyle=*,linecolor=black](6.02,3.14)
\rput[bl](6.11,3.27){\black{$y_{i+1}$}}
\rput[bl](-5.08,1.7){$L'$}
\psdots[dotstyle=*,linecolor=black](5.24,3.73)
\rput[bl](5.32,3.86){\black{$y^{iv}$}}
\psdots[dotstyle=*,linecolor=black](4.76,4.09)
\rput[bl](4.85,4.22){\black{$y''$}}
\psdots[dotstyle=*,linecolor=black](-1.67,1)
\rput[bl](-1.72,1.18){\black{$x_i$}}
\psdots[dotstyle=*,linecolor=black](-2.24,2.04)
\rput[bl](-2.6,2.17){\black{$x_{i+1}$}}
\psdots[dotstyle=*,linecolor=black](4.72,2.04)
\rput[bl](4.81,2.17){\black{$y'$}}
\psdots[dotstyle=*,linecolor=black](5.24,2.04)
\rput[bl](5.32,2.17){\black{$y'''$}}

\rput[bl](-0.4,2.71){\black{$\rho^{i+1}(x_{i+1})$}}
\rput[bl](-3,0){\black{$\rho^{i}(x_{i})$}}

\rput[bl](7,2.5){\black{$\rho^{i+1}(y_{i+1})$}}
\rput[bl](5.8,0.2){\black{$\rho^{i}(y_{i})$}}

\rput[bl](-0.54,1.2){\black{$\alpha_x$}}
\rput[bl](3.7,1.1){\black{$\alpha_y$}}
\rput[bl](4.87,1.72){\black{$\alpha_x$}}

\end{scriptsize}
\end{pspicture*}

\vspace{3mm}

\caption{The points $x_i,x_{i+1},y_i,y_{i+1},y',y'',y''',y^{iv}$ and the different lines 
in the proof of Claim \ref{clageom}.}
 \label{fig1}
\end{center}
\end{figure}
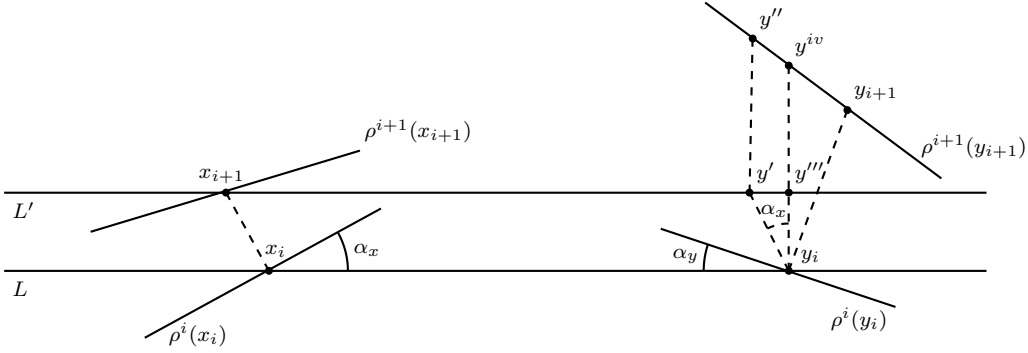

We distinguish two cases. In the first one we assume that both $\alpha_x$, $\alpha_y$ are very small, say $\alpha_x+\alpha_y\leq 1/1000$.
Consider the line $L'$ through $x_{i+1}$ which is parallel to $L$.
See Figure \ref{fig1}.
Let $y_i'\in L'$ be the point such that the segment $[x_i,x_{i+1}]$ is parallel to $[y_i,y']$, so that moreover $|y_i-y'|=|x_i-x_{i+1}|$ and $|y_i-x_i|=|y'-x_{i+1}|$.
So we have
\begin{align}\label{eqak48}
\bigl||y_{i+1}-x_{i+1}| - |y_i-x_i|\bigr|&  = 
\bigl||y_{i+1}-x_{i+1}| - |y'- x_{i+1}|\bigr|\\
&= \frac{\bigl||y_{i+1}-x_{i+1}|^2 - |y'- x_{i+1}|^2\bigr|}
{\bigl||y_{i+1}-x_{i+1}| + |y'- x_{i+1}|\bigr|} \nonumber\\
&\lesssim \frac{\bigl||y_{i+1}-x_{i+1}|^2 - |y'- x_{i+1}|^2\bigr|}
r,\nonumber
\end{align}
since $|y'-x_{i+1}|\approx r$. 
Let $\rho^{i+1}(y_{i+1})$ be the line containing the segment $L_j^{i+1}$ such that  $y_{i+1}\in 
L_j^{i+1}$. Note that the angle between $\rho^{i+1}(y_{i+1})$ and $L$ is small because $\alpha_y\leq 1/1000$ and $\meas(\rho^{i+1}(y_{i+1}),\rho^{i}(y_{i}))$ is also small too.
 Let $y''\in \rho^{i+1}(y_{i+1})$ be such that the angle between the segment $[y'',y']$ and the line $L'$ is a right angle. Then,
by Pythagoras' theorem,
$$|y'-x_{i+1}|^2 + |y'-y''|^2 = |y''-x_{i+1}|^2.$$
Thus
\begin{align}\label{eqak49}
\bigl||y_{i+1}-x_{i+1}|^2 - |y'- x_{i+1}|^2\bigr|&\leq 
\bigl||y_{i+1}-x_{i+1}|^2 - |y'' - x_{i+1}|^2\bigr|+
\bigl||y'' - x_{i+1}|^2 - |y'- x_{i+1}|^2\bigr|\\
& = \bigl||y_{i+1}-x_{i+1}|^2 - |y'' - x_{i+1}|^2\bigr| + |y'-y''|^2.\nonumber
\end{align}
For the last term on the right side we set
$$|y'-y''|^2\lesssim |y'-y_i|^2 + |y_i-y_{i+1}|^2 + |y_{i+1} -y''|^2 =
|x_i-x_{i+1}|^2 + |y_i-y_{i+1}|^2 + |y_{i+1} -y''|^2 .$$
Regarding the first term on the right hand side of \rf{eqak49}, we  have
\begin{align}\label{eqdjhgs32}
\bigl||y_{i+1}-x_{i+1}|^2 - |y'' - x_{i+1}|^2\bigr| & = 
\bigl||y_{i+1}-x_{i+1}| - |y'' - x_{i+1}|\bigr| \cdot 
\bigl||y_{i+1}-x_{i+1}| + |y'' - x_{i+1}|\bigr|\\
&\leq |y_{i+1}- y''| \cdot 
\bigl||y_{i+1}-x_{i+1}| + |y'' - x_{i+1}|\bigr|.\nonumber
\end{align}
We write
$$|y_{i+1}-x_{i+1}|\leq |y_{i+1}-y_i| +|y_{i}-x_{i}|   + |x_i-x_{i+1}|.$$
By the assumption \rf{eqas402}, $|y_{i}-x_{i}|\lesssim r$. The other terms on the right
side above are also bounded by $c\,r$ because for all $z\in\Gamma^i$, $|z-\Pi_i(z)|\lesssim
\ve_0\,2^{-i/2}d_0 \leq\ve_0\,2^{-m/2}d_0\lesssim r$. We also set
$$|y'' - x_{i+1}|\leq |y'' - y_{i+1}| + |y_{i+1}- x_{i+1}| \leq |y'' - y_{i+1}| + c\,r.$$
Thus the left side of \rf{eqdjhgs32} does not exceed
$$|y_{i+1}- y''|^2 + c\,r\,|y_{i+1}- y''|.$$
Then, by \rf{eqak49} and the above inequalities, we obtain
\begin{equation}\label{eqrema32}
\bigl||y_{i+1}-x_{i+1}|^2 - |y'- x_{i+1}|^2\bigr|\leq
|x_i-x_{i+1}|^2 + |y_i-y_{i+1}|^2 + c\,|y'' - y_{i+1}|^2 + c\,r\,|y'' - y_{i+1}|.
\end{equation}

Note that 
\begin{equation}\label{eqdaj31}
|x_i-x_{i+1}|\lesssim \ell(Q^i(x_i)) \, \beta_{\infty,\gex}(Q^i(x_i))
\qquad\mbox{and} \qquad|y_i-y_{i+1}|\lesssim \ell(Q^i(y_i))  \,\beta_{\infty,\gex}(Q^i(y_i)).
\end{equation}
So it remains to estimate the term $|y'' - y_{i+1}|$ from \rf{eqrema32}.
To this end, we consider the points $y'''$, $y^{iv}$, as in Figure \ref{fig1}.
That is, we consider a hyperplane $H$ orthogonal to $L$ through $y_i$ and then we put $\{y'''\}=H\cap L'$
and $\{y^{iv}\} = H\cap \rho^{i+1}(y_{i+1})$.

We write
$$|y_{i+1} - y''| \leq |y_{i+1} - y^{iv}| + |y^{iv} - y''| .$$
By elementary geometry, it follows that
$$|y_{i+1} - y^{iv}|\lesssim \sin\alpha_y\,|y_i-y_{i+1}|$$
and
$$|y^{iv} - y''|\lesssim |y' - y'''|\lesssim \sin\alpha_x\,|y_i-y'|= \sin\alpha_x\,|x_i-x_{i+1}|.$$
So we get
$$|y_{i+1} - y''| \lesssim\sin\alpha_x\,|x_i-x_{i+1}| + \sin\alpha_y\,|y_i-y_{i+1}|\lesssim r.$$
Therefore,
$$
|y_{i+1} - y''|^2 + r\,|y_{i+1} - y''|\lesssim r\,\bigl(\sin\alpha_x\,|x_i-x_{i+1}| + \sin\alpha_y\,|y_i-y_{i+1}|\bigr).
$$

Now we take into account that
\begin{equation}\label{eqak298}
\sin\alpha_x\lesssim  \sum_{Q\in\DG:Q^i(x_i)\subset Q\subset 2B_S}\beta_{\infty,\gex}(Q),
\end{equation}
and analogously for $\sin\alpha_y$. Appealing to \rf{eqdaj31} then we deduce
\begin{align}\label{eqbuu5}
|y_{i+1} - y''|^2 + r\,|y_{i+1} - y''| &\lesssim r\,|x_i-x_{i+1}|\sum_{Q\in\DG:Q^i(x_i)\subset Q\subset 2B_S} \beta_{\infty,\gex}(Q)\\
&\quad  \nonumber
+
r\,|y_i-y_{i+1}|\sum_{Q\in\DG:Q^i(y_i)\subset Q\subset 2B_S} \beta_{\infty,\gex}(Q)
\\
& \lesssim r\,\ell(Q_i(x_i))\Bigl(\sum_{Q\in\DG:Q^i(x_i)\subset Q\subset 2B_S} \beta_{\infty,\gex}(Q)\Bigr)^2 
\nonumber\\
&\quad +r\,   \nonumber
\ell(Q_i(y_i))
\Bigl(\sum_{Q\in\DG:Q^i(y_i)\subset Q\subset 2B_S} \beta_{\infty,\gex}(Q)\Bigr)^2.
\end{align}
By \rf{eqdaj31} again, it is also clear that both $|x_i-x_{i+1}|^2$ and $|y_i-y_{i+1}|^2$ are bounded
by some constant times the right hand side of the preceding inequality.
Then, from \rf{eqrema32} and \rf{eqbuu5} it follows that
\begin{align*}
\frac{\bigl||y_{i+1}-x_{i+1}|^2 - |y'- x_{i+1}|^2\bigr|}r & \lesssim
\ell(Q_i(x_i))\Bigl(\sum_{Q\in\DG:Q^i(x_i)\subset Q\subset 2B_S} \beta_{\infty,\gex}(Q)\Bigr)^2 
\nonumber\\
&\quad +\ell(Q_i(y_i))
\Bigl(\sum_{Q\in\DG:Q^i(y_i)\subset Q\subset 2B_S} \beta_{\infty,\gex}(Q)\Bigr)^2,
\end{align*}
which together with \rf{eqak48} proves the claim in the case when $\alpha_x+\alpha_y\leq1/1000$.

\vv
Suppose now that  $\alpha_x+\alpha_y>1/1000$, so that, for example,  
$\alpha_x>1/2000$. Then we write
\begin{align*}
\bigl||y_{i+1}-x_{i+1}| - |y_i-x_i|\bigr| & \leq |y_{i+1}-y_i| + |x_{i+1}- x_i|\\
& \leq \ell(Q_i(x_i)) + \ell(Q_i(y_i))\\
& = 2\, \ell(Q_i(x_i)) \\
& \lesssim (\sin\alpha_x)^2\,\ell(Q_i(x_i)).
\end{align*}
Using \rf{eqak298}, we obtain
\begin{align*}
\bigl||y_{i+1}-x_{i+1}| - |y_i-x_i|\bigr|
& \lesssim \ell(Q_i(x_i))\,
\bigl(\sin\alpha_x\bigr)^2
\lesssim \ell(Q_i(x_i))\,
\Biggl(\sum_{\substack{Q\in\DG:\\
Q^i(x_i)\subset Q\subset 2B_S}} \beta_{\infty,\gex}(Q)\Biggr)^2,
\end{align*}
which proves the claim in this second case.
\end{proof}

\vv

\begin{proof}[\bf Proof of Claim \ref{claimgd3}]
Note that $\supp\psi_r(x_m-\cdot)$ intersects $L^m_j$, $L^m_{j+1}$, and no other segments of the form $L_h^m$.
So we have
\begin{align*}
\psi_r*\sigma^m(x_m) & = \int_{L_j^m} \psi_r(x_m-y)\, g^m_j\,d\HH^1(y) + \int_{L_{j+1}^m} \psi_r(x_m-y)\, g^m_{j+1}\,d\HH^1(y)\\
& = \left(\int_{L_j^m} \psi_r(x_m-y)\, g^m_j\,d\HH^1(y) + \int_{L_{j+1}^m} \psi_r(x_m-y)\, g^m_j\,d\HH^1(y)
\right)\\
& \quad + \left(\int_{L_{j+1}^m} \psi_r(x_m-y)\, (g_{j+1}^m-g_j^m)\,d\HH^1(y)\right)\\&  = \circled{A} +
\circled{B}.
\end{align*}
It is immediate to check that $\Bigl|\circled{B}\Bigr|\lesssim |g_j^m-g_{j+1}^m|$. 

Regarding \circled{A}, we set
\begin{align*}
\circled{A} 
& = \left(\int_{L_j^m} \psi_r(x_m-y)\, g^m_j\,d\HH^1(y) + \int_{\rho_j^m\setminus L_j^m} \psi_r(x_m-y)\, g^m_j\,d\HH^1(y)\right) \\
& \quad +
\left(\int_{L_{j+1}^m} \psi_r(x_m-y)\, g^m_j\,d\HH^1(y) - \int_{\rho_j^m\setminus L_j^m} \psi_r(x_m-y)
\,g^m_j\,d\HH^1(y)\right).
\end{align*}
The first term on the right hand side vanishes (taking into account that $x_m\in\rho_m^j$), and so we only have to deal with the last one.
To this end, consider a rotation $R$ which transforms $L_{j+1}^m$ into a segment contained in $\overline{ \rho_j^m
\setminus L_j^m}$, fixes $\{x_j^m\}=L_j^m\cap L_{j+1}^m$, and leaves invariant the subspace of $\R^d$
orthogonal to the plane formed by $L_j^m$ and $L_{j+1}^m$ (assuming these segments to be not collinear, otherwise we let $R$ be the identity).
 Since $\HH^1|_{L_{j+1}^m} = R^{-1}\#(\HH^1|_{R(L_{j+1}^m})$, we have
\begin{align*}
\int_{L_{j+1}^m} \psi_r(x_m-y)\, g^m_j\,d\HH^1(y) & = \int \psi_r(x_m-y)\, g^m_j\,d
R^{-1}\#(\HH^1|_{R(L_{j+1}^m)})(y) \\
& = \int \psi_r(x_m-R^{-1}(y))\, g^m_j\,d\HH^1|_{R(L_{j+1}^m)}(y)\\
& = \int_{\rho^m_j\setminus L_j^m} \psi_r(x_m-R^{-1}(y))\, g^m_j\,d\HH^1(y).
\end{align*}
Therefore,
$$\circled{A} = \int_{\rho_j^m\setminus L_j^m} \bigl[\psi_r(x_m-R^{-1}(y))-\psi_r(x_m-y)\bigr]\,g^m_j\,d\HH^1(y).$$

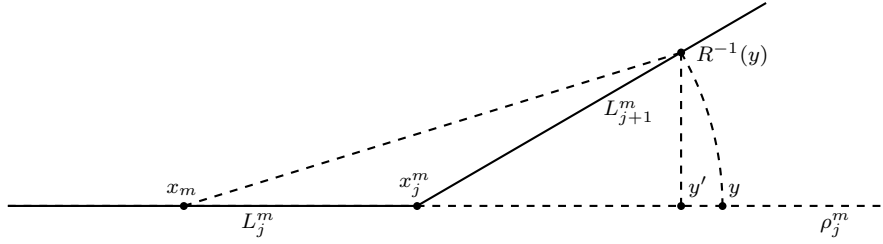
\begin{figure}
\begin{center}
\psset{xunit=1.0cm,yunit=1.0cm,algebraic=true,dimen=middle,dotstyle=o,dotsize=3pt 0,linewidth=0.8pt,arrowsize=3pt 2,arrowinset=0.25}
\begin{pspicture*}(-4.3,0.3)(7.3,5)
\psplot[linestyle=dashed,dash=3pt 3pt]{-4.3}{7.3}{(--7.26-0*x)/10.08}
\psline(1.14,0.72)(5.78,3.42)
\parametricplot[linestyle=dashed,dash=3pt 3pt]{0.0}{0.5270017751987305}{1*4.06*cos(t)+0*4.06*sin(t)+1.14|0*4.06*cos(t)+1*4.06*sin(t)+0.72}
\psline(-4.3,0.72)(1.14,0.72)
\psline[linestyle=dashed,dash=3pt 3pt](-1.96,0.72)(4.65,2.76)
\psline[linestyle=dashed,dash=3pt 3pt](4.65,2.76)(4.65,0.72)
\begin{scriptsize}
\psdots[dotstyle=*,linecolor=black](1.14,0.72)
\psdots[dotstyle=*,linecolor=black](4.65,0.72)
\rput[bl](0.9,0.84){\black{$x_j^m$}}
\rput[bl](3.62,1.8){$L_{j+1}^m$}
\psdots[dotstyle=*,linecolor=black](5.2,0.72)
\rput[bl](5.28,0.84){\black{$y$}}
\rput[bl](4.75,0.84){\black{$y'$}}
\rput[bl](-1.2,0.3){$L_j^m$}
\rput[bl](6.5,0.3){$\rho_j^m$}
\psdots[dotstyle=*,linecolor=black](4.65,2.76)
\rput[bl](4.85,2.58){\black{$R^{-1}(y)$}}
\psdots[dotstyle=*,linecolor=black](-1.96,0.72)
\rput[bl](-2.19,0.86){\black{$x_m$}}
\end{scriptsize}
\end{pspicture*}

\vspace{3mm}

\caption{The points $x_m,x_j^m,y,y',R^{-1}(y)$ and the segments $L_j^m$ and $L_{j+1}^m$
in the proof of Claim \ref{claimgd3}.}
 \label{fig2}
 \end{center}
\end{figure}

For $y\in \supp\bigl[\psi_r(x_m-R^{-1}(\cdot))-\psi_r(x_m-\cdot)\bigr]\cap \rho_j^m$, we
claim that
\begin{equation}\label{eqdkv498}
\bigl||x_m-R^{-1}(y)|-|x_m-y|\bigr| \lesssim r\,\meas(\rho_j^m,\rho_{j+1}^m)^2.
\end{equation}
To see this, consider the orthogonal projection $y'$ of $R^{-1}(y)$ on $\rho_j^m$ (see Figure \ref{fig2}), and set
$$\bigl||x_m-y| - |x_m-R^{-1}(y)|\bigr| \leq 
\bigl||x_m-y| - |x_m-y'|\bigr| + \bigl||x_m-y'| - |x_m-R^{-1}(y)|\bigr|.$$
By Pythagoras theorem, taking into account that $|x_m-y'|\approx|x_m-y|\approx r$, we get
\begin{align*}
\bigl||x_m-y'|-|x_m-R^{-1}(y)|\bigr| & = \frac{\bigl||x_m-y'|^2-|x_m-R^{-1}(y)|^2\bigr|}{\bigl||x_m-y'|+|x_m-R^{-1}(y)|\bigr|}
 \lesssim \frac{|R^{-1}(y)-y'|^2}r\\
& \lesssim 
\sin\bigl(\meas(R^{-1}(y),x_m,y')\bigr)^2\,r\leq \sin\bigl(\meas(R^{-1}(y),x_j^m,y')\bigr)^2\,r.
\end{align*}
On the other hand,
\begin{align*}
\bigl||x_m-y| - |x_m-y'|\bigr| & = |y-y'| = |x_j^m-R^{-1}(y)| - |x_j^m-y'| \\
& = \bigl(1-
\cos\bigl(\meas(R^{-1}(y),x_j^m,y')\bigr)\bigr) \,|x_j^m-R^{-1}(y)|\lesssim 
\sin\bigl(\meas(R^{-1}(y),x_j^m,y')\bigr)^2 \,r,
\end{align*}
which completes the proof of \rf{eqdkv498}.

Now, from \rf{eqdkv498} we deduce
$$\bigl|\psi_r(x_m-R^{-1}(y))-\psi_r(x_m-y)\bigr|\lesssim \frac{\meas(\rho_j^m,\rho_{j+1}^m)^2}r.$$
So we obtain
$$
\Bigl|\circled{A}\Bigr| \lesssim \frac{\meas(\rho_j^m,\rho_{j+1}^m)^2}r\,
\HH^1\bigl(\rho_j^m\cap \supp\bigl[\psi_r(x_m-R^{-1}(\cdot))-\psi_r(x_m-\cdot)\bigr]\bigr)\lesssim 
\meas(\rho_j^m,\rho_{j+1}^m)^2.$$
Together with the estimate we got for \circled{B}, this concludes the proof of the claim.
\end{proof}
\vvv

Next we denote
$$\wt \Delta_{\sigma^k,\vphi}(x,r) =
 \frac1r\int_{|x-y|\leq 4r} \bigl|\vphi_r*\sigma^k(x) - \vphi_r*\sigma^k(y)\bigr|\,d\sigma^k(y)
 $$ 
 and
$$\wt{\wt \Delta}_{\sigma^k,\vphi}(x,r)=
\int\left| \int \bigl[\vphi_r(x-z)\,\vphi_{2r}(z-y) -
\vphi_{2r}(x-z)\vphi_{r}(z-y) \bigr]\sigma^k(z)\right|d\sigma^k(y).
$$

Arguing as above, we will get estimates  for $\wt \Delta_{\sigma^k,\vphi}$
and $\wt{\wt \Delta}_{\sigma^k,\vphi}$ analogous to the ones obtained in Lemma
\ref{lemakey99} for $\Delta_{\sigma^k,\vphi}$. We will not give detailed proofs because the arguments
are very similar to the ones for Lemma \ref{lemakey99}.

\vv

\begin{lemma}\label{lemakey99'}
We have
$$
\int_{\gex}\int_0^\infty \left|\wt \Delta_{\sigma^k,\vphi}(x,r)\right|^2\,\frac{dr}r\,d\HH^1(x)
\leq_{A,\tau}
\sum_{Q\in\DG}\beta_{\infty,\gex}(Q)^4\,\ell(Q).
$$
\end{lemma}

\begin{proof}[Sketch of the proof]
We will just explain the estimate of the integral
$$\int_{\Gamma^k_{ex}}\int_{2^{-(k+2)/2}d_0}^\infty |\wt\Delta_{\sigma^k,\vphi}(x,r)|^2\,\frac{dr}r\,d\HH^1(x).$$
The arguments for remaining integral
\begin{align*}
\circled{3}=\int_{\Gamma^k_{ex}}\int_0^{2^{-(k+2)/2}d_0} 
|\wt\Delta_{\sigma^k,\vphi}(x,r)|^2\,\frac{dr}r\,d\HH^1(x).
\end{align*}
are very similar to the analogous integral in \rf{eq333*} in Lemma \ref{lemakey99}.

Notice that by Cauchy-Schwarz we have
$$\wt \Delta_{\sigma^k,\vphi}(x,r)^2 \lesssim
 \frac1r\int_{|x-y|\leq 4r} \bigl|\vphi_r*\sigma^k(x) - \vphi_r*\sigma^k(y)\bigr|^2\,d\sigma^k(y).$$
Thus,
\begin{align*}
\int_{\Gamma^k_{ex}}\int_{2^{-(k+2)/2}d_0}^\infty&  \wt \Delta_{\sigma^k,\vphi}(x,r)^2\,\frac{dr}r\,d\HH^1(x)\\
& \lesssim \int_{2^{-(k+2)/2}d_0}^\infty \int_{x\in\Gamma^k_{ex}}\int_{y\in\gex:|x-y|\leq4r} \bigl|\vphi_r*\sigma^k(x) - \vphi_r*\sigma^k(y)\bigr|^2\,d\HH^1(y)d\HH^1(x)\,\frac{dr}{r^2}.
\end{align*} 
Given  
 $x\in\Gamma^k_{ex}$ and $r\geq2^{-(k+2)/2}d_0$, take the maximal integer
 $m\leq k$ such that $2^{-(m+2)/2}\,d_0\geq 10r$. 
 As in the proof of Lemma \ref{lemakey99}, consider the points
$x_m,x_{m+1},\ldots,x_k=x$ such that $x_i\in \Gamma^i_{ex}$ and $\Pi_i(x_i)=x_{i+1}$ for $i=m,\,m+1,\ldots,k-1$.
Analogously, for  $y\in\Gamma^k_{ex}$, let
$y_m,y_{m+1},\ldots,y_k=y$ be such that $y_i\in \Gamma^i_{ex}$ and $\Pi_i(y_i)=y_{i+1}$ for $i=m,\,m+1,\ldots,k-1$.
 Then we set
\begin{align*}
\bigl|\vphi_r*\sigma^k(x) - \vphi_r*\sigma^k(y)\bigr| &\leq 
\bigl|\vphi_r*\sigma^k(x) - \vphi_r*\sigma^{m}(x_m)\bigr|+
\bigl|\vphi_r*\sigma^k(y) - \vphi_r*\sigma^m(y_m)\bigr|\\
& \quad+
\bigl|\vphi_r*\sigma^m(x_m) - \vphi_r*\sigma^m(y_m)\bigr|,
\end{align*}
so that
\begin{align*}
\int_{\Gamma^k_{ex}}& \int_{2^{-(k+2)/2}d_0}^\infty \wt \Delta_{\sigma^k,\vphi}(x,r)^2\,\frac{dr}r\,d\HH^1(x)\\ &\lesssim
\int_{2^{-(k+2)/2}d_0}^\infty \int_{x\in\Gamma^k_{ex}}\int_{y\in\gex:|x-y|\leq4r} \bigl|\vphi_r*\sigma^k(x) - \vphi_r*\sigma^{m}(x_m)\bigr|^2\,d\HH^1(y)d\HH^1(x)\,\frac{dr}{r^2}\\
&\quad +
\int_{2^{-(k+2)/2}d_0}^\infty \int_{x\in\Gamma^k_{ex}}\int_{y\in\gex:|x-y|\leq4r} \bigl|\vphi_r*\sigma^k(y) - \vphi_r*\sigma^m(y_m)\bigr|^2\,d\HH^1(y)d\HH^1(x)\,\frac{dr}{r^2}\\
&\quad +
\int_{2^{-(k+2)/2}d_0}^\infty \int_{x\in\Gamma^k_{ex}}\int_{y\in\gex:|x-y|\leq4r} \bigl|\vphi_r*\sigma^m(x_m) - \vphi_r*\sigma^m(y_m)\bigr|^2\,d\HH^1(y)d\HH^1(x)\,\frac{dr}{r^2}\\
& = \circled{1a} +\circled{1b} + \circled{2}.
\end{align*}

To deal with \circled{1a} observe that
$$\circled{1a}\lesssim_{A,\tau}
\int_{2^{-(k+2)/2}d_0}^\infty \int_{x\in\Gamma^k_{ex}} \bigl|\vphi_r*\sigma^k(x) - \vphi_r*\sigma^{m}(x_m)\bigr|^2\,d\HH^1(x)\,\frac{dr}{r}.
$$
By \rf{eqhos33} (which also holds with $\psi_r$ replaced by $\vphi_r$), we get
$$\circled{1a}\lesssim_{A,\tau}
\sum_{Q\in\DG}\beta_{\infty,\gex}(Q)^4\,\ell(Q).
$$

By Fubini, the integral \circled{1b} coincides with \circled{1a}. 
On the other hand, the estimates for \circled{2} are also analogous to the ones for the term also denoted by \circled{2} in the
proof of Lemma \ref{lemakey99} and so we omit the details again.
\end{proof}

\vv
\begin{lemma}\label{lemakey99''}
We have
$$
\int_{\gex}\int_0^\infty \Bigl|\wt {\wt \Delta}_{\sigma^k,\vphi}(x,r)\Bigr|^2\,\frac{dr}r\,d\HH^1(x)
\lesssim_{A,\tau}
\sum_{Q\in\DG}\beta_{\infty,\gex}(Q)^4\,\ell(Q).
$$
\end{lemma}

\begin{proof}[Sketch of the proof]
The arguments are very similar to the ones of the preceding two lemmas.
So we will just explain the estimate of the integral
$$\int_{\Gamma^k_{ex}}\int_{2^{-(k+2)/2}d_0}^\infty \wt {\wt \Delta}_{\sigma^k,\vphi}(x,r)^2\,\frac{dr}r\,d\HH^1(x).$$

Denote 
$$F_r^k(x,y) = r\int \bigl(\vphi_r(x-z)\,\vphi_{2r}(z-y) - \vphi_{2r}(x-z)\,\vphi_{r}(z-y)\bigr)
\,d\sigma^k(z),$$
so that
$$\wt {\wt \Delta}_{\sigma^k,\vphi}(x,r) = \frac1r\int |F_r^k(x,y)|\,d\sigma^k(y).$$
Note that $F_r^k$ vanishes identically if $\sigma^k$ coincides with the arc-length measure
on some line containing  $x$ and $y$. Observe also that $F_r^k(x,y)=0$ if $|x-y|\geq 6r$.
By Cauchy-Schwarz then we have
$$\wt {\wt \Delta}_{\sigma^k,\vphi}(x,r)^2 \lesssim
\frac1r\int_{y\in\gex:|x-y|\leq 6r} |F_r^k(x,y)|^2\,d\HH^1(y).$$
Thus,
\begin{align*}
\int_{\Gamma^k_{ex}}\int_{2^{-(k+2)/2}d_0}^\infty&  \wt {\wt \Delta}_{\sigma^k,\vphi}(x,r)^2\,\frac{dr}r\,d\HH^1(x)\\
& \lesssim \int_{2^{-(k+2)/2}d_0}^\infty \int_{x\in\Gamma^k_{ex}}\int_{y\in\gex:|x-y|\leq6r} \bigl|F_r^k(x,y)\bigr|^2\,d\HH^1(y)d\HH^1(x)\,\frac{dr}{r^2}.
\end{align*} 
Given  
 $x,y\in\gex$ and $r\geq2^{-(k+2)/2}d_0$ let
 $m\leq k$ the maximal integer such that $2^{-(m+2)/2}\,d_0\geq 10r$. 
 As in the proof of Lemma \ref{lemakey99'}, consider the points
$x_m,x_{m+1},\ldots,x_k=x$ such that $x_i\in \Gamma^i_{ex}$ and $\Pi_i(x_i)=x_{i+1}$ for $i=m,\,m+1,\ldots,k-1$ and the analogous ones 
$y_m,y_{m+1},\ldots,y_k=y$.

For each $i$, denote
$$F_{r,a}^i(x_i,y_i) = 
r\int \vphi_r(x_i-z)\,\vphi_{2r}(z-y_i) 
\,d\sigma^i(z),$$
$$F_{r,b}^i(x_i,y_i) = 
r\int \vphi_{2r}(x_i-z)\,\vphi_{r}(z-y_i) 
\,d\sigma^i(z),$$
so that
$F_{r}^i(x_i,y_i) = F_{r,a}^i(x_i,y_i) - F_{r,b}^i(x_i,y_i)$.
Write
\begin{align*}
\int_{\Gamma^k_{ex}}& \int_{2^{-(k+2)/2}d_0}^\infty \wt{\wt \Delta}_{\sigma^k,\vphi}(x,r)^2\,\frac{dr}r\,d\HH^1(x)\\ &
\lesssim
\int_{2^{-(k+2)/2}d_0}^\infty \int_{x\in\Gamma^k_{ex}}\int_{y\in\gex:|x-y|\leq6r} \bigl|F_{r,a}^k(x,y)-
F_{r,a}^m(x_m,y_m)\bigr|^2\,d\HH^1(y)d\HH^1(x)\,\frac{dr}{r^2}\\
&\quad + \int_{2^{-(k+2)/2}d_0}^\infty \int_{x\in\Gamma^k_{ex}}\int_{y\in\gex:|x-y|\leq6r} \bigl|F_{r,b}^k(x,y)-
F_{r,b}^m(x_m,y_m)\bigr|^2\,d\HH^1(y)d\HH^1(x)\,\frac{dr}{r^2}\\
&\quad +
\int_{2^{-(k+2)/2}d_0}^\infty \int_{x\in\Gamma^k_{ex}}\int_{y\in\gex:|x-y|\leq6r} \bigl|
F_r^m(x_m,y_m)\bigr|^2\,d\HH^1(y)d\HH^1(x)\,\frac{dr}{r^2}\\
& = \circled{1a} + \circled{1b} + \circled{2}.
\end{align*}

\vv
\noi{\bf Estimate of \circled{1a}}.\\
By the triangle inequality,
\begin{align*}
\bigl|F_{r,a}^k(x,y)-F_{r,b}^m(x_m,y_m)\bigr|  \leq  \sum_{i=m}^{k-1}\bigl|F_{r,a}^i(x_i,y_i)-
F_{r,a}^{i+1}(x_{i+1},y_{i+1})\bigr|.
\end{align*}
 Since
$\sigma^{i+1} = \Pi_{i,\#}\sigma^i$, we have
\begin{align*}
\bigl|F_{r,a}^i&(x_i,y_i) - F_{r,a}^{i+1}(x_{i+1},y_{i+1})\bigr|\\
& = 
r\left|
\int \vphi_r(x_i-z)\,\vphi_{2r}(z-y_i) \,d\sigma^i(z) - 
\int \vphi_r(\Pi_i(x_i)-\Pi_i(z))\,\vphi_{2r}(\Pi_i(z)-\Pi_i(y_i)) \,d\sigma^i(z)\right|\\
& \leq 
r
\int |\vphi_r(x_i-z)- \vphi_r(\Pi_i(x_i)-\Pi_i(z))|\,\vphi_{2r}(z-y_i) \,d\sigma^i(z) \\
& \quad+ 
r
\int |\vphi_{2r}(z-y_i)- \vphi_{2r}(\Pi_i(z)-\Pi_i(y_i))|\,\vphi_r(\Pi_i(x_i)-\Pi_i(z)) \,d\sigma^i(z).
\end{align*}
As $\vphi_r$ is supported on $B(0,2r)$ and constant in $B(0,r/2)$, we derive
\begin{align}\label{eqigu389}
\bigl|F_{r,a}^i(x_i,y_i) - F_{r,a}^{i+1}(x_{i+1},y_{i+1})\bigr| &\lesssim \frac1{r^2}\int_{c^{-1}r\leq |x_i-z|\leq 6r} \bigl||\Pi_i(x_i)-\Pi_i(z)|-|x_i-z|\bigr|\,
d\sigma^k(z)\\
&\quad +
\frac1{r^2}\int_{c^{-1}r\leq |y_i-z|\leq 6r} \bigl||\Pi_i(y_i)-\Pi_i(z)|-|y_i-z|\bigr|\,
d\sigma^k(z).\nonumber
\end{align}
Observe the similarities between this estimate and the one of $|\psi_r*\sigma^i(x_i) -\psi_r*\sigma^{i+1}(x_{i+1})|$ in \rf{eqash44}.
Then we obtain
\begin{align*}
\circled{1a} & \lesssim 
\int_{2^{-(k+2)/2}d_0}^\infty \int_{x\in\Gamma^k_{ex}}
\left|\sum_{i=m}^{k-1}
\frac1{r^2}\int_{c^{-1}r\leq |x_i-z|\leq 6r} \bigl||\Pi_i(x_i)-\Pi_i(z)|-|x_i-z|\bigr|\,
d\sigma^k(z)\right|^2\!
d\HH^1(x)\,\frac{dr}{r}
\\
& \quad +  
\int_{2^{-(k+2)/2}d_0}^\infty \int_{y\in\Gamma^k_{ex}}
\left|\sum_{i=m}^{k-1}
\frac1{r^2}\int_{c^{-1}r\leq |y_i-z|\leq 6r} \bigl||\Pi_i(y_i)-\Pi_i(z)|-|y_i-z|\bigr|\,
d\sigma^k(z)\right|^2\!
d\HH^1(y)\,\frac{dr}{r}.
\end{align*}
By Fubini, both terms on the right hand side coincide. Moreover,
arguing as in the estimate of the term \circled{1} in the proof of Lemma \ref{lemakey99} it
follows that both 
are bounded by
$$C(A,\tau)\sum_{Q\in\DG}\beta_{\infty,\gex}(Q)^4\,\ell(Q).$$

\vv
\noi{\bf Estimate of \circled{1b}}.\\
The arguments are almost the same as the ones for \circled{1a}.

\vv
\noi{\bf Estimate of \circled{2}}.\\
We will use the following result.

\vv
\begin{claim}\label{claimgd3''}
Let $x_m\in L^m_j\subset\Gamma^m$ be such that $B(x_m,6r)\cap L^m_{j+1}\neq\varnothing$.
Denote by $g_j^m$ and $g_{j+1}^m$ the constant densities of $\sigma^m$ on $L_j^m$ and
$L_{j+1}^m$ respectively.
Then
\begin{equation}\label{eqsk330''}
|F_r^m(x_m,y_m)|\lesssim \meas(\rho_j^m,\rho_{j+1}^m)^2 + |g_j^m-g_{j+1}^m|.
\end{equation}
\end{claim}
\vv

The proof is quite similar to the one of Claim \ref{claimgd3}, taking into account
that 
$$\int_L \bigl(\vphi_r(x-z)\,\vphi_{2r}(z-y) - \vphi_{2r}(x-z)\,\vphi_{r}(z-y)\bigr)
\,d\HH^1(z)$$
vanishes when $L$ is a line and $x,y\in L$.
 For the reader's convenience we
show the detailed proof below.

Arguing as we did to estimate the term denoted also by \circled{2} in the
proof of Lemma \ref{lemakey99}, we find that 
$$\circled{2}\lesssim_{A,\tau}
\sum_{Q\in\DG}\beta_{\infty,\gex}(Q)^4\,\ell(Q).
$$
\end{proof}

\vvv
\begin{proof}[\bf Proof of Claim \ref{claimgd3''}]
To simplify notation, we write
$$f_r(x,y,z) = r\bigl(\vphi_r(x-z)\,\vphi_{2r}(z-y) - \vphi_{2r}(x-z)\,\vphi_{r}(z-y)\bigr),
$$
so that
$$F_r^m(x_m,y_m)= \int f_r(x_m,y_m,z)\,d\sigma^m(z).$$

Note that $f_r(x_m,y_m,z)$ vanishes unless $|x_m-y_m|\leq 6r$ and $|x_m-z|\leq 4r$.
So to estimate $F_r^m(x_m,y_m)$ we may assume that $x_m\in L_j^m$ and $y_m,z\in L_j^m\cup L_{j+1}^m$.
Denote by $\Pi_{\rho_j^m}$ the orthogonal projection on the line $\rho_j^m$, and let
${\sigma^m}'=\Pi_{\rho_j^m,\#}\sigma^m$ and $y_m'=\Pi_{\rho_j^m}(y_m)$. Then we have
\begin{align*}
F_r^m(x_m,y_m) -  \int f_r(x_m,y_m',z)\,d{\sigma^m}'(z) =
\int \bigl[f_r(x_m,y_m,z) - f_r(x_m,\Pi_{\rho_j^m}(y_m), \Pi_{\rho_j^m}(z))\bigr]\,
d\sigma^m(z).
\end{align*}
Arguing as in \rf{eqigu389}, we derive
\begin{align*}
\left|F_r^m(x_m,y_m) -  \!\int \!f_r(x_m,y_m',z)\,d{\sigma^m}'(z)\right|  &\lesssim \frac1{r^2}\int_{c^{-1}r\leq |x_m-z|\leq 6r} \bigl||x_m-\Pi_{\rho_j^m}(z)|-|x_m-z|\bigr|\,
d\sigma^m(z)\\
&\hspace{-1.1cm}+
\frac1{r^2}\int_{c^{-1}r\leq |y_m-z|\leq 6r} \bigl||\Pi_{\rho_j^m}(y_m)-\Pi_{\rho_j^m}(z)|-|y_m-z|\bigr|\,
d\sigma^m(z).
\end{align*}
By Pythagoras' theorem it follows easily that
$$\bigl||x_m-\Pi_{\rho_j^m}(z)|-|x_m-z|\bigr|\lesssim \frac{|\Pi_{\rho_j^m}(z)-z|^2}r\lesssim
r\,\meas(\rho_j^m,\rho_{j+1}^m)^2$$ 
and also
$$\bigl||\Pi_{\rho_j^m}(y_m)-\Pi_{\rho_j^m}(z)|-|y_m-z|\bigr|
\lesssim \frac{|\Pi_{\rho_j^m}(z)-z|^2}r + \frac{|\Pi_{\rho_j^m}(y_m)-y_m|^2}r\lesssim
r\,\meas(\rho_j^m,\rho_{j+1}^m)^2.$$ 
Therefore,
\begin{equation}\label{eqddk77}
\left|F_r^m(x_m,y_m) -  \int f_r(x_m,y_m',z)\,d{\sigma^m}'(z)\right|\lesssim 
\meas(\rho_j^m,\rho_{j+1}^m)^2.
\end{equation}

On the other hand, it is easy to check that 
$${\sigma^m}'|_{B(x_m,6r)\cap L_j^m} = g_{j}^m\,\HH^1|_{B(x_m,6r)\cap
 L_j^m} $$
and
$${\sigma^m}'|_{B(x_m,6r)\cap\rho_j^m\setminus L_j^m} = \frac{g_{j+1}^m}{\cos
\meas(\rho_j^m,\rho_{j+1}^m)}\,\HH^1|_{B(x_m,6r)\cap
\rho_j^m\setminus L_j^m}. $$
So taking into account that $\int f_r(x_m,y_m',z)\,d\HH^1|_{\rho_j^m}(z)=0$, we obtain
\begin{align*}
\left|\int f_r(x_m,y_m',z)\,d{\sigma^m}'(z)\right| &= \left|\int f_r(x_m,y_m',z)\,d({\sigma^m}'
- g_j^m\,\HH^1|_{\rho_j^m})(z)\right|\\
& \leq \int |f_r(x_m,y_m',z)|\,\biggl|\frac{g_{j+1}^m}{\cos
\meas(\rho_j^m,\rho_{j+1}^m)} - g_j^m\biggr|\, d\HH^1|_{\rho_j^m}(z).
\end{align*}
Using that 
$$\biggl|\frac{g_{j+1}^m}{\cos
\meas(\rho_j^m,\rho_{j+1}^m)} - g_j^m\biggr|\lesssim \meas(\rho_j^m,\rho_{j+1}^m)^2 + |g_j^m-g_{j+1}^m|,$$
and that $|f_r(x_m,y_m',z)|\lesssim1/r$ and $\supp f_r(x_m,y_m',\cdot)\subset B(x_m,cr)$, we 
infer that
$$\left|\int f_r(x_m,y_m',z)\,d{\sigma^m}'(z)\right|\lesssim \meas(\rho_j^m,\rho_{j+1}^m)^2 + |g_j^m-g_{j+1}^m|,$$
which together with \rf{eqddk77} proves the claim.
\end{proof}

\vv
The following is an immediate consequence of the preceding results.

\begin{lemma}\label{corokey99}
We have
\begin{equation}\label{eqsq}
\int_{\gex}\int_0^\infty \left|\Delta_{\sigma^k,\vphi}(x,r)\right|^2\,\frac{dr}r\,d\HH^1(x)
\leq C(A,\tau,K)\,\ve_0^2\,\ell(R).
\end{equation}
The analogous estimate holds replacing $\Delta_{\sigma^k,\vphi}(x,r)$ by $\wt \Delta_{\sigma^k,\vphi}(x,r)$ or $\wt{\wt\Delta}_{\sigma^k,\vphi}(x,r)$.
\end{lemma}

\begin{proof}
We have shown in Lemma \ref{lemsigmad} that
$$\int_{\gex}\int_0^\infty \left|\Delta_{\sigma^k,\vphi}(x,r)\right|^2\,\frac{dr}r\,d\HH^1(x)
\leq C(A,\tau)
\sum_{Q\in\DG}\beta_{\infty,\gex}(Q)^4\,\ell(Q).$$
Since $\beta_{\infty,\gex}(Q)\lesssim \ve_0$ for all $Q\in\DG$, we have
$$\sum_{Q\in\DG}\beta_{\infty,\gex}(Q)^4\,\ell(Q)\lesssim \ve_0^2 \,\sum_{Q\in\DG}\beta_{\infty,\gex}(Q)^2\,\ell(Q).$$
On the other hand, by Jones' traveling salesman theorem \cite{Jones}, \cite{Okikiolu} , it follows easily that
 the sum on the right is bounded by $c\,\HH^1(\Gamma^k)$, and so by $C(A,\tau,K)\,\ell(R)$.
\end{proof}

\vv


\section{The $L^2(\sigma^k)$ norm of the density of $\nu^k$ with respect to $\sigma^k$}
\label{sec13}

Recall that both $\nu^k$ and $\sigma^k$ are AD-regular measures supported on $\Gamma^k$. In particular, they are mutually absolutely continuous with respect to $\HH^1|_{\Gamma^k}$ and thus
there exists some function $f_k$ bounded above and away from zero such that
$\nu^k = f_k\,\sigma^k$. 
By Lemma \ref{lemadnu}, the density of $\nu^k$ with respect to $\HH^1|_{\gex}$ satisfies
\begin{equation}\label{eqdens467}
\frac{d\nu^k}{d\HH^1|_{\gex}}\approx_{A,\tau} \Theta_\mu(B_R),
\end{equation}
and by Lemmas \ref{lemsigmagrow} and \ref{lemsigmad},
\begin{equation}\label{eqdens468}
\frac{d\sigma^k}{d\HH^1|_{\gex}}\approx_{M} 1.
\end{equation}
So we have
$$f_k = \frac{d\nu^k}{d\sigma^k}\approx_{A,\tau,M} \Theta_\mu(B_R).$$
Recall also that the density $\frac{d\nu^k}{d\HH^1|_{\gex}}$ is constantly equal to $c_0^k$
far away from $B(x_0, K r_0)$. Analogously, $\sigma^k$ coincides with $\HH^1|_{\gex}$ 
out of $B(x_0, K r_0)$. So $f_k-c_0^k$ is compactly supported. The main objective of this section  consists in estimating the $L^2(\sigma^k)$ norm of $f_k-c_0^k$.
To this end, for $r>0$, $x\in\R^d$, and a function $g\in L^1_{loc}(\sigma^k)$ we define
$$D_r g(x)=\frac{\vphi_r*(g\sigma^k)(x)}{\vphi_r*\sigma^k(x)}
- \frac{\vphi_{2r}*(g\sigma^k)(x)}{\vphi_{2r}*\sigma^k(x)}$$
(recall that $\vphi_r(y)=\frac1r\,\vphi\Bigl(\frac{|y|}r\Bigr)$).

\vv
\begin{lemma}\label{lemv21}
For $f_k=\dfrac{d\nu^k}{d\sigma^k}$, $r>0$, and $x\in\gex$, we have
$$\int_\gex\int_0^\infty  |D_r f_k|^2\,\frac{dr}r\,
d\sigma^k \lesssim_{A,\tau,K,M} \ve_0^{1/10}\,\Theta_\mu(B_R)^2\,\ell(R).$$
\end{lemma}

\begin{proof}
By Lemmas \ref{lem116} and \ref{corokey99}, it is enough to show that
\begin{equation}\label{eqag33}
|D_r f_k(x)|\lesssim_{A,\tau,M} |\psi_r*\nu^k(x)| + \Theta_\mu(B_R)\,|\psi_r*\sigma^k(x)|.
\end{equation}
Then we write
\begin{align*}
|D_r f_k(x)| & = \left|\frac{\vphi_r*\nu^k(x)}{\vphi_r*\sigma^k(x)}
- \frac{\vphi_{2r}*\nu^k(x)}{\vphi_{2r}*\sigma^k(x)}\right|\\
& = \left|\frac{\vphi_r*\nu^k(x)\,\vphi_{2r}*\sigma^k(x) - \vphi_{2r}*\nu^k(x)\,\vphi_r*\sigma^k(x)}
{\vphi_r*\sigma^k(x)\,\vphi_{2r}*\sigma^k(x)}\right|\\
& \leq \left|\frac{\vphi_r*\nu^k(x)-\vphi_{2r}*\nu^k(x)}{\vphi_{r}*\sigma^k(x)}\right| +
\left|\frac{\bigl(\vphi_{2r}*\sigma^k(x)-\vphi_r*\sigma^k(x)\bigr)\vphi_{2r}*\nu^k(x)}{\vphi_{r}*\sigma^k(x)\,
\vphi_{2r}*\sigma^k(x)}\right|.
\end{align*}
The inequality \rf{eqag33} just follows then from \rf{eqdens467} and \rf{eqdens468}, which imply that $\vphi_{r}*\sigma^k(x)\gtrsim_M \!1$,
 $\vphi_{2r}*\sigma^k(x)\gtrsim_M \!1$, and $\vphi_{2r}*\nu^k(x)\lesssim_{A,\tau} \Theta_\mu(B_R)$.
\end{proof}

\vv
Notice that the operators $D_r$ vanish on constant functions. That is, $D_r1\equiv0$ for all $r>0$. In order to apply some quasiorthogonality arguments, we would also need their adjoints to satisfy $D_r^*1\equiv0$. Unfortunately this property is not fulfilled, in general. For this reason, we are going to introduce a variant of the operator $D_r$ which
we will denote by $\wt D_r$ that will be better suited for the quasiorthogonality techniques we intend to
apply. 

For a function $g\in L^1_{loc}(\sigma^k)$ and $r>0$, we denote
$$S_r g(x) = \frac{\vphi_r*(g\sigma^k)(x)}{\vphi_r*\sigma^k(x)},$$
so that $D_rg = S_rg -S_{2r}g$.
Let $W_r$ be the operator of multiplication by $1/S_r^*1$. Then we consider the operator
$$\wt S_r = S_r\,W_r\,S_r^*,$$
and we define $\wt D_r = \wt S_r-\wt S_{2r}$. Notice that $\wt S_r$, and thus $\wt D_r$, is self-adjoint. Moreover
$\wt S_r1 \equiv1$, so that 
$$\wt D_r1 = \wt D_r^*1 =0.$$

We denote by $s_r(x,y)$ the kernel of $S_r$ with respect to $\sigma^k$. That is, $s_r(x,y)$ is the function
such that $S_r g(x)  =\int s_r(x,y)\,g(y)\,d\sigma^k(y)$. Observe that this equals
$$s_r(x,y) = \frac1{\vphi_r*\sigma^k(x)}\,\vphi_r(x-y).$$
On the other hand, the kernel of $\wt S_k$ is the following:
$$\wt s_r(x,y) = \int s_r(x,z)\,\frac1{S_r^*1(z)}\,s_r(y,z)\,d\sigma^k(z).$$

The following is, by now, a standard result from Littlewood-Paley theory in homogeneous spaces
first proved by David, Journ\'e, and Semmes, relying on some ideas from Coifman.

\begin{theorem}\cite{DJS}
Let $1<r_0<2$ and let $g\in L^2(\sigma^k)$. Then
\begin{equation}\label{eqp44}
\|g\|_{L^2(\sigma^k)}^2 \approx_{A,\tau,M}\, \int_\gex \sum_{j\in\Z}  |\wt D_{2^{-j}r_0} g|^2
\,d\sigma^k.
\end{equation}
\end{theorem}

Note that the constants involved in the estimate \rf{eqp44} 
do not depend on $r_0$.
An easy consequence is the following.

\begin{lemma}\label{lemdjs}
For $g\in L^2(\sigma^k)$, we have
\begin{equation*}\label{eqp244'}
\|g\|_{L^2(\sigma^k)}^2 \approx_{A,\tau,M}\, 
\int_\gex\int_0^\infty  |\wt D_r g|^2\,\frac{dr}r
d\sigma^k.
\end{equation*}
\end{lemma}

\begin{proof}
Just notice that
$$\int_\gex\int_0^\infty  |\wt D_r g|^2\,\frac{dr}r
d\sigma^k = \int_\gex\,\sum_{j\in\Z}\, \int_1^2 
 |\wt D_{r2^{-j}} g|^2\,\frac{dr}r\,
d\sigma^k,$$
and then use Fubini and \rf{eqp44}.
\end{proof}

\vv
\begin{lemma}\label{lemv22}
Let $f\in L^\infty(\sigma^k)$. The following estimates hold:
\begin{equation}\label{eqsss1}
\int_{\gex}\int_0^\infty |S_r f -  S_r^* f|^2\,\frac{dr}r\,d\sigma^k\lesssim_{A,\tau ,K ,M} 
\ve_0^{2}\,\|f\|_{L^\infty(\sigma^k)}^2\, \ell(R),
\end{equation}
and
\begin{equation}\label{eqsss2}
\int_{\gex}\int_0^\infty |S_r\,S_{2r}f -  S_{2r}\,S_r f|^2\,\frac{dr}r\,d\sigma^k\lesssim_{A,\tau ,K ,M} 
\ve_0^{2}\,\|f\|_{L^\infty(\sigma^k)}^2\, \ell(R).
\end{equation}
\end{lemma}

\begin{proof} To see the first estimate, we write
\begin{align*}
|S_r f(x) -  S_r^* f(x)| & = \left|\int_\gex \left(\frac1{\vphi_r*\sigma^k(x)} - \frac1{\vphi_r*
\sigma^k(y)}\right)  \,\vphi_r(x-y) \,f(y)\,d\sigma^k(y)\right| \\
 & \lesssim_{A,\tau,M} \|f\|_{L^\infty(\sigma^k)} \int_\gex \bigl|\vphi_r*\sigma^k(x) - \vphi_r*\sigma^k(y)\bigr| \,\vphi_r(x-y) \,d\sigma^k(y).
\end{align*}
Since $\supp\vphi_r(x-\cdot)\subset B(x,2r)$, the last integral is bounded by 
$$
 \frac cr\int_{|x-y|\leq 2r} \bigl|\vphi_r*\sigma^k(x) - \vphi_r*\sigma^k(y)\bigr|\,d\sigma^k(y)
 \leq c\,\wt \Delta_{\sigma^k,\vphi}(x,r).
 $$ 
By Lemma \ref{corokey99}, we derive
$$
\int_{\gex}\int_0^\infty \left|\wt\Delta_{\sigma^k,\vphi}(x,r)\right|^2\,\frac{dr}r\,d\HH^1(x)
\lesssim_{A,\tau,K,M} 
\ve_0^{2}\, \Theta_\mu(B_R)^2\,\ell(R),
$$
and thus \rf{eqsss1} follows.
 
\vv
To prove \rf{eqsss2} we write
\begin{align*}
S_r\,S_{2r}f(x) -  S_{2r}\,S_r f(x) &= 
\iint \frac1{\vphi_r*\sigma^k(x)\,\vphi_{2r}*\sigma^k(z)}\,\vphi_r(x-z)\,\vphi_{2r}(z-y)\, f(y)\,d\sigma^k(z)\,d\sigma^k(y)\\
&\quad \!- \!
\iint \!\frac1{\vphi_{2r}*\sigma^k(x)\,\vphi_{r}*\sigma^k(z)}\,\vphi_{2r}(x-z)\,\vphi_{r}(z-y)\, f(y)\,d\sigma^k(z)\,d\sigma^k(y)\\
& = I_1+ I_2 + I_3,
\end{align*}
where
\begin{align*}
I_1\!&= \!\!\iint \!\left[\frac1{\vphi_r*\sigma^k(x)\,\vphi_{2r}*\sigma^k(z)}
- \frac1{\vphi_{2r}*\sigma^k(x)\,\vphi_{r}*\sigma^k(z)}\right]\!
\vphi_r(x-z)\,\vphi_{2r}(z-y) f(y)d\sigma^k(z)d\sigma^k(y),\\
I_2 &= \iint \left[\frac1{\vphi_{2r}*\sigma^k(x)\,\vphi_{r}*\sigma^k(z)} - \frac1{\vphi_r*\sigma^k(x)\,\vphi_{2r}*\sigma^k(x)}\right]\\
&\qquad\qquad \qquad\qquad\qquad\;\;\cdot\bigl[\vphi_r(x-z)\,\vphi_{2r}(z-y) -
\vphi_{2r}(x-z)\,\vphi_{r}(z-y)\bigr]\, f(y)\,d\sigma^k(z)d\sigma^k(y),\\
I_3 & = \frac1{\vphi_r*\sigma^k(x)\,\vphi_{2r}*\sigma^k(x)}
\iint \left[\vphi_r(x-z)\,\vphi_{2r}(z-y) -
\vphi_{2r}(x-z)\vphi_{r}(z-y) \right]f(y)\sigma^k(z)d\sigma^k(y).
\end{align*}
To estimate $I_1$ we set
\begin{align*}
\left|\frac1{\vphi_r*\sigma^k(x)\,\vphi_{2r}*\sigma^k(z)}
- \frac1{\vphi_{2r}*\sigma^k(x)\,\vphi_{r}*\sigma^k(z)}\right| & \leq 
\frac{\bigl|\vphi_{2r}*\sigma^k(x) - \vphi_{r}*\sigma^k(x)\bigr|}{\vphi_r*\sigma^k(x)\,\vphi_{2r}*\sigma^k(x)\,
\vphi_{2r}*\sigma^k(z)} \\
&\quad + 
\frac{\bigl|\vphi_{r}*\sigma^k(z) - \vphi_{2r}*\sigma^k(z)\bigr|}{\vphi_{2r}*\sigma^k(x)\,\vphi_r*\sigma^k(z)\,
\vphi_{2r}*\sigma^k(z)}\\
&\lesssim_{A,\tau,M} |\psi_r*\sigma^k(x)| + |\psi_r*\sigma^k(z)|.
\end{align*}
Then we obtain
$$|I_1|\lesssim_{A,\tau,M} \|f\|_{L^\infty(\sigma^k)}\,\bigl[|\psi_r*\sigma^k(x)| +  \vphi_r*(|\psi_r*\sigma^k|
\sigma^k)(x)
\bigr].$$
So writing $I_1\equiv I_{1,r}(x)$, we have
\begin{equation}\label{eqfdl55}
\|I_{1,r}\|_{L^2(\sigma^k)}\lesssim_{A,\tau,M} \|f\|_{L^\infty(\sigma^k)}\,\|\psi_r*\sigma^k\|_{L^2(\sigma^k)}.
\end{equation}

Concerning $I_2$, we have
\begin{align*}
\left|\frac1{\vphi_{2r}*\sigma^k(x)\,\vphi_{r}*\sigma^k(z)} - \frac1{\vphi_r*\sigma^k(x)\,\vphi_{2r}*\sigma^k(x)}\right| & = 
\frac{\bigl|\vphi_r*\sigma^k(x) - \vphi_r*\sigma^k(z)\bigr|}
{\vphi_{r}*\sigma^k(x)\,\vphi_{r}*\sigma^k(z)\, \vphi_{2r}*\sigma^k(x)}\\
&\lesssim_{A,\tau,M}
\bigl|\vphi_r*\sigma^k(x) - \vphi_r*\sigma^k(z)\bigr|.
\end{align*}
Notice that if $z,y$ belong to the domain of integration of $I_2$, then $|x-z|\leq 4r$ and $|x-y|
\leq 6r$. 
So we can write 
\begin{align*}
|I_2|& \lesssim_{A,\tau,M} \frac1{r^2}\int_{|x-z|\leq 4r} \bigl|\vphi_r*\sigma^k(x) - \vphi_r*\sigma^k(z)\bigr|\,d\sigma^k(z) \;\int_{|x-y|\leq6r} |f(y)|\,d\sigma^k(y)\\
& \lesssim_{A,\tau,M} \|f\|_{L^\infty(\sigma^k)} \frac1r\int_{|x-z|\leq 4r} \bigl|\vphi_r*\sigma^k(x) - \vphi_r*\sigma^k(z)\bigr|\,d\sigma^k(z)\\
& =_{A,\tau,M} \|f\|_{L^\infty(\sigma^k)}\,\wt \Delta_{\sigma^k,\vphi}(x,r).
\end{align*}

To deal with $I_3$ we just write
\begin{align*}
I_3& \lesssim_{A,\tau,M}  \|f\|_{L^\infty(\sigma^k)}\,
\int\left| \int \bigl[\vphi_r(x-z)\,\vphi_{2r}(z-y) -
\vphi_{2r}(x-z)\vphi_{r}(z-y) \bigr]\sigma^k(z)\right|d\sigma^k(y)\\
& =_{A,\tau,M} \|f\|_{L^\infty(\sigma^k)}\,\,\wt{\wt \Delta}_{\sigma^k,\vphi}(x,r).
\end{align*}

Gathering the estimates obtained for $I_1$, $I_2$ and $I_3$ and applying Lemma \ref{corokey99}
we obtain
$$
\int_{\gex}\int_0^\infty \left|
S_r\,S_{2r}f(x) -  S_{2r}\,S_r f(x)\right|^2\,\frac{dr}r\,d\HH^1(x)
\leq_{A,\tau,K,M} \ve_0^{2}\,\|f\|_{L^\infty(\sigma^k)}^2\, \ell(R).
$$ 
\end{proof}

\vv
\begin{lemma}\label{lemv23}
For $f_k=\dfrac{d\nu^k}{d\sigma^k}$, we have
\begin{equation}\label{eqp2445}
\|f_k-c_0^k\|_{L^2(\sigma^k)}^2
\lesssim_{A,\tau,K,M} \ve_0^{1/10}\,\Theta_\mu(B_R)^2\,\ell(R).
\end{equation}
\end{lemma}

\begin{proof}
As explained at the beginning of this section, the function $f_k-c_0^k$ is compactly supported
and moreover it is bounded. Thus it belongs to $L^2(\sigma^k)$.
Then by Lemma \ref{lemdjs},
$$\|f_k-c_0^k\|_{L^2(\sigma^k)}^2 \approx_{A,\tau,M} \int_{\gex}\int_0^\infty |\wt D_r (f_k-c_0^k)|^2\,\frac{dr}r\,d\sigma^k.$$
Since $\wt D_r$ vanishes on constant functions, it turns out that $\wt D_r (f_k-c_0^k) = \wt D_r f_k$.
Thus, using also Lemma \ref{lemv21}, to prove \rf{eqp2445} it suffices to show that
\begin{equation}\label{eqdkj31}
\int_{\gex}\int_0^\infty |\wt D_r f_k|^2\,\frac{dr}r\,d\sigma^k\lesssim_{A,\tau ,K ,M} \int_{\gex}\int_0^\infty |D_r f_k|^2\,\frac{dr}r\,d\sigma^k +\ve_0^{2}\,\|f_k\|_{L^\infty(\sigma^k)}^2\, \ell(R).
\end{equation}

We are going to show that \rf{eqdkj31} holds for any function $f\in L^\infty(\sigma^k)$. To this
end,
recall that $\wt D_r=\wt S_r-\wt S_{2r}$ and $\wt S_r = S_r \,W_r\,S_r^*$,
where $W_r$ is the operator of multiplication by $1/S_r^* 1$.
Note that for any $x\in\R^d$,
$$|\wt S_r f(x) - S_r\,S_r^*f(x)| = \Bigl|S_r \Bigl(\Bigl(\frac1{S_r^*1}-1\Bigr)\,S_r^*f\Bigr)(x)
\Bigr|.$$
Since, for any $y\in\R^d$, $|\frac1{S_r^*1(y)}-1|\lesssim_{A,\tau,M} |S_r^*1(y)-1|$ 
and $|S_r^*f(y)|\lesssim \|f\|_{L^\infty(\sigma^k)}$, we get
$$|\wt S_r f(x) - S_r\,S_r^*f(x)| \lesssim_{A,\tau,M} \|f\|_{L^\infty(\sigma^k)}\,\bigl|S_r \bigl(|S_r^*1-1|\bigr)(x)\bigr|.$$
As $S_r$ is bounded in $L^2(\sigma^k)$ uniformly on $r$, we obtain
$$\|\wt S_r f- S_r\,S_r^*f\|_{L^2(\sigma^k)} \lesssim_{A,\tau,M} \|f\|_{L^\infty(\sigma^k)}\,\|S_r^*1-1\|_{L^2(\sigma^k)} .$$
Applying \rf{eqsss1} to $f=1$, taking into account that $S_r1\equiv1$, we deduce that
\begin{align*}
\int_{\gex}\int_0^\infty |\wt S_r f- S_r\,S_r^*f|^2\,\frac{dr}r\,d\sigma^k & \lesssim_{A,\tau,M} \|f\|_{L^\infty(\sigma^k)}^2\int_{\gex}\int_0^\infty |S_r^*1-S_r1|^2\,\frac{dr}r\,d\sigma^k \\
&\lesssim_{A,\tau,K,M}
\ve_0^{2}\,\|f\|_{L^\infty(\sigma^k)}^2\, \ell(R).
\end{align*}
So we infer that
$$
\int_{\gex}\int_0^\infty |\wt D_r f- (S_r\,S_r^*f - S_{2r}\,S_{2r}^*f)|^2\,\frac{dr}r\,d\sigma^k  
\lesssim_{A,\tau,K,M}
\ve_0^{2}\,\|f\|_{L^\infty(\sigma^k)}^2\, \ell(R).$$
As a consequence, to prove \rf{eqdkj31} for $f_k=f$ it is enough to show that
\begin{equation}\label{eqdkj32}
\int_{\gex}\int_0^\infty |S_r\,S_r^*f - S_{2r}\,S_{2r}^*f|^2\,\frac{dr}r\,d\sigma^k\lesssim_{A,\tau ,K ,M} \int_{\gex}\int_0^\infty |D_r f|^2\,\frac{dr}r\,d\sigma^k +\ve_0^{2}\,\|f\|_{L^\infty(\sigma^k)}^2\, \ell(R).
\end{equation}

We write
\begin{align*}
\|S_r\,&S_r^*f - S_{2r}\,S_{2r}^*f\|_{L^2(\sigma^k)} \\
& \leq 
\|S_r\,S_rf - S_{2r}\,S_{2r}f\|_{L^2(\sigma^k)}  + 
\|S_r\,S_rf - S_{r}\,S_{r}^*f\|_{L^2(\sigma^k)} + \|S_{2r}\,S_{2r}f - S_{2r}\,S_{2r}^*f\|_{L^2(\sigma^k)}\\
&\lesssim_{A,\tau,M} \|S_r\,S_rf - S_{2r}\,S_{2r}f\|_{L^2(\sigma^k)}  + \|S_rf - S_{r}^*f\|_{L^2(\sigma^k)}+ \|S_{2r}f - S_{2r}^*f\|_{L^2(\sigma^k)}.
\end{align*}
To estimate the last two terms on the right side we will use \rf{eqsss1}. For the first one we set
\begin{align*}
\|S_r&\,S_rf - S_{2r}\,S_{2r}f\|_{L^2(\sigma^k)} \\&\leq
\|S_r\,S_rf - S_{r}\,S_{2r}f\|_{L^2(\sigma^k)} +
\|S_r\,S_{2r}f - S_{2r}\,S_{r}f\|_{L^2(\sigma^k)} + 
\|S_{2r}\,S_{r}f - S_{2r}\,S_{2r}f\|_{L^2(\sigma^k)}\\
& \lesssim_{A,\tau,M} 
\|S_rf - S_{2r}f\|_{L^2(\sigma^k)} +
\|S_r\,S_{2r}f - S_{2r}\,S_{r}f\|_{L^2(\sigma^k)} + 
\|S_{r}f - S_{2r}f\|_{L^2(\sigma^k)},
\end{align*}
because of the $L^2(\sigma^k)$ boundedness of $S_r$ and $S_{2r}$.
Thus,
\begin{align*}
\int_{\gex}\int_0^\infty &|S_r\,S_r^*f - S_{2r}\,S_{2r}^*f|^2\,\frac{dr}r\,d\sigma^k \\&
\lesssim_{A,\tau,M} \int_0^\infty \|S_r\,S_{2r}f - S_{2r}\,S_{r}f\|_{L^2(\sigma^k)}^2\,\frac{dr}r\\
& \quad+ \int_0^\infty \bigl(\|S_rf - S_{r}^*f\|_{L^2(\sigma^k)}^2+ \|S_{2r}f - S_{2r}^*f\|_{L^2(\sigma^k)}^2\bigr)\,\frac{dr}r
+ \int_0^\infty \|S_{r}f - S_{2r}f\|_{L^2(\sigma^k)}^2\,\frac{dr}r.
\end{align*}
By Lemma \ref{lemv22}, the first and second integrals on the right hand side do not exceed
$$C(A,\tau,K,M)\,
\ve_0^{2}\,\|f\|_{L^\infty(\sigma^k)}^2\, \ell(R),$$ while the
last one equals
$$\int_{\gex}\int_0^\infty |D_r f|^2\,\frac{dr}r\,d\sigma^k.$$
So \rf{eqdkj31} is proved for any $f\in L^\infty(\sigma^k)$ and consequently the lemma follows.
\end{proof}
\vv


\section{The end of the proof of the Main Lemma \ref{mainlemma}} \label{sec14}

In this section first we will show that the measure of the union of the cells from $\HD$ which are contained
in $R$ is small.
The estimate of the $L^2(\sigma^k)$ norm of $f_k-c_0^k$ will play a key role in the arguments. 
Afterwards we will finish the proof of the Main Lemma.

First we need a technical result:

\begin{lemma}\label{lemtec}
Let $Q\in\nterm$. There exists some cell $P\in\nreg$ such that
\begin{equation}\label{eqppu22}
P\cap 1.1B_Q\neq \varnothing,\qquad \ell(P)\approx\ell(Q),\qquad \mu(P)\gtrsim\mu(1.1B_Q).
\end{equation}
If moreover $Q\in\HD$, then $P$ can be taken so that it also satisfies
$$\wt\mu(P)\approx\mu(P)\approx\mu(1.1B_Q),$$
assuming $\eta$ small enough.
\end{lemma}

\begin{proof}
By Lemma \ref{lem74**}, it follows easily that any cell $S\in\nreg$ with $S\cap 1.1B_Q$ satisfies $\ell(S)\lesssim\ell(Q)$.
Let $0<t<1/100$ be some constant to be fixed below.
Suppose first that all the cells $S\in\nreg$ which intersect $1.1B_Q$ satisfy $\ell(S)\geq t\,\ell(Q)$ and that $N\!W_0\cap 1.1B_Q=\varnothing$.
In this case, the number of such cells is bounded above by some constant depending only on $t$, and so
if we let $P$ be a cell of this family with maximal $\mu$-measure, then 
we have
\begin{equation}\label{eqcas1073}
\mu(P)\gtrsim_t\mu(1.1B_Q) \qquad\mbox{and}\qquad  \ell(P)\approx_t\ell(Q).
\end{equation}

Suppose now that there exists some cell $S\in\nreg$ which intersects $1.1B_Q$ such that $\ell(S)<
t\,\ell(Q)$, or that $N\!W_0\cap 1.1B_Q\neq\varnothing$. We claim that this implies that $Q\in\DD^{db}$
if $t$ is small enough.
To prove this, note that if there exists a cell $S$ with $S\cap1.1B_Q\neq\varnothing$ and $\ell(S)<
t\,\ell(Q)$, then 
$S\subset 1.2B_Q$, and taking a suitable ancestor of $S$ we infer that there exists some cell $S'\in\ngood$ with $\ell(S')\approx\ell(Q)$, $\dist(S',S)\lesssim \ell(S')$. The same holds in the case when 
$N\!W_0\cap 1.1B_Q\neq\varnothing$.
Further, if $t$ is small enough,
then we can assume that $3.3B_{S'}\subset 1.3B_Q$.
Let $a\geq3.3$ be the maximal number such that $a\,B_{S'}\subset 1.5B_Q$.
 Notice that
$r(a\,B_{S'})\geq r(1.5B_Q)-r(1.3B_Q) = 0.2\,r(B_Q)$.
Since $100^2B(Q)$ is contained in $c\,a\,B_{S'}$ for some constant $c\lesssim1$ (independent of $C_0$),
by Remark \ref{remdens} we deduce that
$$\mu(100^2B(Q))\leq c_{13}\,\mu(a\,B_{S'})\leq c_{13}\, \mu(1.5B_Q) = c_{13}\, \mu(1.5\cdot28 B(Q))
\leq c_{13}\,\mu(100B(Q)),$$
with $c_{13}$ independent of $C_0$. Then \rf{eqdob23} does not hold for $Q$ if $C_0$ is taken big enough, which ensures that $Q\in\DD^{db}$ as claimed.

The fact that $Q\in\DD^{db}$ guaranties that $\mu(B(Q))\approx\mu(1.1B_Q)$, by \rf{eqdob22}. 
Using also the small boundaries condition \rf{eqsmb2} we infer that, for some $l$
big enough, the set
$G(Q) = B(Q)\setminus N_l(Q)$ has $\mu$-measure comparable to $\mu(B(Q))$, and so to $\mu(1.1B_Q)$.
Since $Q\not\in\ngood$, from the definitions of $\wt d(\cdot)$ and $\nreg$, it follows easily that any cell
$S\in\nreg$ which intersects $G(Q)$ satisfies $\ell(S)\approx\ell(Q)$. Thus letting $P$ be a
cell from $\nreg$ with $P\cap G(Q)\neq\varnothing$ having maximal $\mu$-measure, as in \rf{eqcas1073}
we deduce that
$$\mu(P)\gtrsim \mu(G(Q))\approx\mu(1.1B_Q) \qquad\mbox{and}\qquad  \ell(P)\approx\ell(Q).$$

It remains now to show that if $Q\in\HD\cap\nterm$, then $\wt\mu(P)\approx\mu(P)\approx\mu(1.1B_Q)$.
In this case, $\mu(1.1B_Q)\gtrsim A\,\Theta_\mu(R)\,\ell(Q)$, and thus
$$\mu(P)\lesssim A\,\Theta_\mu(B_R)\,\ell(P)\approx  A\,\Theta_\mu(B_R)\,\ell(Q)\lesssim\mu(1.1B_Q),$$
and so $\mu(P)\approx\mu(1.1B_Q)$. To prove that $\wt\mu(P)\approx\mu(P)$, let $\wh Q$ be the parent of $Q$
and let $c_{14}>0$ be such that $P\subset c_{14}B_{\wh Q}$. Since $\wh Q\in\good$, by Lemma \ref{lemk12}
we have
$\mu(c_{14}B_{\wh Q} \setminus \wt E)\leq \eta^{1/10}\,\mu(c_{14} B_{\wh Q})$,
and thus
$$\mu(P\setminus \wt E)\leq \mu(c_{14}B_{\wh Q} \setminus \wt E)\leq \eta^{1/10}\,\mu(c_{14} B_{\wh Q})
\lesssim \eta^{1/10}\,A\,\Theta_\mu(B_R) \ell(Q) \lesssim\eta^{1/10}\,\mu(1.1B_Q) \lesssim\eta^{1/10}\,\mu(P),$$
which ensures that $\wt\mu(P)=\mu(P\cap\wt E)\approx \mu(P)$ for $\eta$ small enough.
\end{proof}

\vv

\begin{lemma}\label{lemhdpetit}
We have
$$\mu\biggl(\,\bigcup_{Q\in\HD\cap\DD(R)} Q\biggr)\lesssim \left(\frac{c(A,\tau ,K )}M + 
c(A,\tau,K,M) \,\ve_0^{1/10}\right)\,\mu(R).$$
\end{lemma}

\begin{proof}
Notice that
$$\mu\biggl(\,\bigcup_{Q\in\HD\cap\DD(R)} Q\biggr) \leq
\mu\biggl(\,\bigcup_{Q\in\bsb\cap\DD(R)} Q\biggr) + \mu\biggl(\,\bigcup_{Q\in\HD\cap\nterm\cap\DD(R)} Q\biggr).$$
By Lemma \ref{lembsb}, the first term on the right side does not exceed
$ \frac{c(A,\tau ,K )}M\,\mu(R)$. So it is enough to show that
\begin{equation}\label{eqeno77}
\mu\biggl(\,\bigcup_{Q\in\HD_1} Q\biggr)\leq c(A,\tau,K,M) \,\ve_0^{1/10}\,\mu(R),
\end{equation}
where
$$\HD_1= \HD\cap\nterm\cap\DD(R).$$

Consider a cell $Q\in\HD_1$. 
We wish to relate the measure $\mu$ on $Q$ to the measure $\nu^k$
on some appropriate ball $B_{j(Q)}^k$. To this end, let $P=P(Q)\in\nreg$ be a cell satisfying
\rf{eqppu22}. 
Suppose that $k$ is big enough so that 
\begin{equation}\label{eqsupj54}
2^{-k/2}\,d_0\leq \ell(P(Q)).
\end{equation}
By Lemma \ref{lemfac34} we know that $\wt \mu$-almost all $P(Q)$ is contained in the union of the
balls $B_j^k$, $j=1,\ldots,N_k$, and
by Lemma \ref{lem74} (a),  the balls $B_j^k$ which intersect $P$ have radii comparable to $\ell(P(Q))$.
Thus the number of such balls does not exceed some absolute constant. So letting $B_{j(Q)}^k$ be the ball
of this family which has maximal $\wt\mu$-measure, it turns out that
$$\wt\mu(B_{j(Q)}^k)\approx\wt\mu(P(Q))\approx\mu(1.1B_Q)\gtrsim A\,\Theta_\mu(B_R)\,\ell(Q)
\approx A\,\Theta_\mu(B_R)\,r(B_{j(Q)}^k).$$

Recall now that $\nu^k=\sum_{j=0}^{N_k} \nu_j^k$, with $\supp\nu_j^k\subset \frac32 \bar B_j^k$ for 
$j\in[1,N_k]$. Further, if $\frac32\bar B_j^k\cap \bar B_{j(Q)}^k\neq\varnothing$, by Lemma \ref{lemnofac} (d), we have $r(B_j^k)=\ell_j^k\approx\ell_{j(Q)}^k=r(B_{j(Q)}^k)$ and thus
$\frac32\bar B_j^k\subset c_{15}B_{j(Q)}^k$ for some absolute constant $c_{15}$. So we have
$$\nu^k\bigl(c_{15}B_{j(Q)}^k\bigr)\geq \sum_{j:\frac32\bar B_j^k\cap \bar B_{j(Q)}^k\neq\varnothing}
\|\nu_j^k\|.$$
Since $\|\nu_j^k\|=\int\theta_j^k\,d\wt\mu$ (see \rf{eqnujk}) and by Lemma \ref{lem10.1}
$$\sum_{j:\frac32\bar B_j^k\cap \bar B_{j(Q)}^k\neq\varnothing}\theta_j^k\geq \chi_{B_{j(Q)}^k},$$
we infer that
\begin{equation*}
\nu^k\bigl(c_{15}B_{j(Q)}^k\bigr)\geq \sum_{j:\frac32\bar B_j^k\cap \bar B_{j(Q)}^k\neq\varnothing}
\int \theta_j^k\,d\wt\mu \geq \wt\mu(B_{j(Q)}^k)\approx A\,\Theta_\mu(B_R)\,r(B_{j(Q)}^k).
\end{equation*}

From the preceding estimate, taking into account that $c_0^k\approx\Theta_\mu(B_R)$ (see Remark
\ref{remb22}) and the key fact that $\sigma^k$ has linear growth with an absolute constant (see Lemma
\ref{lemsigmagrow}), we deduce that the density $f_k$ of $\nu^k$ with respect $\sigma^k$ satisfies
\begin{align*}
\int_{c_{15}B_{j(Q)}^k}|f_k-c_0^k|\,d\sigma^k &\geq \nu^k(c_{15}B_{j(Q)}^k) - 
c_0^k\,\sigma^k(c_{15}B_{j(Q)}^k) \\
&\geq c\,A\,\Theta_\mu(B_R)\,r(B_{j(Q)}^k) - c'\,\Theta_\mu(B_R)\,r(B_{j(Q)}^k)\geq \Theta_\mu(B_R)\,r(B_{j(Q)}^k),
\end{align*}
assuming $A$ big enough. By Cauchy-Schwarz and the linear growth of $\sigma^k$, the left hand side is bounded above by $c\,\|f_k-c_0^k\|_{L^2(\sigma^k)}\,r(B_{j(Q)}^k)^{1/2}$. Then for some constant $c_{16}>1$ big enough so that $c_{15}B_{j(Q)}^k\subset c_{16}B_Q$ we get
\begin{equation}\label{eqdj48}
\|\chi_{c_{16}B_Q}\,\bigl(f_k-c_0^k\bigr)\|_{L^2(\sigma^k)}^2\geq
\|\chi_{c_{15}B_{j(Q)}^k}\,\bigl(f_k-c_0^k\bigr)\|_{L^2(\sigma^k)}^2\gtrsim \Theta_\mu(B_R)^2
\,r(B_{j(Q)}^k)\gtrsim_{A,\tau}\,\Theta_\mu(B_R)\, \mu(1.1B_Q).
\end{equation}

Consider now a finite family $\HD_2\subset\HD_1$ such that 
\begin{equation}\label{eq2894}
\mu\biggl(\,\bigcup_{Q\in\HD_2}Q\biggr) \geq \frac12\,\mu\biggl(\bigcup_{Q\in\HD_1}Q\biggr),
\end{equation}
and take $k$ big enough so that \rf{eqsupj54} holds for all the cells $P(Q)$ associated to any $Q\in\HD_2$ as explained above. Consider a subfamily $\HD_3\subset \HD_2$ such that the balls $\{c_{16}B_Q\}_{Q\in\HD_3}$ are pairwise disjoint and
$$\bigcup_{Q\in\HD_2} c_{16}B_Q\subset \bigcup_{Q\in\HD_3} 3c_{16}B_Q.$$
Taking into account that $\mu(3c_{16}B_Q)\lesssim A\,\Theta_\mu(B_R)\,\ell(Q)\lesssim \mu(1.1B_Q)$ and using
\rf{eq2894}, \rf{eqdj48} and \rf{eqp2445}, we get
\begin{align*}
\mu\biggl(\,\bigcup_{Q\in\HD_1} Q\biggr) & \leq 2 \,\mu\biggl(\,\bigcup_{Q\in\HD_2} c_{16}B_Q
\biggr) 
 \leq 2 \,\sum_{Q\in\HD_3} \mu(3c_{16}B_Q)
 \lesssim \sum_{Q\in\HD_3} \mu(1.1B_Q)\\
& \lesssim_{A,\tau} \frac1{\Theta_\mu(B_R)}\sum_{Q\in\HD_3}\|\chi_{c_{16}B_Q}\,\bigl(f_k-c_0^k\bigr)\|_{L^2(\sigma^k)}^2
 \leq \frac1{\Theta_\mu(B_R)}\,\|f_k-c_0^k\|_{L^2(\sigma^k)}^2\\ & \lesssim_{A,\tau ,K ,M} \ve_0^{1/10}\,\Theta_\mu(B_R)\,\ell(R),
\end{align*}
which proves \rf{eqeno77}.
\end{proof}
\vvv

The preceding lemma was the last step for the proof of Main Lemma \ref{mainlemma}. For the reader's
convenience, we state it here again. Recall that $F\subset \supp\mu=E$ is an arbitrary compact set such that
$$
\int_{F}\int_0^1 \Delta_\mu(x,r)^2\,\frac{dr}r\, d\mu(x)<\infty.
$$

\vv

\begin{mainlemma*}
Let $0<\ve<1/100$. Suppose that 
 $\delta$ and $\eta$ are small enough positive constants (depending only on $\ve$).
Let $R\in\DD^{db}$ be a doubling cell with $\ell(R)\leq \delta$ such that 
$$
\mu(R\setminus F)\leq\eta\,\mu(R),\qquad\quad
\mu(\lambda B_R\setminus F)\leq \eta\,\mu(\lambda B_R) \quad\mbox{for all $2<\lambda\leq\delta^{-1}$},$$
and
$$
\mu\bigl(\delta^{-1}B_R\cap F\setminus \GG(R,\delta,\eta)\bigr)
\leq \eta\,\mu(R\cap F).
$$
Then there exists an AD-regular curve $\Gamma_R$ (with the AD-regularity constant bounded
by some absolute constant) and a family of pairwise disjoint cells $\sss(R)\subset \DD(R)
\setminus \{R\}$ such that, denoting by
$\tree(R)$ the subfamily of the cells from $\DD(R)$ which are not strictly contained in any cell
from $\sss(R)$,
the following holds:
\vv
\begin{itemize}

\item[(a)] $\mu$-almost all $F\cap R\setminus \bigcup_{Q\in\sss(R)}Q$ is contained in 
$\Gamma_R$, and moreover $\mu|_{F\cap R\setminus \bigcup_{Q\in\sss(R)}Q}$ is absolutely continuous with
respect to $\HH^1|_{\Gamma_R}$. 
\vv

\item[(b)] For all $Q\in\tree(R)$, $\Theta(1.1B_Q)\leq A\,\Theta_\mu(1.1B_R)$, where $A\geq 100$ is
some absolute constant.\vv

\item[(c)] The cells from $\sss(R)$ satisfy
\begin{align*}
\sum_{Q\in\sss(R)}\Theta_\mu(1.1B_Q)^2\,\mu(Q)& \leq\ve\,\Theta_\mu(B_R)^2\,\mu(R) \\
&\quad + 
c(\ve)\sum_{Q\in\tree(R)} \int_{F\cap\delta^{-1}B_Q}\int_{\delta\ell(Q)}^{\delta^{-1}\ell(Q)}
\Delta_\mu(x,r)^2\,\frac{dr}r\,d\mu(x).
\end{align*}
\end{itemize}
\end{mainlemma*}
\vvv

Notice that the curve $\Gamma_R$ mentioned in the Main Lemma is not the limit in the Hausdorff distance
of the curves $\Gamma^k$, but the limit of the curves $\Gamma_R^k$ which are described in Remark \ref{remgr}.
On the other hand, the statement in (b) is a consequence of Lemma \ref{claf22}, possibly after adjusting
the constant $A$ suitably.

The inequality in (c) follows from Lemmas \ref{lempoctot} and \ref{lemhdpetit}. Indeed, recall that
Lemma \ref{lempoctot} asserts that, for $\eta$ and $\delta$ are small enough, 
\begin{align*}
\sum_{\substack{Q\in\DD(R):\\
Q\subset \BCF\cup\LD\cup\BCG\cup\BSD}} 
 \Theta_\mu(1.1B_Q)^2\,\mu(Q) &\lesssim A^2\,(\eta^{1/5}+
 \tau^{1/4} +\delta^{1/2})
  \,\Theta_\mu(B_R)^2\,\mu(R)\\
 &\quad + \frac{A^2}{\eta^3}
\sum_{Q\in\tree} \int_{\delta^{-1}B_Q\cap F}\int_{\delta^{5}\,\ell(Q)}^{\delta^{-1}
\ell(Q)}\Delta_\mu(x,r)^2\,\frac{dr}r\,d\mu(x),
\end{align*}
while from Lemma \ref{lemhdpetit} we deduce that
\begin{align*}
\sum_{Q\in\HD\cap\DD(R)}\Theta_\mu(Q)^2\,\mu(Q) & \lesssim A^2
\left(\frac{c(A,\tau ,K )}M + 
c(A,\tau,K,M) \,\ve_0^{1/10}\right)\,\Theta_\mu(B_R)^2\,\mu(R)\\
& \leq\left(\frac{c'(A,\tau ,K )}M + 
c'(A,\tau,K,M) \,\ve_0^{1/10}\right)\,\Theta_\mu(B_R)^2\,\mu(R).
\end{align*}
So choosing $M$ big enough and $\ve_0$ (and thus $\eta$ and $\delta$) small enough, the inequality in (c)
follows, replacing $\delta$ by $\delta^5$, say. Recall that the choice of the parameters
$\eta,\delta,\tau,A,K,M$ has been explained in \rf{eqconstants*} and \rf{eqconstants*2}.

\vvv

\section{Proof of Theorem \ref{teocauchy}: boundedness of $T_\mu$ implies boundedness of the Cauchy transform
}\label{secauchy1}

For the reader's convenience, we state again Theorem \ref{teocauchy}:

\begin{theorem*}
Let $\mu$ be a finite Radon measure in $\C$ satisfying the linear growth condition
\begin{equation}\label{eqlingro11}
\mu(B(x,r))\leq c\,r\qquad\mbox{for all $x\in\C$ and all $r>0$.}
\end{equation}
 The Cauchy transform $\CC_\mu$ is bounded in $L^2(\mu)$ if and only if
\begin{equation}\label{eqdfhal44}
\int_{x\in Q}\int_0^\infty\left|\frac{\mu(Q\cap B(x,r))}{r} - \frac{\mu(Q\cap B(x,2r))}{2r}\right|^2\,\frac{dr}r\,d\mu(x)\leq
c\,\mu(Q)
\qquad\mbox{\!\!\!for every square $Q\subset\C$.}
\end{equation}
\end{theorem*}

\vv

Given $f\in L^1_{loc}(\mu)$, we denote
$$T_\mu f(x) = \left(\int_0^\infty\left|\frac{(f\mu)(B(x,r))}{r} - \frac{(f\mu)(B(x,2r))}{2r}\right|^2\,\frac{dr}r
\right)^{1/2},$$
where $(f\mu)(A)=\int_Af\,d\mu$, and we write $T\mu(x)= T_\mu1(x)$.
In this way, the condition \rf{eqdfhal44} states that
$$\|T_\mu \chi_Q\|_{L^2(\mu|_Q)}^2\leq c\,\mu(Q)\qquad\mbox{\!\!\!for every square $Q\subset\C$.}$$


In this section we will prove that if $\mu$ has linear growth and
\begin{equation}\label{eqhyp933}
\|T_\mu \chi_Q\|_{L^2(\mu|_Q)}\leq c\,\mu(Q)^{1/2}\qquad \mbox{for every square $Q\subset \C$,}
\end{equation}
 then $\CC_\mu$ is bounded in $L^2(\mu)$.
To prove this, we will use the relationship between the Cauchy transform of $\mu$ and
its curvature 
\begin{equation}\label{curvvv}
c^2(\mu) = \iiint\frac1{R(x,y,z)^2}\,d\mu(x)\,d\mu(y)\,d\mu(z),
\end{equation}
where $R(x,y,z)$ stands for the radius of the circumference passing through $x,y,z$.
The notion of
curvature of a measure was introduced by Melnikov \cite{Melnikov} while studying a discrete version of analytic capacity.
Because of its relationship to the Cauchy transform on the one hand and to $1$-rectifiability on the other hand (see \cite{Tolsa-llibre}, for example),
curvature of measures has played a key role in the solution of some old problems regarding analytic capacity, such us Vitushkin's conjecture by David \cite{David-vitus} and
the semiadditivity of analytic capacity by Tolsa \cite{Tolsa-sem}. 

If in the integral in \rf{curvvv} one integrates over $\{(x,y,z)\in\C^3:
|x-y|>\ve,|y-z|>\ve,|x-z|>\ve\}$, one gets 
the $\ve$-truncated curvature $c_\ve^2(\mu)$.
The following result is due to Melnikov and Verdera \cite{MV}.

\begin{propo} \label{propocurv}
Let $\mu$ be a finite Radon measure on $\C$ with $c_0$-linear growth. That is, $\mu$ satisfies
\rf{eqlingro11} with $c=c_0$.
For all $\ve>0$, we have
\begin{equation}\label{eqmv}
\|\CE\mu\|_{L^2(\mu)}^2 = \frac{1}{6} c^2_\ve(\mu) +
O(\mu(\C)),
\end{equation}
with
$$|O(\mu(\C))|\leq c'\,c_0^2\,\mu(\C),$$ where $c'$
is some absolute constant.
\end{propo}

Another important tool to show that the condition \rf{eqhyp933} implies the $L^2(\mu)$ boundedness of $\CC_\mu$ is the so called non-homogeneous $T1$ theorem, which in the particular case of the Cauchy transform reads as follows.

\begin{theorem}\label{teot1cauchy}
Let $\mu$ be a Radon measure on $\C$ with linear growth. The Cauchy transform $\CC_\mu$ is bounded in
$L^2(\mu)$ if and only if 
for all $\ve>0$ and all the squares $Q\subset \C$,
$$\|\CC_{\mu,\ve} \chi_Q\|_{L^2(\mu\rest Q)} \leq c\,\mu(Q)^{1/2},$$
with $c$ independent of $\ve$.
\end{theorem}

See Theorem 3.5 of \cite{Tolsa-llibre} for the proof, for example.

By Proposition \ref{propocurv} and Theorem \ref{teot1cauchy}, to prove that  \rf{eqhyp933}
implies the $L^2(\mu)$ boundedness of $\CC_\mu$,
 it suffices to show that for any measure $\mu$ with linear growth
$$c^2(\mu|_Q)\leq C\,\mu(Q) + C\,\|T_\mu \chi_Q\|_{L^2(\mu|_Q)}^2\qquad \mbox{for every square $Q\subset \C$}.$$
Clearly, this is equivalent to proving the following.

\vv

\begin{theorem}\label{temcauchy1}
Let $\mu$ be a compactly supported Radon measure on $\C$ with linear growth. Then we have
$$c^2(\mu)\leq C\,\|\mu\| + C\iint_0^\infty\left|\frac{\mu(B(x,r))}{r} - \frac{\mu(B(x,2r))}{2r}\right|^2\,\frac{dr}r\,d\mu(x).$$
\end{theorem}

\vv
To obtain the preceding result we will construct a suitable corona type decomposition of $\mu$ by using the following variant of the Main Lemma \ref{mainlemma}:

\begin{lemma}\label{mainlemma2}
Let $\mu$ be a compactly supported Radon measure on $\C$.
Let $0<\ve<1/00$. Suppose that 
 $\delta$ and $\eta$ are small enough positive constants (depending only on $\ve$).
Let $R\in\DD^{db}$ be a doubling cell such that 
$$
\mu\bigl(\delta^{-1}B_R\setminus \GG(R,\delta,\eta)\bigr)
\leq \eta\,\mu(R).
$$
Then there exists an AD-regular curve $\Gamma_R$ (with the AD-regularity constant bounded above
by some absolute constant) and a family of pairwise disjoint cells $\sss(R)\subset \DD(R)
\setminus \{R\}$ such that, denoting by
$\tree(R)$ the subfamily of the cells from $\DD(R)$ which are not strictly contained in any cell
from $\sss(R)$,
the following holds:
\vv
\begin{itemize}

\item[(a)] $\mu$-almost all $R\setminus \bigcup_{Q\in\sss(R)}Q$ is contained in 
$\Gamma_R$, and moreover $\mu|_{R\setminus \bigcup_{Q\in\sss(R)}Q}$ is absolutely continuous with
respect to $\HH^1|_{\Gamma_R}$. 

\item[(b)] There exists an absolute constant $c$ such that every $Q\in\sss(R)$ satisfies $cB_Q\cap \Gamma_R\neq\varnothing$.

\vv

\item[(c)] For all $Q\in\tree(R)$, $\Theta(1.1B_Q)\leq A\,\Theta_\mu(1.1B_R)$, where $A\geq 100$ is
some absolute constant.\vv

\item[(d)] The cells from $\sss(R)$ satisfy
\begin{align*}
\sum_{Q\in\sss(R)}\Theta_\mu(1.1B_Q)^2\,\mu(Q)& \leq\ve\,\Theta_\mu(B_R)^2\,\mu(R) \\
&\quad + 
c(\ve)\sum_{Q\in\tree(R)} \int_{\delta^{-1}B_Q}\int_{\delta\ell(Q)}^{\delta^{-1}\ell(Q)}
\Delta_\mu(x,r)^2\,\frac{dr}r\,d\mu(x).
\end{align*}
\end{itemize}
\end{lemma}

\vv
Recall that given a cell $Q\in\DD$, 
we denoted by $\GG(Q,\delta,\eta)$ the set of points 
$x\in\C$ such that
$$
\int_{\delta\,\ell(Q)}^{\delta^{-1}\ell(Q)} \Delta_\mu(x,r)^2\,\frac{dr}r\leq 
\eta\,\Theta_\mu(2B_Q)^2.
$$

Basically, Lemma \ref{mainlemma2} corresponds to the Main Lemma \ref{mainlemma} in the particular case
when $F=\supp\mu$.
Further, in (b) we stated the fact that every $Q\in\sss(R)$ satisfies $cB_Q\cap \Gamma_R\neq\varnothing$, which comes for free from the construction of the curve $\Gamma_R$ in Section \ref{sec88},
recalling that given $Q\in\sss(R)$, if $c$ is big enough, then the ball $cB_Q$ contains some cell $Q'\in\good$
which in turn contains either some cell from the family $\reg$ with positive $\wt\mu$ measure or
some point from $\supp\wt\mu\cap \Gamma_R$, by Lemma \ref{lemk12} (recall also Remark \ref{remgr}).
Moreover, unlike in the Main Lemma \ref{mainlemma}, above we do not ask $\ell(R)\leq\delta$. Indeed, this
assumption was present in the Main Lemma only because we cared about the truncated square function
$$\left(\int_0^1\left|\frac{\mu(B(x,r))}{r} - \frac{\mu(B(x,2r))}{2r}\right|^2\,\frac{dr}r
\right)^{1/2}.$$

Let $R_0\in\DD$ be a cell which contains $\supp\mu$ with $\ell(R_0)\approx\diam(\supp\mu)$.
Consider the family of cells $\ttt$ constructed in 
 Section \ref{secprovam} (with $F=\supp\mu$ and $\BZ_1=\BZ_2=\varnothing$ now). 
Recall that this is a family of doubling cells (i.e., $\ttt\subset\DD^{db}$) 
contained in $R_0$ and that $R_0\in\ttt$. 

 Given a cell $Q\in\ttt$, we let 
 $\eend(Q)$ be
the subfamily of the cells
$P\in \ttt$ satisfying
\begin{itemize}
\item $P\subsetneq Q$,
\item $P$ is maximal, in the sense that there does not exist
another cell $P'\in \ttt$ such that $P\subset P'\subsetneq Q$.
\end{itemize}
In fact, it turns out that $\eend(Q)$ coincides with the family $\MD(Q)$ from Section 
\ref{secprovam}. 

Also, we denote by $\tr(Q)$ (the tree associated with $Q$) the family of cells $\DD$ which are contained in
$Q$ and are not contained in any cell from $\eend(Q)$.
The set of good points for $Q$ is
$$G(Q):= Q\setminus \bigcup_{P\in\eend(Q)}P.$$
Further, given two cells $Q,R\in\DD$ with $Q\subset R$, we set
$$\delta_\mu(Q,R) := \int_{2B_R\setminus Q} \frac1{|y-z_Q|}\,d\mu(y),$$
where $z_Q$ stands for the center of $Q$.

We have:

\begin{lemma}[The corona decomposition] \label{lemcorona}
Let $\mu$ be a compactly supported measure on $\C$. 
 The family $\ttt\subset\DD^{db}$ constructed above 
satisfies the following.
For each cell $Q\in \ttt$ there exists an
 AD regular curve $\Gamma_Q$ (with the AD-regularity constant uniformly bounded above by some absolute constant) such that:
\begin{itemize}
\item[(a)] $\mu$ almost all $G(Q)$ is contained in $\Gamma_Q$.

\item[(b)] For each $P\in \eend(Q)$ there exists some cell
$\wt{P}$ containing $P$ such that $\delta_\mu(P,\wt{P})\leq
C\Theta_\mu(Q)$ and $B_{\wt{P}}\cap \Gamma_Q\neq \varnothing$.

\item[(c)] If $P\in\tr(Q)$, then $\Theta_\mu(1.1B_P)\leq
C\,\Theta_\mu(B_Q).$
\end{itemize}
Further, the following packing condition holds:
 \begin{equation} \label{pack}
\sum_{Q\in \ttt} \Theta_\mu(B_Q)^2 \mu(Q) \leq C\,\Theta_\mu(B_{R_0})^2\,\mu(R_0) + C\iint_0^\infty\left|\frac{\mu(B(x,r))}{r} - \frac{\mu(B(x,2r))}{2r}\right|^2\,\frac{dr}r\,d\mu(x).
\end{equation}
\end{lemma}

\vv
The preceding lemma follows immediately from Lemmas \ref{mainlemma2} and \ref{lemkey62}.
Let us remark that the property (b) in Lemma \ref{lemcorona} is a consequence of the property (b) of
Lemma \ref{mainlemma}, the construction of the family $\MD(Q)=\eend(Q)$, and Lemma \ref{lemcad23}.

The corona decomposition of Lemma \ref{lemcorona} is a variant of the one in 
\cite[Main Lemma 3.1]{Tolsa-bilip}. In the latter reference, the corona decomposition is stated in terms of the usual
dyadic squares of $\C$ instead of the dyadic cells of David-Mattila, and the left hand side of
\rf{pack} is estimated in terms of the curvature of $\mu$, instead of the $L^2(\mu)$ norm of the square integral
 $T\mu$.

We have now the following.

\begin{lemma}\label{lemcorona2}
Let $\mu$ be a compactly supported measure on $\C$
such that
$\mu(B(x,r))\leq c_0\,r$ for all $x\in\C, \,r>0$.
Suppose that there exists a family $\ttt\subset\DD^{db}$ such that $\ttt$ contains a cell $R_0$ such that
$\supp\mu\subset R_0$, and so that for each cell $Q\in \ttt$ there exists an
 AD regular curve $\Gamma_Q$ (with the AD-regularity constant uniformly bounded by some absolute constant) such that the properties (a), (b) and (c) of Lemma \ref{lemcorona} hold (with the set $G(Q)$ and the families $\eend(Q)$, $\tr(Q)$ defined in terms of the family $\ttt$ as above). Then,
  \begin{equation} \label{pack2}
 c^2(\mu) \leq c\,\sum_{Q\in \ttt} \Theta_\mu(B_Q)^2 \mu(Q).
 \end{equation}
\end{lemma}

The proof of this lemma is very similar to the one of 
Main Lemma 8.1 of \cite{Tolsa-bilip}, where this is proved to hold for bilipschitz images of the corona
decomposition of \cite[Main Lemma 3.1]{Tolsa-bilip}.
We will skip the details.

Clearly, Theorem \ref{temcauchy1} follows from Lemmas \ref{lemcorona} and \ref{lemcorona2}. Indeed, 
by \rf{pack2} and \rf{pack} we have
$$ c^2(\mu) \leq c\,\sum_{Q\in \ttt} \Theta_\mu(B_Q)^2 \mu(Q)
\leq C\,\|\mu\| + C\iint_0^\infty\left|\frac{\mu(B(x,r))}{r} - \frac{\mu(B(x,2r))}{2r}\right|^2\,\frac{dr}r\,d\mu(x).$$
\vvv


\section{Some Calder\'on-Zygmund theory for $T_\mu$}
\label{secauchy2}

Before proving that the boundedness of $\CC_\mu$ in $L^2(\mu)$ implies the $L^2(\mu)$
boundedness of $T_\mu$, we need to show that some typical results from Calder\'on-Zygmund theory also
hold for the operator $T_\mu$. Since the kernel of $T_\mu$ is not smooth, the results available in the literature (of which I am aware) are not suitable for $T_\mu$.

For more generality, we consider the $n$-dimensional version of $T_\mu$:
$$T_\mu^n f(x) = \left(\int_0^\infty\left|\frac{(f\mu)(B(x,r))}{r^n} - \frac{(f\mu)(B(x,2r))}{(2r)^n}\right|^2\,\frac{dr}r
\right)^{1/2}.$$
For any arbitrary signed Radon measure $\nu$ on $\R^d$, we also write
$$T\nu^n f(x) = \left(\int_0^\infty\left|\frac{\nu(B(x,r))}{r^n} - \frac{\nu(B(x,2r))}{(2r)^n}\right|^2\,\frac{dr}r
\right)^{1/2}.$$

Relying on some previous results from \cite{CGLT}, the following result was obtained in \cite[Theorem 4.1]{TT}:

\begin{theorem}\label{teott}
Let $\mu$ be a uniformly $n$-rectifiable measure in $\R^d$. Then $T_\mu^n$ is bounded in $L^2(\mu)$.
\end{theorem}

We will prove that the following also holds.

\begin{propo}\label{propo175}
Let $\mu$ be an $n$-AD-regular measure in $\R^d$. If $T_\mu^n$ is bounded in $L^2(\mu)$, then
$T_\mu^n$ is also bounded in $L^p(\mu)$ for $1<p<\infty$.
\end{propo}

In \cite[Theorem 5.1]{TT} it was shown that $T^n$ is bounded from the space of measures $M(\R^d)$ to $L^{1,\infty}(\mu)$ whenever
$\mu$ in a uniformly  $n$-rectifiable measure in $\R^d$. However, 
 an examination of the proof of this result in \cite{TT} shows that what is really shown is that 
if $\mu$ in an $n$-AD-regular measure such that $T_\mu^n$ is bounded in $L^2(\mu)$, then $T^n$ is bounded from $M(\R^d)$ to $L^{1,\infty}(\mu)$. By interpolation, this yields the Proposition \ref{propo175} in the case $1<p<\infty$.
Next we consider the case $p>2$.

\begin{proof}
 To show $L^p(\mu)$ the boundedness of $T_\mu^n$ for
$2<p<\infty$, by interpolation again, it is enough to show that $T_\mu$ is bounded from $L^\infty(\mu)$
to $BMO(\mu)$. The arguments to prove this are rather standard.

Consider $f\in L^\infty(\mu)$ and let $Q$ be some cell of the dyadic lattice $\DD$
associated with $\mu$. We have to show that, for some constant $c_Q$,
$$\frac1{\mu(Q)} \int_Q |T_\mu^n f - c_Q|\,d\mu \leq c\,\|f\|_{L^\infty(\mu)}.$$
Set $f_1=f\,\chi_{4B_Q}$ and $f_2=f - f_1$. Since $T_\mu^n$ is sublinear and positive we have
$|T_\mu^n(f_1+f_2)(x)-T_\mu^nf_2(x)|\leq T_\mu^n f_1(x)$. Thus,
$$|T_\mu^n f(x) - c_Q| \leq 
|T_\mu^n (f_1+f_2)(x) - T_\mu^n f_2(x)| + |T_\mu^n f_2(x) - c_Q| \leq 
T_\mu^n f_1(x) + |T_\mu^n f_2(x) - c_Q|.$$
Hence,
\begin{equation}\label{eqdk825}
\int_Q |T_\mu^n f - c_Q|\,d\mu \leq \int_Q T_\mu^n f_1\,d\mu + \int_Q |T_\mu^n f_2(x) - c_Q|\,d\mu.
\end{equation}
The first term on the right hand side is estimated by using Cauchy-Schwarz inequality and the $L^2(\mu)$ boundedness of $T_\mu$:
$$\int_Q T_\mu^n f_1\,d\mu \leq \|T_\mu^n f_1\|_{L^2(\mu)}\,\mu(Q)^{1/2} \leq 
 c\,\|f_1\|_{L^2(\mu)}\,\mu(Q)^{1/2} \leq c\,\|f\|_{L^\infty(\mu)}\,\mu(Q).$$

To deal with the last integral on the right hand side of \rf{eqdk825} we choose $c_Q=T_\mu^n f_2(z_Q)$, where
$z_Q$ stands for the center of $Q$. To show that
$$\int_Q |T_\mu^n f_2(x) - c_Q|\,d\mu\leq c\,\|f\|_{L^\infty(\mu)}\,\mu(Q)$$
and finish the proof of the proposition
it is enough to show that 
\begin{equation}\label{eqff632}
|T_\mu^n f_2(x) - T_\mu^n f_2(z_Q)|\leq c\,\|f\|_{L^\infty(\mu)} \quad\mbox{ for $x\in Q$.}
\end{equation}
To this end, write
\begin{align}\label{eqsjk318}
|T_\mu^n f_2(x) - T_\mu^n f_2(z_Q)|& = \left|\left(\int_0^\infty \Delta_{f_2\,\mu}(x,r)^2\,\frac{dr}r\right)^{1/2} - \left(
\int_0^\infty \Delta_{f_2\,\mu}(z_Q,r)^2\,\frac{dr}r\right)^{1/2}\right|\\
& \leq 
\left(\int_0^\infty |\Delta_{f_2\,\mu}(x,r) - \Delta_{f_2\,\mu}(z_Q,r)|^2\,\frac{dr}r\right)^{1/2}\nonumber\\
& \lesssim
\left(\int_0^\infty \bigl|(|f_2|\,\mu)(A(z_Q,r-r(B_Q), r+r(B_Q))\bigr|^2\,\frac{dr}{r^{2n+1}}\right)^{1/2}\nonumber\\
&\quad+\left(\int_0^\infty \bigl|(|f_2|\,\mu)(A(z_Q,2r-r(B_Q), 2r+r(B_Q))\bigr|^2\,\frac{dr}{r^{2n+1}}\right)^{1/2}\nonumber\\
& \leq \|f\|_{L^\infty(\mu)} \left(\int_{r\geq 3r(B_Q)} \bigl|\mu(A(z_Q,r-r(B_Q), r+r(B_Q))\bigr|^2\,
\frac{dr}{r^{2n+1}}\right)^{1/2},\nonumber
\end{align}
using that $f_2=f\,\chi_{\R^d\setminus 4B_Q}$ for the last inequality. 
To estimate the last integral, first we take into account that $\mu(A(z_Q,r-r(B_Q), r+r(B_Q))\lesssim\,|r+r(B_Q)|^n
\lesssim r^n$ for $r\geq 3r(B_Q)$ and $x\in Q$ and then we use Fubini:
\begin{align}\label{eqsjk3189}
\int_{r\geq 3r(B_Q)} \bigl|\mu(A(z_Q,r-r(B_Q), r+r(B_Q))\bigr|^2\frac{dr}{r^{2n+1}} & \lesssim\!
\int_{r\geq 3r(B_Q)}\! \mu(A(z_Q,r-r(B_Q), r+r(B_Q))\frac{dr}{r^{n+1}}\\ 
& \lesssim \int_{|x-z_Q|\geq 2r(B_Q)}
\int_{|x-z_Q|-r(B_Q)}^{|x-z_Q|+r(B_Q)} \frac{dr}{r^{n+1}}\,d\mu(x).\nonumber
\end{align}
In the last double integral, for $x$ and $r$ in the domain of integration we have $r\approx|x-z_Q|$, and 
so we get
\begin{align*}
\int_{|x-z_Q|\geq 2r(B_Q)}
\int_{|x-z_Q|-r(B_Q)}^{|x-z_Q|+r(B_Q)} \frac{dr}{r^{n+1}} \,d\mu(x)& \approx
\int_{|x-z_Q|\geq 2r(B_Q)} \frac1{|x-z_Q|^{n+1}}
\int_{|x-z_Q|-r(B_Q)}^{|x-z_Q|+r(B_Q)} dr \,d\mu(x)\\
& \approx \int_{|x-z_Q|\geq 2r(B_Q)} \frac{r(B_Q)}{|x-z_Q|^{n+1}} \,d\mu(x)\lesssim 1.
\end{align*}
Gathering \rf{eqsjk318}, \rf{eqsjk3189}, and the last inequality, \rf{eqdk825} follows and we are done.
\end{proof}

\vv

In the next proposition we intend to prove a Cotlar type inequality involving the operators
$$T_{\mu,\ell}^n f(x) = \left(\int_{r>\ell} \Delta_{f\mu}(x,r)^2\,\frac{dr}r\right)^{1/2},$$
$$M_{\mu,\ell} f(x) = \sup_{r> \ell} \frac1{\mu(B(x,2r))} \int_{B(x,r)}|f|\,d\mu\quad\text{and,}\quad
M^n_{\mu,\ell} f(x) = \sup_{r>\ell} \frac1{r^n}\int_{B(x,r)}|f|\,d\mu,$$
where $\ell$ is some non-negative constant.

\vv

\begin{propo}\label{propocotlar}
Let $\mu$ be a doubling measure in $\R^d$. That is
$$\mu(B(y,2r))\leq c_{db}\,\mu(B(x,r))\quad \mbox{for all $y\in\supp\mu$ and all $r>0$.}$$
Let $f\in L^p(\mu)$, for some $1\leq p <\infty$. For all $x\in\R^d$ and all $\ell\geq0$ we have
\begin{equation}\label{eqcotlar}
T_{\mu,\ell}^n f(x) \lesssim_{c_{db}}\,M_{\mu,\ell} (T_{\mu,\ell}f)(x) + M^n_{\mu,\ell} f(x).
\end{equation}
\end{propo}

Notice that the preceding statement is for all $x\in\R^d$, not only for $x\in\supp\mu$.

\begin{proof}
Take $x\in\R^d$ and let $t=\max(\ell,\dist(x,\supp\mu))$.
It is straightforward to check that
\begin{equation}\label{eqclaim8427}
T_{\mu,\ell}^n f(x) \leq T_{\mu,5t}^n f(x)  + c\,M^n_{\mu,\ell} f(x).
\end{equation}
We claim now that for all $y\in B(x,2t)$, 
\begin{equation}\label{eqclaim8428}
|T_{\mu,5t}^n f(x) - T_{\mu,5t}^n f(y)|\leq c\,M^n_{\mu,\ell} f(x).
\end{equation}
To see that \rf{eqcotlar} follows from the preceding claim, just take the mean over the ball $B(x,2t)$
of the inequality \rf{eqclaim8428} to get
\begin{align*}
T_{\mu,5t}^n f(x) & \leq \frac1{\mu(B(x,2t))}\,\int_{B(x,2t)} T_{\mu,5t}^n f(y)\,d\mu(y) + c\,M^n_{\mu,\ell} f(x)\\
& \lesssim_{c_{db}} \frac1{\mu(B(x,4t))}\,\int_{B(x,2t)} T_{\mu,5t}^n f(y)\,d\mu(y) + c\,M^n_{\mu,\ell} f(x),
\end{align*}
where we took into account that $\mu(B(x,2t))\approx \mu(B(x,4t))$, since there exists $x'\in\supp\mu$ 
such that $|x-x'|=t$ and $\mu$ is doubling.
Together with \rf{eqclaim8427} and the fact that $T_{\mu,5t}^n f(y)\leq T_{\mu,\ell}^n f(y)$, this yields
the inequality \rf{eqcotlar}.

We turn our attention to \rf{eqclaim8428}. Some of the estimates will be similar to the ones in the previous
proposition in connection with the boundedness from $L^\infty(\mu)$ to $BMO(\mu)$.
We write
\begin{align}\label{eqfig42}
|T_{\mu,5t}^n f(x) - T_{\mu,5t}^n f(y)| & \leq \left(\int_{r>5t} 
\bigl|\Delta_{f\,\mu}(x,r)-\Delta_{f\,\mu}(y,r)\bigr|^2\,\frac{dr}r\right)^{1/2}\\
& \leq 
\left(\int_{r> 5t} |\Delta_{f\,\mu}(x,r) - \Delta_{f\,\mu}(y,r)|^2\,\frac{dr}r\right)^{1/2}\nonumber\\
& \lesssim 
\left(\int_{r> 5t}  \bigl|(|f|\,\mu)(A(x,r-2t, r+2t)\bigr|^2\,\frac{dr}{r^{2n+1}}\right)^{1/2}\nonumber\\
&\quad +
\left(\int_{r> 5t}  \bigl|(|f|\,\mu)(A(x,2r-2t, 2r+2t)\bigr|^2\,\frac{dr}{r^{2n+1}}\right)^{1/2}\nonumber\\
&\lesssim 
\left(\int_{r> 5t}  \bigl|(|f|\,\mu)(A(x,r-2t, r+2t)\bigr|^2\,\frac{dr}{r^{2n+1}}\right)^{1/2}.\nonumber
\end{align}
To estimate the last integral we use the fact that for $x$ and $r$ in the domain of integration
$$\frac{(|f|\,\mu)(A(x,r-2t, r+2t)}{r^n}\leq \frac1{r^n}\,\int_{B(x,2r)}|f|\,d\mu\leq c\,M_{\mu,\ell}^nf(x)$$
and Fubini:
\begin{align*}
\int_{r\geq 5t}  \bigl|(|f|\,\mu)(A(x,r-2t, r+2t)\bigr|^2\,\frac{dr}{r^{2n+1}}&\lesssim
M_{\mu,\ell}^n f(x) \int_{r> 5t}  (|f|\,\mu)(A(x,r-2t, r+2t)\,\frac{dr}{r^{n+1}}\\
& \leq M_{\mu,\ell}^n f(x) \int_{|x-y|\geq 3t} |f(y)|\int_{|x-y|-2t}^{|x-y|+2t} \frac{dr}{r^{n+1}}\,d\mu(y)\\
&\lesssim M_{\mu,\ell}^n f(x) \int_{|x-y|> 3t} \frac{|f(y)|}{|x-y|^{n+1}}\int_{|x-y|-2t}^{|x-y|+2t} dr\,d\mu(y)\\
&\lesssim M_{\mu,\ell}^n f(x) \int_{|x-y|> 3t} \frac{t\,|f(y)|}{|x-y|^{n+1}}\,d\mu(y) \,\lesssim \,M_{\mu,\ell}^n f(x)^2.
\end{align*}
Plugging this estimate into \rf{eqfig42} yields the claim \rf{eqclaim8428}.
\end{proof}
\vv

If $\ell$ is not a constant but a function depending on $x$,  abusing notation, we also write:
$$T_{\sigma,\ell}^n f(x) := T_{\sigma,\ell(x)}^n f(x),\quad M_{\sigma,\ell} f(x) = M_{\sigma,\ell(x)} f(x),\quad
M^n_{\sigma,\ell} f(x) := M^n_{\sigma,\ell(x)} f(x).$$

\vv

\begin{propo}\label{propolplp}
Let $\Gamma$ be an $n$-AD-regular set in $\R^d$, and for a given $a>0$ set $\sigma=a\,\HH^n|_\Gamma$.
Suppose that $T_{\HH^n|_\Gamma}^n$ is bounded in $L^2(\HH^n|_\Gamma)$.
Let $\mu$ be a Radon measure in $\R^d$ and $\ell:\supp\mu\to[0,\infty)$ some function.
Suppose that $\mu(B(x,r))\leq a\,r^n$ for all $x\in\supp\mu$ and all 
$r\geq \ell(x)$.
Then $T_{\sigma,\ell}^n:L^p(\sigma) \to L^p(\mu)$ is bounded for $1<p<\infty$ with norm not exceeding
$c_{18}\,a$, with $c_{18}$ depending only on $p$, the $n$-AD-regularity constant of $\Gamma$, and the $L^2(\HH^n|_\Gamma)$
norm of $T_{\HH^n|_\Gamma}^n$.
\end{propo}

\begin{proof}
By Proposition \ref{propo175}, $T_{\HH^n|_\Gamma}^n$ is bounded in $L^p(\HH^n|_\Gamma)$, and so we deduce that
$T_\sigma^n$ is bounded in $L^p(\sigma)$ with norm not exceeding 
$c_{19}\,a$, with $c_{19}$ depending only on $p$, the AD-regularity constant of $\Gamma$,  and the $L^2(\HH^n|_\Gamma)$
norm of $T_{\HH^n|_\Gamma}^n$.
By \rf{eqcotlar}, for all $x\in\supp\mu$ we have
$$
T_{\sigma,\ell(x)}^n f(x) \lesssim_{c_{db}}\,M_{\sigma,\ell(x)} (T_{\sigma,\ell(x)}^nf)(x) + M^n_{\sigma,\ell(x)} f(x)
\leq M_{\sigma,\ell(x)} (T_{\sigma}f)(x) + M^n_{\sigma,\ell(x)} f(x).$$
That is,
\begin{equation}\label{eqcot22}
T_{\sigma,\ell}^n f(x) \lesssim_{c_{db}} M_{\sigma,\ell} (T_{\sigma}f)(x) + M^n_{\sigma,\ell} f(x).
\qquad \mbox{for all $x\in\supp\mu$}.
\end{equation}
Note that the doubling constant $c_{db}$ of $\sigma$ depends on the AD-regularity constant of $\HH^n|_{\Gamma}$ but not on $a$.

By \rf{eqcot22}, to prove the proposition it is enough to show that $M_{\sigma,\ell}$
and $M_{\sigma,\ell}^n$
are bounded from $L^p(\sigma)$ to $L^p(\mu)$ with
$$\|M_{\sigma,\ell}\|_{L^p(\sigma)\to L^p(\mu)}\leq c \quad\text{and}\quad \|M_{\sigma,\ell}^n\|_{L^p(\sigma)\to L^p(\mu)}\leq c\,a.$$
The arguments to show this are very standard. For completeness, we will show the details.

Concerning $M_{\sigma,\ell}$, it is clear that it is bounded from $L^\infty(\sigma)$ to $L^\infty(\mu)$.
 Also, it is bounded from $L^1(\mu)$ to $L^{1,\infty}(\sigma)$. Indeed, given $\lambda>0$ and $f\in L^1(\sigma)$, denote
 $$\Omega_\lambda = \{x:M_{\sigma,\ell}f(x)>\lambda\}.$$
 For each $x\in\Omega_\lambda\cap \supp\mu$, consider a ball $B(x,r_x)$ with $r_x\geq \ell(x)$ and
 $B(x,r_x)\cap\supp\sigma\neq\varnothing$
 such that
 $$\frac1{\sigma(B(x,2r_x))}\int_{B(x,r_x)}|f|\,d\sigma>\lambda.$$
 Consider a Besicovitch covering of $\Omega_\lambda\cap \supp\mu$ with balls $B(x_i,r_{x_i})$ with finite overlap, with 
 $x_i\in\Omega_\lambda\cap\supp\mu$. Then we have
\begin{align*}
 \mu(\Omega_\lambda) & \leq \sum_i \mu(B(x_i,r_{x_i})) \leq a\,\sum_i r_{x_i}^n\\
 & \leq c\,\sum_i \sigma(B(x_i,2r_{x_i}))
 \leq c\sum_i \int_{B(x_i,r_{x_i})}|f|\,d\sigma\leq c\,\|f\|_{L^1(\sigma)}.
\end{align*}
Above we took into account that $a \,r_{x_i}^n\lesssim \sigma(B(x_i,2r_{x_i}))$, which follows from 
the fact that $B(x_i,r_{x_i})\cap\supp\sigma\neq\varnothing$. So $M_{\sigma,\ell}$ is bounded from 
$L^1(\sigma)$ to $L^{1,\infty}(\mu)$. Together with the trivial boundedness from $L^\infty(\sigma)$ to $L^\infty(\mu)$,  by interpolation this shows that $M_{\sigma,\ell}$ is bounded from $L^p(\sigma)$ to
$L^p(\mu)$ for $1<p<\infty$.

On the other hand, regarding $M_{\sigma,\ell}^n$, note that if $x\in\supp\mu$ and
$B(x,r)\cap \supp\sigma\neq\varnothing$, with $r\geq \ell(x)$, then 
$$\frac1{r^n}\int_{B(x,r)} |f|\,d\sigma \leq 
\frac{c\,a}{\sigma(B(x,2r))}\int_{B(x,r)} |f|\,d\sigma \leq c\,a\,M_{\sigma,\ell}f(x).$$
Taking the supremum over the radii $r\geq\ell(x)$ such that $B(x,r)\cap \supp\sigma\neq\varnothing$, we infer 
that $M_{\sigma,\ell}^nf(x)\leq c\,a\,M_{\sigma,\ell}f(x)$, and thus
$M_{\sigma,\ell}^n$ is bounded from $L^p(\sigma)$ to $L^p(\mu)$ with its
norm not exceeding $c(p)\,a$.
\end{proof}

\vv


\section{Proof of Theorem \ref{teocauchy}: boundedness of the Cauchy transform implies boundedness of $T_\mu$
}\label{secauchy3}

In this section we will show that if 
 $\mu$ has linear growth and the Cauchy transform $\CC_\mu$ is bounded in $L^2(\mu)$, then
$$\|T_\mu \chi_Q\|_{L^2(\mu|_Q)}\leq c\,\mu(Q)^{1/2}$$
for every square $Q\subset \C$.
Because of the connection between the Cauchy kernel and curvature, the preceding result is an immediate 
corollary of the following.
\vv

\begin{theorem}\label{teocauchy2}
Let $\mu$ be a finite Radon measure on $\C$ with linear growth. Then we have
$$\iint_0^\infty\left|\frac{\mu(B(x,r))}{r} - \frac{\mu(B(x,2r))}{2r}\right|^2\,\frac{dr}r\,d\mu(x)  \leq C\,\|\mu\| + C\,c^2(\mu).$$
\end{theorem}

To prove this theorem we will use the corona decomposition of \cite{Tolsa-bilip}.
To state the precise result we need, first we will introduce some terminology which is very similar to the
one of Section \ref{sec14}. The most relevant differences are that it involves the the curvature of $\mu$ instead of the squared $L^2(\mu)$ norm of $T\mu$ and that it is stated in terms of 
the usual dyadic lattice $\DD(\C)$ instead
of the David-Mattila lattice $\DD$. 

Let $\mu$ be a finite Radon measure, and assume that there exists a dyadic square $R_0\in\DD(\C)$
such that $\supp\mu\subset R_0$ with $\ell(R_0)\leq 10\,\diam(\supp(\mu))$, say. 
Let $\ttt_{*}\subset
\DD(\C)$ be
a family of dyadic squares contained in $R_0$, with $R_0\in\ttt_{*}$.

 Given $Q\in\ttt_{*}$, we denote by
 $\eend_{*}(Q)$ 
the subfamily of the squares
$P\in \ttt_*$ satisfying
\begin{itemize}
\item $P\subsetneq Q$,
\item $P$ is maximal, in the sense that there does not exist
another square $P'\in \ttt_*$ such that $P\subset P'\subsetneq Q$.
\end{itemize}
 Also, we denote by $\tr_*(Q)$ the family of squares $\DD(\C)$ which intersect $\supp\mu$, are contained in
$Q$, and are not contained in any square from $\eend_*(Q)$. 
Notice that 
$$\{P\in\DD(\C):P\subset R_0,\,P\cap\supp\mu\neq\varnothing\} = \bigcup_{Q\in \ttt_*} \tr_*(Q).$$
We set
$$G_*(Q):= Q\cap\supp(\mu)\setminus \bigcup_{P\in\eend_*(Q)}P.$$
Given a square $Q\subset\C$, we denote
$$\Theta_\mu(Q)= \frac{\mu(Q)}{\ell(Q)},$$
and given two squares $Q\subset R$, we set
$$\delta_{*,\mu}(Q,R) := \int_{2R\setminus Q} \frac1{|y-z_Q|}\,d\mu(y),$$
where $z_Q$ stands for the center of $Q$. 

We have:

\begin{lemma}[The dyadic corona decomposition of \cite{Tolsa-bilip}] \label{lemcorona3}
Let $\mu$ be a Radon measure on $\C$ with linear growth and finite curvature $c^2(\mu)$. Suppose that
there exists a dyadic square $R_0\in\DD(\C)$
such that $\supp\mu\subset R_0$ with $\ell(R_0)\leq 10\,\diam(\supp(\mu))$. 
Then there exists a family $\ttt_*$ as above which satisfies the following. 
For each square $Q\in \ttt_*$ there exists an
 AD-regular curve $\Gamma_Q$ (with the AD-regularity constant uniformly bounded by some absolute constant) such that:
\begin{itemize}
\item[(a)] $\mu$ almost all $G_*(Q)$ is contained in $\Gamma_Q$.

\item[(b)] For each $P\in \eend_*(Q)$ there exists some square
$\wt{P}\in\DD(\C)$ containing $P$, concentric with $P$, such that $\delta_{*,\mu}(P,\wt{P})\leq
C\Theta_\mu(7Q)$ and $\frac12{\wt{P}}\cap \Gamma_Q\neq \varnothing$.

\item[(c)] If $P\in\tr_*(Q)$, then $\Theta_\mu(7P)\leq
C\,\Theta_\mu(7Q).$
\end{itemize}
Further, the following packing condition holds:
 \begin{equation} \label{pack3}
\sum_{Q\in \ttt_*} \Theta_\mu(7Q)^2 \mu(Q) \leq C\,\|\mu\| + C\,c^2(\mu).
\end{equation}
\end{lemma}

\vv

Let us remark that the squares from the family $\ttt_*$ may be non-doubling.

The preceding lemma is not stated explicitly in \cite{Tolsa-bilip}. However it follows 
immediately from the Main Lemma 3.1 of \cite{Tolsa-bilip}, just by splitting the so called $4$-dyadic squares in 
 \cite[Lemma 3.1]{Tolsa-bilip} into dyadic squares. Further, the family $\ttt_*$ above is the same as the family
 $\ttt_{\rm dy}$ from \cite[Section 8.2]{Tolsa-bilip}.
 
Quite likely, by arguments analogous to the ones used 
to prove Lemma 3.1 of \cite{Tolsa-bilip} (or the variant stated in Lemma \ref{lemcorona3} of the present text), one can prove an analogous result in terms of cells from the dyadic lattice of David and Mattila. 
This would read exactly as Lemma \ref{lemcorona}, but one should replace the inequality \rf{pack} by 
the following:
$$
\sum_{Q\in \ttt} \Theta_\mu(B_Q)^2 \mu(Q) \leq C\,\Theta_\mu(B_{R_0})^2\,\mu(R_0) + C\,c^2(\mu).
$$
Perhaps this would simplify some of the technical difficulties arising from the lack of a well
adapted dyadic lattice to the measure $\mu$ in \cite{Tolsa-bilip}.
However, proving this would take us too long and so this is out of the reach of this paper.

To prove Theorem \ref{teocauchy2}, we split $T\mu(x)$ as follows. 
Given $Q\in\DD(\C)$, we denote
$$T_Q \mu(x)^2 = \chi_Q(x)\int_{\ell(Q)/2}^{\ell(Q)}\Delta_\mu(x,r)^2\,\frac{dr}r.$$
For $x\in\supp\mu$, then we have
\begin{equation}\label{eqd9459}
T\mu(x)^2 =\sum_{Q\in\DD(\C)} T_Q\mu(x)^2 = \sum_{R\in\ttt_*}\sum_{Q\in\tr_*(R)} T_Q\mu(x)^2 + \sum_{Q\in\DD(\C):Q\not\subset R_0}
T_Q\mu(x)^2.
\end{equation}
The last sum is easy to estimate:

\vv
\begin{lemma}\label{lemfac99}
We have
$$\sum_{Q\in\DD(\C):Q\not\subset R_0}
\|T_Q\mu\|_{L^2(\mu)}^2\leq c\,\Theta_\mu(R_0)^2\,\|\mu\|.$$
\end{lemma}

\begin{proof}
Since $\supp\mu\subset R_0$, we have
\begin{align*}
\sum_{Q\in\DD(\C):Q\not\subset R_0} \|T_Q\mu\|_{L^2(\mu)}^2 & = \sum_{Q\in\DD(\C):Q\supset R_0} 
\|T_Q\mu\|_{L^2(\mu)}^2 \\
& = \int_{R_0}\int_{\ell(R_0)}^\infty\Delta_\mu(x,r)^2\,\frac{dr}r\\
&\lesssim \mu(R_0)^2\int_{R_0}\int_{\ell(R_0)}^\infty\,\frac{dr}{r^3}\approx \Theta_\mu(R_0)^2\,\mu(R_0).
\end{align*}

\end{proof}

To deal with the first term on the right hand side of \rf{eqd9459} we need
 a couple of auxiliary results from \cite{Tolsa-bilip}. The first one is the following.
\vv

\begin{lemma}\label{lemreg*}
Let $\ttt_*$ be as in Lemma \ref{lemcorona3}.
For each $R\in\ttt_*$ there exists a family of dyadic squares $\reg_*(R)$ which satisfies the following
properties:
\begin{itemize}
\item[(a)] The squares from $\reg_*(R)$ are contained in $Q$ and are pairwise disjoint.

\item[(b)] Every square from $\reg_*(R)$ is contained in some square from $\eend_*(R)$.

\item[(c)] $\bigcup_{Q\in\reg_*(R)}2Q \subset \C\setminus G_*(R)$ and
$\supp\mu\cap R\setminus\bigcup_{Q\in\reg_*(R)}Q\subset G_*(R)\subset \Gamma_R$.

\item[(d)] If $P,Q\in \reg_*(R)$ and $2P\cap 2Q\neq \varnothing$, then $\ell(Q)/2 \leq \ell(P) \leq 2 \ell(Q)$.
 
\item[(e)] If $Q\in\reg_*(R)$ and $x\in Q$, $r\geq \ell(Q)$, then $\mu(B(x,r)\cap
4R) \leq C\Theta_\mu(7R)\,r.$

\item[(f)] For each $Q\in \reg_*(R)$, there exists some square $\wt{Q}$, concentric with $Q$,
which contains $Q$, such that $\delta_{*,\mu}(Q,\wt{Q})\leq
C\Theta_{\mu}(7R)$ and $\frac12\wt{Q}\cap\Gamma_R\neq
\varnothing$.
\end{itemize}
\end{lemma}

This result is proved in Lemmas 8.2 and 8.3 of \cite{Tolsa-bilip}. For the reader's convenience, let us
say that this follows by a regularization procedure analogous to the one used in the present paper to construct
the families $\reg$ and $\nreg$.

The next lemma shows how, in a sense, the measure $\mu$  can be approximated on a tree $\tr_*(R)$ by another measure supported on 
$\Gamma_R$ which is absolutely continuous with respect to the arc length measure. This is proved in Lemma 
8.4 of \cite{Tolsa-bilip}.

\begin{lemma} \label{repart}
For $R\in\ttt_*$, denote
$\reg_*(R)=:\{P_i\}_{i\geq1}$. For each
$i$, let $\wt{P}_i\in\DD(\C)$ be a closed square containing $P_i$ such that
$\delta_{*,\mu}(P_i,\wt{P}_i)\leq C\Theta_\sigma(7R)$ and $\frac12 \wt{P}_i\cap \Gamma_{R}\neq
\varnothing$ (as in (e) of Lemma \ref{lemreg*}).
For each $i\geq1$ there exists some function $g_i\geq0$ supported
on $\Gamma_R\cap \wt{P}_i$ such that
\begin{equation} \label{co1}
\int_{\Gamma_R} g_i\,d\HH^1 = \mu(P_i),
\end{equation}
\begin{equation} \label{co2}
\sum_i g_i \lesssim \Theta_\mu(7R),
\end{equation}
and
\begin{equation} \label{co3}
\|g_i\|_\infty \,\ell(\wt{P}_i) \lesssim\mu(P_i).
\end{equation}
\end{lemma}

Recalling the splitting in \rf{eqd9459},
to prove Theorem \ref{teocauchy2}, it suffices to show that for every $R\in\ttt_*$
$$\sum_{Q\in\tr_*(R)}\int T_Q\mu^2\,d\mu \leq c\,\Theta_\mu(7R)^2\,\mu(R),$$
because of the packing condition \rf{pack3}.
To this end, denote
$$S_R\mu(x)^2
= \chi_R(x)\int_{\ell(x)/2}^{\ell(R)}\Delta_\mu(x,r)^2\,\frac{dr}r,$$
where $\ell(x)=\ell(Q)$ if $x\in Q\in\reg_*$ and $\ell(x)=0$ if $x\not\in\bigcup_{Q\in\reg_*}Q$.
By (b) of Lemma \ref{lemreg*},
$$\sum_{Q\in\tr_*(R)} T_Q\mu(x)^2 \leq S_R\mu(x)^2.$$
Thus the proof of Theorem \ref{teocauchy2} will be concluded after proving the next result.
\vv

\begin{lemma}\label{lemultim}
For every $R\in\ttt_*$, we have
$$\|S_R\mu\|_{L^2(\mu)}^2\leq c\,\Theta_\mu(7R)^2\,\mu(R).$$
\end{lemma}

\begin{proof}
Consider the measure $\sigma =\Theta_\mu(7R)\,\HH^1|_{\Gamma_R}$, and 
take the functions $g_i$, $i\geq0$, from Lemma \ref{repart}. Set
$$\beta_i = \mu|_{P_i} - g_i\,\HH^1|_{\Gamma_R},$$
and denote $h_i= \Theta_\mu(7R)^{-1}g_i$, so that $g_i\,\HH^1|_{\Gamma_R} = h_i\,\sigma$.
Denote also  $h= \sum_i h_i$ and notice that 
$$\mu = \sum_i \beta_i + h\, \sigma.$$
As $S_R$ is subadditive, we have
\begin{equation}\label{eqspl88}
S_R\mu\leq S_R(h\,\sigma) + \sum_i S_R\beta_i.
\end{equation}

Since $\mu(B(x,r)\cap R)\lesssim \Theta_\mu(7R)\,r$ for all $x\in R\cap\supp\mu$ and all $r\geq \ell(x)$, by 
Theorem \ref{teott} and Proposition \ref{propolplp}
$T_{\sigma,\ell}:L^2(\sigma) \to L^2(\mu|_R)$ is bounded with norm not exceeding
$c\,\Theta_\mu(7R)$. So we get
$$\|S_R(h\,\sigma)\|_{L^2(\mu)}^2\leq \|T_{\sigma,\ell} h\|_{L^2(\mu|_R)}^2\leq c\,
\Theta_\mu(7R)^2\,\|h\|_{L^2(\sigma)}^2.$$
To estimate $\|h\|_{L^2(\sigma)}^2$, write
$$\|h\|_{L^2(\sigma)}^2\leq \|h\|_{L^\infty(\sigma)}\,\|h\|_{L^1(\sigma)},$$
and recall that
$$\|h\|_{L^\infty(\sigma)}= \Theta_\mu(7R)^{-1}\|g\|_{L^\infty(\sigma)}\lesssim 1$$
and
$$\|h\|_{L^1(\sigma)} = \|g\|_{L^1(\HH^1|_{\Gamma_R})}\leq \mu(R).$$
Hence we obtain
\begin{equation}\label{eqsr421}
\|S_R(h\,\sigma)\|_{L^2(\mu)}^2\lesssim\Theta_\mu(7R)^2\,\mu(R).
\end{equation}

Now we will estimate the term $\sum_i S_R\beta_i$ from \rf{eqspl88}. 
We split $S_R\beta_i(x)$ as follows:
\begin{align}\label{eqabc12}
S_R\beta_i(x)&\leq  \left(\chi_R(x)\int_{\ell(x)}^{4\ell(\wt P_i)} \Delta_{\beta_i}(x,r)^2\,\frac{dr}r\right)^{1/2} + \left(\chi_R(x)\int_{4\ell(\wt P_i)}^{\ell(R)} \Delta_{\beta_i}(x,r)^2\,\frac{dr}r\right)^{1/2}
\\
&\leq 
\left(\chi_R(x)\int_{\ell(x)}^{\ell(P_i)/8} \Delta_{\mu|_{P_i}}(x,r)^2\,\frac{dr}r\right)^{1/2}
+
\left(\chi_R(x)\int_{\ell(P_i)/8}^{4\ell(\wt P_i)} \Delta_{\mu|_{P_i}}(x,r)^2\,\frac{dr}r\right)^{1/2} 
\nonumber\\
&\quad +
\left(\chi_R(x)\int_{\ell(x)}^{4\ell(\wt P_i)} \Delta_{h_i\,\sigma}(x,r)^2\,\frac{dr}r\right)^{1/2}
+
\left(\chi_R(x)\int_{4\ell(\wt P_i)}^{\ell(R)} \Delta_{\beta_i}(x,r)^2\,\frac{dr}r\right)^{1/2}\nonumber\\
&= 
A_i(x) + B_i(x) + C_i(x) + D_i(x).\nonumber
\end{align}

To deal with $A_i(x)$, note if $x\not\in 2P_i$, then $A_i(x)$ vanishes. Recall now that, by Lemma \ref{lemreg*},
if $x\in2P_i$, then $\ell(x)\approx\ell(P_i)$. So we obtain
$$A_i(x) \leq \chi_{2P_i}(x)
\left(\int_{c\ell(P_i)}^{\ell(P_i)/8} \Delta_{\mu|_{P_i}}(x,r)^2\,\frac{dr}r\right)^{1/2} \lesssim 
\chi_{2P_i}(x)\,\frac{\mu(P_i)}{\ell(P_i)}.$$

Let us turn our attention to the term $B_i(x)$ from \rf{eqabc12}.
Notice that $B_i(x) =0$ if $x\not\in 20\wt P_i$, and moreover
$\Delta_{\mu|_{P_i}}(x,r) = 0$ if $2r<\dist(x,P_i)$. Hence, in the domain of integration
of $B_i(x)$ we can assume both that $r\geq \ell(P_i)/8$ and that $r\geq \frac12\,\dist(x,P_i)$, which imply that
$$r\geq \frac12\,\Bigl(\frac18\,\ell(P_i) + \frac12\,\dist(x,P_i)\Bigr) \approx |x-z_{P_i}|+ \ell(P_i).$$
So we have
\begin{align*}
B_i(x)^2 & \leq \chi_{20\wt P_i}(x)\int_{c(|x-z_{P_i}|+ \ell(P_i))}^{4\ell(\wt P_i)} \Delta_{\mu|_{P_i}}(x,r)^2\,
\frac{dr}r \\
& \lesssim \chi_{20\wt P_i}(x)\,\mu(P_i)^2\int_{c(|x-z_{P_i}|+ \ell(P_i))}^{4\ell(\wt P_i)} 
\frac{dr}{r^3} \lesssim \chi_{20\wt P_i}(x)\,\frac{\mu(P_i)^2}{\bigl(|x-z_{P_i}|+ \ell(P_i)\bigr)^2}.
\end{align*}
Thus,
$$B_i(x)\lesssim \chi_{20\wt P_i}(x)\,\frac{\mu(P_i)}{|x-z_{P_i}|+ \ell(P_i)}.$$

Concerning $C_i(x)$, again it is easy to check that $C_i(x) =0$ if $x\not\in 20\wt P_i$. So we have
$$C_i(x) \leq \chi_{20\wt P_i} T_{\sigma,\ell} h_i(x).$$

Next we consider the term $D_i(x)$ from \rf{eqabc12}. Since $\int d\beta_i=0$ and $\supp\beta_i\subset \wt P_i$, it turns out that
$\Delta_{\beta_i}(x,r)=0$ unless 
$$\bigl(\partial B(x,r)\cup \partial B(x,2r)\bigr)\cap \wt P_i \neq\varnothing.$$
If this condition holds, we write $r\in I(i,x)$.
This condition, together with the fact that  $r\geq 4\,\ell(\wt P_i)$ in the domain of integration of the integral that defines $C_i(x)$, implies that
$r\approx |x-z_{P_i}|\approx |x-z_{P_i}| + \ell(\wt P_i)$. Therefore,
$$\int_{4\ell(\wt P_i)}^{\ell(R)} \Delta_{\beta_i}(x,r)^2\,\frac{dr}r
\lesssim \frac{\|\beta_i\|^2}{\bigl(|x-z_{P_i}|+\ell(\wt P_i)\bigr)^3} \int_{r\in I(i,x)}dr\leq 
\frac{\mu(P_i)^2\,\ell(\wt P_i)}{\bigl(|x-z_{P_i}|+\ell(\wt P_i)\bigr)^3}.$$
Hence,
$$D_i(x) \lesssim \frac{\mu(P_i)\,\ell(\wt P_i)^{1/2}}{\bigl(|x-z_{P_i}|+\ell(\wt P_i)\bigr)^{3/2}}.$$

Gathering the estimates we have obtained for $A_i(x)$, $B_i(x)$, $C_i(x)$, and $D_i(x)$, we get
\begin{align*}
S_R\beta_i(x)&\lesssim \chi_{2P_i}(x)\,\frac{\mu(P_i)}{\ell(P_i)} + 
\chi_{20\wt P_i}(x)\,\frac{\mu(P_i)}{|x-z_{P_i}|+ \ell(P_i)}
+ \frac{\mu(P_i)\,\ell(\wt P_i)^{1/2}}{\bigl(|x-z_{P_i}|+\ell(\wt P_i)\bigr)^{3/2}}
+ \chi_{20\wt P_i} T_\sigma h_i(x)\\
& \lesssim \left[
\chi_{20\wt P_i}(x)\,\frac{\mu(P_i)}{|x-z_{P_i}|+ \ell(P_i)}
+ \frac{\mu(P_i)\,\ell(\wt P_i)^{1/2}}{\bigl(|x-z_{P_i}|+\ell(\wt P_i)\bigr)^{3/2}}\right]
+ \chi_{20\wt P_i} T_\sigma h_i(x)\\
& =: E_i(x)  
+ \chi_{20\wt P_i} T_\sigma h_i(x).
\end{align*}

We will estimate the $L^2(\mu|_R)$ norm of $\sum_i S_R\beta_i$ by duality. So consider a
non-negative function $f\in L^2(\mu|_R)$ and write
\begin{equation}\label{eqakk893}
\int_R f\,\sum_i S_R\beta_i\,d\mu \lesssim \sum_i \int f\,E_i\,d\mu + 
\sum_i \int_{20\wt P_i} f\,T_{\sigma,\ell} h_i \,d\mu =: \circled{1} + \circled{2}.
\end{equation}

First we deal with \circled{1}. We consider the centered maximal Hardy-Littlewood operator
$$M_{\mu|_R}^c f(x) = \sup_{r>0} \frac1{\mu(B(x,r)\cap R)}\int_{B(x,r)\cap R}|f|\,d\mu.$$
It is easy well known that $M_{\mu|_R}^c$ is
bounded in $L^p(\mu|_R)$, $1<p\leq\infty$, and of weak type $(1,1)$ with respect to $\mu|_R$. 

For each $i$ we have
$$\int f\,E_i\,d\mu = \int_{20\wt P_i} \frac{\mu(P_i)}{|x-z_{P_i}|+ \ell(P_i)}\,f(x)\,d\mu(x) +
\int \frac{\mu(P_i)\,\ell(\wt P_i)^{1/2}}{\bigl(|x-z_{P_i}|+\ell(\wt P_i)\bigr)^{3/2}}\,f(x)\,d\mu(x).
$$
We claim now that the following holds:
\begin{equation}\label{eqcla41*}
\int_{20\wt P_i} \frac{1}{|x-z_{P_i}|+ \ell(P_i)}\,f(x)\,d\mu(x) \lesssim \Theta_\mu(7R)\,\inf_{y\in P_i}M_{\mu|_R}^c f(y)
\end{equation}
and
\begin{equation}\label{eqcla42*}
\int \frac{\ell(\wt P_i)^{1/2}}{\bigl(|x-z_{P_i}|+\ell(\wt P_i)\bigr)^{3/2}}\,f(x)\,d\mu(x)
\lesssim \Theta_\mu(7R)\,\inf_{y\in P_i}M_{\mu|_R}^c f(y).
\end{equation}
Assuming these estimates for the moment, we deduce
$$\int f\,E_i\,d\mu \lesssim \Theta_\mu(7R)\,\inf_{y\in P_i} M_{\mu|_R}^c f(y)\,\mu(P_i)\leq
\Theta_\mu(7R)\,\int_{P_i} M_{\mu|_R}^c f(y)\,d\mu(y),$$
and then, since the squares $P_i$ are pairwise disjoint and contained in $R$,
\begin{align*}
\circled{1}& \lesssim \Theta_\mu(7R)\sum_i \int_{P_i} M_{\mu|_R}^c f(y)\,d\mu(y)
\leq \Theta_\mu(7R) \int_R M_{\mu|_R}^c f\,d\mu \\
& \leq \Theta_\mu(7R)\,\|M_{\mu|_R}^c f\|_{L^2(\mu|_R)}\,\mu(R)^{1/2} \leq \Theta_\mu(7R)\,\|f\|_{L^2(\mu)}\,\mu(R)^{1/2}.
\end{align*}

Next we estimate the term \circled{2} from \rf{eqakk893}. By H\"older's inequality and the $L^p(\mu)$
boundedness of $T_{\sigma,\ell}$ from $L^p(\sigma)$ to $L^p(\mu|_R)$ with norm not exceeding $c\,\Theta_\mu(7R)$, we get
\begin{align*}
\circled{2} & \leq 
\sum_i \left(\int_{20\wt P_i\cap R} |T_{\sigma,\ell} h_i|^4 \,d\mu\right)^{1/4}
\left(\int_{20\wt P_i} |f|^{4/3} \,d\mu\right)^{3/4}\\
& \lesssim \Theta_\mu(7R) \sum_i \| h_i\|_{L^4(\sigma)}
\left(\int_{20\wt P_i\cap R} |f|^{4/3} \,d\mu\right)^{3/4}.
\end{align*}
Consider the following centered maximal operator
$$M_{\mu|_R}^{c,4/3} f(x) =\sup_{r>0} \left(\frac1{\mu(B(x,r)\cap R)} \int_{B(x,r)\cap R}|f|^{4/3}\,d\mu\right)^{3/4}.$$
This is bounded in $L^p(\mu|_R)$ for $4/3<p\leq\infty$ and of weak type $(4/3,4/3)$ with respect
to $\mu|_R$.
Notice that for all $y\in P_i$
\begin{align*}
\left(\int_{20\wt P_i\cap R} |f|^{4/3} \,d\mu\right)^{3/4} & \leq 
\left(\int_{B(y,\ell(40\wt P_i))\cap R} |f|^{4/3} \,d\mu\right)^{3/4} \\
& \leq 
\mu(B(y,\ell(40\wt P_i))\cap R)^{3/4}\,M_{\mu|_R}^{c,4/3} f(y)\\
& \lesssim
\Theta_\mu(7R)^{3/4}\,\ell(\wt P_i)^{3/4}\,M_{\mu|_R}^{c,4/3} f(y).
\end{align*}
Therefore,
\begin{equation}\label{eqfi949}
\circled{2}  \lesssim \Theta_\mu(7R)^{1+3/4} \sum_i \| h_i\|_{L^4(\sigma)}\,
\ell(\wt P_i)^{3/4}\,\inf_{y\in P_i} M_{\mu|_R}^{c,4/3} f(y).
\end{equation}
Recalling that $h_i= \Theta_\mu(7R)^{-1}\,g_i$ and that $\|g_i\|_\infty \,\ell(\wt{P}_i) \lesssim\mu(P_i)$, we get
\begin{align*}
\|h_i\|_{L^4(\sigma)} \,\ell(\wt P_i)^{3/4} & \leq \|h_i\|_{L^\infty(\sigma)} \,
\sigma(\wt P_i)^{1/4}\,\ell(\wt P_i)^{3/4} \\
& \lesssim \Theta_\mu(7R)^{-1+1/4} \|g_i\|_{L^\infty(\sigma)}\,
\ell(\wt P_i)
\lesssim \Theta_\mu(7R)^{-3/4} \,\mu(P_i).
\end{align*}
Plugging this estimate into \rf{eqfi949} and using the $L^2(\mu)$ boundedness of $M_{\mu|_R}^{c,4/3}$,
we obtain
\begin{align*}
\circled{2} & \lesssim \Theta_\mu(7R) \sum_i \mu(P_i)\,\inf_{y\in P_i} M_{\mu|_R}^{c,4/3} f(y)\leq
\Theta_\mu(7R) \int_R M_{\mu|_R}^{c,4/3} f(y)\,d\mu(y) \\
& \leq 
\Theta_\mu(7R) \,\|M_{\mu|_R}^{c,4/3} f\|_{L^2(\mu)}\,\mu(R)^{1/2}\lesssim \Theta_\mu(7R)\, \|f\|_{L^2(\mu)}\,\mu(R)^{1/2}.
\end{align*}

Gathering the estimates obtained for \circled{1} and \circled{2}, we get
$$\int_R f\,\sum_i S_R\beta_i\,d\mu \lesssim \Theta_\mu(7R)\, \|f\|_{L^2(\mu)}\,\mu(R)^{1/2}$$
for any non-negative function $f\in L^2(\mu)$,
which implies that 
$$\Bigl\|\sum_i S_R\beta_i\Bigr\|_{L^2(\mu)} \lesssim \Theta_\mu(7R)\, \mu(R)^{1/2},$$
as wished.

To conclude the proof of the lemma it just remains to prove the claims \rf{eqcla41*} and \rf{eqcla42*}.
We carry out this task in the following lemma.
\end{proof}
\vv

\begin{lemma}
Let $f\in L^2(\mu|_R)$ be non-negative. 
We have
\begin{equation}\label{eqcla41**}
\int_{20\wt P_i} \frac{1}{|x-z_{P_i}|+ \ell(P_i)}\,f(x)\,d\mu(x) \lesssim \Theta_\mu(7R)\,\inf_{y\in P_i}M_{\mu|_R}^c f(y)
\end{equation}
and
\begin{equation}\label{eqcla42**}
\int \frac{\ell(\wt P_i)^{1/2}}{\bigl(|x-z_{P_i}|+\ell(\wt P_i)\bigr)^{3/2}}\,f(x)\,d\mu(x)
\lesssim \Theta_\mu(7R)\,\inf_{y\in P_i}M_{\mu|_R}^c f(y).
\end{equation}
\end{lemma}

\begin{proof}
First we deal with the inequality  \rf{eqcla41**}. Given  a non-negative function $f\in L^2(\mu|_R)$ and $y\in P_i$, we set
\begin{align*}
\int_{20\wt P_i} \frac{1}{|x-z_{P_i}|+ \ell(P_i)}\,f(x)\,d\mu(x) &
\lesssim \frac1{\ell(P_i)} \int_{B(y,2\ell(P_i))} f(x)\,d\mu|_R(x) \\
&\quad+ 
\int_{\frac12\ell(P_i)\leq |x-z_{P_i}|\leq 40\ell(\wt P_i)} \frac{1}{|x-z_{P_i}|}\,f(x)\,d\mu|_R(x)
\\&= I_1+ I_2.
\end{align*}
Concerning $I_1$, we have
$$I_1\lesssim \frac{\mu(B(y,2\ell(P_i))\cap R)}{\ell(P_i)} \frac1{\mu(B(y,2\ell(P_i))\cap R)} 
\int_{B(y,2\ell(P_i))\cap R} f\,d\mu\lesssim \Theta_\mu(7R)\,M_{\mu|_R}^cf(y).$$
To deal with $I_2$ we apply Fubini:
\begin{align*}
I_2 & = \int_{\frac12\ell(P_i)\leq |x-z_{P_i}|\leq 40\ell(\wt P_i)}
f(x)\int_{r>|x-z_{P_i}|}\frac1{r^2}\,dr\, d\mu|_R(x)\\
& = \int_{\frac12\ell(P_i)}^\infty \frac1{r^2}\int_{|x-z_{P_i}|<\min(r,40\ell(\wt P_i))} f(x)\,d\mu|_R(x)\,dr\\
& \leq \int_{\frac12\ell(P_i)}^\infty \frac1{r^2}\int_{B(y,\min(2r,80\ell(\wt P_i)))} f\,d\mu|_R\,dr\\
& \leq M_{\mu|_R}^cf(y) \int_{\frac12\ell(P_i)}^\infty \frac{\mu\bigl(B(y,\min(2r,80\ell(\wt P_i)))\cap R\bigr)}{r^2}\,dr.
\end{align*}
Since $B(y,\min(2r,80\ell(\wt P_i)))\subset B(z_{P_i},\min(4r,160\ell(\wt P_i)))$, we get
$$I_2\leq
M_{\mu|_R}^cf(y) \int_{\frac12\ell(P_i)}^\infty \frac{\mu\bigl(B(z_{P_i},\min(4r,160\ell(\wt P_i)))\cap R\bigr)}{r^2}\,dr.$$
Notice now that, by Fubini, the integral above equals
\begin{align*}
\int_{\frac12\ell(P_i)}^\infty \frac1{r^2}\int_{|x-z_{P_i}|<\min(4r,160\ell(\wt P_i))} \,d\mu|_R(x)\,dr  & = 
\int_{\frac12\ell(P_i)\leq |x-z_{P_i}|\leq 160\ell(\wt P_i)}
\int_{r>\frac14|x-z_{P_i}|}\frac1{r^2}\,dr\, d\mu|_R(x) \\
&= \int_{\frac12\ell(P_i)\leq |x-z_{P_i}|\leq 160\ell(\wt P_i)} \frac4{|x-z_{P_i}|}\,d\mu|_R(x).
\end{align*}
From the condition $\delta_{*,\mu}(P_i,\wt P_i)\lesssim\Theta_\mu(7R)$, it follows easily that the last integral above is bounded by $c\,\Theta_\mu(7R)$.
Thus,
$$I_2\lesssim\Theta_\mu(7R)\,M_{\mu|_R}^cf(y),$$
and \rf{eqcla41**} follows.

\vv
The arguments for \rf{eqcla42**} are quite similar. Consider 
 a non-negative function $f\in L^2(\mu|_R)$ and $y\in P_i$, and write
\begin{align*}
\int \frac{\ell(\wt P_i)^{1/2}}{\bigl(|x-z_{P_i}|+\ell(\wt P_i)\bigr)^{3/2}}\,f(x)\,d\mu(x)
 &
\lesssim \frac1{\ell(\wt P_i)} \int_{B(y,2\ell(\wt P_i))} f(x)\,d\mu|_R(x) \\
&\quad+ 
\int_{|x-z_{P_i}|\geq \frac12\ell(\wt P_i)^{3/2}} \frac{\ell(\wt P_i)^{1/2}}{|x-z_{P_i}|^{3/2}}\,f(x)\,d\mu|_R(x)
\\&= J_1+ J_2.
\end{align*}
Arguing as in the case of $I_1$, we get
$$J_1\lesssim  \Theta_\mu(7R)\,M_{\mu|_R}^cf(y).$$
To deal with $J_2$ we apply Fubini again:
\begin{align*}
J_2 & = c\,\int_{|x-z_{P_i}|\geq \frac12\ell(\wt P_i)}
f(x)\int_{r>|x-z_{P_i}|}\frac{\ell(\wt P_i)^{1/2}}{r^{5/2}}\,dr\, d\mu|_R(x)\\
& = \int_{\frac12\ell(\wt P_i)}^\infty \frac{\ell(\wt P_i)^{1/2}}{r^{5/2}}\int_{|x-z_{P_i}|< r} f(x)\,d\mu|_R(x)\,dr\\
& \leq \int_{\frac12\ell(\wt P_i)}^\infty \frac{\ell(\wt P_i)^{1/2}}{r^{5/2}}\int_{B(y,2r)} f\,d\mu|_R\,dr\\
& \leq M_{\mu|_R}^cf(y) \int_{\frac12\ell(\wt P_i)}^\infty \frac{\ell(\wt P_i)^{1/2}\,\mu\bigl(B(y,2r)\cap R\bigr)}{r^{5/2}}\,dr.
\end{align*}
Since 
$$\frac{\mu\bigl(B(y,2r)\cap R\bigr)}r\lesssim \Theta_\mu(7R)\quad\mbox{for $r\geq \frac12\ell(P_i)$,}$$
we obtain
$$J_2\lesssim \Theta_\mu(7R)\,
M_{\mu|_R}^cf(y) \int_{\frac12\ell(\wt P_i)}^\infty \frac{\ell(\wt P_i)^{1/2}}{r^{3/2}}\,dr\lesssim 
\Theta_\mu(7R)\,M_{\mu|_R}^cf(y),$$
and so \rf{eqcla42**} is proved.
\end{proof}

\vvv

\end{document}